\documentclass[11pt]{book}

\usepackage{graphicx} 

\usepackage{calc}  
\usepackage{amsmath}  
\usepackage{amsthm}
\usepackage{amsfonts}
\usepackage{mathrsfs}
\usepackage{amssymb}
\usepackage{amstext}
\usepackage{amscd}
\usepackage{mathabx}
\usepackage{enumerate}
\usepackage[all]{xy}
\usepackage{color}
\usepackage{todonotes}
\usepackage{tikz-cd}
\usepackage{tikz}
\usetikzlibrary{cd}

\usetikzlibrary{arrows}
\usetikzlibrary{calc}
\usetikzlibrary{decorations.pathmorphing}
\tikzstyle{crossing}=[circle,fill=white,minimum height=6pt, inner sep=0pt, outer sep=0pt, style={transform shape=false}]

\newtheorem{Lemma}{Lemma}

\numberwithin{equation}{section}
\newtheorem{thm}{Theorem}[section]
\newtheorem{prop}[thm]{Proposition}
\newtheorem{cor}[thm]{Corollary}
\newtheorem{lem}[thm]{Lemma}

\theoremstyle{definition}

\newtheorem{ex}[thm]{Example}
\newtheorem{dfn}[thm]{Definition}
\newtheorem{rem}[thm]{Remark}

\newtheorem{conj}[thm]{Conjecture}

\newtheorem{acknow}{Acknowledgments}

\def\Z{\mathbb{Z}}

\def\R{\mathbb{R}}
\def\C{\mathbb{C}}
\def\H{\mathbb{H}}
\def\bbS{\mathbb{S}}

\def\cE{\mathcal{E}}
\def\cF{\mathcal{F}}

\def\cX{\mathcal{X}}
\def\X{\cX}
\def\cY{\mathcal{Y}}

\def\bbS{\mathbb{S}}

\def\fs{\mathfrak{s}}

\def\End{\operatorname{End}}
\def\grad{\operatorname{grad}}

\def\id{\operatorname{id}}
\def\im{\operatorname{im}}

\def\Ind{\operatorname{Ind}}
\def\ind{\operatorname{ind}}

\def\Pic{\operatorname{Pic}}

\def\rank{\operatorname{rank}}

\newcommand{\faml}{\mathcal{F}}
\newcommand{\pic}{\mathrm{Pic}}
\newcommand{\inv}{\mathrm{inv}}

\newcommand{\bigu}{\mathbf{U}} 
\newcommand{\exsp}{\mathcal{K}_{G,Z}} 
\newcommand{\hoex}{\mathrm{Ho}\; \mathcal{K}_{G,Z}} 
\newcommand{\fpsw}{PSW} 

\newcommand{\colim}{\mathrm{colim}}

\newcommand{\from}{:}
\newcommand{\spanierwhitehead}{\mathfrak{C}}

\newcommand{\preBF}{\mathcal{BF}}
\newcommand{\BF}{\mathbf{BF}}
\newcommand{\preSWF}{\mathcal{SWF}}
\newcommand{\SWFn}{\mathbf{SWF}}
\newcommand{\unfoldeda}{\underline{\mathrm{swf}}^A}
\newcommand{\unparamSWF}{\mathcal{SWF}^u}
\newcommand{\unparamSWFn}{\mathbf{SWF}^u}

\newcommand{\thom}{\mathrm{Th}}

\newenvironment{abstract}{\begin{center} {\bf Abstract} \end{center}}

\title{Seiberg-Witten Floer spectra for $b_1>0$}

\author{Hirofumi Sasahira and Matthew Stoffregen}

\date{\empty}

\begin{document}

\maketitle


\begin{abstract}
%

The Seiberg-Witten Floer spectrum is a stable homotopy refinement of the monopole Floer homology of Kronheimer and Mrowka.  The Seiberg-Witten Floer spectrum was defined by Manolescu for closed, spin$^c$ 3-manifolds with $b_1 = 0$  in  an $S^1$-equivariant stable homotopy category and has been producing interesting topological applications.  Lidman and Manolescu showed that the $S^1$-equivariant homology of the spectrum is isomorphic to the monopole Floer homology.

%
%
%

For  closed spin$^c$ 3-manifolds $Y$ with $b_1(Y) > 0$, there are analytic and homotopy theoretic   difficulties to define the Seiberg-Witten Floer spectrum.  In this memoir, we address the difficulties and construct the Seiberg-Witten Floer spectrum for $Y$,  provided that the first Chern class of the spin$^c$ structure is torsion and that the triple cup product on $H^1(Y;\Z)$ vanishes.  We conjecture that its $S^1$-equvariant homology is isomorphic to the monopole Floer homology.


For a  4-dimensional spin$^c$ cobordism $X$ between $Y_0$ and $Y_1$, we define the Bauer-Furuta map  on these new spectra of $Y_0$ and $Y_1$, which is conjecturally a refinement of the relative Seiberg-Witten invariant of $X$.  
As an application, for  a compact spin 4-manifold $X$ with boundary $Y$, we prove a $\frac{10}{8}$-type inequality for $X$ which is written in terms of the intersection form of $X$ and  an invariant $\kappa(Y)$ of $Y$.

Additionally, we compute the Seiberg-Witten Floer spectrum for some 3-manifolds.


\end{abstract}




\begin{acknow}
 We are  grateful to  Mikio Furuta, Ciprian Manolescu,  Shinichiro Matsuo and Chris Scaduto for helpful conversations, as well as Jianfeng Lin and Daniel Ruberman for useful comments on a previous version of this work.  We also  thank the organizers of the ``Gauge Theory and  Applications" conference at Regensburg, where this work was started.

We started this work with Tirasan Khandhawit and discussed together from July 2018 until July 2019. We are grateful to him for many discussions and insights. 

The first author was supported by JSPS KAKENHI Grant Numbers JP19K 03493, JP23K03115.   The second author was supported by NSF grant DMS-2203828.

\end{acknow}



\tableofcontents


\chapter{Introduction}  \label{sec:intro}

\section{Background}

The Seiberg-Witten equations \cite{witten} have been an important tool in the study of $4$-manifolds since their introduction.  Soon after these equations appeared, Kronheimer-Mrowka \cite{KM} extended their study to define the monopole Floer homology of $3$-manifolds, and established its relationship with the $4$-manifold invariant; the resulting theory has since had many applications in low-dimensional topology.

In both gauge theory and symplectic geometry, certain Floer homology theories have since been shown to arise as the homology of well-defined \emph{Floer spectra} as envisioned by Cohen, Jones and Segal \cite{CJS}, and some invariants, obtained by counting solutions of certain PDE, are now either known or conjectured to come from the degree of certain maps between spectra.  One of the first examples of such a construction is the Bauer-Furuta invariant \cite{furuta},\cite{bauer-furuta}, which associates an element in stable homotopy $\pi^{st}(S^0)$ to certain $4$-manifolds, refining the ordinary Seiberg-Witten invariant.  
Building on the finite-dimensional approximation technique introduced by Furuta,  Manolescu \cite{Manolescu-b1=0} constructed an $S^1$-equivariant stable homotopy type $\mathit{SWF}(Y,\mathfrak{s})$ associated to rational homology three-spheres with $\mathrm{spin}^c$-structure $(Y,\mathfrak{s})$.


It is a natural question to extend Manolescu's construction to $3$-manifolds with $b_1(Y)>0$.  In the case $b_1(Y)=1$, Kronheimer-Manolescu \cite{kronheimer-manolescu} constructed a \emph{periodic pro-spectrum} for pairs $(Y,\mathfrak{s})$.  Later, together with T. Khandhawit and  J. Lin, the  first author constructed the \emph{unfolded} Seiberg-Witten Floer spectrum for arbitrary closed, oriented $(Y,\mathfrak{s})$ in \cite{KLS1}, \cite{KLS2}.  

The \emph{unfolded} spectrum comes in multiple flavors.  For now, we consider only the type-$A$ unfolded invariant $\unfoldeda(Y,\mathfrak{s})$, which depends on $(Y,\mathfrak{s})$ as well as some additional data we suppress.  This invariant is a directed system in the $S^1$-equivariant stable homotopy category.  In particular, it is not \emph{per se} a spectrum, and tends to be very large. 

In \cite{KLS2}, T. Khandhawit,  J. Lin and the first author showed that the unfolded invariant allowed for gluing formulas, in a very general setting, for the calculation of the Bauer-Furuta invariant of a $4$-manifold cut along $3$-manifolds with $b_1>0$.  In particular, this enables one to prove vanishing formulas for the Bauer-Furuta invariant in many situations.

However, the invariant $\unfoldeda(Y,\mathfrak{s})$ is not expected to recover the monopole Floer homology, but is instead expected to recover a version of monopole Floer homology with fully twisted coefficients.

In the present memoir, we construct a new Seiberg-Witten Floer spectrum $\SWFn(Y,\mathfrak{s})$ for $b_1(Y)>0$, as follows.

\begin{thm}\label{thm:main} Let $(Y,\mathfrak{s})$ be a closed,  $\mathrm{spin}^c$ $3$-manifold which satisfies that the first Chern class $c_1(\mathfrak{s})\in H^2(Y;\mathbb{Z})$ is  torsion, and so that the triple-cup product on $H^1(Y;\mathbb{Z})$ vanishes.  Associated to a \emph{Floer framing} $\mathfrak{P}$ (see Section \ref{subsec:defn-of-invar} for this notation), there is a well-defined parameterized, $S^1$-equivariant stable homotopy type $\preSWF(Y,\mathfrak{s}, \mathfrak{P})$, 
over the Picard torus $\Pic(Y) = H^1(Y;\mathbb{R})/H^1(Y;\mathbb{Z})$, called the \emph{Seiberg-Witten Floer stable homotopy type} of $(Y,\mathfrak{s},\mathfrak{P})$.  Moreover, there is a well-defined (unparameterized) $S^1$-equivariant connected simple system of spectra $\SWFn^u(Y,\mathfrak{s},\mathfrak{P})$, the \emph{Seiberg-Witten Floer spectrum}.  
	
	If $\mathfrak{s}$ is self-conjugate and $\mathfrak{P}$ is a $\mathrm{Pin}(2)$-equivariant Floer framing, then the equivariant, parameterized stable homotopy type
 $\preSWF(Y,\mathfrak{s},\mathfrak{P})$ naturally comes with the structure of a parameterized $\mathrm{Pin}(2)$-equivariant stable homotopy type,   where the Picard torus has a $\mathrm{Pin}(2)$-action factoring through $\pi_0(\mathrm{Pin}(2))$ by conjugation.  Similarly, $\SWFn^u(Y,\mathfrak{s},\mathfrak{P})$ has an underlying (unparameterized) $\mathrm{Pin}(2)$-equivariant spectrum, $\SWFn^{u,\mathrm{Pin}(2)}(Y,\mathfrak{s},\mathfrak{P})$.
	
	 The homotopy type $\preSWF(Y,\mathfrak{s},\mathfrak{P})$, viewed without its parameterization, has the homotopy type of a finite $S^1$ (respectively $\mathrm{Pin}(2)$)-CW complex.  The Seiberg-Witten Floer spectrum $\SWFn^u(Y,\mathfrak{s},\mathfrak{P})$ (respectively $\SWFn^{u,\mathrm{Pin}(2)}(Y,\mathfrak{s},\mathfrak{P})$) has the homotopy type of a finite $S^1$ (respectively $\mathrm{Pin}(2)$) CW-spectrum.  
	 
	 If $b_1(Y)=0$, $\preSWF(Y,\mathfrak{s},\mathfrak{P})$ agrees with the invariant $\mathit{SWF}(Y,\mathfrak{s})$ in \cite{Manolescu-b1=0}, in that:
	 \[
	 \preSWF(Y,\mathfrak{s},\mathfrak{P})\simeq\Sigma^{n\mathbb{C}}\mathit{SWF}(Y,\mathfrak{s}),
	 \]
	 for some $n\in \mathbb{Z}$, depending only on $\mathfrak{P}$.  
\end{thm}

For the notion of parameterized spaces we use, ex-spaces, we refer to Chapter \ref{sec:homotopy}, as well as for the notion of a connected simple system.  In particular, see Definition \ref{def:sw-cat} for the notion of a parameterized equivariant stable homotopy type.

The collection of Floer framings of $(Y,\mathfrak{s})$, should any exist, is an affine space over $K(\mathrm{Pic}(Y))\cong \mathbb{Z}^{2^{b_1(Y)-1}}$.  Moreover, there is an explicit relationship between the Floer spectra constructed for different spectral sections; see Corollary \ref{cor:change-of-suspensions}.

In order to explain the context of Theorem \ref{thm:main}, and its apparent difference from the unfolded invariant, we review below the process of finite-dimensional approximation, introduced by Furuta, and used by Manolescu \cite{Manolescu-b1=0} in his construction of the three-manifold invariant for rational homology three-sphere input, as well as in \cite{kronheimer-manolescu},\cite{KLS1},\cite{KLS2}.

\section{Finite-dimensional approximation}

There are two main approaches to refining the construction of Floer-theoretic invariants from homology theories to generalized homology theories (and, in some instances, spectra).  There is the approach by constructing \emph{framed flow categories} (or variations on this type of category) as envisioned originally by \cite{CJS}.  A very general version of this has just been accomplished in \cite{Abouzaid-Blumberg} (while the present work was in its final stages of preparation).  There is also the method of finite-dimensional approximation, mentioned above, which we now summarize.

Manolescu's construction of $\mathit{SWF}(Y,\mathfrak{s})$ takes place inside the \emph{Coulomb gauge slice} of the Seiberg-Witten equations.  All that matters for this introduction is that, roughly speaking, the Coulomb slice is some Hilbert space on which the Seiberg-Witten equations admit a particularly simple form, as a linear operator plus a compact perturbation.  For certain linear subspaces of the Coulomb slice (adapted to the linear part of the Seiberg-Witten equations), Manolescu considers an approximation of the formal $L^2$-gradient flow of the Seiberg-Witten equations.  The approximations tend to stabilize as larger and larger finite-dimensional subspaces are chosen.  Associated to suitable flows on suitable topological spaces, there is a convenient invariant, the \emph{Conley index}, which is a well-defined homotopy type associated to the flow (along with some extra data).  The invariant $\mathit{SWF}(Y,\mathfrak{s})$ is then taken as the Conley index of these approximated flows.

The most pressing difficulty facing finite-dimensional approximation to other equations of gauge theory or symplectic geometry is that the configuration space in these other situations is usually not linear, so that it is not obvious which finite-dimensional submanifolds one should consider ``approximations" on.

For $b_1(Y)>0$ the gauge slice of the Seiberg-Witten equations is no longer linear, but Kronheimer-Manolescu \cite{kronheimer-manolescu}, and the authors of \cite{KLS1}, \cite{KLS2}, avoided the problem of having a more general configuration space by considering the Seiberg-Witten equations on the universal cover (which is once again a Hilbert space) of a gauge slice to the Seiberg-Witten equations, where finite-dimensional approximation is still possible, but where the usual compactness of the space of Seiberg-Witten trajectories is lost.  The loss of compactness leads to the resulting invariant $\unfoldeda$ not being a single spectrum, but rather a system of them.

The problem of performing finite-dimensional approximation in nonlinear situations has been open for some time (though see \cite{Kragh}).  In this memoir our objective is to resolve it in one (relatively simple) case, for the Seiberg-Witten equations.  We hope that this method may be useful in other situations where one would like to apply finite-dimensional approximation for topologically complicated configuration spaces.

The main work of the present memoir is showing that there exist families of submanifolds of the configuration space of the Seiberg-Witten equations (for $b_1(Y)>0$) on which the Seiberg-Witten equations can be approximated very accurately.  This comes down to carefully controlling spectral sections of the Dirac operator, in the sense of Melrose-Piazza \cite{MP}, and in particular relies on some control of spectra of Dirac operators.  Once the submanifolds are constructed, there also remains the problem of showing that the approximate Seiberg-Witten equations thereon are sufficiently accurate; for this we use a refined version of the original argument of Manolescu which requires weaker assumptions than the original, but does not yield the same strength of convergence as in Manolescu's case.

A word is also in order about the hypotheses on the input in Theorem \ref{thm:main}.  Cohen-Jones-Segal conjectured that there should exist Floer spectra for many of the familiar Floer homology theories - but only in the event that the \emph{polarization} is trivial.  The hypotheses in the Theorem are necessary for the vanishing of the polarization (indeed, a Floer framing is the same thing as a trivialization of the polarization), as observed in \cite{kronheimer-manolescu}.

However, in spite of usually having a dependence on the Floer framing, we can consider generalized homology theories applied to $\SWFn^u(Y,\mathfrak{s},\mathfrak{P})$ that are insensitive to the framing.  In the following theorem, $n(Y,\mathfrak{s},\mathfrak{P})$ is a certain numerical invariant of a Floer framing, introduced in Chapter \ref{sec:froyshov}, and $MU$ and $MU_{S^1}$ denote, respectively, complex cobordism and $S^1$-equivariant complex cobordism.  For the notion of an equivariant complex orientation, see Section \ref{subsec:elementary-properties} (and for more detail, \cite{Cole-Greenlees-Kriz}),

\begin{thm}\label{thm:monopole-complex-cobordism-introduction}
	Let $E$ be a (possibly $S^1$-equivariant) complex-oriented (resp. $S^1$- equivariantly complex oriented) cohomology theory.  Then
	\[
	E^{*-2n(Y,\mathfrak{s},\mathfrak{P})}(\SWFn^u(Y,\mathfrak{s},\mathfrak{P}))
	\]
	is (canonically) independent of $\mathfrak{P}$.
	
	In particular, the complex-cobordism theories 
	\[
	\begin{split}
	\mathit{FMU}^*(Y,\mathfrak{s})&=\widetilde{MU}^{*-2n(Y,\mathfrak{s},g,\mathfrak{P})}(\SWFn^u(Y,\mathfrak{s},\mathfrak{P})),\\
	\mathit{FMU}^*_{S^1}(Y,\mathfrak{s})&=\widetilde{MU}_{S^1}^{*-2n(Y,\mathfrak{s},g,\mathfrak{P})}(\SWFn^u(Y,\mathfrak{s},\mathfrak{P})),
	\end{split}
	\]
	are invariants of the pair $(Y,\mathfrak{s})$, which we call the \emph{Floer (equivariant) complex cobordism} of $(Y,\mathfrak{s})$.
	
\end{thm}

As $MU,MU_{S^1}$ are the universal complex-oriented cohomology theories, in some sense $\mathit{FMU}^*(Y,\mathfrak{s})$ and $\mathit{FMU}^*_{S^1}(Y,\mathfrak{s})$ might be interpreted as the universal monopole Floer-type invariants that are independent of the framing.  

More speculatively, we remark that the independence of $\mathit{FMU}^*$ on the framing suggests that its definition could be extended to pairs $(Y,\mathfrak{s})$ which do not admit a Floer framing.  We plan to pursue this in future work.

It would also be desirable to relate the (generalized) homology theories of the Seiberg-Witteh Floer spectrum  $\SWFn^u(Y,\mathfrak{s},\mathfrak{P})$ to the monopole-Floer homology of Kronheimer-Mrowka.
 In particular, we conjecture:
\begin{conj}\label{con:lidman-manolescu-generalized}
	For $(Y,\mathfrak{s})$ a pair as in Theorem \ref{thm:main}, 
 \begin{align*}
&	 H_{\bullet-2n(Y,\mathfrak{s},\mathfrak{P})}^{S^1}(\SWFn^u(Y,\mathfrak{s},\mathfrak{P}))\cong  \widecheck{\mathit{HM}}_{\bullet}(Y,\mathfrak{s}),         \\
&	 cH^{S^1}_{\bullet-2n(Y,\mathfrak{s},\mathfrak{P})}(\SWFn^u(Y,\mathfrak{s},\mathfrak{P}))\cong  \widehat{\mathit{HM}}_{\bullet}(Y,\mathfrak{s}),\\  
&	 tH_{\bullet-2n(Y,\mathfrak{s},\mathfrak{P})}^{S^1}(\SWFn^u(Y,\mathfrak{s},\mathfrak{P}))\cong \overline{\mathit{HM}}(Y,\mathfrak{s}),     \\
&	 H_{\bullet-2n(Y,\mathfrak{s},\mathfrak{P})}(\SWFn^u(Y,\mathfrak{s},\mathfrak{P}))\cong   \widetilde{\mathit{HM}}_{\bullet}(Y,\mathfrak{s}).
	 \end{align*}
	Note that ordinary homology is (equivariantly) complex-orientable, and so the homology theories on each left side are independent of the choice of spectral section (and we have been somewhat imprecise about the gradings on the right).  Here $H^{S^1}$, $cH^{S^1}$, $tH^{S^1}$ are, respectively, Borel, coBorel, and Tate homology.
\end{conj}

 This conjecture is already established by Lidman-Manolescu in the case that $Y$ is a rational-homology sphere \cite{LidmanManolescu}.

We note that there is a natural generalization of Conjecture \ref{con:lidman-manolescu-generalized} to include the case of local coefficient systems on monopole Floer homology $\mathit{HM}^\circ$; this involves using other parameterized cohomology theories (as in \cite[Section 20.3]{MS}) applied to $\mathcal{SWF}(Y,\mathfrak{s},\mathfrak{P})$.  There is also a further generalization of the conjecture to relate the $\mathrm{Pin}(2)$-equivariant cohomology of $\SWFn^u(Y,\mathfrak{s},\mathfrak{P})$, for $(Y,\mathfrak{s})$ admitting a $\mathrm{Pin}(2)$-equivariant Floer framing, to the equivariant monopole Floer homology defined by F. Lin \cite{lin-pin2}.

We remark that Theorem \ref{thm:main} should yield a well-defined connected simple system $\SWFn(Y,\mathfrak{s},\mathfrak{P})$ of equivariant, parameterized spectra.  Indeed, this would follow if the \emph{parameterized} Conley index of a dynamical system were known to be well-defined as a connected simple system (rather than as a homotopy type.  The ordinary Conley index is known \cite{Salamon} to be a connected simple system).  We hope to return to this point, and other improvements to naturality, in future work.

\section{Four-manifolds}
In this memoir we also define a Bauer-Furuta invariant associated to a $\mathrm{spin}^c$ $4$-manifold with boundary.

Let $(Y, \mathfrak{s})$ be a closed spin$^c$ 3-manifold and $\mathfrak{P}$ be a Floer framing of $(Y, \mathfrak{s})$. Recall that, in the parameterized setting,  we only define the ex-space $\preSWF(Y,\mathfrak{s},\mathfrak{P})$ up to stable homotopy equivalence.  To fix notation, define a \emph{map class} of maps $P \to Q$ between two spaces $P, Q$, themselves only well-defined up to homotopy-equivalence, to mean just a homotopy class, up to the action of self-homotopy equivalences on $P$ or $Q$. 

For an $S^1$-equivariant virtual vector bundle $V$ over a base $B$, let $S^V_B$ denote the corresponding sphere bundle over $B$.  We then construct a Bauer-Furuta invariant $\preBF$ as follows:

\begin{thm}\label{thm:bauer-furuta-main}
	Let $(X,\mathfrak{t})$ be a smooth, compact, $\mathrm{spin}^c$ four-manifold with boundary $(Y,\mathfrak{s})$, and fix a Floer framing $\mathfrak{P}$ of $(Y,\mathfrak{s})$.  Then there is a well-defined (parameterized, $S^1$-equivariant, stable) map class
	\[
	\preBF(X,\mathfrak{t}): S^{\ind(D_X,\mathfrak{P})}_{\mathrm{Pic}(Y)}\to \preSWF(Y,\mathfrak{s},\mathfrak{P}).
	\]
\end{thm}

For the definition of the index $\ind (D_X, \mathfrak{P})$, see Chapter \ref{section 4-mfd with boundary}.  There is also a version of Theorem \ref{thm:bauer-furuta-main} at the spectrum level, which is more complicated to state; see Corollary \ref{cor:bauer-furuta-main-equivalent}.

As a byproduct of our proof of well-definedness of $\preSWF(Y,\mathfrak{s},\mathfrak{P})$, we also obtain an invariant of families:

\begin{thm}\label{thm:main-families}
 Let $\mathcal{F}$ be a Floer-framed family of $\mathrm{spin}^c$ 3-manifolds, with compact base $B$ and fibers denoted by $\mathcal{F}_{b}$ for $b\in B$.  Let $\mathrm{Pic}(\mathcal{F})$ denote the bundle over $B$ with fiber $\mathrm{Pic}(\mathcal{F}_b)$. There is a well-defined parameterized, $S^1$-equivariant stable-homotopy type $\preSWF(\mathcal{F})$, which is parameterized over $\mathrm{Pic}(\mathcal{F})$.  
	\end{thm}

A similar families invariant exists for the Bauer-Furuta invariant, but we omit its discussion, as we do not have need for it in the present memoir.

As an application of our construction, we construct Fr{\o}yshov-type invariants associated to the Seiberg-Witten Floer stable homotopy type.  In particular, we define a generalization of Manolescu's $\kappa$-invariant, from $\mathrm{Pin}(2)$-equivariant K-theory of three-manifolds with $b_1(Y)=0$, to $Y$ with $b_1(Y)>0$.  We show:

\begin{thm}\label{thm 10/8 inequality-intro}
  Let $(X, \mathfrak{t})$ be a compact,  spin 4-manifold with boundary $-Y_0 \coprod Y_1$.  Assume that $Y_0$ is a rational homology $3$-sphere and the index $\operatorname{Ind} D$ for $(Y_1, \mathfrak{t}|_{Y_1})$ is zero in $KQ^1(\mathrm{Pic}(Y_1))$.  Here $KQ^1$ stands for the quaternionic K-theory. (See \cite{dupont, lin_rokhlin}.)   
		Then we have
		\[
		- \frac{\sigma(X)}{8} +  \kappa(Y_0, \mathfrak{t}|_{Y_0})  - 1
		\leq   b^+(X) + \kappa(Y_1, \mathfrak{t}|_{Y_1}).
		\]

\end{thm}

See Remark \ref{rem:duality} for the reason why we assume $b_1(Y_0)=0$ in this theorem. 

We also define invariants associated to the $S^1$-equivariant monopole Floer homology, corresponding roughly to the generalized $d$-invariants introduced by Levine-Ruberman \cite{Levine-Ruberman} in Heegaard Floer homology.

We also calculate the Seiberg-Witten-Floer homotopy-type invariant in some relatively simple situations; see Chapter \ref{sec:computation}.

\section{Further Directions}
We do not prove any gluing theorems for the Bauer-Furuta invariant, or for its families analog, and this is a natural point of departure, remaining within Seiberg-Witten theory.  In this direction, we expect the surgery exact triangles \cite[Section 42]{KM} (and variations) to hold for homology theories other than ordinary homology.  For this, it would be particularly desirable to obtain a description of the map on $\mathit{FMU}^*$ induced by the Bauer-Furuta invariant, independent of choices like the Floer framing.  It is also natural to ask how the unfolded spectrum $\unfoldeda(Y,\mathfrak{s})$ is related to the folded spectrum $\preSWF(Y,\mathfrak{s})$.  

A technical problem that may make the invariant $\preSWF(Y,\mathfrak{s},\mathfrak{P})$ more wieldy is to establish a natural topological description (on $Y$) of the set of Floer framings.  We hope to address some of these points in the future.

Furthermore, we expect that it should be possible to consider more detailed applications to the question of when a family of three-manifolds extends to a family of $4$-manifolds with boundary.  Compare with recent work of Konno-Taniguchi \cite{konno-taniguchi-rel-boundary} in the case that the boundary family of $3$-manifolds is the trivial family of a rational homology sphere.

Finally, given an extension of $\mathit{FMU}^*(Y,\mathfrak{s})$ to three-manifolds that do not admit a Floer framing, it seems likely that the excision argument of \cite{Kronheimer-Mrowka-excision} should apply, in which case we would expect there to exist generalizations of sutured monopole Floer homology to various generalized homology theories.

\section{Organization}
This memoir  is organized as follows.  We first construct special families of spectral sections to the Dirac operator in Chapter \ref{sec:findim-appx}, and show that certain subsets of the (approximate) Seiberg-Witten configuration space are isolating neighborhoods in the sense of Conley index theory.  In Chapter \ref{sec:well-def}, we show that the resulting invariant is well-defined, as a consequence of this process we establish a Seiberg-Witten Floer homotopy type for families.  This consists of showing that all of the possible choices for different approximations to the Seiberg-Witten equations are compatible.  In Chapter \ref{sec:computation} we give various example calculations of $\mathcal{SWF}(Y,\mathfrak{s},\mathfrak{P})$.  In Chapter \ref{section 4-mfd with boundary}, we construct a relative Bauer-Furuta invariant, and show that it is well-defined.  Finally, in Chapter \ref{sec:froyshov}, we establish various Fr{\o}yshov-type inequalities that are a consequence of the existence of the new relative Bauer-Furuta invariant.

There is one appendix, Chapter \ref{sec:homotopy}, on homotopy-theoretic background, as well as an afterword on potential further applications outside of Seiberg-Witten theory.


\chapter{Finite dimensional approximation on 3-manifolds}\label{sec:findim-appx}

\section{Spectral sections}  \label{section:def of spectral section}

 In order to define Seiberg-Witten Floer spectra, we will make use of spectral sections of a family of Dirac operators introduced by Melrose-Piazza \cite{MP}. 
We will  recall definitions and basic things on spectral sections in this section.

Suppose that we have a closed, oriented  $(2n-1)$-manifold $Y$ and that we have  a fiber bundle
\[
          \mathcal{Y} \rightarrow B 
\]
with fiber $Y$. Here $B$ is a compact Hausdorff space. 
Also suppose that we are given a finite dimensional vector bundle
\[
       F_{Y} \rightarrow  \cY
\] 
with  metric.  
We consider  an infinite dimensional vector bundle on $B$ defined by
\[
            \cE_{Y,   \infty} :=  \bigcup_{z \in B}  \Gamma(F_{Y}|_{\cY_{z}}).
\]
Let 
\[
     D_{Y} :   \cE_{Y, \infty} \rightarrow \cE_{Y, \infty}
\]
be a family of first order elliptic, self-adjoint differential operators. That is, $D_Y$ preserves the fibers  of $\cE_{Y, \infty}$ and for each $z \in B$, 
\[
     D_{Y, z} : \cE_{Y, \infty, z}  \rightarrow  \cE_{Y, \infty, z}
\]
is a first order, elliptic, self-adjoint  differential operator.  Here $\cE_{Y, \infty, z}$ is the fiber of $\cE_{Y, \infty}$ over $z$. 

We assume that for each $z \in B$, there is an open neighborhood $U$ of $z$ such that we have a trivialization
\begin{equation}  \label{eq F_Y U}
          F_Y|_{\cY_{U}} \cong U \times F_{Y, z}, 
\end{equation}
where $\cY_{U}$ is the restriction of  the bundle $\cY$ to $U$, and we can write
\[
         D_{Y, w} = D_{Y, z} + A_{Y, w}
\]
for $w \in U$ through the isomorphism $\cE_{Y, \infty, z} \cong \cE_{Y, \infty, w}$ induced by (\ref{eq F_Y U}).  Here $A_{Y,w}$ is the operator acting on $\cE_{Y, \infty, w}$ induced by a fiberwise linear bundle map $F_Y|_{\cY_{w}} \rightarrow F_Y|_{\cY_{w}}$ which  continuously depends on $w$.

For $k \geq 0$, define the $L^2_{k}$-inner product  on $\cE_{Y,  \infty}$ by
\[
    \langle \phi_1, \phi_2 \rangle_{k}
           = \int_{\cY_{z}}  \langle \phi_1, \phi_2 \rangle + \langle |D_{Y, z} |^k \phi_1, | D_{Y, z}|^{k} \phi_2 \rangle d\mu. 
\]
Here $|D_{Y, z}|$ denote the absolute value of  $D_{Y,z}$ defined as in \cite[Chapter VIII,  \S 9]{reed-simon}. 
We write $\cE_{Y,  k}$ for the completions with respect to the $L^2_k$-norm.  The operator $D_Y$ extends to a bounded operator
\[
       D_{Y} : \cE_{Y,  k}  \rightarrow \cE_{Y,  k-1}. 
\] 
For $w \in U$, the algebraic operator $A_{Y, w}$ extends to a bounded operator $\cE_{Y, k, w} \to \cE_{Y, k, w}$ which continuously depends on $w$ with respect to the operator norm, and $D_{Y, w} = D_{Y, z} + A_{Y, w}$ as operators $\cE_{Y, k, w} \rightarrow \cE_{Y, k-1, w}$ through the  local trivialization (\ref{eq F_Y U}).


We now recall the definition of a spectral section from \cite{MP}.

\begin{dfn}[{\cite{MP}}]   \label{def:spectral section}
    	A \emph{spectral section} for $D_{Y} : \cE_{Y,   k} \rightarrow \cE_{Y, k-1}$ over a compact base $B$ is a family of self-adjoint projections $P : \cE_{Y,  0} \rightarrow \cE_{Y,  0}$ so that there is a constant $C > 0$ such that the following holds.   Suppose that  $z\in B$, $u \in \cE_{Y, \infty, z}$, $D_{Y,z} u = \lambda u$ for some $\lambda \in \R$. Then $P_z u =u$ if $\lambda>C$ and $P_z u=0$ if $\lambda<-C$.  Here, a \emph{family} is meant to be a continuous family in the $L^2$-operator norm topology, parameterized by $B$.
\end{dfn}

We note that the condition that $P$ be  continuous families in the $L^2$-norm topology is equivalent to $P$ being  continuous families in any $L^2_k$-norm topology with $k>0$, using the interaction of $P$ with the spectrum of $D_{Y}$.  
Also note that since $P$ is self-adjoint, $P$ is an orthogonal projection onto its image with respect to the $L^2$-inner product. 
In fact, for $\phi_1, \phi_2 \in \cE_{Y, \infty, z}$, we have
\[
   \langle P \phi_1, (1-P) \phi_2 \rangle_0 =  \langle \phi_1,  P (1 -P) \phi_2 \rangle_0 = 0. 
\]
Here we have used the fact that $P$ is self-adjoint and $P^2 = P$.

 About the existence of spectral section, Melrose and Piazza proved the following.

\begin{thm} \cite[Proposition 1]{MP}
There exists a spectral section of $D_Y$ if and only if the index $\ind D_Y$ is zero in $K^1(B)$.  Here $\ind D_Y$ is the index defined in \cite{AS_skew-adjoint}. 
\end{thm}

Using a spectral  section, we can define the  Atiyah-Patodi-Singer index  for  a family of differential operators on a manifold with boundary.  Let $X$ be a compact, oriented  $2n$-manifold with boundary $Y$. Suppose that we have a fiber bundle 
\[
     \cX \rightarrow B
\]
with fiber $X$,  such that the family obtained by taking the boundary of each fiber of $\cX$ is $\cY$.     Also suppose that we have finite dimensional vector bundles
\[
          F_X^0, F^1_X \rightarrow \cX 
\]
and that isomorphisms 
\[  
     F_{X}^0|_{\cY} \cong F_{X}^1|_{\cY} \cong F_Y
\]
are given.  
Define infinite dimensional vector bundles over $B$ by 
\[
        \cE_{X, \infty}^0   = \bigcup_{z \in B} \Gamma (F_{X}^0|_{\cX_z}), \quad
        \cE_{X, \infty}^1 = \bigcup_{z \in B} \Gamma(F_{X}^1|_{\cX_z}).
\]
 We consider a family of  first order elliptic differential operators
\[
     D_X : \cE_{X, \infty}^0  \rightarrow \cE_{X, \infty}^1 
\]
such that
\[
      D_X = \frac{\partial}{\partial t} + D_Y
\]
near the boundary $\cY$. 
Here $t$ is the coordinate of the first component of a neighborhood of $\cY$  in $\cX$ which is diffeomorphic to $[0,1] \times \cY$. 
As before, we assume that for $z \in B$, there is an open neighborhood $U$ of $z$ and we can write $D_{X, w} = D_{X, z} + A_{X, w}$ for $w \in U$ through local trivializations of $F_{X}^0, F_{X}^1$. Here $A_{X, w}$ is an algebraic operator  induced by a  linear bundle map $F_{X}^0 |_{\cX_w} \to F_{X}^{1}|_{\cX_w}$ depending on $w$ continuously. 

We define Hilbert bundles  $\cE_{X, k}^0$, $\cE_{X, k}^1$ over $B$ for $k \geq 0$ using $D_X$ as before. 
Note that  $\ind D_Y = 0$ in $K^1(B)$ because of the cobordism invariance of the index. Hence there is  a spectral section of $D_Y$.

Let $(\cE_{Y, k-\frac{1}{2}})^{0}_{-\infty}$ be the subspace spanned by non-positive eigenvectors of $D_{Y}$ and $p^0$ be the $L^2_{k-\frac{1}{2}}$-orthogonal projection onto $(\cE_{Y,  k-\frac{1}{2}})^{0}_{-\infty}$. 
Let us consider the family of operators with the APS boundary condition. That is, we consider   the family of operators
\[
    (D_{X},   p^0  \circ r) : \cE_{X, k}^0 \rightarrow \cE_{X, k-1}^1 \oplus (\cE_{Y,  k-\frac{1}{2}})^{0}_{-\infty}.
\] 
Here $r$ is the restriction to  $\cY$. 
Note that this family is not continuous because of the spectral flow of $D_Y$.  Hence we can not use this family to define the index. A spectral section enables us to avoid this issue. 
Since our sign convention is different from that of \cite{MP},  taking a spectral section of $-D_Y$ rather than $D_Y$ is more convenient for us. 

\begin{prop}    \label{prop:Ind(D,P)}
Fix $k \geq 1$. 
Let $P$ be a spectral section of $-D_Y$.  We also denote by $P$ the image of $P$ in $\cE_{Y, 0}$, which is a Hilbert subbundle.  Let $\pi_{P}$ be the $L^2_{k-\frac{1}{2}}$-projection onto $P \cap \cE_{Y, k-\frac{1}{2}}$.  Then 
\[
          (D_X,  \pi_{P} \circ r ) :  \cE_{X, k}^0 \rightarrow \cE_{X, k-1}^1 \oplus   (P \cap \cE_{Y, k-\frac{1}{2}})
\] 
is a continuous family of Fredholm operators and we can define the index $\Ind (D_{X}, P) \in K(B)$.  
 The index $\ind (D_X, P)$ is independent of the choice of $k$. 
\end{prop}

Let $P$ be a spectral section of $-D_Y$.  We write $P$ for the image of $P$ in $\cE_{Y,0}$ too. Then we can take other spectral sections $Q, R$ of $-D_Y$ such that
\[
          Q \subset  P \subset   R. 
\]
See our construction of spectral sections in Section \ref{section : spectral sections}. Define a family of  operators
\[ 
         D_Y' := QD_YQ + (1-R)D_Y (1-R)     -   (1-Q)     P     +    R (1-P). 
\] 
We can see that  $D_Y'$ is injective   and that  $P$ is equal to the subspace spanned by negative eigenvectors of $D_{Y}'$.   Also we see that the operator $\mathbb{A} = D_{Y}' - D_Y$ is a family of smoothing operators acting on $\cE_{Y, k}$.  In fact, the image of $\mathbb{A}$ is included in the subspace spanned by finitely many eigenvectors of $D_Y$. 

Take a smooth function $f : \cX \rightarrow [0,1]$ such that
\[
         f(x) =
         \begin{cases} 
             1  & \text{for $x \in    [\frac{1}{2}, 1] \times \cY$, }  \\
             0 & \text{for $x \in \cX \smallsetminus  ([0,1] \times \cY)$}. 
           \end{cases} 
\]
Define $D_X' : \cE_{X, k}^0 \rightarrow \cE_{X, k-1}^1$ by
\[
     D_X' = D_X  + f \mathbb{A}. 
\]
Then 
\[
        D_X' = \frac{\partial}{\partial t} + D_Y'
\]
near  $\cY$ and there is no spectral flow of $D_Y'$. 
Therefore the family of operators $D_X'$ with the APS boundary condition defines the index $\ind D_X' \in K(B)$, and 
\[
      \ind D_{X}' = \ind (D_X, P). 
\]

\section{Connections on Hilbert bundles}   \label{subsec:Hilbert bundles} 

Since we will consider a connection on a Hilbert bundle later, we give the definition of a connection on a Hilbert bundle.  

Let $M$ be a connected, smooth $n$-manifold and $H$ be a Hilbert space.  We write $\operatorname{Aut} H$ and $\operatorname{End} H$ for the group of bounded linear isomorphisms $H \rightarrow H$ and the ring of bounded operators $H \rightarrow H$ respectively.   

Take a coordinate chart $(U, \varphi)$   of $M$.  For a map 
\[ 
     f : U \rightarrow H, 
\]
we define  the partial derivative $ \frac{\partial f}{\partial x^i}(x)$ at $x \in U$ by
\[
     \frac{\partial f}{\partial x^i}(x) = \lim_{h \rightarrow 0} \frac{1}{h}  \big( f \circ \varphi^{-1}(\varphi(x) + h e_i) - f (x) \big) 
\]
if the limit exists in $H$.  Here $e_i$ is the $i$-th standard basis of $\R^n$. 
For $\alpha = (\alpha_1, \dots, \alpha_n) \in (\Z_{\geq 0})^n$, we define $ \frac{\partial^{\alpha} f}{\partial x^{\alpha}}$ to be $\big( \frac{\partial}{\partial x^1} \big)^{\alpha_1}   \cdots \big( \frac{\partial}{\partial x^n} \big)^{\alpha_n} f$.  We say that $f$ is smooth if the derivatives $\frac{\partial^{\alpha} f}{\partial x^{\alpha}}$ exist and are continuous  on $U$ for all $\alpha \in (\Z_{\geq 0})^n$.

Let  $p : \cE \rightarrow M $ be a smooth Hilbert bundle on $M$ with fiber $H$.     
By a smooth Hilbert bundle we mean  that for each small open set $U$ in $M$, we have a local trivialization
\[
      \psi : \cE|_{U} \rightarrow U \times H
\]
such that if $\psi' : \cE|_{U'} \rightarrow U' \times H$ is another local trivialization with $U \cap U' \not= \emptyset$,  we can write
\[
         \psi' \circ      \psi^{-1} (x, v) = (x, g(x) v)
\]
for $x \in U \cap U'$ and $v \in H$, and $g$ is a map $U \cap U' \rightarrow \operatorname{Aut} H$ which  is smooth with respect to the operator norm. 
We always assume that Hilbert bundles are smooth.

A section $s : M \rightarrow \cE$ is said to be smooth if for each local trivialization $\psi : \cE|_{U} \rightarrow U \times H$, the map
\[
          \psi \circ s |_{U} : U \rightarrow U \times H
\]
is smooth.  We denote by $\Gamma(\cE)$  the space of smooth sections of $\cE$.

A connection $\nabla$ on $\cE$ is defined to be a map
\[
        \nabla : \Gamma(\cE) \rightarrow \Gamma( T^* M \otimes \cE) 
\]
having the following properties. 

\begin{enumerate}[(i)]
\item
For any sections $s_1, s_2 \in \Gamma(\cE)$,  
\[
    \nabla (s_1 + s_2) = \nabla s_1 + \nabla s_2. 
\]

\item
For any section $s \in \Gamma(\cE)$,  vector fields $X_1, X_2 \in \Gamma(TM)$ and smooth functions $f_1, f_2 \in C^{\infty}(M)$, 
\[
   \nabla_{f_1 X_1 + f_2 X_2} s = f_1 \nabla_{X_1} s + f_2 \nabla_{X_2} s. 
\]

\item
For any section $s \in \Gamma(\cE)$ and function $f \in C^{\infty}(M)$, 
\[
     \nabla(fs) = df \otimes s + f \nabla s. 
\]

%

\end{enumerate}

We define a connection $\nabla$ on the dual Hilbert bundle $\cE^*$ by
\[
            (\nabla_X \alpha )(s) : =   X( \alpha (s)) - \alpha(\nabla_X s). 
\]
Here $s \in \Gamma(\cE),   \alpha  \in \Gamma(\cE^*), X \in \Gamma(TM)$.

For connections $\nabla_1$, $\nabla_{2}$ on Hilbert bundles  $\mathcal{E}_1$, $\mathcal{E}_2$ over $M$, we define connections $\nabla_1 \oplus \nabla_2$, $\nabla_1 \otimes \nabla_2$ on $\mathcal{E}_1 \oplus \mathcal{E}_2$, $\mathcal{E}_1 \otimes \mathcal{E}_2$ by
\[
    \begin{split}
      & ( \nabla_1 \oplus \nabla_2) (s_1 \oplus s_2) :=  (\nabla_1 s_1 ) \oplus  (\nabla_2 s_2 ),  \\
      & (\nabla_1 \otimes \nabla_2) (s_1 \otimes s_2) := (\nabla_1s_1) \otimes s_2 + s_1 \otimes (\nabla_2 s_2). 
    \end{split}
\]

Write $\Omega^i(M; \mathcal{E})$ for the space of $i$-forms on $M$ with values in $\mathcal{E}$: 
\[
       \Omega^i(M; \mathcal{E}) :=   \Gamma( \Lambda^i T^*M  \otimes \cE ).
\]
For a connection $\nabla$ on $\cE$,  we have  the exterior derivative
\[
   d_{\nabla}: \Omega^i(M; \mathcal{E}) \rightarrow \Omega^{i+1}(M; \mathcal{E})
\]
defined by
\[
  \begin{split}
      & d_{\nabla} (\eta s ) =   (d \eta) s +  \eta \wedge  (\nabla s), \\
      & d_{\nabla} ( \eta_1 + \eta_2) = d_{\nabla} \eta_1 + d_{\nabla} \eta_2. 
  \end{split}
\]
Here $s \in \Gamma (\mathcal{E})$, $\eta \in \Omega^i(M)$, $\eta_1, \eta_2 \in \Omega^i(M; \mathcal{E})$.

We will make an assumption on the smoothness of $\nabla$.  Take a local trivialization $\psi : \cE|_{U} \rightarrow U  \times H$.  We can write  
\begin{equation}  \label{eq:psi nabla s}
            \psi    \nabla_{X}  s 
            = X   (\psi s ) + \omega (X) (\psi s )
\end{equation}
for  $s \in \Gamma(\cE|_{U})$ and  $X \in \Gamma(TU)$. 
Here for each $x \in U$ and $X \in T_xU$, $\omega(X)$ is a linear map $H \rightarrow H$. 
The  assumption is   that  $\omega(X)$ is bounded  and  the map  $\omega : TU \rightarrow \End H$ is  smooth with respect to the operator norm.   In particular, for a compact set $K$ in $U$,  the restriction $\omega(X)|_{K}$ is a Lipschitz continuous map $K \rightarrow \operatorname{End} H$. 

Under the above assumption, for any smooth curve  $c : [-\epsilon, \epsilon] \rightarrow U$  and $e \in \cE_{c(0)}$,  where $\epsilon > 0$, we have a unique smooth section $s$ of $\cE$ along  $c$ which solves the  ordinary differential equation in the Hilbert space: 
\[
     \frac{d}{dt} \psi(s (t))  + \omega  \bigg(  \frac{dc}{dt}(t)  \bigg) ( \psi s(t) ) = 0,  \quad
     s(0) = e. 
\] 
We call $s$ a parallel section of $\cE$ along $c$ or a  horizontal lift of $c$.   See \cite{Deimling} for the existence and uniqueness of solutions to the equation.  

Take  $x \in U$ and let $x^1, \dots, x^n$ be local coordinates around $x$. For $i = 1, \dots, n$, let $c_i$ be a smooth curve  $[-\epsilon, \epsilon] \rightarrow  U$ such that
\[
         c_i(0) = x,  \quad \frac{dc_i}{dt}(0) = \frac{\partial}{\partial x^i}. 
\]
 For $e \in \cE_x$, we define the horizontal component  $(T_{e} \cE)_H$  of $T_e \cE$ to be the subspace spanned by $\{  ds_i (\frac{\partial}{\partial t}) \}_{i=1, \dots, n}$.  Here $s_i$ is the parallel section of $\cE$ along $c_i$ with $s_i(0) = e$.  We can show that $(T_{e} \cE)_H$ is independent of the choice of the local coordinates $x^1, \dots, x^n$.   The connection $\nabla$ defines a decomposition 
\[
        T\cE = (T\cE)_H \oplus p^* \cE. 
\]
Note that we have a natural isomorphism
\[
         (T\cE)_{H} \cong p^* TM. 
\] 

As usual, there is a unique  $2$-form $F_{\nabla} \in \Omega^2(M; \End \cE)$ such that
\[
        d_{\nabla} \circ d_{\nabla} \eta = F_{\nabla} \wedge \eta 
\]
for   $\eta \in \Omega^i(M;  \cE)$.  We can write 
\[
         \psi F_{\nabla} = d\omega + \omega \wedge \omega
\]
on $U$, where $\omega$ is the $1$-form with values in $\operatorname{End} H$ in (\ref{eq:psi nabla s}). 
 We call $F_{\nabla}$ the curvature of $\nabla$. 
We say that $\nabla$ is flat if $F_{\nabla} = 0$. 

We can associate a flat connection to a representation 
\[
     \rho : \pi_1(M) \rightarrow   \mathrm{Aut} (H)
\]
in the usual way. 
Let $\cE$ be the Hilbert bundle on $M$ defined by
\[
    \cE := \tilde{M} \times_{\rho}  H, 
\]
where $\tilde{M}$ is the universal cover of $M$.  A smooth section $s : M \rightarrow \cE$ corresponds to a smooth map $\tilde{s} : \tilde{M} \rightarrow H$ such that 
\[
     \tilde{s}(\gamma \cdot x) =  \rho(\gamma) \tilde{s}(x) 
\]
for $x \in \tilde{M}$, $\gamma \in \pi_1(M)$.
Taking the exterior derivative,  we have
\[
      d \tilde{s}(\gamma \cdot x) = \rho(\gamma) d\tilde{s}(x) 
\]
and hence $d\tilde{s}$ descends to  a section of $T^* M \otimes \cE$, which we denote by $\nabla s$.  We can show that the map 
 \[
      \nabla : \Gamma(\cE) \rightarrow \Gamma(T^*M \otimes \cE)
 \]
is a flat connection on $\cE$.


\section{Notation and main statements}   \label{sec:main results}

Let $Y$ be a connected, closed, oriented 3-manifold and take a Riemannian  metric $g$ and $\mathrm{spin}^c$ structure $\frak{s}$ with $c_1(\frak{s})$ torsion on $Y$.   We denote the spinor bundle over $Y$ by $\mathbb{S}$. 
Fix a $\mathrm{spin}^c$ connection $A_0$ on $Y$ with $F_{A_0} = 0$. For a $1$-form $a \in  \Omega^1(Y)$,  we write $D_a$ for the Dirac operator $D_{A_0+ ia}$ which acts on the space  $C^{\infty}(\mathbb{S})$ of smooth sections of $\mathbb{S}$.   The family $\{ D_{a} \}_{a \in  \mathcal{H}^1(Y)}$ parameterized by the harmonic $1$-forms on $Y$ induces an operator $D$ acting on the vector bundle 
\[
           \mathcal{E}_{\infty} = \mathcal{H}^1(Y) \times_{H^1(Y;\mathbb{Z})} C^{\infty}(\mathbb{S})
\] 
over the Picard torus $\mathrm{Pic}(Y) = H^1(Y ; \mathbb{R}) / H^1(Y; \mathbb{Z})$.     
The action of $H^1(Y;\mathbb{Z})$ is defined by
\[
     h (a, \phi) = (a - h, u_h \phi)
\]
for $h \in H^1(Y;\mathbb{Z})$, $a \in \mathcal{H}^1(Y)$, $\phi \in C^{\infty}(\mathbb{S})$, where  $u_h$ is the harmonic gauge transformation  $Y \rightarrow U(1)$ with $-i u_h^{-1} du_{h} = h$ in $ \mathcal{H}^1(Y)$.

For $k \in \mathbb{R}_{\geq 0}$, define a Hilbert bundle on $\mathrm{Pic}(Y)$ by 
\[
     \mathcal{E}_k := \mathcal{H}^1(Y) \times_{H^1(Y;\mathbb{Z})} L^2_k(\mathbb{S}).
\]
For $k \geq 1$, the operator $D$ on $\mathcal{E}_{\infty}$ extends to a bounded operator
\[
       D : \mathcal{E}_{k} \rightarrow \mathcal{E}_{k-1}. 
\]

We have a canonical  flat connection $\nabla$ on $\mathcal{E}_{k}$ corresponding to the representation 
\[
      \begin{array}{ccc}  
      \pi_1(B) =  H^1(Y;\mathbb{Z}) & \rightarrow & \operatorname{Aut}(L^2_k(\mathbb{S}))    \\
        h  & \mapsto & u_h, 
      \end{array}
\]
where $B = \mathrm{Pic}(Y)$,  $\operatorname{Aut} (L^2_k(\mathbb{S}))$ is the group of bounded linear automorphisms on $L^2_k(\mathbb{S})$.    See Section \ref{subsec:Hilbert bundles}. 

A smooth section $s : B \rightarrow   \mathcal{E}_{k}$ can be considered to be a smooth map
\[
         \tilde{s} : \mathcal{H}^1(Y) \rightarrow L^2_k(\mathbb{S})
\]
such that 
\[
      \tilde{s}( a - h) =  u_{h} \tilde{s}(a)
\]
for  $h \in \operatorname{im} ( H^1(Y; \mathbb{Z}) \rightarrow \mathcal{H}^1(Y) )$. The covariant derivative $\nabla s$ corresponds to  the usual exterior derivative $d\tilde{s}$ of $\tilde{s}$.

Denote $\langle \cdot,  \cdot \rangle_{a, k}$ for the $L^2_k$-inner product with respect to $D_{a}$:
\[
      \langle \phi_1, \phi_2 \rangle_{a, k} 
      = \langle  \phi_1,  \phi_2 \rangle_{0} + 
         \langle  |D_a|^{k} \phi_1,  |D_a|^{k} \phi_2 \rangle_{0}
\]
where $\langle \cdot, \cdot \rangle_0$ is the $L^2(Y)$-inner product.  Here, we write $|D_{a}|$ for the absolute value of the Dirac operator $D_a$, defined using the spectral theorem (see e.g. \cite[Chapter VIII, \S 9]{reed-simon}).  
Then the family $\{ \langle \cdot, \cdot \rangle_{a, k} \}_{a \in \mathcal{H}^1(Y)}$ of $L^2_k$-inner products induces a fiberwise inner product $\langle \cdot, \cdot \rangle_k$ on $\mathcal{E}_k$. To see this,   take  sections $s_1, s_2 : B \rightarrow \mathcal{E}_k$ and $h \in \operatorname{im} (  H^1(Y;\mathbb{Z}) \rightarrow \mathcal{H}^1(Y))$.   Let $\tilde{s}_1, \tilde{s}_2 : \mathcal{H}^1(Y) \rightarrow L^2_k(\mathbb{S})$ be the maps corresponding to $s_1, s_2$.  Note that 
\[
       \tilde{s}_i(a-h) = u_h \tilde{s}_i(a), \quad D_{a-h} = u_h D_a u_{h}^{-1}. 
\]
Therefore
\[
    \begin{split}
        & \langle \tilde{s}_1(a - h), \tilde{s}_2(a-h) \rangle_{a-h, k}   \\
         & =  \langle  \tilde{s}_1(a-h),   \tilde{s}_2(a-h) \rangle_{0} + \langle |D_{a-h}|^{k} \tilde{s}_1(a-h), |D_{a-h}|^{k} \tilde{s}_2(a-h)\rangle_{0}    \\
         & =   \langle u_{h} \tilde{s}_1(a), u_{h} \tilde{s}_2(a) \rangle_{0} + 
           \langle  (u_{h} |D_{a}|^k u_{h}^{-1}) u_{h} \tilde{s}_1(a), (u_{h} |D_{a}|^k u_{h}^{-1}) u_{h} \tilde{s}_2(a)  \rangle_{0}   \\
         & = \langle u_h \tilde{s}_1(a), u_h \tilde{s}_2(a) \rangle_{0} +  \langle  u_h  |D_a|^k \tilde{s}_1(a), u_h  |D_a|^k \tilde{s}_2(a) \rangle_{0}   \\
         & = \langle  \tilde{s}_1(a),  \tilde{s}_2(a) \rangle_{a, k}. 
   \end{split}
\]
This implies that the family $\{ \langle \cdot, \cdot \rangle_{a, k} \}_{a \in \mathcal{H}^1(Y)}$  descends to a fiberwise inner product $\langle \cdot, \cdot \rangle_{k}$ on $\mathcal{E}_k$.   We write $\| \cdot \|_{k}$ for the fiberwise norm on $\mathcal{E}_k$ induced by $\langle \cdot, \cdot \rangle_{k}$.

The flat connection $\nabla$, with respect to $k = 0$,  defines a decomposition
\begin{equation}  \label{eq TE}
           T \mathcal{E}_0 = p^* TB \oplus p^*  \mathcal{E}_0, 
\end{equation} 
where $p : \mathcal{E}_0 \rightarrow B$ is the projection, $p^* TB$ is the horizontal component and $p^* \mathcal{E}_0$ is the vertical component.   See Section \ref{subsec:Hilbert bundles}. 
  Note that the flat connection $\nabla$ is not compatible with the inner product $\langle \cdot, \cdot \rangle_{k}$ on $\mathcal{E}_k$ for $k > 0$.

Put
\[
             \mathcal{W}_k  = B \times  L^2_k(  \operatorname{im} d^*  ), 
\]
where $d^* : i \Omega^2(Y) \rightarrow i \Omega^1(Y)$ is the adjoint of the exterior derivative.  We consider $\mathcal{W}_k$ to be a trivial Hilbert bundle on $B$. 
The Seiberg-Witten equations on $Y \times \mathbb{R}$ are equations for $\gamma = (\phi, a, \omega) : \mathbb{R} \rightarrow L^2_k(\mathbb{S}) \times \mathcal{H}^1(Y) \times L^2_k(\im d^*)$ written as
\begin{equation}\label{eq:seiberg-witten}
    \begin{aligned}
        \frac{d\phi}{d t} &= - D_a \phi(t) - c_1(\gamma(t)), \\
        \frac{d a}{dt} &= - X_{H}(\phi), \\
        \frac{d\omega}{dt} &= - *d \omega - c_2(\gamma(t)). 
    \end{aligned}
\end{equation}
  The terms  $X_H(\phi)$, $c_1(\gamma(t))$,  $c_2(\gamma(t))$ are defined by
\begin{equation} \label{eq quadratic terms}
  \begin{aligned}
       & q(\phi) =  \rho^{-1} \bigg( \phi \otimes \phi^{*} -  \frac{1}{2} | \phi |^2 \operatorname{id} \bigg) \in \Omega^1(Y),  \\
       & X_H(\phi) = q(\phi)_{\mathcal{H}} \in \mathcal{H}^1(Y), \\
       &  c_1(\gamma(t)) =   ( \rho(\omega(t)) - i \xi (\phi(t)) ) \phi(t), \\
       & c_2(\gamma(t)) =  \pi_{\operatorname{im} d^*} (q(\phi(t))), 
  \end{aligned}
\end{equation}
where $\rho$ is the Clifford multiplication which defines an isomorphism
\[
                  T^* Y \otimes \mathbb{C} \rightarrow \mathfrak{sl}(\mathbb{S}), 
\]    
$q(\phi)_{\mathcal{H}}$ is the harmonic component of $q(\phi)$, 
$\pi_{\operatorname{im} d^*}$ is the $L^2$-projection on $\mathcal{W}_k$, and $\xi(\phi)$ is the function $Y \rightarrow \mathbb{R}$  satisfying
\[
            d \xi(\phi) =  i \pi_{ \operatorname{im} d } (q(\phi)), \quad
            \int_{Y} \xi(\phi) \operatorname{vol} = 0. 
\]

The equations (\ref{eq:seiberg-witten}) do not correspond to the Seiberg-Witten equations in Coulomb gauge in $Y\times \mathbb{R}$ (i.e., solutions of the equations are not Seiberg-Witten trajectories in Coulomb gauge).  Instead, we use the \emph{pseudo-temporal gauge} of \cite[Definition 5.2.1]{LidmanManolescu} (see also \cite[Section 3]{Manolescu-b1=0}).  The correspondence between solutions of (\ref{eq:seiberg-witten}) and the Seiberg-Witten equations modulo gauge is given by Proposition 5.4.2 of \cite{LidmanManolescu}.  Note that Lidman-Manolescu work in the setting of $b_1=0$, however, the argument is local in the configuration space and passes over without change to the $b_1>0$ case.  We will, however, call solutions of (\ref{eq:seiberg-witten}) \emph{Seiberg-Witten trajectories}.

The equations descend to equations for  $\gamma = (\phi, \omega) : \mathbb{R} \rightarrow \mathcal{E}_k \oplus  \mathcal{W}_{k}$:
\begin{equation}   \label{SW eq}
    \begin{aligned}
           \Big( \frac{d\phi}{dt} (t)  \Big)_{V}  &=   - D \phi(t) - c_1(\gamma(t)),   \\
           \Big(  \frac{d \phi}{dt} (t) \Big)_{H} &=   -X_H(\phi(t)), \\ 
           \frac{d\omega}{dt}(t) & = -*d \omega(t) -  c_2(\gamma(t)). 
    \end{aligned}
\end{equation}
Here $\big( \frac{d\phi}{dt} \big)_{V}$,  $\big( \frac{d\phi}{dt} \big)_{H}$ are the vertical component and horizontal component of $\frac{d\phi}{dt}$ respectively,  and we have suppressed the subscript from $D$.

Assume that the family index of the family of Dirac operators $D$ over $\mathrm{Pic}(Y)$ vanishes, that is:
\[
        \operatorname{ind} D = 0 \in K^1(B). 
\]

Then we can choose a spectral section $P_0$ of $-D$, and using $P_0$, we can define a self-adjoint (with respect to $L^2$)  operator
\[
        \mathbb{A} : C^{\infty}(\mathbb{S}) \rightarrow C^{\infty}(\mathbb{S})
\]
such that the image of $\mathbb{A}$ is included in a subspace spanned by finitely many eigenvectors of $D$, and so that $\ker (D + \mathbb{A}) = 0$. 
 Put $D' = D + \mathbb{A}$.  The $L^2$-closure of the subspace spanned by the negative eigenvectors of $D'$ is exactly the image of $P_0$, acting on $L^2$ (See \cite{MP}  and Section \ref{section:def of spectral section} for all of these assertions.).  In the future, for a spectral section $P$, we will also often write $P$ to refer to the image of $P$.   
We have a decomposition
\[
         \mathcal{E}_{\infty} = \mathcal{E}_{\infty}^{+} \oplus \mathcal{E}_{\infty}^{-}, 
\]
where $\mathcal{E}^{+}_{\infty}$ and $\mathcal{E}^{-}_{\infty}$ are the subbundles of $\mathcal{E}$ spanned by positive eigenvectors and negative eigenvectors of $D'$.

For positive numbers $k_+, k_-$ and $s_1, s_2 \in C^{\infty}(\mathbb{S})$,  we define an inner product $\langle  s_1, s_2 \rangle_{a, k_+, k_-}$ by 
\begin{equation} \label{eq L2 k_+ k_-}
      \langle   s_1, s_2 \rangle_{a, k_+, k_-} := 
      \langle   |D_a'|^{k_+} s_1^+,     |D_a'|^{k_+} s_2^+ \rangle_{0} + 
      \langle   |D_a'|^{k_-} s_1^-,     |D_a'|^{k_-} s_2^-  \rangle_{0}, 
\end{equation}
where $s_j = s_j^+ + s_j^-$ and $s_j^+ \in \mathcal{E}^+_{\infty}$, $s_{j}^{-} \in \mathcal{E}_{\infty}^{-}$.  
 Note that we do not need the term $\langle s_1, s_2 \rangle_0$,  since the kernel of $D'_a$ is zero. 
 We call this inner product the $L^2_{k_+, k_-}$-inner product. 


As before, the family $\{  \langle \cdot, \cdot \rangle_{a, k_+, k_-}   \}_{a \in \mathcal{H}^1(Y)}$ induces a fiberwise inner product on $\mathcal{E}_{\infty}$  and we denote by $\mathcal{E}_{k_+, k_-}$ the completion of $\mathcal{E}_{\infty}$ with respect to the norm $\| \cdot \|_{k_+, k_-}$. 

On the space $\operatorname{im} d^*  \cap \Omega^1(Y)$, we define an inner product $\langle \cdot, \cdot \rangle_{k_+, k_-}$ by
\[
       \langle  \omega_1, \omega_2 \rangle_{k_+, k_-}
       = \langle  |*d|^{k_+} \omega_1^+,   |*d|^{k_+} \omega_2^+ \rangle_0 +
          \langle  |*d|^{k_-} \omega_1^-,  |*d|^{k_-} \omega_2^- \rangle_0,
\]
where $\omega_j = \omega_j^{+} + \omega_j^{-}$ and $\omega_j^+$ is in the subspace spanned by positive eigenvectors of the operator $*d$ and $\omega_j^-$ is in the negative one.  We denote by $\mathcal{W}_{k_+, k_-}$ the completion of the vector bundle $B \times \operatorname{im} d^*$ over $B$ with respect to $\| \cdot \|_{k_+, k_-}$.   We will use $L^2_{k-\frac{1}{2}, k}$-norm in Section \ref{section 4-mfd with boundary} to define the relative Bauer-Furuta invariant. 
See Remark \ref{rem:L^2_{k-1/2, k}} for the reason why we use the $L^2_{k-\frac{1}{2}, k}$ norm.

We recall the definition of finite-type trajectories (from e.g. \cite[Definition 1]{Manolescu-b1=0}):
\begin{dfn}\label{def:finite-type}
    A Seiberg-Witten trajectory $\gamma(t)=(\phi(t),a(t),\omega(t))$ is \emph{finite-type} if $CSD(\gamma(t))$ and $\lVert \phi(t) \rVert_{C^0}$ are bounded functions of $t$, where $CSD$ is the Chern-Simons-Dirac functional.
\end{dfn}

The following is a direct consequence of a standard argument in Seiberg-Witten theory, see e.g. \cite[Proposition 1]{Manolescu-b1=0}.

\begin{prop}  \label{prop compactness}
For positive numbers $k_+, k_- > 0$,   there is a positive constant $R_{k_+, k_-} > 0$ such that for any finite-type solution $\gamma : \mathbb{R} \rightarrow  \mathcal{E}_2 \times \mathcal{W}_2$ to (\ref{SW eq}), we have
\[
       \lVert \gamma(t)  \rVert_{k_+, k_-} \leq R_{k_+, k_-}
\]
for all $t \in \mathbb{R}$. 
\end{prop}

Write $\mathcal{E}_0(D)^{b}_{b'}$ for the span of eigenvectors of $D$ with eigenvalue in $(b',b]$, as a space over $\mathcal{H}^1(Y)$ (Note that it will not usually be a bundle).  For a spectral section $P$ of $D$, we also write $P$ for the image of the projection $P$.  By Theorem \ref{thm spectral section mu mu+delta} below, we can take sequences of smooth spectral sections $P_n, Q_n$, of $-D, D$, respectively,  such that
\begin{equation}\label{eq:spectral-conditions-1}
    \begin{aligned}
         & (\mathcal{E}_0(D))^{\mu_{n, -}}_{-\infty} \subset P_n \subset  (\mathcal{E}_0(D))^{\mu_{n,+}}_{-\infty}, \\
         & (\mathcal{E}_0(D))_{\lambda_{n, +}}^{\infty} \subset Q_n \subset ( \mathcal{E}_0(D))_{\lambda_{n, -}}^{\infty}, 
    \end{aligned}
\end{equation}
with 
\begin{equation}\label{eq:spectral-conditions-2}
  \begin{aligned}
         & \mu_{n, -}  + 10  < \mu_{n, +}  < \mu_{n, +} + 10  < \mu_{n + 1, -},    \\
         & \lambda_{n+1, +} < \lambda_{n, -} - 10 < \lambda_{n, -} < \lambda_{n, +} - 10, \\
         & \mu_{n, + } - \mu_{n,-} < \delta,  \\
         &   \lambda_{n, +} - \lambda_{n, -} < \delta. 
   \end{aligned}
\end{equation}
Here $\delta > 0$ is a positive constant independent of $n$, and a smooth spectral section means a spectral section which depends smoothly on the base space $B$. 



We define a finite rank subbundle  $F_n$ in $\mathcal{E}_{\infty}$ by
\[
         F_n = P_n \cap Q_n. 
\]
Define a connection $\nabla_{F_n}$ on $F_n$ by
\[
         \nabla_{F_n} = \pi_{F_n} \nabla, 
\]
where $\pi_{F_n}$ is the $L^2_{k_+, k_-}$-projection on $F_n$. 
The connection $\nabla_{F_n}$ defines a decomposition
\begin{equation}  \label{eq TF}
               TF_n = (TF_n)_{H, \nabla_{F_n}} \oplus (TF_n)_{V}  \cong p^* TB \oplus p^* F_n. 
\end{equation}
A calculation shows that  the horizontal component $(T_{\phi}F_n)_{H, \nabla_{F_n} }$ of $TF_n$ at $\phi \in F_n$ is given by 
\begin{equation}   \label{eq hor TF}
              \{  (v,  (\nabla_v \pi_{F_n}) \phi ) | v \in T_a B  \}   \subset   
              (p^* TB \oplus p^* \mathcal{E}_{0})_{\phi} = T_{\phi} \mathcal{E}_0.
\end{equation}
Here $a = p(\phi) \in B$.

Let $W_n$ be the finite dimensional subbundle  of the   Hilbert bundle $\mathcal{W}_k$ spanned by the eigenvectors of the operartor $*d$ whose eigenvalues are in the interval  $(\lambda_{n, -}, \mu_{n, +}]$:
\[
       W_n = (\mathcal{W}_k)_{\lambda_{n,-}}^{\mu_{n,+}} = B \times  L^2_{k}(\im d^*)_{\lambda_{n,-}}^{\mu_{n, +}}
\]

 Fix a positive number   $R'$ with  $R' \geq 100 R_{k_+, k_-}$ and  a smooth function 
 \[
   \chi : \mathcal{E}_{k_+, k_-} \oplus \mathcal{W}_{k_+, k_-} \rightarrow [0,1]
 \]
 with compact support such that $\chi (\phi, \omega) = 1$ if $\| (\phi, \omega) \|_{k_+, k_-} \leq R'$. We consider the following equations for $\gamma  = (\phi, \omega) : \mathbb{R} \rightarrow F_n \oplus W_n$ which we call the \emph{finite dimensional approximation} of (\ref{SW eq}): 
\begin{equation}   \label{eq for gamma}
   \begin{aligned}
  \Big( \frac{d\phi}{dt} (t) \Big)_{V} &= - \chi\{  (\nabla_{X_{H}} \pi_{F_n})  \phi(t)  +  \pi_{F_n} (D \phi(t) + c_1(\gamma(t)) ) \} ,   \\ 
  \Big( \frac{d \phi}{dt}(t)    \Big)_{H} & =  - \chi X_{H}(\phi(t)),   \\
            \frac{d \omega}{dt}(t)  &=  - \chi \{ *d \omega(t) +  \pi_{W_n} c_2(\gamma(t)) \}.
   \end{aligned}
\end{equation}
Here  
$\big( \frac{d\phi}{dt} \big)_{V},  \big(\frac{d\phi}{dt} \big)_{H}$ are the vertical component and the horizontal component with respect to the fixed decomposition (\ref{eq TE}) rather than (\ref{eq TF}).  It follows from (\ref{eq hor TF}) that the right hand side of (\ref{eq for gamma}) is a tangent vector on $F_n \oplus W_n$. 
Hence the equations (\ref{eq for gamma}) define a flow
\[
     \varphi_n = \varphi_{n, k_+, k_-} :(F_n \oplus W_n) \times \mathbb{R} \rightarrow F_n \oplus W_n. 
\]
(This flow depends on $k_+, k_-$ because $\pi_{F_n}$ does.)

We have  decompositions
\[
     F_n = F_n^{+} \oplus F_{n}^{-},   \quad
     W_n = W_n^+ \oplus W_n^-, 
\]
where $F^+_n, W_n^+$ are the positive eigenvalue components of $D'$, $*d$ and $F^{-}_n, W_n^{-}$ are the negative eigenvalue components. 
In the remainder of Chapter \ref{sec:findim-appx}, we will prove the following:

\begin{thm}  \label{thm isolating nbd}
Let $k_+, k_-$ be  half integers (that is, $k_+, k_- \in \frac{1}{2}\mathbb{Z}$)  with $k_+, k_- > 5$ and with $| k_+ - k_- | \leq \frac{1}{2}$. 
Fix a positive number $R$ with $R_{k_+, k_-} < R < \frac{1}{10} R'$, where $R_{k_+, k_-}$ is the constant of Proposition \ref{prop compactness}.    Then 
\[
   \big(  B_{k_+} (F_n^+; R) \times_B B_{k_-}(F_n^-; R) \big)  \times_{B}   \big( B_{k_+}(W_n^+; R) \times_{B} B_{k_-}(W_n^-; R) \big)
 \]
is an isolating neighborhood of  the flow $\varphi_{n, k_+,k_-}$ for $n \gg 0$.  Here $B_{k_{\pm}}(F_n^{\pm};R)$ are the disk bundle of $F_n^{\pm}$ of radius $R$ in $L^2_{k_{\pm}}$ and $B_{k_+}(F_n^+; R) \times_B B_{k_-}(F_n^-; R)$ is the fiberwise product.  Similarly for $B_{k_{\pm}}(W_n^{\pm};R)$. 
\end{thm}

The general strategy to prove Theorem \ref{thm isolating nbd} is as follows: once we have Theorem \ref{thm spectral section mu mu+delta} in hand, we must control the gradient term $(\nabla_{X_H}\pi_{F_n}) \phi(t)$ appearing in the approximate Seiberg-Witten equations (\ref{eq for gamma}); a number of bounds for this are obtained in Sections \ref{subsec:derivatives} and \ref{subsec:weighted-spaces}.  The proof proper is in Section \ref{subsec:proof-of-isolation}, where Theorem \ref{thm isolating nbd} follows from establish that, for sufficiently large approximations, the linear term in the approximate Seiberg-Witten equations (\ref{eq for gamma}) tends to dominate the other terms with respect to appropriate norms.

We also note that the total space $B_{n,R}$ appearing in Theorem \ref{thm isolating nbd} is an ex-space over $B=\mathrm{Pic}(Y)$ in the sense of Section \ref{subsec:homotopy1}, with projection given by restricting $p: \mathcal{E}_k\to B$ to $B_{n,R}$, and with a section $s_B: \mathrm{Pic}(Y)\to B_{n,R}$ given by the zero-section.


\section{Construction of spectral sections} \label{section : spectral sections}

We will prove the following: 

\begin{thm}  \label{thm spectral section mu mu+delta}
Assume that $\operatorname{Ind} D = 0$ in $K^1(B)$.  Take a sequence $\mu_n$ of positive numbers $\mu_{n} \ll \mu_{n+1}$, where $\mu_n\to\infty$ as $n\to \infty$. 
There is a sequence of  spectral sections $P_n$ of $-D$ with the following properties. 

\begin{enumerate}[(i)]

\item
We have 
\[
         \mathcal{E}_0(D)_{-\infty}^{\mu_{n}} \subset P_n \subset \mathcal{E}_0(D)_{-\infty}^{\mu_{n} + \delta},
\]
where $\delta$ is a positive constant independent of $n$. 

\item
We can write
\[ 
       P_{n+1} = P_n \oplus \langle f^{(n)}_1, \dots, f^{(n)}_{r_n} \rangle,
\]
where $\{ f_1^{(n)}, \dots, f^{(n)}_{r_n} \}$ is a frame of $P_{n}^{\bot}$ (where $P_n^\bot$ is the $L^2$-orthogonal complement of $P_n$ inside of $P_{n+1}$).  In particular, 
\[
      P_{n+1} \cong P_n \oplus \underline{\mathbb{C}}^{r_n}, 
\]
where $\underline{\mathbb{C}}^{r_n}$ is the trivial vector bundle over $B$. 

\end{enumerate}
\end{thm}

Before we start proving Theorem \ref{thm spectral section mu mu+delta},  we will show the following: 

\begin{prop} \label{prop [D, pi]}
Take any non-negative numbers $k, l$. 
Let $P_n$ be a  sequence of spectral sections of $-D$ having Property (i) of Theorem \ref{thm spectral section mu mu+delta}. 
Let $\pi_n : \mathcal{E}_k  \rightarrow P_n \cap \mathcal{E}_k$ be the $L^2_k$-projection.

\begin{enumerate}
\item
 The commutators
\[
       [D, \pi_n] : \mathcal{E}_{\infty} \rightarrow \mathcal{E}_{\infty}
\]
extend to bounded operators
\[
    [D, \pi_n] : \mathcal{E}_l \rightarrow \mathcal{E}_l
\]
and we have
\[
        \lVert   [D, \pi_n] : \mathcal{E}_l \rightarrow \mathcal{E}_l   \rVert  < C, 
\]
where $C$ is a positive constant independent of $n$.   Moreover for any $l  > 0$, $\epsilon > 0 $ with $0 < \epsilon \leq l$, 
\[
         \sup_{a \in B}   \lVert| [D_{a}, \pi_{n, a}] : L^2_{l}(\mathbb{S}) \rightarrow L^2_{l-\epsilon}(\mathbb{S}) \rVert  \rightarrow 0
\] 
as $n \rightarrow \infty$. 

\item
The operator $\pi_n : \mathcal{E}_{\infty} \rightarrow \mathcal{E}_{\infty}$ extends to a bounded  operator $\mathcal{E}_{l} \rightarrow \mathcal{E}_{l}$ for each nonnegative real number $l$. 
Moreover, there is a positive constant $C$ independent of $n$ such that
\[
      \| \pi_n : \mathcal{E}_l \rightarrow \mathcal{E}_l \| < C.
\]

\end{enumerate}
\end{prop}

\begin{proof}
Take $a \in B$ and let $\{  e_j \}_j$ be an orthonormal basis of $L^2(\mathbb{S})$ with
\[
           D_a e_j = \eta_j e_j,
\]
where $\eta_j \in \mathbb{R}$.    

Let $P_{n, a}$ be the fiber of $P_n$ over $a$. Take $\phi \in \mathcal{E}_{\infty} \cap P_{n, a}$.  We can write
\[
        \phi =    \sum_{\eta_j \leq \mu_n + \delta} c_j e_j, 
\]
where $c_j \in \C$. 
Note that 
\[
         \sum_{\eta_j \leq \mu_n} c_j e_j \in \mathcal{E}_{\infty} \cap P_{n,a},   \
          \sum_{\mu_n < \eta_j \leq \mu_n + \delta} c_j e_j \in \mathcal{E}_{\infty} \cap P_{n, a}.
\]
We have
\begin{equation}  \label{eq [D, pi] phi}
   \begin{aligned}
     &  [D_a, \pi_{n,a}] \phi  \\
      & =  (D_a \pi_{n,a} - \pi_{n,a} D) \phi  \\
      & =  \sum_{\eta_j \leq \mu_n + \delta} \eta_j c_{j} e_j - \pi_{n,a} \sum_{\eta_j \leq \mu_n + \delta} \eta_j c_{j}  e_j  \\
      & = (1-\pi_{n,a}) \sum_{\mu_n < \eta_j \leq \mu_{n}+\delta} \eta_j c_j  e_j \\
      & = (1 - \pi_{n, a} )
        \Bigg\{ \sum_{  \mu_n < \eta_j \leq \mu_n + \delta  }  (\eta_j - \mu_n) c_j e_j  + 
           \mu_n \sum_{\mu_n < \eta_j \leq \mu_n + \delta}  c_j e_j
        \Bigg\}   \\
        & = \sum_{\mu_n < \eta_j \leq \mu_n + \delta} (\eta_j - \mu_{n}) c_j  (1 - \pi_{n,a}) e_j. 
   \end{aligned}
\end{equation}
Since 
\[
              \pi_n = \pi_{-\infty}^{\mu_n} + \pi_{P_n \cap \mathcal{E}(D)_{\mu_n}^{\mu_n + \delta}}, 
\]
for $j$ with $\mu_n < \eta_j \leq \mu_n + \delta$, we have
\[
         ( 1 - \pi_{n,a}) e_j \in \mathcal{E}_0(D_a)_{\mu_n}^{\mu_n + \delta}. 
\]
Hence we can write
\begin{equation}  \label{eq (1 - pi) e_j}
       (1-\pi_{n, a}) e_j = \sum_{\mu_n < \eta_p \leq \mu_n + \delta} \alpha_{j p} e_{p}
\end{equation}
for $j$ with $\mu_n < \eta_j \leq \mu_n + \delta$.  Here $\alpha_{jp} \in \C$.     Since
\[
        \| ( 1 - \pi_{n,a} ) : L^2_k \rightarrow L^2_k \| = 1, \  \| e_j \|_{k} = (1+|\eta_j|^{2k})^{\frac{1}{2}}
\]
we have
\[
     \| (1-\pi_{n, a}) e_j \|_{k}^2 =
       \sum_{\mu_n < \eta_p \leq \mu_n + \delta}  | \alpha_{jp} |^2  ( 1 + |\eta_p|^{2k} )
       \leq   ( 1 + |\eta_j|^{2k}). 
\]
For  $j$ with $\mu_n < \eta_j \leq \mu_n + \delta$ ,  
\begin{equation}  \label{eq alpha_{jp}}
   \begin{aligned}
        \sum_{\mu_n < \eta_p \leq \mu_n + \delta} | \alpha_{jp} |^2
        & = \sum_{\mu_n < \eta_p \leq \mu_n + \delta} | \alpha_{jp} |^2   (1 + |\eta_{p}|^{2k})  \frac{1}{1 + |\eta_{p}|^{2k}}   \\
        & \leq  \frac{C_1}{1+ (\mu_n+\delta)^{2k}} \sum_{\mu_n < \eta_p \leq \mu_n + \delta}  |\alpha_{jp}| ^2   (1 + |\eta_{p}|^{2k})  \\
        & \leq  \frac{C_1 (1 + |\eta_j|^{2k})}{1 + ( \mu_n + \delta)^{2k}}  \\
        & \leq C_1, 
   \end{aligned}
\end{equation}
where $C_1$ is a positive constant independent of $j, n$.

By (\ref{eq [D, pi] phi}), (\ref{eq (1 - pi) e_j}) and (\ref{eq alpha_{jp}}), 
\begin{align*}
        \| [D_a, \pi_{n,a}] \phi \|_{l}^2
      & =  \sum_{\mu_n < \eta_j \leq \mu_n + \delta} \sum_{\mu_n < \eta_p \leq \mu_n + \delta}  |\eta_j - \mu_n|^2 (1 + |\eta_p|^{2l} ) |c_j|^2   | \alpha_{jp} |^2  \\
      & \leq \delta^2 \sum_{\mu_n < \eta_j \leq \mu_n + \delta} \sum_{\mu_n < \eta_p \leq \mu_n + \delta} (1 + |\eta_j|^{2l} ) |c_j|^2   | \alpha_{jp}  |^{2} \cdot  \frac{1+|\eta_p|^{2l}}{1 + |\eta_j|^{2l}}  \\
     & \leq C_2     \sum_{\mu_n < \eta_j \leq \mu_n + \delta}  ( 1 + |\eta_{j}|^{2l}) |c_j|^2  \Bigg( \sum_{\mu_n < \eta_p \leq \mu_n + \delta} | \alpha_{jp} |^{2}  \Bigg)  \\
     & \leq C_3   \sum_{\mu_n < \eta_j \leq \mu_n + \delta}  ( 1 + |\eta_{j}|^{2l}) |c_j|^2   \\
     & \leq C_3   \| \phi \|_{l}^2.
\end{align*}
Here $C_2, C_3 > 0$ are positive constants independent of $n, \phi, a$. Also we have
   \begin{align*}
     & \| [D_a, \pi_{n,a}] \phi \|_{l- \epsilon}^2  \\
      & =  \sum_{\mu_n < \eta_j \leq \mu_n + \delta} \sum_{\mu_n < \eta_p \leq \mu_n + \delta}  |\eta_j - \mu_n|^2(1 + |\eta_p|^{2(l-\epsilon)} ) |c_j|^2   | \alpha_{jp} |^2  \\
      & \leq \delta^2 \sum_{\mu_n < \eta_j \leq \mu_n + \delta} \sum_{\mu_n < \eta_p \leq \mu_n + \delta} (1 + |\eta_j|^{2(l-\epsilon)} ) |c_j|^2   | \alpha_{jp}  |^{2} \cdot  \frac{1+|\eta_p|^{2(l-\epsilon)}}{1 + |\eta_j|^{2(l-\epsilon)}}  \\
     & \leq C_4     \sum_{\mu_n < \eta_j \leq \mu_n + \delta}  ( 1 + |\eta_{j}|^{2(l-\epsilon)}) |c_j|^2  \Bigg( \sum_{\mu_n < \eta_p \leq \mu_n + \delta} | \alpha_{jp} |^{2}  \Bigg)  \\
     & \leq C_5   \sum_{\mu_n < \eta_j \leq \mu_n + \delta}  ( 1 + |\eta_{j}|^{2(l-\epsilon)}) |c_j|^2   \\
     & \leq C_6( \mu_n^{-2l} + \mu^{-2\epsilon}_{n})  \| \phi \|_{l}^2. 
   \end{align*}
Here $C_4, C_5, C_6$ are positive constants independent of $n, \phi, a$.

On the other hand, consider $\phi \in  \mathcal{E}_{\infty} \cap P_{n, a}^{\bot_k}$, where $P_{n, a}^{\bot_k}$ is the $L^2_k$-orthogonal complement of $P_{n, a} \cap L^2_k(\mathbb{S})$ in $L^2_k(\mathbb{S})$.   We can write
\[
       \phi = \sum_{\eta_j > \mu_n} c_j e_j. 
\]
Note that
\[
      \sum_{\eta_j > \mu_{n} + \delta}  c_j e_j \in \mathcal{E}_{\infty} \cap P_{n, a}^{\bot_k}, 
      \sum_{ \mu_n < \eta_j \leq \mu_n + \delta} c_j e_j \in \mathcal{E}_{\infty} \cap P_{n, a}^{\bot_k}. 
\]
We have
  \begin{align*}
      [D_a, \pi_{n, a}] \phi
      &= \pi_{n, a} \sum_{\mu_n < \eta_j \leq \mu_n + \delta} \eta_j c_j e_j     \\
      &= \pi_{n, a} \Bigg(  \sum_{\mu_n < \eta_j \leq \mu_n + \delta} (\eta_j - \mu_n) c_j e_j +
          \mu_n \sum_{\mu_n < \eta_j \leq \mu_n + \delta} c_j e_j  \Bigg)   \\
      &= \sum_{\mu_n < \eta_j \leq \mu_n + \delta} (\eta_j - \mu_n) c_j  \pi_{n, a}  e_j 
  \end{align*}
As before, using this equality,   we can show that
\[
    \| [D_{a}, \pi_{n, a}] \phi \|_{l} \leq C_7 \| \phi \|_{l}, \ 
     \| [D_{a}, \pi_{n,a}] \phi \|_{l - \epsilon} \leq C_8  \mu_n^{-\epsilon}  \| \phi \|_{l}
\]
for some positive constants $C_7, C_8$ independent of $n, \phi, a$.  

Therefore $[D_a, \pi_{n,a}]$ extend to  bounded maps $L^2_l \rightarrow L^2_l$ with
\[
          \| [D_a, \pi_{n,a}] : L^2_l \rightarrow L^2_l \| \leq C_9, 
\]
for some constant $C_9$ independent of $n, a$. Also
\[
      \sup_{a \in B} \| [D_a, \pi_{n, a}] : L^2_l(\mathbb{S}) \rightarrow L^2_{l - \epsilon}(\mathbb{S}) \| \rightarrow 0 
\]
as $n \rightarrow \infty$.  We have proved (1).

We will prove (2). 
It is easy to see that if $\mu_{n} < \eta_j \leq \mu_{n} + \delta$, we have
\[
        \pi_{n} e_j \in  (\mathcal{E}_l)_{\mu_n}^{\mu_n+ \delta}. 
\]  
So we can write
\[
         \pi_n e_j = \sum_{\mu_n < \eta_{p} \leq \mu_n+\delta}  \alpha_{jp} e_p
\]
Because the operator norm of $\pi_n : L^2_k \rightarrow L^2_k$ is $1$ and $\| e_j \|_{k}^2 = 1 + |\eta_j|^{2k}$ , we have
\[
       |\mu_n|^{2k} \sum_{\mu_n < \eta_p \leq \mu_n + \delta} |\alpha_{jp}|^2 \leq 
       \sum_{\mu_n < \eta_p \leq \mu_{n}+\delta}  |\alpha_{jp}|^2 (1+|\eta_p|^{2k})  \leq 1+|\eta_j|^{2k}. 
\]
Therefore for $j$ with $\mu_n < \eta_j \leq \mu_n+\delta$, 
\begin{equation}  \label{eq a_{jp}}
         \sum_{\mu_n < \eta_p \leq \mu_n + \delta} |\alpha_{jp}|^2 
         \leq  \frac{1+|\eta_j|^{2k}}{ |\mu_n|^{2k}}
         \leq C_9. 
\end{equation}
Here $C_9 >0$ is a constant independent of $n, j$.   Take $\phi \in \mathcal{E}_{\infty}$. We can write as
\[
      \phi = 
      \sum_{\eta_j \leq \mu_n} c_j e_j 
      + \sum_{\mu_n < \eta_j \leq \mu_n + \delta}c_j e_j
      + \sum_{ \mu_n + \delta < \eta_j } c_j e_j. 
\]
Then 
\[
     \pi_n \phi =
     \sum_{\eta_j \leq \mu_n} c_j e_j
     + \sum_{\substack{\mu_n < \eta_j \leq \mu_n + \delta  \\ \mu_n < \eta_p \leq \mu_n + \delta }  }  c_j \alpha_{jp} e_p. 
\]
Hence we obtain
   \begin{align*}
  &   \| \pi_n \phi \|_{l}^2     \\
  &= \sum_{\eta_j \leq \mu_n} |c_j|^2 (1+|\eta_j|^{2l})    
       +  \sum_{\substack{\mu_n < \eta_j \leq \mu_n + \delta  \\ \mu_n < \eta_p \leq \mu_n + \delta }  }  |c_j|^2 |\alpha_{jp}|^2   (1 + | \eta_{p}|^{2l})   \\
 &   \leq  C_{10} \Bigg( \sum_{\eta_j \leq \mu_n} |c_j|^2 (1+|\eta_j|^{2l})    
       + (1+|\mu_n|^{2l}) \sum_{\substack{\mu_n < \eta_j \leq \mu_n + \delta  \\ \mu_n < \eta_p \leq \mu_n + \delta }  }   |c_j|^2 |\alpha_{jp}|^2
         \Bigg) \\
&     \leq  C_{11} \Bigg( \sum_{\eta_j \leq \mu_n} |c_j|^2 (1+|\eta_j|^{2l})    
       + (1+|\mu_n|^{2l}) \sum_{  \substack{ \mu_n < \eta_j \leq \mu_n + \delta}}  |c_j|^2 
         \Bigg)   \\
   &  \leq C_{12} \| \phi_n \|_{l}^2, 
    \end{align*}
where we have used (\ref{eq a_{jp}}) and $C_{10}, C_{11}, C_{12}$ are constant independent of $n$. 
Therefore $\| \pi_n : L^2_l \rightarrow L^2_l \| \leq C_{12}$.

\end{proof}

To prove Theorem \ref{thm spectral section mu mu+delta}, we need the following theorem and lemma: 

\begin{thm}\cite[$Theorem \ 1^*$]{Atiyah}   \label{thm Atiyah}
Let $W$ be a closed, spin manifold of odd dimension.   Then there is $C_* > 0$ such that each  interval of length $C_*$ contains an eigenvalue of  $D_A$. Here $A$ is a connection on a complex vector bundle $V$ over $W$ and $D_A : C^{\infty} (\mathbb{S} \otimes V) \rightarrow C^{\infty} (\mathbb{S} \otimes V)$ is the twisted Dirac operator. 
\end{thm}

Assume that $\operatorname{Ind} D = 0$. By \cite{MP}, we have a spectral section $P_0$ of $-D$. By Lemma 8 of \cite{MP},  using $P_0$,  we can construct a  smoothing operator $\mathbb{A} : \mathcal{E}_0 \rightarrow \mathcal{E}_{\infty}$ whose image is included in the space spanned by finitely many eigenvectors of $D$ such that $\ker D' = 0$ and
\[
              \mathcal{E}_0(D')_{-\infty}^{0} = P_0, 
\] 
where $D' = D + \mathbb{A}$.  Moreover there is $\nu_0 \gg 0$ such that  $\mathbb{A} = 0$ on $\mathcal{E}_0(D)_{-\infty}^{-\nu_0}$ and $\mathcal{E}_0(D)_{\nu_0}^{\infty}$.  From the construction of $\mathbb{A}$ in the proof of Lemma 8 of \cite{MP}, it is easy to see that for $\lambda \ll 0$ and $\mu \gg 0$, 
\[
       \mathcal{E}_0(D)_{-\infty}^{\mu} = \mathcal{E}_0(D')_{-\infty}^{\mu}, \ 
       \mathcal{E}_0(D)^{\infty}_{\lambda} = \mathcal{E}_0(D')^{\infty}_{\lambda}, \
       \mathcal{E}_0(D)_{\lambda}^{\mu} = \mathcal{E}_0(D')_{\lambda}^{\mu}. 
\]

\begin{lem} \label{lem dim E a a'}
There is a  constant  $\delta > 0$  such that  for any  $\mu > 0$ and $a, a' \in B$, 
\[
       \dim  \mathcal{E}_0(D'_{a})_{0}^{\mu}   \leq \dim \mathcal{E}_0(D'_{a'})_{0}^{\mu + \delta}. 
\]
\end{lem}

\begin{proof}
Put
\[
        M = \max \{  \| \nabla_v D' : L^2(\mathbb{S}) \rightarrow L^2(\mathbb{S}) \|  :  v \in TB,  \| v \| = 1   \}.
\]
Take a smooth path $\{ a_t \}_{t = 0}^{\ell}$ in $B$ from $a$ to $a'$ with $\| \frac{d}{dt} a_t \| = 1$.  Here $\ell$ is the length of the path.  Since $B$ is compact, we may assume that there is a constant $C > 0$ independent of $a, a'$ such that $\ell \leq C$.  Put
\[
     I = \big\{ t \in [0, \ell] : \forall s  \leq t, \dim \mathcal{E}_0(D_{a}')_0^{\mu}  \leq \dim \mathcal{E}_0(D_{a_s}')_0^{\mu + s M}  \big\}.
\]
Note that $0 \in I$ and that $I$ is closed in $[0, \ell]$ by the continuity of the eigenvalues of $D_{a_s}'$. It is sufficient to prove that $\sup I = \ell$.   

 Put $t_ 0 = \sup I$ and assume that $t_0 < \ell$. Choose $t_+ \in (t_0, \ell]$  with 
\[
   t_+ - t_0 \ll 1.
\]
Let $\nu_1(t), \dots, \nu_m(t)$ be the eigenvalues of $D_{a_t}'$ with 
\[
        0 <   \nu_1(t_0) \leq  \dots \leq \nu_{m}(t_0) \leq \mu + t_0 M
\]
such that $\nu_j(t)$ are continuous in $t \in [t_0, t_+]$ and $\dim \mathcal{E}(D_{a_{t_0}}')_0^{\mu + t_0 M} = m$.   Note that $t_0 \in I$ since $I$ is closed in $[0, \ell]$ and that 
\[
\dim \mathcal{E}_0(D_{a}')_0^{\mu} \leq m
\] 
by the definition of $I$.    Let $\nu'$ be the smallest eigenvalue of $D_{a_{t_0}}'$ with $\nu' > \nu_m(t_0)$.   We may assume that
\begin{equation} \label{eq nu' - nu_j}
      M(t_+ - t_0)  \ll \nu' - \nu_m(t_0).
\end{equation}
By Theorem 4.10 in page 291 of \cite{Kato},   we have
\[
          \operatorname{dist}(\nu_j(t), \Sigma(D_{a_{t_0}}')) \leq   M(t-t_0)
\]
for $t \in [t_0, t_+]$.  Here $\Sigma(D_{a_{t_0}}')$ is the set of eigenvalues of $D_{a_{t_0}}'$. It follows from this inequality and (\ref{eq nu' - nu_j}) that
\[
                   0 < \nu_j(t)  \leq \nu_m(t_0)  + M(t - t_0)  \leq \mu + Mt
\]
for $t \in [t_0, t_+]$ and $j \in \{ 1, \dots, m \}$. This implies that
\[
          \dim \mathcal{E}_0(D_{a}')_{0}^{\mu} \leq  m \leq  \dim \mathcal{E}_0(D_{a_t}')_0^{\mu + tM}
\]
fon $t \in [t_0, t_+]$. This is a contradiction and we obtain $t_0 = \ell$. 

\end{proof}

\noindent
{\it Proof of Theorem \ref{thm spectral section mu mu+delta}}   

For some $\mu \gg 0$, to construct a spectral section $P$ between $\mathcal{E}(D)_{-\infty}^{\mu}$ and $\mathcal{E}(D)_{-\infty}^{\mu + \delta}$, it is sufficient to find a frame $\{ f_1, \dots, f_r \}$ in $\mathcal{E}_0(D')_{0}^{\mu+\delta}$ such that
\begin{equation}   \label{eq E f}
         \mathcal{E}_0(D')_0^{\mu} \subset  
         \operatorname{span} \{ f_1, \dots, f_r \}  
         \subset \mathcal{E}_0(D')_{0}^{\mu + \delta}
\end{equation}
because the direct sum $\mathcal{E}_0(D')_{-\infty}^{0} \oplus \operatorname{span} \{ f_1, \dots, f_r \}$ is a spectral section between $\mathcal{E}_0(D)_{-\infty}^{\mu}$ and $\mathcal{E}_0(D)_{-\infty}^{\mu + \delta}$.

Put $d = \dim B$.    Fix an integer $N$ with $N \gg d$. By Theorem \ref{thm Atiyah},  there is $\delta_0 > 0$ such that
\begin{equation} \label{eq N}
         \dim ( \mathcal{E}_0(D_a'))_{\mu}^{\mu+\delta_0}  \geq N
\end{equation}
for all  $a \in B$ and $\mu \in \mathbb{R}$. By Lemma \ref{lem dim E a a'},  we may  assume that
\begin{equation}  \label{eq dim E_b E_b'}
     \dim \mathcal{E}_0(D_{a'}')_{0}^{\mu - \delta_0} 
     \leq \dim  \mathcal{E}_0(D_{a}')^{\mu}_{0} 
     \leq \dim  \mathcal{E}_0(D_{a'}')_{0}^{\mu+\delta_0}
\end{equation}
for all $a, a' \in B$ and $\mu \in \mathbb{R}$ with $\mu > \delta_0$. 

Fix a positive number $\delta$ with $\delta > 10 \delta_0$.  Take $\mu \in \mathbb{R}$ with $\mu \gg 0$. For $j \in \{ 0, 1, \dots, d \}$,  choose positive numbers 
\[
      \mu < a_j^{-} < b_{j}^- <   c^- < c^{+}   <   a_{j}^+ < b_j^{+} < \mu + \delta
\]
such that
\begin{align*}
       & b_{j+1}^{-} < a_{j}^{-}, \quad   b_{j}^+ < a_{j+1}^+,  \\
       &  b_j^- < c^- - 2\delta_0, \quad c^+ + 2\delta_0 < a_j^+. 
 \end{align*}
Take a CW complex structure of $B$ such that for each $j$-dimensional cell $e$ there are real numbers $\mu^{-}(e)$, $\mu^{+}(e)$ such that $\mu^{-}(e)$, $\mu^{+}(e)$ are spectral gaps of $D'_a$ for $a \in e$ with
\[
          a_{j}^{-}  <    \mu^{-}(e)  < b_{j}^-, \quad
          a_{j}^{+} < \mu^{+}(e) < b_{j}^+. 
\]
Choose a $0$-dimensional cell $e_0 ( = 1 pt)$ and $\mu_0 \in (c^-, c^{+})$, and put $r := \dim  \mathcal{E}_0(D_{e_0}')_{0}^{\mu_0}$.

\begin{lem} \label{lem r}
For any cell $e$ and $a \in e$,  we have
\[
       \dim  \mathcal{E}_0(D_a')_{0}^{\mu^-(e)} + N \leq r \leq \dim  \mathcal{E}_0(D_a')_{0}^{\mu^+(e)} - N. 
\]
\end{lem}

\begin{proof}
Because $\mu_0 + 2\delta_0 < \mu^+(e)$, by (\ref{eq N}) and (\ref{eq dim E_b E_b'}),  we have
  \begin{align*}
         \dim \mathcal{E}_0(D_a')_{0}^{\mu^+(e)}  
         & \geq \dim \mathcal{E}_0(D'_a)_{0}^{\mu_0+2\delta_0} \\
        &  = \dim \mathcal{E}_0(D_a')_{0}^{\mu_0+\delta_0} + \dim \mathcal{E}_0(D_a')_{\mu_0 + \delta_0}^{\mu_0 + 2\delta_0}   \\
       &  \geq \dim \mathcal{E}_0(D'_{e_0})_{0}^{\mu_0} + N \\
        & = r + N. 
\end{align*}
Hence 
\[
   r \leq \dim \mathcal{E}_0(D_a')_{0}^{\mu^+(e)} - N.
\]
The proof of the inequality $\dim \mathcal{E}_0(D_a')_{0}^{\mu^-(e)} + N \leq r$ is similar. 
\end{proof}

By Lemma \ref{lem r},  for each $0$-dimensional cell $e$, we can take a frame (meaning a linearly independent collection)  $\{ f_1, \dots, f_{r} \}$ of $\mathcal{E}_0(D'_e)_0^{\mu^+(e)}$ such that
\[
            \mathcal{E}_0(D'_{e})_{0}^{\mu^-(e)} \subset  \langle  f_1, \dots, f_r \rangle  \subset \mathcal{E}_0(D_e')_{0}^{\mu^+(e)} . 
\]

Assume that we have a frame  $\{ f_1, \dots, f_r \}$ in $\mathcal{E}_0(D')_{0}^{\infty}$ on the $(j-1)$-dimensional skeleton of $B$ such that
\[
         \mathcal{E}_0(D_{a}')_{0}^{\mu^-(e)} \subset  \langle f_{1, a}, \dots, f_{r, a} \rangle  \subset \mathcal{E}_0(D_a')_{0}^{\mu^+(e)}
\]
for each cell $e$ with $\dim e \leq j-1$ and $a \in e$.

Take a cell $e'$ of $B$ with $\dim e' = j$. 
Note that $\mathcal{E}_0(D')_0^{\mu^+(e')}$, $\mathcal{E}_0(D')^{\mu^-(e')}_0$ are  vector bundles over $e'$.  We denote by $\mathcal{F}$ the bundle
\[
          \bigcup_{a \in e'}  \{   \text{frames of rank $r$ in $\mathcal{E}_0(D'_a)_{0}^{\mu^+(e')}$} \}
\]
over $e'$.

Note that $\mu^+(e) \leq \mu^+(e')$ for any cell $e$ with $\dim e \leq j-1$.   Hence the frame $\{f_1, \dots, f_r \}$  defines a section of  $\mathcal{F}$ on the boundary $\partial e'$. 

We have a homeomorphism 
\[
         \mathcal{F}_a \cong GL(m;\mathbb{C}) / GL(m-r; \mathbb{C}),
\]
where $a \in e'$, $\mathcal{F}_a$ is the fiber of $\mathcal{F}$ over $a$ and $m = \dim \mathcal{E}_0(D_a')_{0}^{\mu^+(e')}$.   By Lemma \ref{lem r}, 
\[
         m  = \dim \mathcal{E}_0(D_{a}')_{0}^{\mu^+(e')} \geq  r + N. 
\] 
Because $N \gg d$, we have
\[
         m, m - r  \gg d.
\]
By the homotopy exact sequence, 
\[
       \pi_{i}(GL(m;\mathbb{C}) / GL(m-r; \mathbb{C})) = 0
\]
for $i = 0, 1, \dots, d$.   Therefore we can extend $\{ f_1, \dots, f_r \}$ to a frame in $\mathcal{E}_{0}(D')_0^{\mu^+(e')}$ over $e'$.  We will denote the extended frame on $e'$ by the same notation $\{ f_1, \dots, f_r \}$. 
We will modify $\{ f_1, \dots, f_r \}$ on the interior $\operatorname{Int} e'$ of $e'$ to get a frame $\{ f_1', \dots, f_r' \}$  such that 
\[
           \mathcal{E}_0(D')_{0}^{\mu^-(e')} 
           \subset \langle f_1', \dots, f_r' \rangle
           \subset \mathcal{E}_0(D')_0^{\mu^+(e')}
\]
on $e'$.  
Since $\mu^-(e')  \leq \mu^-(e)$, on $\partial e'$  we have
\[
          \mathcal{E}_0(D')_0^{\mu^-(e')} \subset \mathcal{E}_0(D')_{0}^{\mu^-(e)} \subset \operatorname{span} \{ f_1, \dots, f_r \}. 
\]

As mentioned before, $\mathcal{E}_0(D')_0^{\mu^-(e')}$ and $\mathcal{E}_0(D')_0^{\mu^+(e')}$ are vector bundles over $e'$.  
Let 
\[
    p :  \mathcal{E}_0(D')^{\mu^+(e')}_0 \Big|_{e'} \rightarrow  \mathcal{E}_0(D')^{\mu^-(e')}_0 \Big|_{e'}
\] 
be the orthogonal projection.

\begin{lem} 
We can  perturb $f_1, \dots, f_r$ slightly on $\operatorname{Int} e'$ such that
\[
         \mathcal{E}_0(D')_0^{\mu^-(e')}   =   p(  \langle f_1, \dots, f_r \rangle )
\]
on $e'$. Here $\operatorname{Int} e'$ is the interior of $e'$. 
\end{lem}

\begin{proof}
We may suppose that
\[
     \mathcal{E}_0(D')^{\mu^+(e')}_{0}  \Big|_{e'} 
    = e' \times (\mathbb{C}^{n} \oplus \mathbb{C}^{n'}),  \quad
        \mathcal{E'}_0(D')^{\mu^-(e')}_{0} \Big|_{e'} 
    = e' \times  (\mathbb{C}^{n} \oplus \{ 0 \}). 
\]
For each $a \in e'$, we can write
\[
     f_{j, a} = g_{j, a} \oplus g_{j, a}', 
\]
where
\[
      g_{j, a} \in \mathbb{C}^{n}, \quad g_{j, a}' \in \mathbb{C}^{n'}. 
\]
Note that
\[
     \mathbb{C}^n = p( \langle f_{1, a}, \dots, f_{r, a}  \rangle)
\]
if and only if the $(n \times r)$-matrix $(g_{1, a} \dots g_{r, a})$ is of rank $n$.  Let $M$ be the set of  $(n \times r)$-complex matrices, which is naturally a smooth manifold of dimension $2nr$. We denote by $R_l$ the set of $(n \times r)$-matrices of rank $l$. Then $R_l$ is a smooth submanifold of $M$ of codimension $2(n-l)(r-l)$. If $l \leq n-1$ we have
\[
       \operatorname{codim}_{\mathbb{R}} (R_l \subset M) = 2(n-l)(r-l) \geq 2(r- n+ 1) \geq 2(N + 1) \gg d. 
\]
Here we have used 
\[
          n = \dim \mathcal{E}_0(D'_a)_{0}^{\mu^-(e')} \leq r - N.
\]
 See Lemma \ref{lem r}.
So we can slightly perturb $(g_1 \dots g_{r})$ on $\operatorname{Int} e'$ such that for all $a \in e'$ and $l \in \{ 0, 1, \dots, n-1 \}$ 
\[
      (g_{1, a} \dots g_{r, a}) \not \in R_l.
\]
Hence the rank of $(g_{1, a} \dots g_{r, a})$ is $n$.  Therefore $\mathbb{C}^n = p(  \langle f_{1, a}, \dots, f_{r, a} \rangle )$ for all $a \in e'$. We can assume that the perturbation is enough small such that after the perturbation, $f_1, \dots, f_r$ is still linear independent.

\end{proof}

By this lemma, we may suppose that
\[
          \mathcal{E}_0(D')_{0}^{\mu^-(e')} = p( \langle f_1, \dots, f_r \rangle)
\]
on $e'$. 
For $a \in e'$, define $F_a : \mathbb{C}^r \rightarrow \mathcal{E}_0(D'_{a})_{0}^{\mu^+(e')}$ by
\[
     F_a(c_1, \dots, c_r) = c_1 f_{1, a} + \dots + c_r f_{r, a}. 
\]
We have
\[
       \mathcal{E}_0(D_a')_{0}^{\mu^-(e')} = \operatorname{im} ( p \circ F_{a})
\]
Put
\[
       K := \bigcup_{a \in e'} \ker  ( p \circ F_a ). 
\]
Then $K$ is a subbundle of the trivial bundle $\underline{\mathbb{C}}^{r}$ on $e'$.  We have the orthogonal decomposition
\[
         \underline{\mathbb{C}}^r  =  K \oplus K^{\bot}. 
\]
We define
\[
      F' : \underline{\mathbb{C}}^r \rightarrow  \mathcal{E}(D')_0^{\mu^+(e')} \Big|_{e'}
\]
by
\[
         F' =  F|_{ K }  + p \circ F|_{ K^{\bot} }.
\]
Then 
\[
            \mathcal{E}(D')_0^{\mu^-(e')} \Big|_{e'} \subset \operatorname{im} F'.
\]

\begin{lem}

\begin{enumerate}
\item
 $F = F'$ on $\partial e'$. 

\item
The map $F'$ is injective on $e'$.

\end{enumerate}
\end{lem}

\begin{proof}

(1)
Take $a \in \partial e'$.  It is sufficient to show that $F_a|_{K^{\bot}} = F_{a}'|_{K^{\bot}}$. 
Recall that 
\[
    \mathcal{E}_0(D_a') ^{\mu^-(e')}_0  \subset \operatorname{im} F_a. 
\]       
Since $\operatorname{im} F_a|_{K_a} \subset (\mathcal{E}_0(D_a')_0^{\mu^-(e')})^{\bot}$ and $\dim \mathcal{E}_0(D_{a}')_0^{\mu^-(e')} = \dim K^{\bot}_{a}$, we have
\[
         \operatorname{im} (F_a|_{K^{\bot}}) = \mathcal{E}_0(D_a')_{0}^{\mu^-(e')}. 
\]
Therefore for $v \in K_{a}^{\bot}$,  $F'_v(v) = p F_a(v) = F_a(v)$.

(2)
Suppose that
\[
       F'(v, v') =0 
\]
for $v \in K, v' \in K^{\bot}$.  Then
\[
         F(v) + p F(v') = 0. 
\]
So we have
\[
      p F(v) + p^2 F(v') = 0. 
\]
Since $v \in K = \ker p \circ F$ and $p^2 = p$, 
\[
        p F(v') = 0. 
 \]
Because $p \circ F$ is an isomorphism on $K^{\bot}$, we have
\[
        v ' = 0. 
\]
Hence
\[
          F(v) = 0
\]
which implies that $v = 0$ because  $F$ is injective. 
\end{proof}

Put
\[
     f_{1, a}' := F'_a(e_1), \dots, f_{r, a}' := F'_a(e_r)
\]
for $a \in e'$. Here $e_1, \dots, e_r$ is the standard basis of $\mathbb{C}^{r}$. 
Then the frame $\{ f_1', \dots, f_r' \}$ of $\mathcal{E}_0(D')_{0}^{\mu^+(e')}$ on $e'$, which is an extension of the frame on $\partial e'$,   has the property that
\[
         \mathcal{E}(D')_0^{\mu^-(e')} 
         \subset \langle f_1', \dots, f_r' \rangle
         \subset \mathcal{E}(D')_0^{\mu^+(e')}. 
\]

We have obtained a frame $f_1, \dots, f_r$ satisfying (\ref{eq E f}). 
Putting  
\[
P = \mathcal{E}_0(D')_{-\infty}^{0} \oplus \langle f_1, \dots, f_r \rangle,
\]
 we obtain a spectral section  with
\[
     \mathcal{E}_0(D)_{-\infty}^{\mu} \subset P \subset \mathcal{E}_0(D)_{-\infty}^{\mu + \delta}, 
\]
where $\delta > 0$ is a constant independent of $\mu$.

Take another positive number $\tilde{\mu}$ with $\mu \ll \tilde{\mu}$.  Doing this procedure one more time,  we get a frame $\{ \tilde{f}_{1}, \dots, \tilde{f}_{s} \}$  of $P^{\bot} \cap \mathcal{E}(D')_{0}^{\tilde{\mu} + \delta}$ such that 
\[
      \mathcal{E}_0(D)_{-\infty}^{\tilde{\mu}} 
      \subset  P \oplus \langle   \tilde{f}_1, \dots, \tilde{f}_s \rangle
      \subset \mathcal{E}_0(D)_{-\infty}^{ \tilde{\mu} + \delta}. 
\]
Repeating this, we get a sequence of spectral sections satisfying the conditions of Theorem \ref{thm spectral section mu mu+delta}. \qed

We will state a $\rm{Pin}(2)$-equivariant version of Theorem \ref{thm spectral section mu mu+delta}.  If $\frak{s}$ is a self-conjugate $\mathrm{spin}^c$ structure of $Y$,  we have an action of $\rm{Pin}(2)$ on  $\mathcal{E}_k$. The action is induced by the action of $\rm{Pin}(2)$ on $\mathcal{H}^1(Y) \times L^2_k(\mathbb{S})$, which is an extension of the $S^1$-action,  defined by
\[
         j(a, \phi)  = (-a, j \phi). 
\]
The Dirac operator $D$ is $\rm{Pin(2)}$-equivariant and we have the index
\[
        \operatorname{ind} D \in KQ^1(B). 
\]
Here $KQ^1(B)$ is the quaternionic K-theory defined in \cite{dupont}, which is used in \cite{lin_rokhlin}.

\begin{thm}  \label{thm Pin(2) spectral sections}
If $\frak{s}$ is a self-conjugate $\mathrm{spin}^c$ structure of $Y$ and $\operatorname{ind} D = 0$ in $KQ^1(B)$, then we have a sequence $P_n$ of ${\rm Pin}(2)$-equivariant spectral sections having the properties of Theorem \ref{thm spectral section mu mu+delta}. 
\end{thm}

\begin{proof}  
We will show an outline of the proof.  Since $\operatorname{Ind} D = 0$ in $KQ^1(B)$,  it follows from the arguments in Section 1 of \cite{lin_rokhlin} that the family $D$ of Dirac operators is $\rm{Pin}(2)$-equivariantly homotopic to a constant family.  Hence we can apply the proof of Proposition 1 of \cite{MP} to show that there exists a $\rm{Pin}(2)$-equivariant spectral section $P_0$ of $-D$.  

Choose a CW complex structure of $B$ such that for each cell $e$, $(-1) \cdot e$ is also a cell.  Note that 
\[ 
         \pi_{i}(Sp(m)/Sp(m-r)) = 0
\]
for $i = 1, \dots, d$, provided that $m, m-r \gg d$. 
Hence for $\mu \gg 0$,  we can construct a $\rm{Pin}(2)$-equivariant frame $f_1, \dots, f_r$ of $P_0^{\bot}$ with
\[
             \mathcal{E}_{0}(D')_0^{\mu} \subset \langle f_1, \dots, f_r \rangle \subset \mathcal{E}_0(D')_{0}^{\mu + \delta}
\]
as in the proof of Theorem \ref{thm spectral section mu mu+delta}.  Here $\delta$ is the positive constant from the proof of Theorem \ref{thm spectral section mu mu+delta}.  Then
\[
       P_0 \oplus \langle f_1, \dots, f_r \rangle
\]
is a $\rm{Pin}(2)$-equivariant spectral section between $\mathcal{E}_0(D)_{-\infty}^{\mu}$ and $\mathcal{E}_0(D)_{-\infty}^{\mu+\delta}$. Repeating this construction, we obtain the desired sequence $P_n$. 
\end{proof}


\section{Derivative of projections}\label{subsec:derivatives}

Let $D : \mathcal{E}_k \rightarrow \mathcal{E}_{k-1}$ be the original Dirac operator.  Recall that we have a canonical flat connection $\nabla$ on $\mathcal{E}_k$.  See Section \ref{sec:main results}. 
Note that for $a \in B$, $v \in T_a B = \mathcal{H}^1(Y)$, we have 
\[
      \nabla_v D =  \frac{d}{dt} \bigg|_{t=0}  D_{a+tv} 
        =     \frac{d}{dt} \bigg|_{t=0} (D_a + t\rho(v))
        = \rho(v). 
\]
Here $\rho(v)$ is the Clifford multiplication.
Since $v$ is a harmonic (and hence smooth) $1$-form, we have $\| v \|_{k} < \infty$ for any $k \geq 0$.  Therefore  $\nabla_v D$ is a bounded operator from $L^2_k(\mathbb{S})$ to $L^2_k(\mathbb{S})$ for each $k \geq 0$.

Take $ \mu \in \mathbb{R}$.    We write $\pi^{\mu}_{-\infty}$ for the $L^2$-projection on  $\mathcal{E}_0(D)_{-\infty}^{\mu}$.  Similarly, $\pi_{\lambda}^{\mu}$ is the $L^2$-projection on $\mathcal{E}_0(D)_{\lambda}^{\mu}$.  We have


\begin{prop} \label{prop nabla pi}
Fix $a \in B$. 
Let $\{ e_{i} \}_{i = -\infty}^{\infty}$ be an  $L^2$-orthonormal basis of $L^2(\mathbb{S})$ such that
\[
         D_a e_{i} = \eta_{i}  e_i.  
\]
Here $\eta_i$ are the eigenvalues of $D_a$.  Take $\lambda, \mu \in \mathbb{R}$ with $\lambda < \mu$. Suppose that $\lambda, \mu$ are not eigenvalues of $D_a$. 
For $v \in T_a B = \mathcal{H}^1(Y)$, 
\begin{equation}  \label{eq nabla pi lambda mu}
  \begin{aligned}
     & \langle (\nabla_v \pi_{\lambda}^{\mu}) e_i, e_j  \rangle_0  \\
  &    = 
         \begin{cases} 
            \frac{\langle \rho(v) e_i, e_j \rangle_0}{\eta_i - \eta_j}  & \text{if $\eta_i < \lambda < \eta_j < \mu$ or $\lambda < \eta_j < \mu < \eta_i$ }  \\
         \frac{\langle \rho(v) e_i, e_j \rangle_0}{\eta_j - \eta_i} &  \text{if $\eta_j < \lambda < \eta_i < \mu$ or $\lambda < \eta_i < \mu < \eta_j$},   \\
         0 & \text{otherwise, }
          \end{cases} 
 \end{aligned}
\end{equation}
and
\begin{equation}  \label{eq nabla pi mu}
     \langle (\nabla_v \pi_{-\infty}^{\mu})  e_i,  e_j \rangle_0 =
       \begin{cases}
    \frac{\langle  \rho(v) e_i, e_j \rangle_0}{ \eta_i - \eta_j } & \text{if $\eta_j < \mu < \eta_i$, }  \\
   \frac{\langle\rho(v)e_i,  e_j \rangle_0}{\eta_j - \eta_i }   &  \text{if $\eta_i < \mu < \eta_j$,} \\
       0 & \text{otherwise. }
      \end{cases} 
\end{equation}
Here $\rho(v)$ is the Clifford multiplication by $v$.  
\end{prop}

\begin{proof}
Since the connection $\nabla$ is induced by the trivial connection on $\mathcal{H}^1(Y) \times C^{\infty}(\mathbb{S})$,  to compute $\nabla_v \pi_{\lambda}^{\mu}$, $\nabla_v \pi_{-\infty}^{\mu}$, we can do computations over $\mathcal{H}^1(Y)$ where we have the canonical trivialization  and the covariant derivative is equal to the usual exterior derivative.

Take a loop $\Gamma_{\lambda}^{\mu}$ in $\C$ defined by 
\[
 \begin{split}
       \Gamma_{\lambda}^{\mu}
        &= \{  x - i \epsilon | \lambda \leq x \leq \mu  \}  \cup 
            \{  \mu + i y | -\epsilon \leq y \leq \epsilon   \}   \\
       & \quad  \cup   \{  x + i \epsilon | \lambda \leq x \leq \mu   \}   \cup
            \{  \lambda + iy | -\epsilon \leq y \leq \epsilon   \} 
\end{split}
\]
for some $\epsilon > 0$.   We orient $\Gamma_{\lambda}^{\mu}$ counterclockwise. 
We will show that for $\phi \in C^{\infty}(\mathbb{S})$, 
\[
      (\pi_a)_{\lambda}^{\mu} \phi = \frac{1}{2\pi i} \int_{\Gamma_{\lambda}^{\mu}} (z - D_a)^{-1} \phi dz. 
\]
See also \cite[Chapter II, Section 4]{Kato}.
We can write
\[
       \phi = \sum_{i=-\infty}^{\infty} c_i e_i
\]
for some $c_i \in \C$ with 
\[
   \sum_{i=-\infty}^{\infty} |c_i|^2 (1 + |\eta_i|^{2k}) < \infty
\] 
for any $k \geq 0$. 
For $z \in \C$ which is not an eigenvalue of $D_a$,  the operator $z - D_a$ is invertible and 
\begin{equation}  \label{eq:(z-D_a)^{-1}}
         (z - D_a)^{-1} \phi = \sum_{i =-\infty}^{\infty} \frac{c_i}{z-\eta_i} e_i. 
\end{equation}
Note that the sum in (\ref{eq:(z-D_a)^{-1}}) converges uniformly on $\Gamma_{\lambda}^{\mu}$ in the $L^2_k$-norm for any $k \geq 0$ since 
\[
      \Bigg|   \frac{c_i}{z - \eta_i}  \Bigg| \leq |c_i|   \quad (z \in \Gamma_{\lambda}^{\mu})
\]
 if  $| i | \gg 0$. 
Hence, by  the residue formula,
  \begin{align*}
      \frac{1}{2 \pi i} \int_{\Gamma_{\lambda}^{\mu}} (z - D_a)^{-1}(\phi) dz
      & = \sum_{i=-\infty}^{\infty}  \frac{1}{2\pi i} \Bigg( \int_{\Gamma_{\lambda}^{\mu}} \frac{c_i}{z - \eta_i}  dz  \Bigg)   e_i   \\
      & = \sum_{\lambda < \eta_i < \mu}  c_i  e_i  \\
      & = (\pi_a)_{\lambda}^{\mu} \phi. 
 \end{align*}
Here we have used the fact that we are allowed to take the term-by-term integration because of the uniform convergence.

Take $v \in T_a B = \mathcal{H}^1(Y)$.  
Then by the above  formula for $\pi_{\lambda}^{\mu}$,  we have
  \begin{align*}
    ( \nabla_{v} \pi_{\lambda}^{\mu}) e_i
     &  = - \frac{1}{2 \pi i} \int_{\Gamma_{\lambda}^{\mu}} (z-D_a)^{-1} (\nabla_v D) (z - D_a)^{-1} e_i dz  \\
   & = - \frac{1}{2\pi i}  \int_{\Gamma_{\lambda}^{\mu}} (z - D_a)^{-1} \rho(v) (z - \eta_i)^{-1} e_i dz  \\
   & = -\frac{1}{2\pi i} \int_{\Gamma_{\lambda}^{\mu}} ( z - \eta_i)^{-1} (z - D_a)^{-1} \rho(v) e_i dz. 
 \end{align*}
Therefore
   \begin{align*}
      \langle (\nabla_v \pi_{\lambda}^{\mu}) e_i, e_j   \rangle_0
      & =  -\frac{1}{2\pi i}  \int_{\Gamma_{\lambda}^{\mu}} (z-\eta_i)^{-1}   \langle  \rho(v) e_i, (\bar{z} - D_a)^{-1} e_j      \rangle_0 dz    \\
      &= -\frac{1}{2 \pi i} \int_{\Gamma_{\lambda}^{\mu}}  (z - \eta_i)^{-1} \langle \rho(v) e_i, (\bar{z} - \eta_j)^{-1} e_j \rangle_0 dz   \\
     & = -\frac{\langle \rho(v) e_i,  e_j \rangle_0}{2 \pi i} \int_{\Gamma_{\lambda}^{\mu}} (z - \eta_i)^{-1} (z - \eta_j)^{-1} dz.
      \end{align*}
From this, we obtain the formula (\ref{eq nabla pi lambda mu}) for $\langle (\nabla_v \pi_{\lambda}^{\mu}) e_i, e_j \rangle_0$.

Note that  since $\rho(v)$ defines a bounded operator $L^2 \rightarrow L^2$,  we can see that the operators $(T_a)_{\lambda}^{\mu}$, $(T_a)_{-\infty}^{\mu}$ defined by the right hand side of (\ref{eq nabla pi lambda mu}) and (\ref{eq nabla pi mu}) are bounded from $L^2$ to $L^2$. Moreover for each compact set $K$ in $\mathcal{H}^1(Y)$, $(T_a)_{\lambda}^{\mu}$ converges to $(T_a)_{-\infty}^{\mu}$ on $K$ uniformly  as $\lambda \rightarrow -\infty$.  We have
   \begin{align*}
      \langle  (\pi_{a+tv})_{\lambda}^{\mu} (e_i), e_j \rangle_0 - 
      \langle (\pi_{a})_{\lambda}^{\mu} (e_i), e_j \rangle_0    
      & \quad = \int_{0}^{t} \frac{d}{ds} \langle  \pi_{a+sv} e_i, e_j \rangle_0 ds \\
      & \quad = \int_0^{t} \langle  (\nabla_v \pi_{a+sv} e_i), e_j    \rangle_0 ds \\
      & \quad = \int_{0}^{t} \langle (T_{a+sv})_{\lambda}^{\mu}(e_i), e_j \rangle_0 ds.
 \end{align*}
Taking the limit as $\lambda \rightarrow -\infty$,  we obtain
\[
     \langle (\pi_{a + tv})_{-\infty}^{\mu} (e_i), e_j    \rangle_0 - \langle (\pi_a)_{-\infty}^{\mu} e_i  \rangle_0
      =
      \int_{0}^{t} \langle   (T_{a+sv})_{-\infty}^{\mu} (e_i), e_j  \rangle_0 ds.
\]
Therefore
\[
       \langle (\nabla_v \pi_{-\infty}^{\mu}) e_i, e_j   \rangle_0 
       = \frac{d}{dt} \bigg|_{t=0}  \langle (\pi_{a + tv})_{-\infty}^{\mu} (e_i), e_j    \rangle_0
       = \langle (T_a)_{-\infty}^{\mu} (e_i), e_j \rangle_0. 
\]
We have obtained (\ref{eq nabla pi mu}). 
\end{proof}

\begin{cor} \label{cor nabla pi}
Suppose that $\mu$ is not an eigenvalue of $D_a$. 
Then for each $v \in TB$ and nonnegative $k$,
\[
    \nabla_v  \pi_{-\infty}^{\mu} : L^2_k(\mathbb{S}) \rightarrow L^2_{k+1}(\mathbb{S})
\]
is a bounded operator.    Moreover if $| \mu | \geq 2$,  $\alpha<k$ and if there is no eigenvalue of $D_a$ in the interval $[\mu - \mu^{-\alpha}, \mu + \mu^{-\alpha}]$,     for $v \in T_aB$ with $\| v \| \leq 1$, 
\[
     \| \nabla_v  \pi_{-\infty}^{\mu}  : 
       L^2_k(\mathbb{S}) \rightarrow L^2_{k- \alpha}(\mathbb{S})   \|  \leq C. 
\]
Here $C > 0$ is a constant independent of $v, \mu$. 
Similar statements hold for $\nabla_v \pi_{\lambda}^{\mu}$, $\nabla_{v} \pi_{\mu}^{\infty}$. 
\end{cor}

\begin{proof}
Let $e_i, \eta_i$ be as in Proposition \ref{prop nabla pi}. 
Take $v \in T_aB  = \mathcal{H}^1(Y)$. Put
\[
     \rho_{ij} := \langle \rho(v) e_i, e_j  \rangle_{0}. 
\]
Take $\phi = \sum_{i} c_i e_i \in C^{\infty}(\mathbb{S})$ with $\| \phi \|_{k} = 1$.   Since $\rho(v)$ is a bounded operator from $L^2_k$ to $L^2_k$ we have
\[
    \| \rho(v) \phi \|_{k}^2 
    = \Big\lVert  \sum_{i, j}   c_i \rho_{ij} e_j \Big\rVert_k^2
    = \sum_{j} \Big|  \sum_{i}   c_i  \rho_{ij}   \Big|^2 ( 1 + |\eta_j|^{2k})
    \leq C_1, 
\]
where $C_1 > 0$ is a constant independent of $\phi$.

By Proposition \ref{prop nabla pi}, we have

\begin{align*}
     & \Big\lVert  (\nabla_v \pi_{-\infty}^{\mu} ) \phi \Big\rVert_{k+1}^{2}   \\
    &=  \Bigg\lVert  \sum_{\eta_i < \mu < \eta_j}  \frac{c_i  \rho_{ij}}{\eta_j - \eta_i} e_j
      + \sum_{\eta_j < \mu < \eta_i } \frac{c_i \rho_{ij}}{\eta_i - \eta_j} e_j   \Bigg\rVert_{k+1}^{2}   \\
   &=   \sum_{\mu < \eta_j} \Bigg|  \sum_{ \eta_i <  \mu }   \frac{c_i \rho_{ij}}{\eta_{j} - \eta_i}     \Bigg|^{2}  ( 1 + |\eta_{j}|^{2k+2} )   +
   \sum_{\eta_j < \mu} \Bigg| \sum_{\mu < \eta_i} \frac{c_{i} \rho_{ij}}{\eta_{i} - \eta_{j}}   \Bigg|^2 ( 1 + | \eta_j|^{2k+2}).
 \end{align*}
Note that there is a constant $C_2 > 0$ independent of $i, j$ such that
\[
    \frac{1+|\eta_j|^{2k+2}}{|\eta_j - \eta_i|^2}
    \leq C_2 ( 1 + |\eta_{j}|^{2k})
\]
for $i, j$ with $\eta_i < \mu < \eta_j$ or $\eta_j < \mu < \eta_i$. 
Hence
\[
     \|  (\nabla_{v} \pi_{-\infty}^{\mu}) \phi   \|_{k+1}^2 
     \leq 
     C_2 \sum_{j}   \left|  \sum_{i}  c_{i} \rho_{ij}  \right|^2 ( 1 + |\eta_{j}|^{2k})
     \leq 
     C_1 C_2. 
\]
Therefore $\nabla_v \pi_{-\infty}^{\mu}$  extends to a bounded operator $L^2_k \rightarrow L^2_{k+1}$. 

Next assume that there is no eigenvalue of $D_a$ in the interval $[\mu - \mu^{-\alpha}, \mu + \mu^{-\alpha}]$.  Take $v \in T_a B$ with $\| v \| = 1$. 
It is easy to see that if $\eta_i < \mu < \eta_j$ or $\eta_j < \mu < \eta_i$ we have
\[
    \frac{1 + |\eta_j|^{2k - 2\alpha}}{|\eta_i - \eta_j|^2}
    \leq C_3  ( 1 + | \eta_j|^{2k}), 
\]
where $C_3 > 0$ is independent of $i, j$.  It follows from this and Proposition \ref{prop nabla pi}  that 
\[
     \| \nabla_{v} \pi_{-\infty}^{\mu} : L^2_k \rightarrow L^2_{k-\alpha} \| \leq C_4, 
\]
where $C_4 > 0$ is a constant independent of $\mu$ and $v$. 
\end{proof}

\begin{lem} \label{lem spectral gap}
Fix positive numbers $\alpha, \beta$ with $\alpha + 3 < \beta$ and $a \in \mathcal{H}^1(Y)$.
For $\mu \in \mathbb{R}$ with $|\mu|  \gg 0$,  there exists $\mu'  \in (\mu - |\mu|^{-\alpha}, \mu + |\mu|^{-\alpha} ]$ such that there is no eigenvalue of $D_a$ in the interval $(\mu' - |\mu|^{-\beta}, \mu' + |\mu|^{-\beta}]$.  

\end{lem}

\begin{proof}
Suppose that the statements  is not true.  Then there is a sequence $\mu_n$ with $|\mu_n| \rightarrow \infty$ such that for any $\mu' \in (\mu_n - |\mu_n|^{-\alpha}, \mu_n + |\mu_n|^{-\alpha}]$  there is an eigenvalue of $D_a$ in $(\mu' - |\mu_n|^{-\beta}, \mu' + |\mu_n|^{-\beta}]$. Therefore for each integer $m$ with $1 \leq m \leq |\mu_n|^{\beta - \alpha}$, there is an eigenvalue of $D_a$ in the interval $(\mu_n +  (m-1) |\mu_n|^{-\beta}, \mu_n + m |\mu_n|^{-\beta}]$.  This implies that 
\[
        \dim (\mathcal{E}_0(D_a))_{\mu_n - |\mu_n|^{-\alpha}}^{\mu_n + |\mu_n|^{-\alpha}} 
        \geq
        |\mu_n|^{\beta-\alpha} -1.
\]
On the other hand, by the Weyl law, 
\[
        \dim (\mathcal{E}(D_a))_{\mu_{n} - |\mu_{n}|^{-\alpha} }^{\mu_{n}+|\mu_n|^{-\alpha} } \leq C |\mu_n|^{3}.  
\]
We have obtained a contradiction. 


\end{proof}

\begin{cor} \label{cor nabla pi 2}
For $\mu \in \mathbb{R}$ with $|\mu| \gg 0$, there is $\mu' \in [\mu, \mu+1]$, such that for $v \in TB$ with $\| v \| = 1$, 
\[
       \| \nabla_{v} \pi^{\mu'}_{-\infty} : L^2_k(\mathbb{S}) \rightarrow L^2_{k-4}(\mathbb{S}) \| \leq C. 
\]
Here $C > 0$ is a constant independent of $v, \mu$. 
Similar statements hold for $\pi_{\lambda}^{\infty}$, $\pi_{\lambda}^{\mu}$. 
\end{cor}

\begin{proof}
This is a direct consequence of Corollary \ref{cor nabla pi} and Lemma \ref{lem spectral gap}.  
\end{proof}

\begin{prop} \label{prop L^2 nabla pi_P} 
Take a non-negative real number  $m$ and  a smooth spectral section $P$ of $-D$ with
\[
       (\mathcal{E}_0(D))^{\mu_-}_{-\infty} \subset P \subset (\mathcal{E}_0(D))^{\mu_+}_{-\infty}. 
\]
Let $\pi_{P}$ be the $L^2$-projection onto $P$. Then for each $v \in TB$, $\nabla_v \pi_P$ is a bounded operator from $L^2_m(\mathbb{S})$ to $L^2_{m+1}(\mathbb{S})$. 
\end{prop}

\begin{proof}
We can take an open covering $\{ U_{i} \}_{i=1}^{N}$ of $B$ such that there are real numbers $\lambda_i$, $\nu_i$ with $\lambda_i < \mu_-$, $\mu_+ < \nu_i$,  which are not eigenvalues of $D_a$ for $a \in U_i$.  Also we may assume that for each $i$, we have a trivialization 
\begin{equation}  \label{eq E local trivialization}
         \mathcal{E}_0|_{U_i} \cong U_i \times L^2(\mathbb{S})
\end{equation}
such that the flat connection $\nabla$ is equal to the exterior derivative $d$ through this trivialization. 
Also for each $i$,  we have smooth  $L^2$-orthonormal frames $f_{i, 1}, \dots, f_{i, r_i}$ of the normal bundle of $(\mathcal{E}_0)^{\lambda_i}_{-\infty}|_{U_i}$ in $P|_{U_i}$. We can write
\[
     \pi_P = \pi^{\lambda_i}_{-\infty} + \sum_{l=1}^{r_i} f_{i, l}^* \otimes f_{i,l}
\]
over $U_i$.  We have
\[
     \nabla_v \pi_P = 
        \nabla_v \pi^{\lambda_i}_{-\infty} + \sum_{l=1}^{r_i} (\nabla_v f_{i, l}^* \otimes f_{i,l} + f_{i,l}^* \otimes \nabla_v f_{i, l}).
\]
By Corollary \ref{cor nabla pi},  $\nabla_v \pi^{\lambda_i}_{-\infty}$ is a bounded operator from $L^2_m$ to $L^2_{m+1}$.   
Also we have
\[
          \nabla_v f_{i, l} 
          = \nabla_{v} (\pi^{\nu_i}_{\lambda_i} f_{i, l})
          = (\nabla_{v} \pi^{\nu_i}_{\lambda_i} ) f_{i, l} + \pi^{\nu_i}_{\lambda_i}  (\nabla_{v} f_{i,l}). 
\]
Since $f_{i, l}(b) \in C^{\infty}(\mathbb{S})$ for $b \in U_i$ and $\nabla_{v} \pi_{\lambda_i}^{\nu_i}$ is a bounded operator $L^2_m \rightarrow L^2_{m+1}$, we have
\[
        \nabla_v f_{i, l}(b) \in C^{\infty}(\mathbb{S})
\]
for $b \in U_i$.    
Also we have
\[
       \left|  f_{i,l}^*(\phi) \right| = | \langle f_{i, l}, \phi \rangle_0 |   \leq \| \phi \|_{0}
\]
for $\phi \in C^{\infty}(\mathbb{S})$. 
Therefore
\[
         \sum_{l=1}^{r_i} f_{i,l}^* \otimes \nabla_v f_{i, l} : L^2_m \rightarrow L^2_{m+1}
\]
is bounded.

Take $\phi \in C^{\infty}(\mathbb{S})$.   We have
\[
       (\nabla_v f_{i, l}^{*} ) (\phi) 
       =  \langle  \phi, \nabla_{v} f_{i, l}\rangle_0. 
\]
Note that $\nabla_v f_{i, l}(b)\in C^{\infty}(\mathbb{S})$ for $b \in U_{i}$.  Hence
\[
      \left\|
        (\nabla_v f_{i, l}^* \otimes f_{i,l} ) (\phi)
      \right\|_{m+1}
      = \left| (\nabla_v f_{i, l}^*)(\phi) \cdot f_{i, l} \right|_{m+1}
      \leq 
      C \| \phi \|_{0}. 
\]
 Therefore
\[
      \sum_{l=1}^{r_i} \nabla_v f_{i,l}^{*} \otimes f_{i,l} : L^2_m \rightarrow L^2_{m+1}
\]
is bounded. 
\end{proof}

\begin{cor}  \label{cor nabla D'}
Suppose that $\operatorname{Ind} D = 0$ in $K^1(B)$ and let $P_0$ be a spectral section of $-D$.  Then there is a family of smoothing operators $\mathbb{A}$ acting on $\mathcal{E}_0$  such that the kernel of $D' = D + \mathbb{A}$ is trivial and  
\[
          P_0 =  \mathcal{E}_0(D')_{-\infty}^{0}. 
\]
Moreover for each positive number $k$ and $v \in TB$, 
\[
      \nabla_v  D'  : L^2_k(\mathbb{S}) \rightarrow L^2_k(\mathbb{S})
\]
is bounded. 
\end{cor}

\begin{proof}
The operator $\mathbb{A}$ is obtained as follows. (See the proof of Lemma 8 of \cite{MP}.)  We can take smooth spectral sections $Q, R$ of $D$ and a positive number $s$  with
\[
          (\mathcal{E}_0)^{-s}_{-\infty} \subset P_0 \subset (\mathcal{E}_0)_{-\infty}^{s}, \
         (\mathcal{E}_0)^{-2s}_{-\infty} \subset  Q \subset (\mathcal{E}_0)^{-s}_{-\infty}, \
          (\mathcal{E}_0)^{s}_{-\infty} \subset R   \subset  (\mathcal{E}_0)^{2s}_{-\infty}. 
\]
 Put
\[
        D'  = \pi_Q D \pi_Q -s \pi_{P_0} ( 1- \pi_Q )  + (1- \pi_R) D (1- \pi_R) + s   (1-  \pi_{P_0})   \pi_R, 
\]
where $\pi_{P_0}, \pi_{Q}, \pi_{R}$ are the $L^2$-projections. 
Then $\ker D' = 0$ . The operator $\mathbb{A}$ is given by
\[
       \mathbb{A} = D' - D. 
\]
The image of $\mathbb{A}$ is included in the subspace spanned by finitely many eigenvectors of $D$. 
By Proposition \ref{prop perturbation P},  $\nabla_v \pi_{P_0}, \nabla_v \pi_{Q}, \nabla_v \pi_{R}$ are bounded operators from $L^2_k(\mathbb{S})$ to $L^2_{k+1}(\mathbb{S})$.  Note that  $\nabla_v D$ is the Clifford multiplication of the harmonic $1$-form $v$. Hence $\nabla_v D$ is a bounded operator $L^2_k(\mathbb{S}) \rightarrow L^2_k(\mathbb{S})$.     Therefore $\nabla_v D'$ is a bounded operator from $L^2_k$ to $L^2_k$. 
\end{proof}

\begin{prop} \label{prop D'}
The statements of Proposition \ref{prop nabla pi}, Corollary \ref{cor nabla pi} and Corollary \ref{cor nabla pi 2} hold for the perturbed Dirac operator $D'$, replacing $\rho(v)$ with $\nabla_v D'$. 
\end{prop}

\begin{proof}
By Corollary \ref{cor nabla D'}, for any non-negative number $k$, 
\[
     \nabla_v D' : L^2_k(\mathbb{S}) \rightarrow L^2_k(\mathbb{S})
\]
is bounded and we can do the same computations as those done for the original Dirac operator $D$. 
\end{proof}

\begin{lem} \label{lem nabla |D'|^k}
For a positive integer $k$, a positive number $l$ with $l \geq k-1$  and $v \in T_aB$, the expression
\[
     \nabla_v |D'|^k : L^2_{l} \rightarrow L^2_{l - k+1}
\]
is bounded. 
\end{lem}

\begin{proof}
Note that
\[
        |D'|^k =    (D')^k (1-\pi_{P_0}) + (-1)^{k} (D')^k  \pi_{P_0}.
\]
Here $\pi_{P_0}$ is the $L^2$-projection on $P_0$. 
We have
\[
   \nabla_v (D')^k =  ( \nabla_v D' ) (D')^{k-1} + D'  (\nabla_v D') (D')^{k-2} + \cdots  + (D')^{k-1} \nabla_v D', 
\]
which implies that $\nabla_v (D')^k$ is  a bounded operator $L^2_{l} \rightarrow L^2_{l-k+1}$ by Corollary \ref{cor nabla D'}.  Also $\nabla_v \pi_{P_0}$ is a bounded operator $L^2_l \rightarrow L^2_{l+1}$ by  Proposition \ref{prop L^2 nabla pi_P}.   
\end{proof}

\begin{rem}
So far the authors have not been able to prove Lemma \ref{lem nabla |D'|^k} in the case when $k$ is not an integer,  though there is an explicit formula
\[
      |D'|^k = \sum_{j}  |\eta_j|^k \pi_j. 
\] 
Here $\pi_j$ is the projection onto the $j$-th eigenspace which can be written as
\[
       \pi_j = \frac{1}{2\pi i} \int_{\Gamma_j} (z - D)^{-1} dz. 
\]
\end{rem}

Suppose that $\Ind D = 0$ in $K^1(B)$ and fix a spectral section $P_0$ and recall the definition of the $L^2_{k_+, k_-}$-innter product $\langle \cdot, \cdot \rangle_{k_+, k_-}$ defined by using the perturbed Dirac operator $D' = D + \mathbb{A}$ of Corollary \ref{cor nabla D'}.  (See (\ref{eq L2 k_+ k_-}).)  Let $\mathcal{E}_{k_+, k_-}$ be the completion of $\mathcal{E}_{\infty}$ with respect to $\langle \cdot, \cdot \rangle_{k_+, k_-}$.

We will prove a generalization of Proposition \ref{prop L^2 nabla pi_P}.

\begin{prop}\label{prop perturbation P}

Take non-negative half integers $k_+, k_-$ and a smooth spectral section $P$ of $-D$ with
\[
       (\mathcal{E}_0)^{\mu_-}_{-\infty} \subset P \subset (\mathcal{E}_0)^{\mu_+}_{-\infty}. 
\]
Let $\pi_{P}$ is the $L^2_{k_+, k_-}$-projection on $P$. 
Then for  each non-negative real number  $m$,  $v \in TB$, $\nabla_v \pi_P$ is a bounded operator from $L^2_m(\mathbb{S})$ to $L^2_{m+1}(\mathbb{S})$.  


\end{prop}

\begin{proof}
Let $U_i, \lambda_i, \nu_i$ be as in the proof of Proposition \ref{prop L^2 nabla pi_P} and $f_{i, 1}, \dots, f_{i, r_i}$ are smooth $L^2_{k_+, k_-}$-orthonormal frames of  the normal bundle of $(\mathcal{E}_0)_{-\infty}^{\lambda_i}|_{U_i}$ in $P$.  We can write
\[
    \pi_P = \pi^{\lambda_i}_{-\infty} + \sum_{l=1}^{r_i}  f_{i,l}^* \otimes f_{i,l}
\]
on $U_i$ Here
\[
     f_{i,l}^{*}(\phi) =  \langle  \pi_{P_0} \phi, |D'|^{2k_-} f_{i, l} \rangle_{0} + \langle (1-\pi_{P_0}) \phi, |D'|^{2k_+} f_{i, l} \rangle_0, 
\]
$P_0$ is the fixed spectral section used to define $L^2_{k_+, k_-}$-norm,  and  $\pi_{P_0}$ is the $L^2$-projection onto $P_0$. 
We have
\[
     \nabla_v \pi_P = 
        \nabla_v \pi^{\lambda_i}_{-\infty} + \sum_{l=1}^{r_i} (\nabla_v f_{i, l}^* \otimes f_{i,l} + f_{i,l}^* \otimes \nabla_v f_{i, l}).
\]
As stated in the proof of Proposition \ref{prop L^2 nabla pi_P}, $\nabla_v \pi^{\lambda_i}$ and $f_{i,l}^* \otimes \nabla_v f_{i, l}$ are bounded operators from $L^2_m$ to $L^2_{m +1}$. 

For $\phi \in C^{\infty}(\mathbb{S})$, 
 \begin{align*}
   & (\nabla_v f_{i,l}^*) (\phi) =    \\
   & \langle (\nabla_v \pi_{P_0}) \phi,  | D'|^{2k_-} f_{i, +} \rangle_0 + 
     \langle \pi_{P_0} \phi, (\nabla_v |D'|^{2k_-}) f_{i, l} \rangle_{0} +  
   \langle  \pi_{P_0} \phi,   |D'|^{2k_-} (\nabla_v f_{i,l}) \rangle_{0}   \\
   & - \langle (\nabla_{v} \pi_{P_0}) \phi, |D'|^{2k_+} f_{i,;} \rangle_0 +  
   \langle (1-\pi_{P_0}) \phi, (\nabla_v |D'|^{2k_+}) f_{i, l} \rangle_{0} \\
& +   \langle  (1-\pi_{P_0}) \phi,   |D'|^{2k_+} (\nabla_v f_{i,l}) \rangle_0.
 \end{align*}
Note that $2k_{\pm}$ are non-negative integers.  By Proposition \ref{prop L^2 nabla pi_P} and Lemma \ref{lem nabla |D'|^k},  
\[
      \| (\nabla_v f_{i, l}^* \otimes f_{i,l}) (\phi) \|_{m +1} =
       \| (\nabla_v f_{i,l}^*)(\phi) \cdot f_{i,l} \|_{m+1} \leq C \| \phi \|_0. 
\]
Hence $\nabla_vf_{i, l}^*  \otimes f_{i,l}$ are bounded operators from $L^2_m$ to $L^2_{m+1}$. 
\end{proof}

\begin{lem}    \label{lem nabla pi  phi psi}
  Let $\nabla$ be a connection on $\mathcal{E}_{k_+, k_-}$  (which is not necessarily the flat connection defined  in Section \ref{sec:main results}).
 Let $F$ be  a  subbundle in $\mathcal{E}_{k_+, k_-}$ of finite rank and $\pi_F : \mathcal{E}_{k_+, k_-} \rightarrow F$ be the $L^2_{k_+, k_-}$-projection. 
For $a \in B$,  $\phi, \psi \in F_a$, and $v \in T_aB$,  we have
\[
     \langle (\nabla_v \pi_F) \phi, \psi \rangle_{k_+, k_-} = 0. 
\]
 Similarly, for $\phi',   \psi' \in F_a^{\bot}$, we have
\[
     \langle (\nabla_v \pi_F) \phi',   \psi' \rangle_{k_+, k_-} = 0. 
\]

\end{lem}

\begin{proof}
Since
\[
          \pi_F \pi_F = \pi_F, 
\]
we have
\[
     (\nabla_v \pi_F ) \pi_F + \pi_F  (\nabla_v \pi_F) = \nabla_v \pi_F. 
\]
Hence
\[
    (\nabla_v \pi_{F}) \phi + \pi_F (\nabla_v \pi_F) \phi =  ( \nabla_v \pi_{F} ) \phi.
\]
Here we have used $\pi_F \phi = \phi$.  Therefore
\[
     \pi_F (\nabla_v \pi_F ) \phi = 0, 
\]
which implies that
\[
     \langle  (\nabla_v \pi_F) \phi, \psi \rangle_{k_+, k_-} = 0. 
\]
The proof of the other equality is similar. 
\end{proof}


\section{Weighted Sobolev space}\label{subsec:weighted-spaces} 
Assume that $\Ind D = 0 $ and fix a spectral section $P_0$ of $-D$. Let $D' = D + \mathbb{A}$ be the perturbed Dirac operator as in Corollary \ref{cor nabla D'}.

From now on, for $k > 0$,  we consider the norm defined by 
\[
    \| \phi \|_k = \|  |D'|^k  \phi \|_0. 
\]
Note that this  norm is equivalent to the original $L^2_k$ norm since $\ker D'=0$. That is, there is a constant $C > 1$ such that
\begin{equation*}  
       C^{-1} \| (1 + |D|^k ) \phi \|_0 \leq \| |D'|^k  \phi \|_0 \leq  C\| ( 1 + |D|^k) \phi \|_{0}.
\end{equation*}
Hence we can apply Corollary \ref{cor nabla pi}, Corollary \ref{cor nabla pi 2}, Proposition \ref{prop D'} to the Sobolev norms with respect to $D'$.

Let $P_n$, $Q_n$ be spectral sections of $-D$, $D$ with
   \begin{align*}
        &   (\mathcal{E}_0(D))_{-\infty}^{\mu_{n,-}} \subset P_n \subset (\mathcal{E}_0(D))_{-\infty}^{\mu_{n,+}},    \\
        &    (\mathcal{E}_0(D))_{\lambda_{n, +}}^{\infty} \subset Q_n \subset (\mathcal{E}_0(D))_{\lambda_{n, -}}^{\infty}. 
    \end{align*}
We may suppose that
      \begin{align*}
         & \mu_{n, -}  + 10  < \mu_{n, +}   < \mu_{n + 1, -} -10,    \\
         & \lambda_{n+1, +} + 10 < \lambda_{n, -} < \lambda_{n, +} - 10 \\
         & \mu_{n, +} - \mu_{n, -} < \delta, \quad \lambda_{n, +} - \lambda_{n, -} < \delta
   \end{align*}
for some positive number $\delta$ independent of $n$. See Theorem \ref{thm spectral section mu mu+delta}.  By the definition of $D' = D + \mathbb{A}$ in the proof of Corollary \ref{cor nabla D'}, we have
   \begin{align*}
       &   \mathcal{E}_0(D)_{-\infty}^{\mu_{n, \pm}} = \mathcal{E}_{0}(D')_{-\infty}^{\mu_{n,\pm}},  \\
       & \mathcal{E}_0(D)^{\infty}_{\lambda_{n, \pm}} = \mathcal{E}_0(D')^{\infty}_{\lambda_{n,\pm}}
  \end{align*}
for $n \gg 0$. 
Fix half integers $k_+, k_-  > 5$.  Put $\ell = \min \{ k_+, k_- \}$.  Let $\pi_{P_n}$, $\pi_{Q_n}$ be the $L^2_{k_+, k_-}$-projections on $P_n, Q_n$. 
By Proposition \ref{prop perturbation P}, we can assume that for each $n$,   there is $C_n > 0$ such that for $v \in TB$ with $\| v \| \leq 1$, 
\begin{equation}   \label{eq nabla pi C_n}
      \| \nabla_v \pi_{P_n} : L^2_{k_+, k_-} \rightarrow L^2_{\ell + 1} \| \leq C_n,   \quad
      \| \nabla_v  \pi_{Q_n} : L^2_{k_+, k_-} \rightarrow L^2_{\ell + 1} \| \leq C_n.
\end{equation}
Define a finite dimensional subundle $F_n$ of $\mathcal{E}_{\infty}$ by 
\[
         F_n = P_n \cap Q_n \subset (\mathcal{E}_0)_{\lambda_{n,-}}^{\mu_{n,+}}. 
\]
We will next introduce weighted Sobolev spaces.  Take positive numbers $\epsilon_n$ with
\begin{equation}  \label{eq C epsilon}
                C_n \epsilon_n \leq \frac{1}{n}, 
\end{equation}
where $C_n$ are the constants from (\ref{eq nabla pi C_n}). 
Fix a smooth function
\[
        w : \mathbb{R} \rightarrow \mathbb{R}
\]
with
  \begin{align*}
   0 < w(x) \leq 1   &  \quad \text{for all $x \in \mathbb{R}$, }    \\
     w(x) = \epsilon_n &  \quad  \text{if $x \in  [\lambda_{n, -} - 3, \lambda_{n,+} + 3] \cup [\mu_{n,-} - 3, \mu_{n,+} + 3] $ for some $n$. }
  \end{align*}

Take  $a \in \mathcal{H}^1(Y)$.  Let $\{ e_{j} \}_j$ be an orthonormal basis of $L^2(\mathbb{S})$ with
\[
          D_a' e_{j} = \eta_j e_{j},
\]
where $\eta_j$ are the eivenvalues of $D_a'$. 

For a positive number $k$ and $\phi = \sum_{j} c_{j} e_{j} \in C^{\infty} (\mathbb{S})$,  we define a weighted Sobolev norm $\| \phi \|_{a, k,  w}$ by  
\[
           \|    \phi \|_{a, k,  w} := 
           \Bigg(  \sum_{j}   |c_{j} |^2     |\eta_j| ^{2k}  w(\eta_j)^2  \Bigg)^{\frac{1}{2}}. 
\]
 Denote by $L^2_{a, k, w}(\mathbb{S})$ the completion of $C^{\infty}(\mathbb{S})$ with respect to $\| \cdot \|_{a, k,  w}$.  The family  $\{ \| \cdot \|_{a, k,  w} \}_{a \in \mathcal{H}^1(Y)}$ of norms induces a fiberwise norm $\| \cdot \|_{k,w}$ on $\mathcal{E}_{\infty}$.  We denote the completion of $\mathcal{E}_{\infty}$ with respect to $\| \cdot \|_{k, w}$ by $\mathcal{E}_{k, w}$.  Note that
 \[
        \| \phi \|_{k, w} \leq \| \phi \|_{k}. 
 \]

\begin{prop} \label{prop pi P_n wighted}
Let $k_+, k_-$ be half integers with $k_+, k_- > 5$ and put $\ell = \min \{ k_+, k_- \}$.  Then
\[
      \sup_{v \in B(TB; 1)}  \Big\| \nabla_v    \pi_{P_n} : L^2_{k_+, k_-} \rightarrow L^2_{\ell - 5, w}   \Big\|
      \rightarrow
      0. 
\]
%
A similar statement holds for $\pi_{Q_n}$. 
\end{prop}

\begin{proof}
For $\lambda, \mu \in \mathbb{R}$, let $\pi_{\lambda}^{\mu}$ be the $L^2$-projection to $(\mathcal{E}_0(D'))_{\lambda}^{\mu}$. 
Take $a \in B$ and  $v \in T_aB$ with $\| v \| \leq 1$.  By Corollary \ref{cor nabla pi 2} and Proposition \ref{prop D'}, for $n \gg 0$, we can take 
\[
    \nu _{n, -}  \in  [\mu_{n, -}-2, \mu_{n,-} -1],  \quad
    \nu_{n, +}   \in [\mu_{n, +} +1,  \mu_{n,+} + 2]
\]
such that 
\begin{align*}
        &  \| \nabla_v \pi_{-\infty}^{\nu_{n, -}} : L^2_{\ell -1}  \rightarrow L^2_{\ell - 5} \| \leq C,    \\
       &  \| (\nabla_v \pi_{\nu_{n, -}}^{\nu_{n,+}} ) : L^2_{\ell - 1} \rightarrow L^2_{\ell -5} \| \leq C,
 \end{align*}
where $C > 0$ is a constant independent of $n$. 
Note that
\begin{align*}
    \pi_{P_n} 
         &= id_{\mathcal{E}_0} \circ \pi_{P_n}   \\
         &= ( \pi_{-\infty}^{\nu_{n, -}} + \pi^{\nu_{n,+}}_{\nu_{n,-}} + \pi_{\nu_{n,+}}^{\infty} ) \circ \pi_{P_n}    \\
         &=   \pi_{-\infty}^{\nu_{n, -}} +  \pi^{\nu_{n, +}}_{\nu_{n, -}} \circ \pi_{P_n}. 
   \end{align*}
Hence 
\begin{equation}   \label{eq nabla pi P_n}
     \nabla_v \pi_{P_n} 
     = \nabla_v \pi_{-\infty}^{\nu_{n, -}} + (\nabla_v \pi_{\nu_{n, -}}^{\nu_{n,+}} )  \pi_{P_n} + 
          \pi^{\nu_{n, +}}_{\nu_{n, -}} (\nabla_v \pi_{P_n}).  
\end{equation}

 For $\epsilon > 0 $, take a positive number $\beta$ with $\beta > \frac{1}{\epsilon}$. Then for any $\phi \in \mathcal{E}_{k_+, k_-}$ with $\| \phi \|_{k_+, k_-} \leq 1$, we have
 \[
         \| \pi_{\beta}^{\infty} \phi \|_{\ell - 1} < \epsilon. 
 \]
By Proposition \ref{prop nabla pi} and Corollary \ref{cor nabla pi 2},  for $n \gg 0$ with  $\beta < \nu_{n,-}$, 
\begin{equation}   \label{eq 1}
   \begin{aligned}
       \| (\nabla_v  \pi_{-\infty}^{\nu_{n, -}})  \phi  \|_{\ell - 5}
   &  = 
       \| (\nabla_v \pi_{-\infty}^{\nu_{n, -}} )  ( \pi_{-\infty}^{\beta} \phi + \pi_{\beta}^{\infty} \phi     )  \|_{\ell - 5}    \\
    & \leq  C' \Big(  \frac{1}{  | \beta - \nu_{n,-}| }  + \epsilon  \Big). 
    \end{aligned}
\end{equation}
Here $C' > 0$ is independent of $n$.  Similarly,
\begin{equation} \label{eq 2}
          \| (\nabla_v \pi_{\nu_{n, -}}^{\nu_{n,+}} )  \pi_{P_n}   \phi  \|_{\ell - 5 } 
          \leq C'' \Bigg(  \frac{1}{  \min \{ | \beta - \nu_{n,+}|,  | \beta - \nu_{n,-}|  \}}  + \epsilon  \Bigg)
\end{equation}  
for $n \gg 0$, where $C'' > 0$ is a constant independent of $n$. 
By the definition of the weighted Sobolev norm $\| \cdot \|_{\ell, w}$ and (\ref{eq C epsilon}), 
\begin{equation}  \label{eq 3}
            \| \pi^{\nu_{n, +}}_{\nu_{n, -}} (\nabla_v \pi_{P_n})  \phi \|_{\ell,w}  \leq C_n \epsilon_n \| \phi \|_{k_+, k_-} \leq  \frac{1}{n}. 
\end{equation}
 The statement follows from (\ref{eq nabla pi P_n}), (\ref{eq 1}), (\ref{eq 2}), (\ref{eq 3}). 
  
 \end{proof}

\begin{lem}   \label{lem norm}
Let $K$ be a compact set in $\mathcal{H}^1(Y)$. 
There is a norm $\| \cdot \|_{K, k, w}$ on $C^{\infty}(\mathbb{S})$ such that for any $a \in K$ and $\phi \in C^{\infty}(\mathbb{S})$ we have
\[
     \| \phi \|_{K, k, w} \leq \| \phi \|_{a, k,  w}. 
\]
Let $L^2_{K, k, w}$ be the completion of $C^{\infty}(\mathbb{S})$ with respect to  $\| \cdot \|_{K, k, w}$. For $l \geq k$, the natural map $L^2_{l} \rightarrow L^2_{K,k,w}$ is injective.
\end{lem}


\begin{proof}
Take a compact set $K$ in $\mathcal{H}^1(Y)$ and fix $a_0 \in K$.  Choose $a \in K$.    Put
\begin{align*}
       & a_t = (1-t)a_0 + t a, \\
      &  r = \| a_0 - a \|,   \\
      &   \delta := \max \{   \| \nabla_v D' : L^2 \rightarrow L^2 \|   :   t \in [0,1], v \in T_{a_t} \mathcal{H}^1(Y), \| v \| = 1   \}.
    \end{align*}
Let $\tilde{\mathcal{E}}_0$ be the trivial bundle $\mathcal{H}^1(Y) \times L^2(\mathbb{S})$ over $\mathcal{H}^1(Y)$, which is the pull-back of $\mathcal{E}_0$ by the projection $\mathcal{H}^1(Y) \rightarrow B$. 
Also take a sequence $\{ \lambda_l \}_{l=-\infty}^{\infty}$ of real numbers with 
\[
         \lambda_l + r \delta \ll  \lambda_{l+1}. 
\]
We will prove that for each $l$, there is a constant $c_l(a) > 0$ such that for $\phi \in \tilde{\mathcal{E}}_0( D_{a_{0}}')_{\lambda_l}^{\lambda_{l+1}}$,  we have
\begin{equation}  \label{eq c_l}
        c_l (a) \| \phi \|_0 \leq    \| ( \pi_a )_{\lambda_{l} - r\delta}^{\lambda_{l+1} +r\delta} \phi \|_0. 
\end{equation}
 Fix an integer $l$. We consider the following set:
\[
        I = \Bigg\{  t \in [0,1]  :
        \begin{array}{l}
        \forall s \in [0, 1], s \leq t, 
        \exists c(s) > 0,  
          \forall \phi \in  \tilde{\mathcal{E}}_0(D_{a_0}')_{\lambda_{l}}^{\lambda_{l+1}} \\
        c(s) \| \phi \|_{0} \leq    \| (\pi_{a_s})_{\lambda_{l} - s r \delta}^{\lambda_{l+1} + s r\delta} \phi \|_{0}
        \end{array}
        \Bigg\}.
\]
Note that $0 \in I$.   To prove (\ref{eq c_l}), it is sufficient to show that $\sup I = 1$.  Put $t_0 = \sup I$ and assume that $t_0 < 1$.    

%
%
%
%
%
%
%
%
%
%
%
%

%
%
  
Then, take $t_+ \in (t_0, 1]$ with $|t_{+} - t_0|$ sufficiently small.  For $t \in [t_0, t_+]$,  let
\[
      \nu_1(t), \dots, \nu_{m}(t)
\]
be the eigenvalues of $D_{a_t}'$ which are continuous in $t$ such that
\begin{align*}
     & \lambda_{l} - t_0 r \delta  < \nu_1(t_0), \nu_2(t_0), \dots, \nu_m(t_0) \leq \lambda_{l+1} + t_0r \delta,   \\
     & \dim \tilde{\mathcal{E}}_0(D_{a_t}')_{\lambda_l-t_0r\delta}^{\lambda_{l+1} t_0 r \delta} = m.
 \end{align*}

Take real numbers $\lambda_-, \lambda_+$ sufficiently close to $\lambda_{l} - t_0 r\delta$, $\lambda_{l+1} + t_0 r \delta$,  which are not eigenvalues of $D_{a_{t}}'$ for $t \in [t_0, t_+]$,  such that
\[
      \tilde{\mathcal{E}}_0(D_{a_{t_0}}')_{\lambda_-}^{\lambda_+} 
   = \tilde{\mathcal{E}}_0(D_{a_{t_0}}')_{\lambda_{l} - t_0 r \delta}^{\lambda_{l+1} + t_0 r \delta}. 
\] 
By Theorem  4.10 in \cite[p. 291]{Kato},  for $t \in [t_0, t_+]$, 
\[
          \lambda_{l} - t r \delta  <  \nu_1(t), \dots, \nu_m(t) \leq \lambda_{l+1} + t r \delta
\]
which implies that
\[
       \tilde{\mathcal{E}}_0(D_{a_t}')_{\lambda_-}^{\lambda_+} =
         \tilde{\mathcal{E}}_0(D_{a_t}')_{\lambda_l - t r\delta}^{\lambda_{l+1} + t r\delta}.
\]
So we have
\[ 
      \|   (\pi_{a_t})_{\lambda_-}^{\lambda_+} \phi \|_0  = \| (\pi_{a_t})_{\lambda_l - t r \delta}^{\lambda_{l+1} + t r \delta} \phi \|_0. 
\]
From the equality 
\[
       \frac{d}{dt} \| (\pi_{a_t})_{\lambda_-}^{\lambda_+} \phi \|_{0}^2 = 
       2 \operatorname{Re} \langle (\nabla_v (\pi_{a_t})_{\lambda_-}^{\lambda_+}) \phi, \phi \rangle_0,
\]
for $t \in [t_0, t_+]$ and $\phi \in \tilde{\mathcal{E}}_0(D_{a_{t_-}}')_{\lambda_l - t_- r \delta}^{\lambda_{l+1} + t_- r\delta}$,  we have
\[
      \left\{  1 - 2 M(t - t_0) \right\} \| \phi \|_{0} \leq 
      \| (\pi_{a_t})_{\lambda_-}^{\lambda_+} \phi \|_0 =
      \Big\| (\pi_{a_{t}})_{\lambda_l - t r\delta}^{\lambda_{l+1} + t r\delta} \phi \Big\|_0,  
\]
where
\[
     M = \max \Big\{   \Big\lVert  \nabla_{v} (\pi_t)_{\lambda_-}^{\lambda_+} : L^2 \rightarrow L^2 \Big\rVert          :    t \in [t_0, t_+]   \Big\}
\]
and $v = a - a_0$.  Taking $t_+$ sufficiently close to $t_0$, we have
\[
         2 M |t_+ - t_0| < 1. 
\]
This implies that
\[
    t_+ \in I
\]
and we get a contradiction. We have obtained (\ref{eq c_l}). 

Take a sufficiently small open neighborhood $U_{l, a}$ of $a$ in $\mathcal{H}^1(Y)$. Then for all $a' \in U_{l, a}$ we have
\[
       \frac{1}{2} c_l(a)  \| \phi \|_0 \leq  \Big\rVert  (\pi_{a'} )_{\lambda_l - r \delta - 1}^{\lambda_{l+1} + r \delta + 1}  \phi \Big\rVert_0
\]
for $\phi \in \tilde{\mathcal{E}}_0(D_{a_0}')_{\lambda_l}^{\lambda_{l+1}}$.  Since $K$ is compact, there exist $a_{l, 1}, \dots, a_{l, N_l} \in K$ such that
\[
      K \subset U_{l,a_1} \cup \cdots \cup U_{l, a_{N_l}}. 
\]
Take a small positive number $\epsilon > 0$ such that there are no eigenvalues of $D_{a}'$ in $[-\epsilon, \epsilon]$ for $a \in K$. 
Put
  \begin{align*}
      & c_l = \min \{ c_{l}(a_{l,1}), \dots, c_l(a_{l, N_l}) \},     \\
      & \underline{w}(l) : =  
      \min \{  |x|^k w(x)   | x \not\in [-\epsilon, \epsilon],  x \in [\lambda_{l-1}, \lambda_{l+2}]   \}
   \end{align*}
For $\phi \in C^{\infty}(\mathbb{S})$, define
\begin{equation}  \label{eq def of wighted norm}
     \lVert  \phi    \rVert_{K, k, w} =
    \Bigg\{ \sum_{l}   \left( \frac{1}{10} c_l  \underline{w}(l) \|   (\pi_{a_0})_{\lambda_l}^{\lambda_{l+1}} \phi \|_0   \right)^2   \Bigg\}^{\frac{1}{2}}. 
\end{equation}
Then 
\[
      \| \phi \|_{K, k, w} \leq \| \phi \|_{a, k, w}
\]
for all $a \in K$ and $\phi \in C^{\infty}(\mathbb{S})$. 
From the definition (\ref{eq def of wighted norm}) of $\| \cdot \|_{K,k,w}$, we have that the natural map $L^2_{l}\to L^2_{K,k,w}$ is injective for $l \geq k$.

\end{proof}

\begin{prop} \label{prop Ascoli}
Let $W$ be a closed, oriented,  smooth manifold and $E$ be a vector bundle on $W$.
Let $k$ be a positive number with $k \geq 1$, $I$ be a compact interval in $\mathbb{R}$ and  $\| \cdot \|$  be any norm on $C^{\infty}(E)$ such that $\| \phi \| \leq \| \phi \|_{k-1}$ for all $\phi \in C^{\infty}(E)$.   Assume that the natural map $L^2_l(E) \rightarrow \overline{C^{\infty}(E)}$ is injective for $l \geq k-1$.  Here $\overline{C^{\infty}(E)}$ is the completion with respect to the norm $\|  \cdot  \|$.   We consider $L^2_l(E)$ to be a subspace of $\overline{C^{\infty}(E)}$ through this map. 

Suppose that we have a sequence $\gamma_n : I \rightarrow C^{\infty}(E)$  such that $\gamma_n$ are equicontinuous in $\| \cdot \|$ and uniformly bounded in $L^2_{k}$.  Then after passing to a subsequence, $\gamma_n$ converges uniformly in $L^2_{k-1}$ to a continuous 
\[
      \gamma : I \rightarrow  L^2_{k-1}(E).
\]
\end{prop}

\begin{proof}
Let $q_1, q_2, \dots, $ be the rational numbers in $I$.   Since $\gamma_n$ are uniformly bounded in $L^2_k$, it follows from the Rellich lemma and the diagonal argument that there is a subsequence $n(i)$ such that $\gamma_{n(i)}(q_m)$ converges in $L^2_{k-1}$ (and hence in $\| \cdot \|$)  as $i \rightarrow \infty$ for each $m$. Since $\gamma_n$ are equicontinous in $\| \cdot \|$,   for any $\epsilon > 0$ and $t \in I$, we can find $q_m$ which is independent of $i$,  with
\[
        \| \gamma_{n(i)}(t) - \gamma_{n(i)}(q_m)  \| < \epsilon. 
\]
So we have, for any $t$, 
   \begin{align*}
    & \| \gamma_{n(i)}(t) - \gamma_{n(j)}(t) \|    \\
   &  \leq \| \gamma_{n(i)}(t) - \gamma_{n(i)}(q_m) \| + \| \gamma_{n(i)}(q_m) - \gamma_{n(j)}(q_m) \| +
            \|  \gamma_{n(j)}(q_m) - \gamma_{n(j)}(t)  \|   \\
      & \leq \| \gamma_{n(i)}(q_m) - \gamma_{n(j)}(q_{m}) \| + 2\epsilon. 
   \end{align*}         
This implies that for each $t \in I$,  $\gamma_{n(i)}(t)$ is a Cauchy sequence in $\| \cdot \|$,  and hence $\gamma_{n(i)}$ has a pointwise limit $\gamma : I \rightarrow \overline{C^{\infty}(E)}$, where $\overline{C^{\infty}(E)}$ is the completion with respect to $\| \cdot \|$.    

Since $\gamma_n$ are equicontinous in $\| \cdot \|$, for any $\epsilon > 0$ there is $\delta > 0$ such that for $t, t' \in I$ with $|t - t'| < \delta$ we have $\| \gamma_{n}(t) - \gamma_n(t') \| < \epsilon$.  Taking the limit, we have $\| \gamma(t) - \gamma(t') \| \leq \epsilon$.   We can choose finitely many rational numbers  $q_{1}, \dots, q_{N}$  in $I$ such that for all $t \in I$ there is $q_{l}$ with $l \in \{ 1, \dots, N \}$ such that $| t - q_{l} | < \delta$.     If $i_0$ is large enough, for $i > i_0$ we have $\| \gamma_{n(i)}(q_m) - \gamma(q_{m}) \| < \epsilon$ for all $m \in \{ 1, \dots, N\}$. Therefore for $i > i_0$, 
  \begin{align*}
     & \| \gamma_{n(i)}(t) - \gamma(t)  \|   \\
     & \leq \|  \gamma_{n(i)}(t) - \gamma_{n(i)}(q_l)  \| + \|  \gamma_{n(i)}(q_l) - \gamma (q_l)   \|
       + \| \gamma (q_l) - \gamma(t) \|    \\
      & < 3 \epsilon. 
  \end{align*}
Hence $\gamma_{n(i)}$ converses uniformly to  $\gamma$ in $\| \cdot \|$.

We first show that the limit $\gamma$ defined above in fact lies in $L^2_{k-\frac{1}{2}}$.  Indeed, for any fixed $t_\infty$ and any sequence $t_i\to t_\infty$ in $I$, we have that $\gamma_{n(i)}(t_i)$ converges, in $(k-\frac{1}{2})$-norm, after extracting a subsequence, to some $\delta$.  However, as above, $\gamma_{n(i)}(t_i)$ also converges in $\|\cdot \|$-norm to $\gamma(t_\infty)$.  Recall that $L^2_{k-\frac{1}{2}}$ is a subspace of $\overline{C^\infty(E)}$, so $\delta\in \overline{C^\infty(E)}$, and we have:
\begin{align*}
 \| \gamma(t_{\infty})-\delta \| &\leq \|\gamma(t_\infty)-\gamma_{n(i)}(t_i)\|+ \| \gamma_{n(i)}(t_i)-\delta\| \\
 &\leq \|\gamma(t_\infty)-\gamma_{n(i)}(t_i)\|+ \| \gamma_{n(i)}(t_i) - \delta\|_{k-\frac{1}{2}}.
\end{align*}
It follows that $\delta=\gamma(t_\infty)$.  This establishes that $\gamma$ is defined as a function $I\to L^2_{k-\frac{1}{2}}$, but not that it is continuous, nor that the $\{\gamma_{n(i)}\}$ converges pointwise in $(k-\frac{1}{2})$-norm.    Note that,  since $\| \gamma_{n}(t) \|_{k}  \leq C$ for a positive constant $C$ independent of $n, t$ by assumption,  we have $\| \gamma(t) \|_{k-\frac{1}{2}} \leq C$ for all $t \in I$.

Assume that $\gamma_{n(i)}$ does not converge uniformly in $L^2_{k-1}$.  
Then after passing to a subsequence, there is $\epsilon_0 > 0$ such that for any $i$ we have $t_i \in I$ with
\[
     \| \gamma_{n(i)}(t_i) - \gamma(t_{i}) \|_{k-1} \geq \epsilon_0. 
\]
After passing to a subsequence, $t_i$ converges to some $t_{\infty} \in I$. Then $\gamma_{n(i)}(t_i)$ converges to $\gamma(t_{\infty})$ in $\| \cdot \|$.  Since $\gamma_{n(i)}(t_i)$ are uniformly bounded in $L^2_k$, by the Rellich lemma, after passing to a subsequence  $\gamma_{n(i)}(t_i)$ converges to some $\delta $ in $L^2_{k-1}$; by the argument to show that $\gamma(t_\infty)\in L^2_{k-\frac{1}{2}}$ above, we see that $\delta=\gamma(t_\infty)$. Similarly, since $\| \gamma(t_i) \|_{k-\frac{1}{2}} \leq C$ for all $i$,  after passing to a subsequence, $\gamma(t_{i})$ converges to some $\delta'$ in $L^2_{k-1}$.  Since $\gamma(t_{i})\to \gamma(t_\infty)$ in $\overline{C^\infty(E)}$, the previous argument gives that $\delta'=\gamma(t_\infty)$.

Therefore, after passing to a subsequence:
\[
       \| \gamma_{n(i)}(t_{i}) - \gamma(t_{i}) \|_{k-1} \rightarrow 0
\]
as $i \rightarrow \infty$.  This is a contradiction. Thus $\gamma_{n(i)}$ converges to $\gamma$ in $L^2_{k-1}$ uniformly.    Since the convergence is uniform in $L^2_{k-1}$, $\gamma$ is continuous in $L^2_{k-1}$. 
\end{proof}

%
%


\section{Proof of Theorem \ref{thm isolating nbd}} \label{subsec:proof-of-isolation}
Take  half integers $k_+, k_-$ with $k_+, k_- > 5$ and with $|k_+ - k_- | \leq \frac{1}{2}$.  We put $\ell = \min \{ k_+, k_- \}$ and
\[
      A_n := 
      (B_{k_+}(F_n^+; R) \times_B B_{k_-}(F_n^-; R))  \times_{B} (B_{k_+}(W^+_n; R) \times_{B} B_{k_-}(W^-_n; R)). 
\]
We want to prove that $A_n$ are isolating neighborhoods for $\varphi_{n,k_+,k_-}=\varphi_n$ for $n$ large. If this is not true,  after passing to a subsequence, 
\[
        \operatorname{inv} A_n  \cap \partial A_n \not = \emptyset 
\]
for all $n$. 
Then we can take
\[
    y_{n, 0} = (\phi_{n,0}, \omega_{n,0}) \in  \operatorname{inv} A_n \cap \partial A_n. 
 \]
After passing to a subsequence, we may suppose that one of the following cases holds for all $n$: 

\begin{enumerate}[(i)]
  \item
   $\phi_{n, 0}^+ \in  S_{k_+}(F_n^+; R)$   
   \item
   $\phi_{n,0}^{-} \in S_{k_-}(F_n^-; R)$, 

   \item
   $\omega_{n,0}^{+}  \in  S_{k_+}(W^+_n; R)$, 
   
   \item
   $\omega_{n,0}^{-} \in S_{k_-}(W^-_n; R)$. 
\end{enumerate}

Let $\gamma_n = (\phi_n, \omega_n) : \mathbb{R} \rightarrow F_n \oplus W_n$  be the solution to (\ref{eq for gamma}) with $\gamma_n(0) = y_{n, 0}$: 
\begin{equation}  \label{eq d/dt gamma_n}
\begin{aligned}
    &  \Big(  \frac{d\phi_n}{dt} (t)  \Big)_{V}  =  -( \nabla_{X_H} \pi_{F_n}) \phi_n(t) - \pi_{F_n} (D \phi_n(t) + c_1(\gamma_n(t))),   \\
    & \Big(  \frac{d\phi_n}{dt}(t) \Big)_{H} = -X_{H}( \phi_n(t)),  \\
    & \frac{d\omega_n}{dt}(t) = -*d \omega_n(t) - \pi_{W_n} c_2(\gamma_n(t)). 
\end{aligned}
\end{equation}
We have
\begin{equation}   \label{eq phi omega norm}
     \|  \phi_n^+(t) \|_{k_+} \leq R,    \ \|  \phi_n^-(t) \|_{k_-} \leq R,  \    
       \| \omega_n^+(t) \|_{k_+} \leq R,   \ \| \omega_n^-(t) \|_{k_-} \leq R
\end{equation}
for all $t \in \mathbb{R}$.  By the Sobolev multiplication theorem, 
    \begin{align*}
      & \| c_1(\gamma_n(t)) \|_{\ell} \leq C \| \gamma_n(t) \|_{\ell}^2 \leq CR^2, \\
      & \| c_2(\gamma_n(t)) \|_{\ell} \leq C \| \gamma_n(t) \|_\ell^{2} \leq C R^2, \\
      & \| X_H(\phi(t)) \|_{\ell} \leq C \| \gamma_n(t) \|_\ell^2 \leq C R^2. 
     \end{align*}

Let $\Delta \subset \mathcal{H}^1(Y)$ be a fundamental domain of the action of $H^1(Y;\mathbb{Z})$ on $\mathcal{H}^1(Y)$, which  is a  bounded set. 
By the path lifting property of the covering space $\mathcal{H}^{1}(Y) \times L^2_{k_+, k_-}(\mathbb{S}) \rightarrow \mathcal{E}_{k_+, k_-}$, we have a lift
\[
       \tilde{\gamma}_n = (\tilde{\phi}_n, \omega_n) : \mathbb{R} \rightarrow \mathcal{H}^1(Y) \times  L^2_{k_+, k_-}(\mathbb{S}) \times L^2_{k_+, k_-}(\im d^*)
\]
of $\gamma_n$ with
\begin{equation}   \label{eq p gamma}
           p_{\mathcal{H}} ( \tilde{\gamma}_n(0)) \in \Delta. 
\end{equation}
By (\ref{eq d/dt gamma_n}), we have
\begin{equation}   \label{eq d phi dt H}
          \Bigg\lVert \Big( \frac{d\phi_n}{dt}(t)   \Big)_{H}   \Bigg\rVert \leq C R^2. 
\end{equation}
Fix $T > 0$.   It follows from (\ref{eq p gamma}) and  (\ref{eq d phi dt H}) that we can take a compact set $K_T$ of $\mathcal{H}^1(Y)$ such that  for any $n$ and $t \in [-T, T]$ we have
\[
         p_{\mathcal{H}} (\tilde{\gamma}_n(t)) \in K_T.
\]

Note that    $\frac{d\tilde{\phi}_n}{dt}$ is uniformly bounded on $[-T, T]$ in $\| \cdot \|_{K_T, \ell  - 5, w}$ by (\ref{eq d/dt gamma_n}),  Proposition \ref{prop pi P_n wighted} and Lemma \ref{lem norm}, which implies that $\tilde{\phi}_n$ are equicontinous in $L^2_{K_T, \ell - 5, w}$ on $[-T, T]$.    Also $\omega_n$ are equicontinuous in $L^2_{\ell -1}$.  
By Proposition \ref{prop Ascoli}, after passing to a subsequence, $\tilde{\gamma}_n|_{[-T, T]}$ converges to a map
\[
              \tilde{\gamma}^{(T)} = (\tilde{\phi}^{(T)}, \omega^{(T)}) : [-T, T] \rightarrow  \mathcal{H}^1(Y) \times  L^2_{\ell-1}(\mathbb{S}) \times L^2_{\ell - 1}(  \im d^*)
\]
uniformly in $L^2_{\ell - 1}$.  By the diagonal argument, we can show that  there is a continuous map
\[
      \tilde{\gamma} = (\tilde{\phi}, \omega): 
      \mathbb{R} \rightarrow   \mathcal{H}^1(Y) \times  L^2_{\ell-1}(\mathbb{S}) \times L^2_{\ell - 1}(\im d^*)
\]
such that, after passing to a subsequence,  $\tilde{\gamma}_n$ converges to $\tilde{\gamma}$ uniformly in  $L^2_{\ell - 1}$ on each compact set in $\mathbb{R}$.


\begin{lem}\label{lem:convergence-to-solution}
The limit $\tilde{\gamma}$ is a solution to the Seiberg-Witten equations over $Y \times \mathbb{R}$. 
\end{lem}

\begin{proof}    
Fix $T>0$. For $t \in [-T, T]$,  we have 
\begin{equation}  \label{eq phi(t) - phi(0)}
   \begin{aligned}
      & \tilde{\phi}_n(t) - \tilde{\phi}_n(0)  \\
      &= \int_{0}^{t}  \frac{d \tilde{\phi}_n}{ds}(s) ds  \\
      &= - \int_{0}^{t}   (\nabla_{X_H} \pi_{\tilde{F}_n}) \tilde{\phi}_n(s) + 
         \pi_{\tilde{F}_n} (D \tilde{\phi}_n(t) + c_1(\tilde{\gamma}_n(t))) + X_{H}(\phi_n(s))  ds.
    \end{aligned}
\end{equation}
We have that $p_{\mathcal{H}}(\tilde{\gamma}_n(t)) \in K_T$ for any $n$ and $t \in [-T, T]$. 
 Note that we have no estimate on $(\nabla_{X_H} \pi_{F_n}) \tilde{\phi}_n$ in \emph{any} $ L^2_{j}$-norm and that we just have control on it in  the auxiliary space $L^2_{K_T, \ell-5, w}$.
 By Proposition \ref{prop pi P_n wighted} and Lemma \ref{lem norm}, 
\[
          (\nabla_{X_H}\pi_{\tilde{F}_n}) \tilde{\phi}_n(s) \rightarrow 0 
\]
uniformly in $L^2_{K_T, \ell - 5, w}$ as $n \rightarrow \infty$.   Recall that $\tilde{\phi}_n$, $\omega_n$ converge in $L^2_{\ell-1}$ uniformly on $[-T, T]$.   
It follows from  Proposition \ref{prop [D, pi]} and the inequality 
   \begin{align*}
      \| \pi_{F_n} D \tilde{\phi}_n - D \tilde{\phi} \|_{\ell - 2}
      & = \| \pi_{F_n} D \tilde{\phi}_n - D \tilde{\phi}_n + D \tilde{\phi}_n - D\tilde{\phi} \|_{\ell -2 }   \\
      & \leq \| [\pi_{F_n}, D] \tilde{\phi}_n   \|_{\ell-2} + \| D \tilde{\phi}_n - D \tilde{\phi} \|_{\ell-2}  
   \end{align*}
that $\pi_{F_n}D\tilde{\phi}_n$ converges to $D\tilde{\phi}$ uniformly in $L^2_{\ell-2}$ on $[-T, T]$. 

Taking the limit with $n \rightarrow \infty$ in (\ref{eq phi(t) - phi(0)}), we obtain
\[
    \tilde{\phi}(t) - \tilde{\phi}(0) 
    = -\int_{0}^{t}  (D \tilde{\gamma}(t) + c_1(\tilde{\gamma}(t)))  + X_{H}(\tilde{\phi}(s)) ds. 
\]
Hence, by the fundamental theorem of calculus, 
\[
        \frac{d \tilde{\phi}}{dt}(t)  =   - (D \tilde{\phi}(t) + c_1(\tilde{\gamma}(t)))  - X_{H}(\tilde{\phi}(t)).
\]
 A priori, the left hand side $\frac{d \tilde{\phi}}{dt}(t)$  only lives in the auxiliary space $L^2_{K_T, \ell -5 , w}$. However,  since $L^2_{\ell-2}$ is a subspace of $L^2_{K_T, \ell-2, w}$ and the right hand side is in $L^2_{\ell-2}$,   $\frac{d \tilde{\phi}}{dt}(t)$ is in $L^2_{\ell-2}$ and both sides are equal to each other as elements of $L^2_{\ell-2}$.   

Similarly, we can show that
\[
     \frac{d\omega}{dt}(t) = -*d\omega(t) - c_2(\tilde{\gamma}(t)). 
\]
Therefore $\tilde{\gamma}$ is a solution to the Seiberg-Witten equations (\ref{SW eq}) and the ordinary theory of elliptic regularity shows that $\tilde{\gamma}$ is in $C^{\infty}$ as a section on any compact set in $Y \times (-T, T)$.  

\end{proof}

Composing $\tilde{\gamma} : \mathbb{R} \rightarrow \mathcal{H}^1(Y) \times L^2_{\ell - 1}(\mathbb{S}) \times L^2_{\ell -1}(\im d^*)$ with the projection
\[
         \mathcal{H}^1(Y) \times L^2_{\ell -1}(\mathbb{S}) \times L^2_{\ell - 1}( \im d^*)
         \rightarrow
         \mathcal{E}_{\ell -1} \oplus  \mathcal{W}_{\ell-1}, 
\]
we get a Seiberg-Witten trajectory 
\[
       \gamma : \mathbb{R} \rightarrow \mathcal{E}_{\ell -1} \oplus \mathcal{W}_{\ell-1}.
\]
Since  $\| \gamma(t) \|_{\ell - 1} \leq R$ for all $t \in \mathbb{R}$,   $\gamma$ has finite energy.   By Proposition \ref{prop compactness},
\begin{equation}   \label{eq gamma t norm}
       \| \gamma(t) \|_{k_+, k_-} \leq R_{k_+, k_-}.
\end{equation}
for all $t \in \mathbb{R}$.

Assume that the case (i) holds for all $n$. We have
\[
       \| \phi_n^+(0) \|_{k_+} = R. 
\]

\begin{lem} \label{lem gamma_n L^2_{k+1/2}}
There is a constant $C > 0$ such that for all $n$, 
\[
   \| \phi_{n}^+(0) \|_{k_+ +\frac{1}{2}} < C.
\]
\end{lem}

\begin{proof}
Note that
\[
       \frac{d}{dt} \bigg|_{t=0} \| \phi_n^+(t) \|^2_{k_+} = 0. 
\]
Let us consider the case when $k_+ \in \frac{1}{2}\mathbb{Z} \smallsetminus \mathbb{Z}$. 

Let $\pi^+$ be the $L^2_{k_+, k_-}$-projection onto $\mathcal{E}_{k_+, k_-}^+$. (That is, $\pi^+ = 1 - \pi_{P_0}$.)  Then we have
  \begin{align*}
 &  \frac{1}{2}   \frac{d}{dt}  \bigg|_{t=0} \| \phi_n^+ (t) \|_{k_+}^2  \\
 &= \frac{1}{2}  \frac{d}{dt} \bigg|_{t = 0} 
  \langle  |D'|^{k_+ + \frac{1}{2}} \pi^+  \phi_n(t),  |D'|^{k_+ - \frac{1}{2}}  \pi^+ \phi_n  (t)  \rangle_{0}  \\
& = 
\langle  (\nabla_{X_H} |D'|^{k_+ + \frac{1}{2}}) \phi_n^+(0),   |D'|^{k_+ - \frac{1}{2}} \phi_n^+(0) \rangle_0  \\
& \quad +   \langle   |D'|^{k_+ + \frac{1}{2}} \phi_n^+(0),  (\nabla_{X_H} |D'|^{k_+ - \frac{1}{2}}) \phi_n^+(0) \rangle_0     \\
 & \quad +  \operatorname{Re}  \langle  (\nabla_{X_H} \pi^+) \phi_n(0),   \phi_{n}^+(0)  \rangle_{k_+} + 
     \operatorname{Re} \Big\langle    \frac{d \phi_n}{dt}(0),   \phi_n^+(0)   \Big\rangle_{k_+}.  
 \end{align*}
Note that $k_+ + \frac{1}{2}$ and $k_+ - \frac{1}{2}$ are integers. 
By Lemma \ref{lem nabla |D'|^k}, 
  \begin{align*}
       \Big|   \langle (\nabla_{X_H} |D'|^{k_+ + \frac{1}{2}}) \phi_n^+(0),  |D'|^{k_+ -\frac{1}{2}} \phi_n^+(0) \rangle_{0} \Big|   
        &\leq C    \| \phi_n^+(0) \|_{k_+ - \frac{1}{2}}^2  \leq C R^2,  \\
      \Big| 
     \langle   |D'|^{k_+ + \frac{1}{2}} \phi_n^+(0),  (\nabla_{X_H} |D'|^{k_+ - \frac{1}{2}}) \phi_n^+(0) \rangle_0
        \Big|
      &  \leq C  \| \phi_{n}^+(0) \|_{k_+ + \frac{1}{2}} \| \phi_n^+(0) \|_{k_+ - \frac{1}{2}}  \\
 &       \leq C R \| \phi_n^+(0) \|_{k_+ +\frac{1}{2}}. 
   \end{align*}
By Proposition  \ref{prop perturbation P},    
  \begin{align*}
    &  \big|    \langle (\nabla_{X_H} \pi^+ ) \phi_n(0), \phi_{n}^+(0) \rangle_{k_+}       \big|   \\
    &  \leq  \|   (\nabla_{X_H} \pi^+ ) \phi_n(0)   \|_{k_+ } \|   \phi_{n}^+(0)  \|_{k_+}  \\
    & \leq C \| \phi_n(0) \|_{k_+ - 1}     \|   \phi_{n}^+(0)  \|_{k_+}  \\
    & \leq C \| \phi_n(0) \|_{\ell}   \  \| \phi_n^+(0) \|_{k_+} \\ 
     &  \leq CR^2. 
   \end{align*}
We have
  \begin{align*}
       & \Big\langle \frac{d\phi_n}{dt}(0), \phi_n^+(0) \Big\rangle_{k_+}   \\
        & = 
        -  \langle (\nabla_{X_H } \pi_{F_n}) \phi_n(0) + \pi_{F_n} ( D' \phi_n (0) - \mathbb{A} \phi_n(0) + c_1(\gamma_n(0))), \phi_n^+(0) \rangle_{k_+}.
\end{align*}
By Lemma \ref{lem nabla pi  phi psi},
\[
     \langle  \left( \nabla_{X_H} \pi_{F_n} \right) \phi_n(0), \phi_n^+(0) \rangle_{k_+} 
    = \langle   \left( \nabla_{X_H} \pi_{F_n} \right) \phi_n(0), \phi_n^+(0)    \rangle_{k_+, k_-} 
     = 0. 
\]
We  have
\begin{align*}
     \langle   \pi_{F_n} D'  \phi_n(0),  \phi_n^{+}(0) \rangle_{k_+}    
        & = \langle   D' \phi_n(0),  \pi_{F_n} \phi_n^+(0) \rangle_{k_+}   \\
        & = \langle D' \phi_n(0), \phi_n^+(0) \rangle_{k_+}  \\
        &=  \| \phi_n^+(0) \|_{k_+ +\frac{1}{2}}^2. 
\end{align*}
Since $\mathbb{A}$ is a smoothing operator, 
\[
           \big|  \langle \pi_{F_n}  \mathbb{A}  \phi_n(0), \phi_n^+(0) \rangle_{k_+} \big|   
           \leq C  \| \phi_n(0) \|_{0} \| \phi_n(0) \|_{k_+} \leq CR^2. 
\]
Since $D'$ is self-adjoint,
  \begin{align*}
    & \big| \langle  \pi_{F_n} c_1(\gamma_n(0)),  \phi_n^+(0) \rangle_{k_+}  \big|      \\
    & = \big| \langle   c_1(\gamma_n(0)),  \phi_n^+(0) \rangle_{k_+}  \big|  \\
    & =  \big|  \langle |D'|^{k_+} c_1(\gamma_n(0)),     |D'|^{k_+} \phi_n^+(0) \rangle_0 \big| \\
    & = \big|  \langle  |D'|^{k_+ - \frac{1}{2}}  c_1(\gamma_n(0)),  |D'|^{k_+ + \frac{1}{2}} \phi_n^+(0) \rangle_0  \big|   \\
    & \leq   \|   c_1(\gamma_n(0))  \|_{k_+ - \frac{1}{2}}
                 \|    \phi_n^+(0)       \|_{k_+ + \frac{1}{2}}    \\
     & \leq C  \|   c_1(\gamma_n(0))   \|_{\ell}
                    \|    \phi_n^+(0)     \|_{k_+ + \frac{1}{2}}  \quad ( \ell = \min \{ k_+, k_- \} ) \\
    & \leq CR^2   \| \phi_n^+(0) \|_{k_+ + \frac{1}{2}}. 
   \end{align*}
Therefore
\[
     0 = \frac{1}{2}  \frac{d}{dt}  \bigg|_{t=0}  \| \phi_n^+  (t)   \|_{ k_+}^2
      \leq    -  \| \phi_n^+(0) \|_{k_++\frac{1}{2}}^2 + C R^2 \| \phi_n^+(0) \|_{k_+ + \frac{1}{2}} + CR^2.
\]
This inequality  implies that the sequence $\| \phi_n^+(0) \|_{k_+ + \frac{1}{2}}$ is bounded. 

The proof in the case $k_+ \in \mathbb{Z}$ is similar. 
\end{proof}

It follows from Lemma \ref{lem gamma_n L^2_{k+1/2}} and the Rellich lemma that after passing to a subsequence,  $\phi_n^+(0)$ converges to $\phi^+(0)$ in $L^2_{k_+}$ strongly.  By the assumption,  $\| \phi_{n}^+(0) \|_{k_+} = R$ for all $n$. Hence, 
\[
      \| \gamma(0) \|_{k_+, k_-} \geq \| \phi^+(0) \|_{k_+,k_-} = R. 
\]
This contradicts (\ref{eq gamma t norm}). 

\vspace{2mm}

Let us consider the case (ii). In this case,  we have
\[
    \| \phi_n^-(0) \|_{k_-} = R. 
\]

\begin{lem} \label{lem phi^-}
There is a constant $C > 0$ such that for all $n$, 
\[
    \| \phi_{n}^-(0) \|_{k_- + \frac{1}{2}} < C.
\]
\end{lem}

\begin{proof}
Note that 
\[
     \langle D' \phi_n(0), \phi_n^-(0) \rangle_{k_-} =  - \| \phi_n^-(0) \|_{k_- + \frac{1}{2}}^2. 
\]
As in the proof of Lemma \ref{lem gamma_n L^2_{k+1/2}}, we can show that
\[
   0 = \frac{d}{dt} \bigg|_{t=0}  \| \phi_{n}^-(t) \|_{k_-}^2  
      \geq \| \phi_{n}^-(0) \|_{k_- + \frac{1}{2}}^2  - CR^2 \| \phi_n^+(0) \|_{k_- + \frac{1}{2}} - CR^2. 
\]
This implies that the sequence $\| \phi_{n}^-(0) \|_{k_- + \frac{1}{2}}$ is bounded. 

\end{proof}

By the Rellich lemma, $\phi_n^-(0)$   converges to $\phi^-(0)$ in $L^2_{k_-}$ strongly.  Hence
\[
    \| \gamma(0) \|_{k_+, k_-} \geq \| \phi^-(0) \|_{k_-} = R.
\]
We get a contradiction. 

In the other cases (iii), (iv) where $y_{n,0}$ is in the other components of $\partial A_n$,   we have a contradiction similarly.

\begin{dfn}
	For this definition we refer to some notions from parameterized homotopy theory and parameterized Conley index theory; refer to Sections \ref{subsec:homotopy1} and \ref{subsec:conley}), respectively.  For notation as in Theorem \ref{thm isolating nbd}, let $\preSWF_{[n]}(Y,\mathfrak{s})$ be the parameterized Conley index of the flow $\varphi_{n,k_+,k_-}$ on the isolated invariant set $A_n$.   
	We call $\preSWF_{[n]}(Y,\mathfrak{s})$ the \emph{pre-Seiberg-Witten Floer invariant} of $(Y,\mathfrak{s})$ (for short, the pre-SWF invariant of $(Y,\mathfrak{s})$).  The object $\preSWF_{[n]}(Y,\mathfrak{s})$ is a(n) (equivariant) topological space, depending on a number of choices (which are not all reflected in its notation).  First, $\preSWF_{[n]}(Y,\mathfrak{s})$ depends on the choice of an index pair, but its (equivariant, parameterized) homotopy type is independent of the choice of index pair - we will abuse notation and also write $\preSWF_{[n]}(Y,\mathfrak{s})$ for its (equivariant, parameterized) homotopy type.  It also depends on a choice of metric on $Y$, as well as spectral sections $P_n,Q_n$ and subspaces $W_n^\pm$, as in the preliminaries to Theorem \ref{thm isolating nbd}.  
	
	The projection used in the parameterized Conley index is from the ex-space $B_{n,R}$ over $\mathrm{Pic}(Y)$, as explained in the discussion after Theorem \ref{thm isolating nbd}.   
	
	 We write $\unparamSWF_{[n]}(Y,\mathfrak{s})$ to refer to the Conley index with trivial parameterization.  By Lemma \ref{lem:conley-index-pushforward},  $\nu_{!}\preSWF_{[n]}(Y,\mathfrak{s})=\preSWF^u_{[n]}(Y,\mathfrak{s})$, where $\nu : B \to *$ is the map collapsing the Picard torus to a point, and $\nu_{!}$ is as defined in Section \ref{subsec:homotopy1} of the Appendix.
	\end{dfn}

If $\mathfrak{s}$ is a self-conjugate $\mathrm{spin}^c$-structure, the bundle $L^2_k(\mathbb{S})\times \mathcal{H}^1(Y)\times L^2_k(\im d^*)$ admits a $\mathrm{Pin}(2)$-action extending the $S^1$-action on spinors, by 
\[
j(\phi,v,\omega)=(j\phi,-v,-\omega).
\]
In the event that the spectral sections $P_n,Q_n$ are preserved by the $\mathrm{Pin}(2)$-action, then the approximate flow on $F_n\oplus W_n$ will be $\mathrm{Pin}(2)$-equivariant, and we define $\preSWF^{\mathrm{Pin}(2)}_{[n]}(Y,\mathfrak{s})$ to be the $\mathrm{Pin}(2)$-equivariant parameterized Conley index, so that its underlying $S^1$-space is $\preSWF_{[n]}(Y,\mathfrak{s})$.  We similarly define $\preSWF^{u,\mathrm{Pin}(2)}_{[n]}(Y,\mathfrak{s})$ (and we will occasionally write $\preSWF^{u,S^1}_{[n]}(Y,\mathfrak{s})$ to distinguish what equivariance is meant). See Theorem \ref{thm Pin(2) spectral sections} for the existence of $\rm{Pin}(2)$-equivariant spectral sections.



\chapter{Well-definedness}\label{sec:well-def}
Here we show how changing the choices in the construction above effect the resulting space output.

\section{Variation of Approximations}\label{subsec:variation-of-appx}
First, we consider the change due to passing between different approximations.  For this section, we fix a $3$-manifold with $\mathrm{spin}^c$-structure $(Y,\mathfrak{s})$.

As before, let $P_n,Q_n$ be spectral sections of $-D, D$ with
\[
\begin{split}
&   (\mathcal{E}_0(D))_{-\infty}^{\mu_{n,-}} \subset P_n \subset (\mathcal{E}_0(D))_{-\infty}^{\mu_{n,+}},    \\
&    (\mathcal{E}_0(D))_{\lambda_{n, +}}^{\infty} \subset Q_n \subset (\mathcal{E}_0(D))_{\lambda_{n, -}}^{\infty}. 
\end{split}
\]
We may assume that $|\mu_{n,+}-\mu_{n,-}|$ and $|\lambda_{n,+}-\lambda_{n,-}|$ are bounded. 
 We call any such sequence of spectral sections a \emph{good sequence of spectral sections}.

Fix half-integers $k_+, k_-  > 5$.  Put $\ell = \min \{ k_+, k_- \}$. 

Let $F_n=P_n\cap Q_n\subset (\mathcal{E}_0)^{\mu_{n,+}}_{\lambda_{n,-}}$, as before.  Fix $\mathbb{H}$ to be the quaternion representation of $\mathrm{Pin}(2)$, and let $B=\mathrm{Pic}(Y)$ denote the Picard torus of $Y$.  We write $I(\varphi,S)$ for the (parameterized) Conley index of a flow $\varphi$ and isolated invariant set $S$; we will usually suppress $S$ from the notation, and $I^u(\varphi,S)$ for the unparameterized version; see Section \ref{subsec:conley}.  Finally, a further bit of notation for the statement of the following theorem.  Let $\mathrm{Th}(E,Z)$, for a vector bundle $\pi: E\to Z$, denote the Thom construction of $\pi$.

\begin{thm}\label{thm change of approximation suspension}
	Let $\eta^P_n: P_{n+1}\to P_n \oplus \C^{k_{P,n}}$ and $\eta^Q_n : Q_{n+1}  \to Q_{n}\oplus \C^{k_{Q,n}} $ be vector-bundle isometries (with respect to the $k_\pm$-metric), where $\C^{k_{P,n}}$ and $\C^{k_{Q,n}}$ are the trivial bundles over $B$ of rank $k_{P,n}$ and $k_{Q,n}$.  Let $\eta^{W,+}_n: W_{n+1}^+\to W_n^+\oplus \R^{k_{W,+,n}}$ and $\eta^{W,-}_n: W_{n+1}^-\to W_{n}^-\oplus \R^{k_{W,-,n}}$ be another pair of isometries.  Then there is a $S^1$-equivariant parameterized homotopy equivalence of Conley indices
	\[
	\eta_*: I(\varphi_{n+1})\to \Sigma_B^{\C^{k_{Q,n}}\oplus \R^{k_{W,-,n}}}I(\varphi_n),
	\]  
which is well-defined up to homotopy for the induced map: 
	\[ 
\nu_{!}\eta_* :  I^u(\varphi_{n+1})\to \Sigma^{\C^{k_{Q,n}}\oplus \R^{k_{W,-,n}}}I^u(\varphi_n).
	\]
	Furthermore, if $\mathfrak{s}$ is a self-conjugate $\mathrm{spin}^c$ structure and instead $\eta^P_n: P_{n+1}\to P_n\oplus \H^{k_{\H,P,n}}$ and $\eta^Q_n: Q_{n+1}\to Q_n\oplus \H^{k_{\H,Q,n}}$, and the maps $\eta^{W,\pm}$ above are equivariant with respect to the $C_2$-action on $W_{n+1}, W_n$ and $\tilde{\R}^{k_{W,\pm,n}}$, then there is a well-defined, up to equivariant homotopy, $\mathrm{Pin}(2)$-equivariant homotopy equivalence
\[
	\nu_!\eta_*: I^u(\varphi_{n+1})\to \Sigma^{\H^{k_{\H,Q,n}}\oplus \tilde{\R}^{k_{W,-,n}}}I^u(\varphi_n),
\]
and similarly for the parameterized version.

The restriction $\eta_*$ to the $S^1$-fixed point set $I(\varphi_{n+1})^{S^1}$ is a fiber-preserving homotopy equivalence to $\Sigma_B^{\mathbb{R}^{k_{W, -,n}}} I_{n}(\varphi)^{S^1}$. 

More generally, without a selection of maps $\eta^\circ_n$ as above, there is an $S^1$-equivariant parameterized homotopy equivalence of Conley indices 
	\[
\eta_*: I(\varphi_{n+1})\to \Sigma_B^{Q_{n+1}/Q_n}\Sigma_B^{W_{n+1}^-/W_n^{-}}I^u(\varphi_n),
\]  
so that the induced, unparameterized map
\[
\nu_!\eta_*: I^u(\varphi_{n+1})\to \mathrm{Th}(Q_{n+1}/Q_n\oplus W_{n+1}^-/W_n^-,I^u(\varphi_n)),
\]
is well-defined up to homotopy, as well as a similar statement for self-conjugate $\frak{s}$.
\end{thm}

\begin{proof}
	By Lemma \ref{lem suspension premain} below and invariance of the Conley index under deformations, there is a well-defined homotopy equivalence $\eta^1: I^u(\varphi_{n+1})\to I^u(\varphi_{n+1}^{\mathrm{split}})$, where $\varphi_{n+1}^{\mathrm{split}}$ is defined in Lemma \ref{lem suspension premain} (and similarly for the parameterized version).  Using the invariance of the Conley index under homeomorphism, we have a well-defined homotopy equivalence 
	\[
	\eta^2: I(\varphi_{n+1}^{\mathrm{split}})\to I(\varphi_{n+1}^{\mathrm{split},\eta}),
	\]
	where $\varphi_{n+1}^{\mathrm{split},\eta}$ is defined at Lemma \ref{lem controlled exit sets}.	 Finally, by Lemma \ref{lem controlled exit sets}, the well-definedness of the Conley index (independent of a choice of index pair), and the definition of the Conley index (using our choice of index pair from Lemma \ref{lem controlled exit sets}), there is a well-defined homotopy equivalence
	\[
	\eta^3 : I(\varphi_{n+1}^{\mathrm{split},\eta})\to \Sigma_B^{Q_{n+1}/Q_n}\Sigma_B^{W_{n+1}^-/W_n^{-}}I(\varphi_n).
	\]  
	In the case that we have fixed trivializations, as above, of $W_{n+1}^{-}/W_n^{-}$ and $Q_{n+1}^-/Q_n^{-}$, the target of $\eta^3$ is identified with
	\[
	     \Sigma_{B}^{\C^{k_{Q,n}} \oplus \R^{k_{W, -, n}}  }   I(\varphi_n)  
	\]

	Since the flows used to define the  homotopy equivalences preserve the fibers of the $S^1$-fixed point sets (that is, $X(\phi)_H = 0$ if $\phi = 0$),  we can see from the formulas for the maps $f, g,  F_{\lambda}, G_{\lambda}$ in the proof of \cite[Theorem 6.2]{MRS}  that the restrictions  of $\eta^{1}, \eta^{2}, \eta^{3}$ to the $S^1$-fixed point sets preserve the fibers.    
	
	The argument adapts immediately to the case in which there is a spin structure, and the Theorem follows.
	
\end{proof}

Let $\Sigma_{n+1}^{\pm}$ be the $L^2_{k_+,k_-}$-orthogonal complement to $P_n$ in $P_{n+1}$ (resp. $Q_n$ in $Q_{n+1}$).  Similarly, let $\Sigma^{W,\pm}_{n+1}$ be the $L^2_{k_+,k_-}$-orthogonal complement to $W_{n}^{\pm}$ in $W_{n+1}^{\pm}$.  Let $\Sigma_{n+1}=\Sigma_{n+1}^+\oplus \Sigma_{n+1}^-$ and $\Sigma^W_{n+1}=\Sigma_{n+1}^{W,+}\oplus \Sigma_{n+1}^{W,-}$.  Then $F_{n+1}=F_n\oplus \Sigma_{n+1}$ and $W_{n+1}=W_n\oplus \Sigma^W_{n+1}$.  Write $\pi_{\Sigma_{n+1}}$ for the projection to $\Sigma_{n+1}$ with respect to $L^2_{k_+,k_-}$-norm.  We also write $\pi_{\Sigma^W_{n+1}}$ for the projection $\Sigma^W_{n+1}$ with respect to $L^2_{k_+,k_-}$-norm.   

Let $\X_n$ be the approximate Seiberg-Witten vector field on $F_{n} \oplus  W_n$, for all $n$, as defined in (\ref{eq for gamma}).  Let $R$ be large enough as in Theorem \ref{thm isolating nbd}.

For a path $\gamma(t)$ in the total space of $F_{n+1} \oplus  W_{n+1}$, we write $\gamma(t)=(\phi^{(1)}(t)+\sigma(t))\oplus( \omega^{(1)}(t)+\omega^{(2)}(t))$, as an element in the fiber over $b(t)=p(\gamma(t))$, where $\phi^{(1)}(t)$ is an element of $(F_n)_{b(t)}$, $\sigma(t)\in (\Sigma_n)_{b(t)}$, $\omega^{(1)}(t)\in (W_n)_{b(t)}$, and $\omega^{(2)}(t)\in (\Sigma^W_{n})_{b(t)}$.  

We then write $\gamma(t)=(\phi^{(1)}(t),\sigma(t),\omega^{(1)}(t),\omega^{(2)}(t), b(t))$ to describe $\gamma$ in terms of these coordinates.  We also write $\phi_{n+1}(t)$ to refer to the path in the total space of $F_{n+1}$ determined by $(\phi_{n+1}^{(1)}(t),\sigma_{n+1}(t),b(t))$.

\begin{lem}\label{lem suspension premain}
	Let $\X_{n}^{\mathrm{split}}$ be the vector field on the total space of $(F_n\oplus \Sigma_n)\oplus (W_n\oplus \Sigma^W_n)$ defined by (\ref{eq split flow}), where 
	\[
	\gamma_{n+1}(t)=(\phi^{(1)}_{n+1}(t),\sigma_{n+1}(t),\omega_{n+1}^{(1)}(t),\omega_{n+1}^{(2)}(t),  b_{n+1}(t))
	\]
 and $\hat{\gamma}_{n+1}(t)$
is the path obtained by (fiberwise) projecting $\gamma_{n+1}(t)$ to $(F_n\oplus W_n)_{b_{n+1}(t)}$.  
	\begin{equation}	
	\begin{aligned}  \label{eq split flow}
	  \frac{d\phi_{n+1}^{(1)}}{dt} (t)    &= - \chi \{ ( \nabla_{X_H} \pi_{F_{n}}) \phi_{n+1}^{(1)}(t) +\pi_{F_{n}} (D \phi_{n+1}^{(1)}(t) + c_1(\hat{\gamma}_{n+1}(t))) \},  \\ \frac{d\sigma_{n+1}}{dt} (t)    &= -\chi \{ ( \nabla_{X_H} \pi_{\Sigma_{n+1}}) \sigma_{n+1}(t) + \pi_{\Sigma_{n+1}} (D \sigma_{n+1}(t)) \} , \\
	\frac{db_{n+1}}{dt}(t)  &= - \chi X_{H}( \phi_{n+1}^{(1)}(t)),  \\
	\frac{d\omega^{(1)}_{n+1}}{dt}(t) &= - \chi \{ *d \omega^{(1)}_{n+1}(t) + \pi_{W_{n}} c_2(\hat{\gamma}_{n+1}(t)) \} \\
	\frac{d\omega^{(2)}_{n+1}}{dt}(t) &= - \chi *d \omega^{(2)}_{n+1}(t). 
	\end{aligned}
	\end{equation} 
 Here $\chi$ is the cut off function  in (\ref{eq for gamma}).
	Then, for $n$ sufficiently large, there is a continuous family of vector fields $\X^\tau_{n+1}$ on  (the total space of) $F_{n+1}\oplus W_{n+1}$ between $\X_{n+1}$ and $\X_{n+1}^{\mathrm{split}}$, with associated flows $\varphi^\tau_{n+1}$, so that $A_{n+1}$ is an isolating neighborhood for all $\tau$, where: 
	 \begin{align*}
	A_{n+1}
	 =& A_n^{o}\times_B B_{k_+}(\Sigma_{n+1}^+;R)\times_B B_{k_-}(\Sigma_{n+1}^-;R)   \\
	&   \times_B B_{k_+}(\Sigma_{n+1}^{W,+};R)\times_B B_{k_-}(\Sigma_{n+1}^{W,-};R),
	\end{align*}
	where $A_n^{o}$ is as $A_n$ in the proof of Theorem \ref{thm isolating nbd}. 
\end{lem}

\begin{proof}
	This is an immediate consequence of Lemmas \ref{lem suspension 1}, \ref{lem suspension lemma 2} and \ref{lem suspension lemma 3}.
\end{proof}
We construct the homotopy $\X^\tau_{n+1}$, with associated flow $\varphi^\tau_{n+1,k_+,k_-}$, in stages.
\begin{lem}\label{lem suspension 1}
	Let $\X^\tau_{n+1}$ for $\tau\in [0,1]$ be defined by:
	\begin{equation}\label{eq first homotopy}	
	\begin{aligned}  
	 \frac{d\phi_{n+1}^{(1)}}{dt} (t)    
	 =&  -\chi \{ ( \nabla_{X_H} \pi_{F_{n}})(\phi_{n+1}^{(1)}(t)+\sigma_{n+1}(t)) \\
	 & + (1-\tau)\pi_{F_n}(D\phi_{n+1}(t)+c_1(\gamma_{n+1}(t)))  \\
	 & + \tau\pi_{F_n}(D(\phi_{n+1}^{(1)})+c_1(\hat{\gamma}_{n+1}(t)))  \},  \\
	 \frac{d\sigma_{n+1}}{dt} (t)    =& - \chi \{ ( \nabla_{X_H} \pi_{\Sigma_{n+1}}) (\phi_{n+1}^{(1)}+\sigma_{n+1}(t)) + (1-\tau)\pi_{\Sigma_{n+1}} (D(\phi_{n+1}^{(1)}(t)  \\ &+ \sigma_{n+1}(t))+c_1(\gamma_{n+1}(t))) + \tau\pi_{\Sigma_{n+1}}D\sigma_{n+1}(t) \}, \\
 \frac{db_{n+1}}{dt}(t)  =& - \chi X_{H}( \phi_{n+1}(t)),  \\
\frac{d\omega^{(1)}_{n+1}}{dt}(t) =& - \chi \{ *d \omega^{(1)}_{n+1}(t) + \tau \pi_{W_{n}} c_2(\hat{\gamma}_{n+1}(t))+(1-\tau)\pi_{W_n}c_2(\gamma_{n+1}(t))  \}, \\
\frac{d\omega^{(2)}_{n+1}}{dt}(t) =& - \chi  \{ *d \omega^{(2)}_{n+1}(t) + (1-\tau)\pi_{\Sigma^W_{n+1}}c_2(\gamma_{n+1}(t))\}.
	\end{aligned}
	\end{equation}
	Here $\chi$ is the cut off function in (\ref{eq for gamma}). 
	Then, for all $n\gg 0$, $A_{n+1}$ is an isolating neighborhood of $\varphi^\tau_{n+1,k_+,k_-}$ for all $\tau\in [0,1]$.
\end{lem}
\begin{proof}
	The lemma is a consequence of Lemmas \ref{lem parameterized bound}, \ref{lem parameterized convergence} and \ref{lem parameterized convergence 2}.  Indeed, let 
	\[A_n^{o}=\big(  B_{k_+} (F_n^+; R) \times_B B_{k_-}(F_n^-; R) \big)  \times_B   \big( B_{k_+}(W_n^+; R) \times_B B_{k_-}(W_n^-; R) \big)\] be as in the proof of Theorem \ref{thm isolating nbd}.  Suppose that \[\mathrm{inv}\, A_{n+1}\not\subset \mathrm{int}\, A_{n+1},\]  for some $\tau_n\in[0,1]$, for all $n$.  Then there is a sequence of finite-energy approximate trajectories $\gamma_{n+1}(t)$, for $\varphi^{\tau_{n+1}}_{n+1,k_+,k_-}$, so that $\gamma_{n+1}(0)\in \partial A_{n+1}.$  There are four cases as in the proof of Theorem \ref{thm isolating nbd}; we only treat the case that 
	   \begin{align*}
   \gamma_{n+1}(0)   &\in  (S_{k_+}(F_{n+1}^+; R) \times_B B_{k_-}(F_{n+1}^-; R))    \\
     & \quad  \times_{B} (B_{k_-}(W^+_{n+1}; R) \times_{B} B_{k_-}(W^-_{n+1}; R))
     \end{align*} 
	for all $n$, the other cases being similar. 
	
	As in the proof of Theorem \ref{thm isolating nbd}, we have a lift
	\[
	\tilde{\gamma}_{n+1} = (\tilde{\phi}_{n+1}, \omega_{n+1}) : \mathbb{R} \rightarrow \mathcal{H}^1(Y) \times  L^2_{k_+, k_-}(\mathbb{S}) \times L^2_{k_+, k_-}(\im d^*)
	\]
	with $p(\tilde{\gamma}_{n+1}(0))\in \Delta$.  
	
	By Lemma \ref{lem parameterized convergence} and Proposition \ref{prop Ascoli}, the sequence $\tilde{\gamma}$ has a subsequence converging, uniformly in $(\ell-1)$-norm to some continuous map 
\[
   \tilde{\gamma}: I \to \mathcal{H}^1(Y)\times L^2_{k_+ - 1, k_- - 1}(\mathbb{S}) \times L^2_{k_+ - 1, k_- - 1}(\im d^*).  
\]   
   By Lemma \ref{lem parameterized convergence 2}, $\tilde{\gamma}$ is a solution of the Seiberg-Witten equations.  Finally, by Lemma \ref{lem parameterized bound}, we obtain that the sequence $\tilde{\phi}_n^+(0)$ converged to $\tilde{\phi}^+(0)$ uniformly in $L^2_{k_+}$-norm, which is a contradiction. 
\end{proof}

\begin{lem}\label{lem parameterized bound}
	Assume that we have a sequence of trajectories $\tilde{\gamma}_{n+1}$ as in the proof of Lemma \ref{lem suspension premain}, with in particular 
	   \begin{align*} 
	\gamma_{n+1}(0)  & \in  (S_{k_+}(F_{n+1}^+; R) \times_B B_{k_-}(F_{n+1}^-; R))    \\
	 & \quad    \times_B (B_{k_-}(W^+_{n+1}; R) \times_B B_{k_-}(W^-_{n+1}; R)).
	   \end{align*} 
	Then there is some $R_1$ so that 
	\[
             \Vert \phi_{n+1}^+(0) \Vert_{k_++\frac{1}{2}}<R_1,
	\]
	for all $n$.  
\end{lem}

\begin{proof}
We emphasize only what must be changed from the proof of Lemma \ref{lem gamma_n L^2_{k+1/2}}.  We check the case where $k_+$ is an integer.  We calculate:
 \begin{align*}
 &  \frac{1}{2}    \frac{d}{dt}  \bigg|_{t=0} \| \phi_{n+1}^+ (t) \|_{k_+}^2  \\
& =  
\operatorname{Re} ( \langle (\nabla_{X_H} (D')^{k_+}) \phi_{n+1}^+(0),   (D')^{k_+} \phi_{n+1}^+(0)  \rangle_{0}  \\
& +     \langle  (\nabla_{X_H} \pi^+) \phi_{n+1}(0), \phi_{n+1}^+(0) \rangle_{k_+}   \\
&    -  \langle  (\nabla_{X_H} \pi_{F_{n+1}}) \phi_{n+1}(0), \phi_{n+1}^+(0)   \rangle_{k_+}  
-  \langle  (1-\tau)  \pi_{F_{n}}  D' \phi_{n+1}(0), \phi_{n+1}^+(0)  \rangle_{k_+}   \\
&  +   \langle  ((1-\tau)\mathbb{A} - (1-\tau)   \pi_{F_{n+1}}) c_1(\gamma_{n+1}(0)), \phi_{n+1}^+(0) \rangle_{k_+} \\
& -\tau\langle \pi_{F_n}D(\phi^{(1)}_{n+1}(0)),\phi^+_{n+1}(0)\rangle_{k_+} 
-\langle \tau \pi_{F_n}c_1(\hat{\gamma}_{n+1}(0)),\phi^+_{n+1}(0)\rangle_{k_+}  \\
&- \langle \pi_{\Sigma_{n+1}}D\sigma_{n+1}(0), \phi^+_{n+1}(0)\rangle_{k_+} \\
&- (1-\tau)\langle \pi_{\Sigma_{n+1}}D(\phi^{(1)}_{n+1}(0) +c_1(\gamma_{n+1}(0))),\phi^+_{n+1}(0)\rangle_{k_+}). 
\end{align*}

Following the argument of Lemma \ref{lem gamma_n L^2_{k+1/2}}, we obtain
\begin{align*}
 & \frac{1}{2}    \frac{d}{dt}  \bigg|_{t=0} \| \phi_{n+1}^+ (t) \|_{k_+}^2  \\ 
 &\leq CR^3  \Vert \phi_{n+1}^+(0) \Vert_{k_++\frac{1}{2}}\\
  & -\langle \pi_{F_{n+1}}D'\phi_{n+1}(0),\phi^+_{n+1}(0) \rangle_{k_+}+ \tau( \langle \pi_{\Sigma_{n+1}}D\phi^{(1)}_{n+1}(0),\phi^+_{n+1}(0)\rangle_{k_+}  \\
 & +\langle \pi_{F_n}D\sigma_{n+1}(0),\phi^+_{n+1}(0)\rangle_{k_+} ).
  \end{align*}

But 
\[
   \langle \pi_{F_{n+1}}D'\phi_{n+1}(t),\phi_{n+1}^+(t) \rangle_{k_+}
   = \Vert \phi^{+}_{n+1}(t) \Vert^2_{k_++\frac{1}{2}}.
\]

	Since $[D',\pi_{\Sigma_{n+1}}]$ is uniformly bounded, we obtain
	\[
	\tau\langle\pi_{\Sigma_{n+1}}D\phi^{(1)}_{n+1},\phi^+_{n+1}\rangle_{k_+}\leq CR^2
	\]
	for some constant $C$ independent of $n$.  
	
	A similar argument applies to $\langle\pi_{F_n}D\sigma_{n+1},\phi^+_{n+1}\rangle_{k_+}$.  The lemma then follows as did Lemma \ref{lem gamma_n L^2_{k+1/2}}.
	\end{proof}

\begin{lem}\label{lem parameterized convergence}
	The sequence $(\tilde{\phi}_n,\omega_n)$ is equicontinuous in $L^2_{K_T,\ell-5,w}$ norm.
\end{lem}
\begin{proof}
	This follows exactly as in the proof of Theorem \ref{thm isolating nbd}.  
	\end{proof}

By Proposition \ref{prop Ascoli}, any sequence which is equicontinuous in $L^2_{K_T,\ell-5,w}$-norm and bounded in $\ell$-norm has a subsequence converging, uniformly in $\Vert\cdot \Vert_{\ell-1}$, to some continuous map $\tilde{\gamma}: I \to \mathcal{H}^1(Y)\times L^2_{\ell-1}(\mathbb{S})\times L^2_{\ell-1}(\im d^*)$.
\begin{lem}\label{lem parameterized convergence 2}
	A limit $\tilde{\gamma}$ for the sequence $(\phi_n,\omega_n)$ as above, is a solution of the Seiberg-Witten equations over $Y\times \R$.  
\end{lem}
\begin{proof}
	
	Take  $T \in \mathbb{Z}_{>0}$ and $t \in [-T, T]$. We have 
	\begin{align*}
	& \tilde{\phi}_{n+1}(t) - \tilde{\phi}_{n+1}(0)  \\
	&= \int_{0}^{t}  \frac{d \tilde{\phi}_{n+1}}{ds}(s) ds  \\
	&= - \int_{0}^{t}   Z_1+Z_2+Z_3 + 
	\pi_{\tilde{F}_{n+1}} (D (\tilde{\phi}^{(1)}_{n+1}(t)+\sigma_{n+1}(t)) + c_1(\tilde{\gamma}_n(t)))\\ &+ X_{H}(\phi_{n+1}(s))  ds,
	\end{align*}
	where 
	\begin{equation}
	\begin{aligned}
	&Z_1=(\nabla_{X_H(\phi_{n+1}(t))}\pi_{F_{n+1}})\tilde{\phi}_{n+1}, \\
		&Z_2= -\tau\pi_{\Sigma_{n+1}}D\phi_{n+1}^{(1)}-\tau\pi_{F_n}D\sigma_{n+1}(t), \\
		&Z_3=-\tau(\pi_{\Sigma_{n+1}}c_1(\tilde{\gamma}_n(t))+\pi_{F_n}c_1(\tilde{\gamma}_n(t))-\pi_{F_n}c_1(\tilde{\hat{\gamma}}_n(t))).
	\end{aligned}
	\end{equation}
	It suffices to show that the $Z_i$ terms approach $0$ uniformly in $L^2_{K_T,\ell-5,w}$, and that 
	\[
	\pi_{\tilde{F}_{n+1}}(D(\tilde{\phi}_{n+1})+c_1(\tilde{\gamma}_{n+1}(t))+X_H(\phi_{n+1}(t))\to D(\tilde{\phi}(t))+c_1(\tilde{\gamma}(t))+X_H(\phi(t)),
	\]also in $L^2_{K_T,\ell-5,w}$.  Indeed, if that is the case, then the limit of integrals on the right-hand side is well-defined, and
	\begin{equation}\label{eq parameterized convergence 2}
	\tilde{\phi}(t)-\tilde{\phi}(0)=-\int^t_0(D\tilde{\phi}+c_1(\tilde{\gamma}(t))+X_H(\phi(s)))ds,
	\end{equation}
	giving the conclusion of the Lemma.
	
	Exactly as in the proof of Theorem \ref{thm isolating nbd}, we obtain that $Z_1$ converges to $0$ uniformly in $L^2_{K_T,\ell-5,w}$.

%

To show that $\pi_{F_n}D\sigma_{n+1}(t)\to 0$ in $L^2_{K_T,\ell-5,w}$, we use an elementary observation about projection with respect to different norms.  That is, if $V$ is a finite-dimensional vector space with norms $\Vert \cdot| \Vert_1$ and $\Vert \cdot \Vert_2$, then for a subspace $V'\subset V$ and projection $\Pi_1$ to $V'$ with respect to $\Vert \cdot \Vert_1$, then $\Vert \Pi_1 x \Vert _2/ \Vert x \Vert_2\leq \rho_1\rho_2$ for $x\in V$, where $\rho_2=\sup_{x\in V^*}\{\Vert x \Vert_2/ \Vert x \Vert_1\}$ and $\rho_1=\sup_{x\in V}\{\Vert x \Vert_1 / \Vert x \Vert_2\}$.  

We say a collection of finite-dimensional vector spaces $V_i$ with norms $\Vert \cdot \Vert_{1,i}$ and $\Vert \cdot \Vert_{2,i}$ is \emph{controlled} if $\rho_{1,i}\rho_{2,i}$ is bounded above.  

We claim that the  orthogonal complement of $F_n$ in $(\mathcal{E}_{\lambda_{n,-}}^{\mu_{n,+}})_a$,
call it $F_n^\perp$, with norms given by the restriction of $L^2_{k_+,k_-}$ and $L^2_{k_+-1,k_--1}$ (respectively), is controlled.  
 Indeed, $F_n^\perp$ is a subspace of $(\mathcal{E}^{\mu_{n,+}}_{\mu_{n,-}})_a$.  On $(\mathcal{E}^{\mu_{n,+}}_{\mu_{n,-}})_a$, by definition we have $\rho_1\rho_2<\mu_{n,+}/\mu_{n,-}$.  By our condition on the growth of the $\mu_{n,\pm}$, we then have that $\rho_{1,n}\rho_{2,n}$ is bounded as a function of $n$.  

We claim that $\pi_{F_n}D\sigma_{n+1}(t)\to 0$ in $L^2_{k_+-2,k_--2}$.  Indeed, $\sigma_{n+1}(t)$ converges to $0$ weakly in $L^2_{k_+,k_-}$ by definition and $\sigma_{n+1}(t)$ converges strongly to $0$ in $L^2_{k_+-1,k_--1}$.  Then $D\sigma_{n+1}(t)$ converges to $0$ in $L^2_{k_+-2,k_--2}$.  Finally, $\pi_{F_n}$ is a bounded family of operators in $L^2_{k_+-2,k_--2}$ by the above argument, giving the claim.  As a consequence, we also have convergence in $L^2_{K_T,\ell-5,w}$.  

To show that $\pi_{\Sigma_{n+1}}D\phi^{(1)}_{n+1}$ converges to $0$, we note that by Proposition \ref{prop [D, pi]}, 
\[
    \Vert [D,\pi_{\Sigma_{n+1}}]: L^2_j\to L^2_j \Vert \leq C
\]
for some constant $C$ independent of $n$, for all half-integers $j\leq k_+$.  Moreover, we have $\pi_{\Sigma_{n+1}} \phi^{(1)}_{n+1}=0$, and so we need only show that the sequence $[\pi_{\Sigma_{n+1}},D]\phi^{(1)}_{n+1}$ converges to zero.  Given the bound on $\sigma D \phi^{(1)}_{n+1}$ from the bound on the commutator $[D,\pi_{\Sigma_{n+1}}]$ above, and using the definition of the norms involved, we see that $\pi_{\Sigma_{n+1}} D \phi^{(1)}_{n+1}\to 0$ in $L^2_{\ell-1}$-norm.  



A very similar argument shows that $\pi_{\Sigma_{n+1}}c_1(\gamma_n(t))\to 0$ in $L^2_{K_T,\ell-5,w}$, and also that $\pi_{F_n}c_1(\gamma_n(t))$ and $\pi_{F_n}c_1(\hat{\gamma}_n(t))$ converge to $c_1(\gamma(t))$ in $L^2_{K_T,\ell-5,w}$, so that $Z_3\to 0$.

A similar argument also shows the convergence in (\ref{eq parameterized convergence 2}), and the proof is complete.

\end{proof}

For $\tau\in [1,2]$, define a flow  $\varphi^\tau_{n+1,k_+,k_-}$ on $F_{n+1} \oplus W_{n+1}$ by 
\begin{align*}
\frac{d\phi^{(1)}_{n+1}}{dt}(t)&=- \chi \{ (2-\tau)(\nabla_{X_H(\phi_{n+1}(t))}\pi_{F_n})(\phi_{n+1}(t)) \\
&+ (\pi_{F_n}D\phi^{(1)}_{n+1}(t)+c_1(\hat{\gamma}_{n+1}(t)))+(\tau-1)(\nabla_{X_H(\phi_{n+1}(t))}\pi_{F_n})\phi^{(1)}_{n+1}(t) \}, \\
\frac{d\sigma_{n+1}}{dt}(t) &=- \chi \{ (2-\tau)(\nabla_{X_H(\phi_{n+1}(t))}\pi_{\Sigma_{n+1}})(\phi^{(1)}_{n+1}(t)+\sigma_{n+1}(t))\\
&+ (\tau-1)(\nabla_{X_H(\phi_{n+1}(t))}\pi_{\Sigma_{n+1}})\sigma_{n+1}(t)+ \pi_{\Sigma_{n+1}}D\sigma_{n+1}(t) \},
\end{align*}
with the other terms unchanged.  Inspection shows that the total space of $F_{n+1} \oplus W_{n+1}$ is preserved by the flow.  
\begin{lem}\label{lem suspension lemma 2}
	For $n\gg 0$, for all $\tau\in [1,2]$, $A_{n+1}$ is an isolating neighborhood for $\varphi^\tau_{n+1,k_+,k_-}$.
\end{lem}
\begin{proof}

We highlight only the difference of the argument from the proof of Lemma \ref{lem suspension 1}.  We have a sequence of trajectories 
\[
\gamma_{n+1}(t)=(\phi^{(1)}_{n+1}(t),\sigma_{n+1}(t),\omega_{n+1}(t))
\]
 exactly as in that argument.  We assume that 
   \begin{align*}
 \gamma_{n+1}(0)  & \in  (S_{k_+}(F_{n+1}^+; R) \times_B B_{k_-}(F_{n+1}^-; R))     \\
   & \quad    \times_{B} (B_{k_-}(W^+_{n+1}; R) \times_{B} B_{k_-}(W^-_{n+1}; R))
   \end{align*}
 for all $n$; the other cases are similar.  The proofs of the analogs of Lemma \ref{lem parameterized convergence} and Lemma \ref{lem parameterized convergence 2} are unchanged, and we obtain that a lift $\tilde{\gamma}_{n}$   of $\gamma_n$ to the universal covering
 converges in $L^2_{K_T, \ell-5, w}$-norm to a solution $\tilde{\gamma}(t)$ of the Seiberg-Witten equations.    We need only prove an analogue of Lemma \ref{lem parameterized bound}, that $\Vert\phi^{+}_{n+1}\Vert_{k_++\frac{1}{2}}$ is bounded independent of $\tau,n$.  Suppose this is false, that is, that 
\[
     \Vert \phi_{n+1}^{(1),+}(0)+\sigma_{n+1}^+(0) \Vert_{k_++\frac{1}{2}}\to \infty.
\]
Then, we study (for the case $k_+\in \mathbb{Z}$, the other case being similar):
\begin{equation}\label{eq moving metrics}
\begin{aligned}
&\frac{1}{2}     \frac{d}{dt} \bigg|_{t=0} \Vert \phi_{n+1}^{(1),+}(t)+\sigma_{n+1}^+(t) \Vert_{k_+}^2 \\
&= \operatorname{Re}(\langle (\nabla_{X_H}\pi^+)\phi_{n+1}^{(1)}(0),\phi^{(1),+}_{n+1}(0)\rangle_{k_+}\\
&+\langle (\nabla_{X_H}(D')^{k_+})\phi^{(1),+}_{n+1}(0),(D')^{k_+}\phi^{(1),+}_{n+1}(0)\rangle_0\\ &-\langle\pi_{F_{n+1}}D'\phi^{(1),+}_{n+1}(0),\phi^{(1),+}_{n+1}(t)\rangle_{k_+}  \\
&+ \langle (\mathbb{A}-\pi_{F_{n+1}})c_1(\hat{\gamma}_{n+1}(0)),\phi_{n+1}^{(1),+}(0)\rangle_{k_+}\\
& -\langle (\nabla_{X_H}\pi_{F_{n+1}})\phi_{n+1}^{(1),+}(0),\phi_{n+1}^{(1),+}(0)\rangle_{k_+}\\
&  -(2-\tau)\langle (\nabla_{X_H}\pi_{F_{n}})\sigma_{n+1}(0),\phi^{(1),+}_{n+1}(0)\rangle_{k_+}\\
&+\langle (\nabla_{X_H}(D')^{k_+})\sigma_{n+1}(0),(D')^{k_+}\sigma^{(1),+}_{n+1}(0)\rangle_0\\
&+\langle (\nabla_{X_H}\pi^+)\sigma_{n+1}(0),\sigma^+_{n+1}(0)\rangle_{k_+} 
 -\langle\pi_{\Sigma_{n+1}}D'\sigma_{n+1}(0),\sigma^+_{n+1}(t)\rangle_{k_+} \\
&-\langle (\nabla_{X_H}\pi_{\Sigma_{n+1}})\sigma_{n+1}(0),\sigma^+_{n+1}(0)\rangle_{k_+} \\
&-(2-\tau)\langle (\nabla_{X_H}\pi_{\Sigma_{n+1}})\phi^{(1),+}_{n+1}(0),\sigma^+_{n+1}(0)\rangle_{k_+}).
\end{aligned}
\end{equation}


All of these terms can be dealt with as in the proof of Lemma \ref{lem parameterized bound}, with the exception of 
\begin{equation}\label{eq spare nabla projection}
\begin{aligned}
&-(2-\tau)\operatorname{Re}\langle (\nabla_{X_H}\pi_{F_{n}})\sigma_{n+1}(0),\phi^{(1),+}_{n+1}(0)\rangle_{k_+}  \\ 
& \qquad -(2-\tau)\operatorname{Re}\langle (\nabla_{X_H}\pi_{\Sigma_{n+1}})\phi^{(1)}_{n+1}(0),\sigma^+_{n+1}(0)\rangle_{k_+}.
\end{aligned}
\end{equation}
To bound this term, consider the expression $\langle \phi_{n+1}^{(1),+}(t),\sigma_{n+1}^+(t)\rangle_{k_+}$ as a function of $t$.  By definition, this is zero, but expanding its derivative gives:
\begin{equation}\label{eq expansion of zero}
\begin{aligned}
0&= \operatorname{Re}\langle (\nabla_{X_H}\pi^+)\phi_{n+1}^{(1)}(t),\sigma^{+}_{n+1}(t)\rangle_{k_+} \\
& +\operatorname{Re}\langle (\nabla_{X_H}(D')^{k_+})\phi^{(1),+}_{n+1}(t),(D')^{k_+}\sigma^{+}_{n+1}(t)\rangle_0
\\ 
&+ \operatorname{Re}\langle  \phi^{(1),+}_{n+1}(t),(\nabla_{X_H}\pi^+)\sigma_{n+1}(t)\rangle_{k_+} \\
&+\operatorname{Re}\langle  (D')^{k_+}\phi^{(1),+}_{n+1}(t),(\nabla_{X_H}(D')^{k_+})\sigma^{+}_{n+1}(t)\rangle_0
\\ 
&+ \operatorname{Re}\langle (\nabla_{X_H} \pi_{F_{n}})\phi^{(1)}_{n+1} (t),\sigma_{n+1}^+(t)\rangle_{k_+}  \\
&  +\operatorname{Re}\langle \phi_{n+1}^{(1),+}(t),\nabla_{X_H}\pi_{\Sigma_{n+1}}\sigma_{n+1}(t)\rangle_{k_+}.
\end{aligned}
\end{equation}
Recall that 
\begin{align*}
\pi_{\Sigma_{n+1}}(\nabla_{X_H}\pi_{F_n})\phi^{(1)}_{n+1}&=-\pi_{\Sigma_{n+1}}(\nabla_{X_H}\pi_{\Sigma_{n+1}})\phi^{(1)}_{n+1}, \\
\pi_{F_n}(\nabla_{X_H}\pi_{F_n})\sigma_{n+1}&=-\pi_{F_{n}}(\nabla_{X_H}\pi_{\Sigma_{n+1}})\sigma_{n+1}.
\end{align*}
Then (\ref{eq expansion of zero}), also using the estimates from the proof of Lemma \ref{lem gamma_n L^2_{k+1/2}}, becomes:
\[
   \big|  \langle (\nabla_{X_H} \pi_{F_{n}})\phi^{(1)}_{n+1}(t) ,\sigma_{n+1}^+(t)\rangle_{k_+}+\langle \phi_{n+1}^{(1),+}(t),(\nabla_{X_H}\pi_{\Sigma_{n+1}})\sigma_{n+1}(t)\rangle_{k_+} \big| \leq CR^2.
\]
Then, using (\ref{eq moving metrics}), we have
\begin{align*}
\frac{1}{2}  \frac{d}{dt}   \bigg|_{t=0} \Vert \phi_{n+1}^{(1),+}(t)+\sigma_{n+1}^+(t) \Vert_{k_+}^2  
&\leq CR^3 \Vert \phi^{(1),+}_{n+1}(0) \Vert_{k_++\frac{1}{2}}\\ 
&-\operatorname{Re}\langle \pi_{\Sigma_{n+1}}D'\sigma_{n+1}(0),\sigma^+_{n+1}(0)\rangle_{k_+}\\
&-\operatorname{Re}\langle\pi_{F_{n+1}}D'\phi^{(1)}_{n+1}(0),\phi^{(1),+}_{n+1}(t)\rangle_{k_+} + C.
\end{align*}
The argument from Lemma \ref{lem gamma_n L^2_{k+1/2}} gives
\begin{align*}
0 &= 
\frac{1}{2}    \frac{d}{dt}  \bigg|_{t=0} \Vert \phi_{n+1}^{(1),+}(t)+\sigma_{n+1}^+(t) \Vert_{k_+}^2   \\
&\leq CR^3 \Vert \phi^{(1),+}_{n+1}(0) \Vert_{k_++\frac{1}{2}}- \Vert \phi^{(1),+}_{n+1}(0) \Vert^2_{k+\frac{1}{2}}- \Vert \sigma^+_{n+1}(0) \Vert_{k+\frac{1}{2}}^2 + C.
\end{align*}
 Thus, $\Vert \phi^{(1),+}_{n+1}(0)+\sigma^+_{n+1}(0) \Vert_{k+\frac{1}{2}}$ is bounded.  The proof of Lemma \ref{lem suspension lemma 2} then follows exactly as Theorem \ref{thm isolating nbd}.
\end{proof}

Finally, for $\tau\in [2,3]$, set
\begin{equation}
\begin{aligned}
\frac{d\phi^{(1)}_{n+1}(t)}{dt}&= -\chi \{ (3-\tau)(\nabla_{X_H(\phi_{n+1}(t))}\pi_{F_{n+1}})\phi_{n+1}^{(1)}(t) \\
& +(\tau-2)(\nabla_{X_H(\phi_{n+1}(t))}\pi_{F_{n+1}})\phi^{(1)}_{n+1}(t) \\
&+ \pi_{F_n}D\phi^{(1)}_{n+1}(t)+c_1(\hat{\gamma}_{n+1}(t)))   \\
&+ (\tau-2)(\nabla_{X_H(\phi_{n+1}(t))}\pi_{F_n})\phi^{(1)}_{n+1}(t) \}, \\
\frac{d\sigma_{n+1}}{dt}(t) &= -\chi \{ (3-\tau)(\nabla_{X_H(\phi_{n+1}(t))}\pi_{\Sigma_{n+1}})\sigma_{n+1}(t)   \\
& +(\tau-2)(\nabla_{X_H(\phi_{n+1}^{(1)}(t))}\pi_{\Sigma_{n+1}}) \sigma_{n+1}(t) \}, \\
\frac{db}{dt}(t) &=- \chi\{ (3-\tau)X_H(\phi_{n+1}(t))+(\tau-2)X_H(\phi^{(1)}_{n+1}(t))\},
\end{aligned}
\end{equation}
with the other terms unchanged.  Note that it is clear that these equations preserve the total space of $F_{n+1} \oplus W_{n+1}$.

\begin{lem}\label{lem suspension lemma 3}
	For $n\gg 0$, for all $\tau\in [2,3]$, $A_{n+1}$ is an isolating neighborhood for $\varphi^\tau_{n+1,k_+,k_-}$.
\end{lem}
\begin{proof}
	This claim is a consequence of the arguments used in Lemma \ref{lem suspension 1} and \ref{lem suspension lemma 2}, and there are no new difficulties.
\end{proof}

%


Write $B(Q_{n+1}/Q_n,R)$ for the $R$-disk bundle of $Q_{n+1}/Q_n$ over $\Pic(Y)$, etc.

\begin{lem}\label{lem controlled exit sets}
	Say that $(A_n^{o},L_n)$ is an index pair for $\X_n$, for some $L_n$, of $\X_n$ on $F_n \oplus W_n$.  Then $(\tilde{A}_{n+1},\tilde{L}_{n+1})$  
	is an index pair for $\X^{\mathrm{split}}_{n+1}$, where
	\[
	\tilde{A}_{n+1}=A_n^{o} \times_B(B(P_{n+1}/P_n\oplus W_{n+1}^+/W_n^+,R))\times_B(B(Q_{n+1}/Q_n\oplus W_{n+1}^-/W_n^-,R)),
	\]   
	for some $R$ sufficiently large, and
	\[
	\tilde{L}_{n+1}=L_n^{o} \times_B(B(P_{n+1}/P_n\oplus W_{n+1}^+/W_n^+,R))\times_B(\partial B(Q_{n+1}/Q_n\oplus W_{n+1}^-/W_n^-,R)).
	\] 
\end{lem}
\begin{proof}
It follows from Lemma \ref{lem suspension premain} that $\mathrm{inv}(\tilde{A}_n\backslash \tilde{L}_n)\subset \mathrm{int}(\tilde{A}_n\backslash \tilde{L}_n)$. 

We next check that $\tilde{L}_n$ is positively invariant in $\tilde{A}_n$.  Write 
\[
(\phi_{n+1}^{(1)}(t),\omega_{n+1}^{(1)}(t),\zeta_{n+1}(t))
\]
 in 
 \[
      (F_{n} \oplus  W_n) \times_B (B(P_{n+1}/P_n\oplus W_{n+1}^+/W_n^+,R))\times_B(B(Q_{n+1}/Q_n\oplus W_{n+1}^-/W_n^-,R))
\]
for a trajectory of $\varphi^{\mathrm{split}}_{n+1,k_+,k_-}$.  The flow on the $\mathcal{F}_n\times_B\mathcal{W}_n$-factor is independent of position on the $B(P_{n+1}/P_n\oplus W_{n+1}^+/W_n^+,R))\times_B(B(Q_{n+1}/Q_n\oplus W_{n+1}^-/W_n^-,R))$ factor, and in particular, if $(\phi_{n+1}^{(1)}(T_0),\omega_{n+1}^{(1)}(T_0))\in L_{n}$, then $(\phi_{n+1}^{(1)}(t),\omega_{n+1}^{(1)}(t))\in L_n$ for all $t\geq T_0$, by our assumption on $L_n$.

We must then show that if $\zeta_{n+1}(T_0)\in \partial B(Q_{n+1}/Q_n\oplus W_{n+1}^-/W_{n}^-,R)$, then 
\[ 
\zeta_{n+1}(t)\in  \partial B(Q_{n+1}/Q_n\oplus W_{n+1}^-/W_n^-,R_1),
\]
or exits $\tilde{A}_{n+1}$, 
for all $t\geq T_0$, if $n$ is large enough.  We regard the path $(\phi_{n+1}^{(1)}(t),\omega_{n+1}^{(1)}(t))$ as fixed, and $\zeta_{n+1}(t)$ as a trajectory of a vector field on the boundary $\partial B(Q_{n+1}/Q_n\oplus W_{n+1}^-/W_n^-,R_1))$.

Write $\zeta_{n+1}(t)=(b(t),\zeta_{n+1}^{(1),+},\zeta_{n+1}^{(1),-},\zeta_{n+1}^{(2),+},\zeta_{n+1}^{(2),-})$, as a section of 
\[
V_n(R_1)=(B(P_{n+1}/P_n\oplus W_{n+1}^+/W_n^+,R_1))\times_B(B(Q_{n+1}/Q_n\oplus W_{n+1}^-/W_n^-,R_1).
\]
We may, and do, assume without loss of generality that $T_0=0$. 
Then if $(\zeta_{n+1}^{(1),-},\zeta_{n+1}^{(2),-})\in \partial B(Q_{n+1}/Q_n\oplus W_{n+1}^-/W_n^-,R)$, either $\zeta_{n+1}^{(1),-}$ or $\zeta_{n+1}^{(2),-}$ has $\Vert\zeta_{n+1}^{(i),-}\Vert_{k_-}\geq R_1/2$.  Assume $i=1$, the other case being similar.  

Recall that $(\phi^{(1)}_{n+1}(t),\omega_{n+1}^{(1)}(t),\zeta_{n+1}(t))$  is equivalent to a trajectory 
\[
  \gamma_{n+1}(t)=(\phi^{(1)}_{n+1}(t),\sigma_{n+1}(t),\omega_{n+1}(t))
\]
of $\X^{\mathrm{split}}_{n+1}$ on $F_{n+1} \oplus W_{n+1}$.    

We consider 
\begin{align*}
&\frac{1}{2}    \frac{d}{dt} \bigg|_{t=0} \Vert \zeta^{(1),-}_{n+1}(t) \Vert^2_{k_-}  \\
&=\frac{1}{2}    \frac{d}{dt} \bigg|_{t=0} \Vert \sigma_{n+1}^-(t) \Vert^2_{k_-}\\
&=\langle -(\nabla_{X_H}\pi_{\Sigma_{n+1}})\sigma_{n+1}(0)-\pi_{\Sigma_{n+1}}D'\sigma_{n+1}(0),\sigma_{n+1}^-(0)\rangle_{k_-}\\
&-\langle \nabla_{X_H}(D')^{k_-}\sigma_{n+1}(0),(D')^{k_-}\sigma_{n+1}^-\rangle_0+\langle(\nabla_{X_H}\pi^-)\sigma_{n+1}(0),\sigma^{-}_{n+1}(0)\rangle_{k_-}\\
&\leq CR^2-\langle\pi_{\Sigma_{n+1}}D'\sigma_{n+1}(0),\sigma_{n+1}^-(0)\rangle_{k_-}\\
&=CR^2-\Vert \sigma_{n+1}^-(0) \Vert^2_{k_-+\frac{1}{2}}.
\end{align*}
Note that we have used that $n$ can be taken sufficiently large that $\Sigma_{n+1}$ is perpendicular to the image of $\mathbb{A}$.

Now, by definition of $\Sigma_{n+1}$, we have
\[
\frac{\Vert\sigma^-_{n+1}(0)\Vert^2_{k_-+\frac{1}{2}}}{\Vert\sigma^-_{n+1}(0)\Vert_{k_-}^2}\to \infty
\]
as $n\to \infty$.

Thus, if $\Vert\sigma^-_{n+1}(0)\Vert_{k_-}\geq R/2$, we have that $\Vert\zeta^{(1),-}_{n+1}(t)\Vert_{k_-}$ is always increasing at $t=0$ (similarly, $\Vert\zeta^{(1),+}_{n+1}(t)\Vert_{k_+}$ is decreasing at $t=0$).  

This shows that $\tilde{L}_{n+1}$ is positively invariant in $\tilde{A}_{n+1}$.  It follows similarly that $\tilde{L}_{n+1}$ is an exit set.  
\end{proof}


\section{Spin$^c$ structure for family of manifolds}  \label{subsec:spinc family}

Since we consider a family of spin$^c$ three-manifolds to show that  the Conley index for the flow $\varphi_{n}$ is independent of the choice of  Riemannian metric of $Y$ in Section \ref{subsec:metric-independence}, we will give the definition of $\mathrm{spin}^c$ structure for a family of Riemannian manifolds. 

Take an $n$-dimensional real, oriented vector space $V$ and an inner product $g$ on $V$.  We denote by $Fr(V, g)$ the space of  orthonormal bases   of $(V, g)$ compatible with the orientation.  Choose another inner product $h$ on $V$.  We define an isomorphism between $Fr(V, g)$ and $Fr(V, h)$. 
For  $\{ e_i \}_{i=1}^{n} \in Fr(V, g)$, put
\[
       h_{ij} = h(e_i, e_j) \in \R. 
\]
Then the matrix $H = (h_{ij})_{i, j =1 ,\dots, n}$ is  symmetric and  positive definite.  We have the square root $\sqrt{H}$ of $H$  defined as follows.  Since $H$ is symmetric and positive definite, we have the eigenspace decomposition
\[
         \R^n = \bigoplus_{i=1}^{r} V_{\lambda_i}, 
\]
where $\lambda_i>0$ are the  distinct eigenvalues of $H$,  and $V_{\lambda_i}$ are the eigenspaces.   Define $\sqrt{H}$ to be the matrix corresponding to the linear map $\R^n \rightarrow \R^n$   defined by $v \mapsto \sqrt{\lambda_i} v$ for $v \in V_{\lambda_i}$.   Define a basis $f_1, \dots, f_n$ of $V$  by
\[
    (f_1  \ \dots \  f_n)   = (e_1 \ \dots \ e_n) \sqrt{H}^{-1}. 
\]
We can see that $f_1, \dots, f_n$ are an orthonormal basis with respect to $h$. So we get a map
\begin{equation}  \label{eq:map fr} 
        Fr(V, g) \rightarrow Fr(V,h). 
\end{equation}
Take $G \in SO(n)$  and put
\[
    (e_1'  \ \dots  \ e_n' ) = (e_1   \ \dots \ e_n) G,   \quad
    H' = (h(e_i', e_j'))_{i,j=1, \dots, n}.
\]
It is easy to see that
\[
     H' = G^{-1} H G, \quad 
     \sqrt{H'} = G^{-1} \sqrt{H} G. 
\]
This implies that the map $(\ref{eq:map fr})$ is  an $SO(n)$-equivariant isomorphism.

For an oriented smooth Riemannian $n$-manifold $(X, g$), let $P_{X, g}$ be the principal $SO(n)$-bundle  of oriented, orthonormal frames in $TX$. 
Recall that a $\mathrm{spin}^c$ structure of $(X, g)$ is a pair of a principal $\mathrm{Spin}^c(n)$-bundle  $\tilde{P}_X$ on $X$ and a smooth map $\xi : \tilde{P}_{X} \rightarrow P_{X, g}$ such that the diagram
\[
        \xymatrix{
                 \tilde{P}_X  \ar[rr]^{\xi} \ar[rd] & & P_{X, g} \ar[ld]  \\
                   & X & 
        }
\]
commutes,  and for $p \in \tilde{P}_X$ and $s \in \mathrm{Spin}^c(n)$ we have
\[
             \xi(p \cdot s)  = \xi(p) \cdot \pi(s). 
\]
Here $\pi : \mathrm{Spin}^c(n) \rightarrow SO(n)$ is the projection.   

Take another Riemannian metric $h$ on $X$.  The $SO(n)$-equivariant isomorphism (\ref{eq:map fr}) induces an isomorphism
\begin{equation}  \label{eq:iso P_X}
         P_{X, g} \cong P_{X, h}
\end{equation}
of principal bundles.   Hence a $\mathrm{spin}^c$ structure $(\tilde{P}_X, \xi)$ of $(X, g)$ naturally defines a $\mathrm{spin}^c$ structure of $(X, h)$. 

A  locally trivial family of $\mathrm{spin}^c$ manifolds   over a topological space $L$ is a tuple $(E, G,  \tilde{P}_E, \xi)$. The first component $E$ stands for a locally trivial  fiber bundle 
\[
      X \rightarrow  E \rightarrow L
\] 
over $L$ with fiber $X$.  For each $\ell \in L$  we have an open neighborhood $U_{\ell}$ of $\ell$ and a trivialization
\[
      E|_{U_{\ell}} \cong U_{\ell} \times E_{\ell}. 
\]
Here $E_{\ell}$ is the fiber of $E$ over $\ell$. 
The second component $G$ is a fiberwise Riemannian metric of $E$. 
Let $P_E$ be the principal $SO(n)$-bundle on $E$ whose fiber over $\ell$ is the principal $SO(n)$-bundle of oriented, orthornormal frames in  $TE_{\ell}$.   
Note that the local trivialization of $E$ on $U_{\ell}$ and the isomorphism (\ref{eq:iso P_X}) induce an isomorphism
\[
       P_{E}|_{U_{\ell}} \cong U_{\ell} \times P_{E_{\ell}}
\]
of principal bundles. 
The third component   $\tilde{P}_E$ is a principal $\mathrm{Spin}^c(n)$-bundle over $E$.  The fourth component  $\xi$ is  a smooth map 
\[
      \tilde{P}_{E} \rightarrow P_{E}
\]
such that the  diagram
\[
   \xymatrix{
            \tilde{P}_E  \ar[rr]^{\xi} \ar[rd] &    &   P_{E}  \ar[ld]  \\
               & E & 
    }
\]
commutes and $\xi(p, \cdot s) = \xi(p) \cdot \pi(s)$ for $p \in \tilde{P}_E$ and $s \in \mathrm{Spin}^c(n)$.   Moreover, we assume that $\tilde{P}_E$ is locally trivial. That is, for each $\ell \in L$  there is an isomorphism 
\[
      \tilde{P}_{E}|_{U_{\ell}} \cong U_{\ell} \times  \big( \tilde{P}_{E}|_{E_{\ell}}  \big)
\] 
of principal bundles such that   the following diagram commutes: 
\[
    \xymatrix{
      \tilde{P}_{E}|_{U_{\ell}} \ar[r]^-{\cong} \ar[d]_{\xi}  &   U_{\ell} \times  \big( \tilde{P}_{E}|_{E_{\ell}} \ar[d]^{id_{U_{\ell}} \times \xi}   \big) \\
            P_{E}|_{U_{\ell}} \ar[r]_-{\cong}  &  U_{\ell} \times P_{E_{\ell}}
    }
\]


\section{Independence of Metric}\label{subsec:metric-independence}


In this section we prove that the approximate Seiberg-Witten flow defined in (\ref{eq for gamma}) varies continuously as we vary the three-manifold.

To make this precise, let $\faml$ be a  locally-trivial family of $\mathrm{spin}^c$ metrized three-manifolds with compact base space $L$, so that $L$ is a CW complex. See Section \ref{subsec:spinc family} for the definition of locally trivial family of $\mathrm{spin}^c$ metrized manifolds. 
Note that associated to $\faml$ there is also a bundle over $L$, $\pic(\faml)$, whose fiber is the Picard-bundle at $\ell\in L$.

%

Suppose that we are given a sequence of continuously varying spectral sections $P_{n,\ell},Q_{n,\ell}$ for $\ell\in L$ so that the $P_{n,\ell}$,$Q_{n,\ell}$ are good as at the beginning of Section \ref{sec:findim-appx}, with $F_{n,\ell}=P_{n,\ell}\cap Q_{n,\ell}$ as a fiber bundle over (the total space of) $\tilde{L}$.  Let $\varphi_{n,\ell,k_+,k_-}$ be the flow defined by projection onto $F_{n,\ell}$.  Here, unlike in the case of a single three-manifold, the flow preserves fibers of $F_{n,\ell}$ over $L$ (though the flow can of course move over $\tilde{L}_\ell$, the fiber of $\tilde{L}\to L$).

There is one subtlety in that now the eigenvalues of $*d$ may vary in the family $\mathcal{F}$.  In particular, we will assume the existence of increasing spectral sections $W_{P,n}$ for $- \! * \! d$, and increasing spectral sections $W_{Q,n}$ for $*d$, satisfying the analogues of (\ref{eq:spectral-conditions-1})-(\ref{eq:spectral-conditions-2}), and set $W_n=W_{P,n}\cap W_{Q,n}$.  With this notation fixed, we define $W_n^+$ and $W_n^-$ as before.

\begin{thm}\label{thm:families}
	Let $\faml$, with compact base $L$, be a family of $\mathrm{spin}^c$ metrized three-manifolds, with fiber $\faml_b$ for $b\in L$.  Let $k_+,k_-$ be half integers with $k_{\pm}>5$ and with $|k_+-k_-|\leq \frac{1}{2}$.  Fix a positive number $R$ with $R>R_{k_+,k_-}$ for some $R_{k_+,k_-}$.  Then 
	\[
	\big(  B_{k_+} (F_n^+; R) \times_B B_{k_-}(F_n^-; R) \big)  \times_{B}   \big( B_{k_+}(W_n^+; R) \times_{B} B_{k_-}(W_n^-; R) \big)
	\]
	is an isolating neighborhood of  the flow $\varphi_{n,\ell, k_+,k_-}$ for $n \gg 0$.  Here $B_{k_{\pm}}(F_n^{\pm};R)$ are the disk bundle of $F_n^{\pm}$ of radius $R$ in $L^2_{k_{\pm}}$ and $B_{k_+}(F_n^+; R) \times_B B_{k_-}(F_n^-; R)$ is the fiberwise product. 
\end{thm}

The proof of this Theorem differs only from the proof of Theorem \ref{thm isolating nbd} in notation, so we will not write out the details.

In particular:

\begin{cor}
	Let $(Y,\fs)$ be a $\mathrm{spin}^c$ manifold, with metrics $g_0,g_1$, and fix a family of good spectral sections $P_{n,0},Q_{n,0}$ over $(Y,g_0)$.  Choose a family of metrics $g_t$ connecting $g_0$ to $g_1$.  Then there exists a family of spectral sections $P_{n,t},Q_{n,t}$ extending $P_{n,0},Q_{n,0}$ and so that the flow $\varphi_{n,0,k_+,k_-}$ on $F_{n,0}$ extends to a continuously-varying flow $\varphi_{n,t,k_+,k_-}$ on $F_{n,t}$, so that
	\[
	\big(  B_{k_+} (F_n^+; R) \times_B B_{k_-}(F_n^-; R) \big)  \times_{B}   \big( B_{k_+}(W_n^+; R) \times_{B} B_{k_-}(W_n^-; R) \big)
	\]
	is an isolating neighborhood of the flow $\varphi_{n,t, k_+,k_-}$ for $n \gg 0$ and all $t\in [0,1]$.  In particular, $I(\varphi_{n,0,k_+,k_-})$ is canonically, up to homotopy equivalence, identified with $I(\varphi_{n,1,k_+,k_-})$.  
\end{cor}
\begin{proof}
	The claim about the existence of the extended spectral sections follows from the homotopy-description of spectral sections and the fact that $[0,1]$ is contractible.  The claim on isolating neighborhoods is a consequence of Theorem \ref{thm:families}.  The well-definedness of the Conley index follows from the continuity property of the Conley index.
\end{proof}


\section{Variation of Sobolev Norms}\label{subsec:sobolev}

\begin{prop}\label{prop:sobolev-independence-3}
	Let $(k_+^1,k_-^1)$ and $(k_+^2,k_-^2)$ be pairs of half-integers $>5$, with $|k_+^i-k_-^i|\leq \frac{1}{2}$ for $i=1,2.$  Fix $R$ sufficiently large.  Then there exists a family of flows $\varphi^\tau_n$ for $\tau\in [0,1]$ so that
	\[
	\big(  B_{g^\tau_+} (F_n^+; R) \times_B B_{g^{\tau}_-}(F_n^-; R) \big)  \times_{B}   \big( B_{g^\tau_+}(W_n^+; R) \times_{B} B_{g^{\tau}_-}(W_n^-; R) \big)
	\]
	is a family of isolating neighborhoods, where $g^\tau_{\pm}$ is the interpolated metric (defined below), and where $\varphi^0_n=\varphi_{n,k^1_+,k^1_-}$ and $\varphi^1_{n}=\varphi_{n,k^2_+,k^2_-}$.  In particular, there is a  homotopy equivalence
	\[
	I(\varphi_{n,k_+^1,k_-^1})\to I(\varphi_{n,k_+^2,k_-^2}),
	\]
	suppressing the spectral section choices from the notation. The restriction to the $S^1$-fixed point set is a  fiber-preserving homotopy equivalence. 
\end{prop} 
\begin{proof}
	Define the \emph{interpolated metric} $g^\tau$ by 
	\[
	g^\tau(x,y):=\langle x,y\rangle_{k_\pm^\tau}:=(1-\tau)\langle x,y \rangle_{k_+^1,k_-^1}+\tau\langle x,y\rangle_{k^2_+,k^2_-}.
	\]
	We abuse notation and also write $g^{\tau}$ for the restriction of $g^\tau$ to subbundles, including $F_n^{\pm}$ and $W^{\pm}_n$.  
	
	The equation (\ref{eq d/dt gamma_n}) defines a flow $\varphi^\tau_{n}$, with $\pi_{F_n}$, $\pi_{W_n}$ replaced appropriately.  Hypothesis (\ref{eq phi omega norm}) continues to hold, with the subscripts $k_\pm$ replaced with $k^{\tau}_{\pm}$.  		Write $\pi^\tau_{F_n}$ for projection with respect to $g^\tau$.  
	
	As usual, we will assume for a contradiction that 
	\[
	y_{n,0}^{\tau_n}=(\phi_{n,0}^{\tau_n},\omega^{\tau_n}_{n,0})\in \inv A_n \cap \partial A_n.
	\]
	Let us treat the case that
	 \[\phi^{\tau_n}_{n,0}\in S_{g^\tau_+}(F_n^+;R)\in \inv A_n\cap \partial A_n,\]
	where $S_{g^\tau_+}(V,R)$, for $V$ a vector bundle over $B$, is the $R$-sphere bundle.  
	
	Exactly as in the proof of Theorem \ref{thm isolating nbd}, we can extract a sequence of approximate solutions $\tilde{\gamma}_n^{\tau_n}=(\tilde{\phi}_n^{\tau_n},\omega_n^{\tau_n})$, for $t\in [-T,T]$, with $T$ fixed.  To see this, we need to control $\frac{d\tilde{\phi}_n^\tau}{dt}$ in $(K_T,\ell-5,w)$-norm.  This amounts to generalizing Proposition \ref{prop pi P_n wighted} to the following situation:

	\begin{prop}\label{prop pi P_n wighted2}
		Let $k_+,k_-$ be half integers, with $k_{\pm}>5$, and set $\ell = \min_{i=1,2}\{k^i_+,k^i_-\}$.  Then
		\[
		\sup_{v \in B(TB; 1)}  \Big\Vert \nabla_v    \pi^\tau_{P_n} : L^2_{ k^\tau} \rightarrow L^2_{\ell - 5, w}   \Big\Vert
		\rightarrow
		0,
		\]
		uniformly in $\tau$.
	\end{prop}
	This proposition holds because the natural modification of the estimate at the end of Corollary \ref{cor nabla pi}  holds.

	Then the sequence $\tilde{\gamma}^{\tau_n}_n(t)$ converges to a map 
	\[
	    \tilde{\gamma}\from [-T,T]\to \mathcal{H}^1(Y) \times L^2_{\ell -1} (\bbS)\times L^2_{\ell -1}(\mathrm{im} \; d^*).
	\]
	 To verify that $\tilde{\gamma}$ solves the Seiberg-Witten equations, we observe that 
	 \[
	     (\nabla_{X_H}\pi^{\tau_n}_{F_n})\tilde{\phi}_n(s)\to 0
	  \]
	   in $L^2_{K_T,\ell-5,w}$-norm, as follows from Proposition \ref{prop pi P_n wighted2}.  
	 
	 We have:	 
	 \begin{align*}
	 \Vert \pi^{\tau_n}_{F_n}D\phi_n - D\phi_n\Vert_{\ell -2}&=\Vert \pi^{\tau_n}_{F_n}D\phi_n -D\phi_n+D\phi_n- D\phi_n\Vert_{\ell -2}\\ &\leq \Vert [\pi_{F_n}^{\tau_n},D]\phi_n\Vert_{\ell-2}+\Vert D\phi_n-D\phi\Vert_{\ell-2}.
	 \end{align*}
	 The first term drops out, using the rule of a sequence of controlled vector spaces, and we obtain that $\pi_{F_n}^{\tau_n} D\phi_n$ converges to $D\phi$ uniformly in $L^2_{\ell-2}$ on $[-T,T]$.  By the proof of Lemma \ref{lem:convergence-to-solution}, the limit $\tilde{\gamma}$ is a solution of the Seiberg-Witten equations.  The proof from this point follows along the same lines as Theorem \ref{thm isolating nbd}. 
\end{proof}


\section{The Seiberg-Witten Invariant}\label{subsec:defn-of-invar}
In this section we repackage the construction of  $\preSWF_{[n]}(Y,\mathfrak{s})$ to take account of the choices made in the construction.

\begin{dfn}\label{def:spectral-system-3}
	A \emph{$3$-manifold spectral system} (abbreviated as just a \emph{spectral system}) for a family $\mathcal{F}$ of metrized $\mathrm{spin}^c$ 3-manifolds, with fiber $(Y,\mathfrak{s})$, is a tuple \begin{equation}\label{eq:def-of-spectral-system}\mathfrak{S}=(\mathbf{P},\mathbf{Q},\mathbf{W}_P,\mathbf{W}_Q,\{\eta^P_n\}_n,\{\eta_n^Q\},\{\eta^{W_P}_n\}_n,\{\eta^{W_Q}_n\}_n)\end{equation}
	where $\mathbf{P}=\{P_n\}_n$ (for $n\geq 0$) is a sequence of good (increasing) spectral sections of the Dirac operator $-D$; similarly $\mathbf{Q}=\{Q_n\}_n$ is a sequence of good increasing spectral sections of $D$ parameterized by $\mathrm{Pic}(\mathcal{F})$.  The $\mathbf{W}_P=\{W_{P,n}\}_n$ are good spectral sections of the operator $- \! * \! d$; and similarly $\mathbf{W}_Q=\{W_{Q,n}\}_n$ are good spectral sections of $*d$.  We require $W_{P,0}$ be the sum of all negative eigenspaces of $*d$, as we may, since the nullspace of $*d$, acting on the bundle $L^2_k(\operatorname{im} d^*)$ is trivial, and similarly $W_{Q,0}$ will be the sum of positive eigenspaces.   The $\eta_n$ are exactly as in Theorem \ref{thm change of approximation suspension}.
\end{dfn}	

We have not established that there exist good sequences of spectral sections for $*d$ for all families $\mathcal{F}$.  However, they exist in many situations, as for example when the family $\mathcal{F}$ is obtained as a mapping torus of a self-diffeomorphism preserving the fiber metric. In this case, $\mathcal{F}$ is a family over $S^1$  and the eigenvalues of $*d$ are constant functions on $S^1$.  More generally, if there is a neighborhood  $U$ of $b$ for each $b \in L$ such that $\mathcal{F}$ has a local trivialization $\mathcal{F}|_{U} \cong U \times Y$ preserving the fiber metric, then the eigenvalues of $*d$ are constants. So we have a good sequence of spectral sections of $*d$. 


\begin{dfn}\label{def:swf}
	The \emph{unparameterized Seiberg-Witten-Floer spectrum} 
	\[\unparamSWFn(\mathcal{F},\mathfrak{S},k_+,k_-)\]
	 of a family $\mathcal{F}$ as in Definition \ref{def:spectral-system-3} associated to a spectral system $\mathfrak{S}$, and $k_\pm $ half integers with $k_\pm> 5$ and $|k_+-k_-|\leq 1/2$, is the (partially-defined) equivariant  spectrum, whose sequence of spaces is defined as follows.

	Let $\mathfrak{S}$ be a spectral system with components as named in (\ref{eq:def-of-spectral-system}).  Let 
\[ 
   D_n=(\dim (P_n-P_0),
   \dim (Q_n-Q_0),
   \dim (W_{P,n}- W_{P,0}),
   \dim (W_{Q,n}- W_{Q,0})), 
\]
    whose components we denote $D^\ell_n$ for $\ell=1,\ldots,4$.   Recall (cf. Section \ref{subsec:spectra}) that we must assign, for a certain collection of representations, a space to each representation, together with structure maps.  The spaces in the Seiberg-Witten Floer spectrum are most naturally defined at those representations $\mathbb{C}^{D_n^2}\oplus \mathbb{R}^{D_n^4}$; in order to define the spectra at other levels, we extrapolate from the definitions at these levels; see also Remark \ref{rmk:other-def}.   
       
       Let $\mathbb{N}_0$ be the set of non-negative integers. 
	For $(i_1,i_2)\in \mathbb{N}_0^2$ sufficiently large, let $A(i_1,i_2)=(A(i_1,i_2)_1,A(i_1,i_2)_2)$ denote the largest pair $(D_n^2,D_n^4)$ among pairs $(D_j^2,D_j^4)$ for which $(D_j^2,D_j^4)\leq (i_1,i_2)$.   	We can write 
	\[
	           A(i_1, i_2) = (D_{n(i_1, i_2)}^{2},  D_{n(i_1, i_2)}^4)
	\]	
	for some $n(i_1, i_2) \in \mathbb{N}_0$. 
       Set  $\unparamSWFn_{i_1,i_2}(\mathcal{F},\mathfrak{S},k_+,k_-)$ to be 
	\[
	\Sigma^{\mathbb{C}^{i_1-A(i_1,i_2)_1}\oplus \mathbb{R}^{i_2-A(i_1,i_2)_2}}\unparamSWF_{[n(i_1, i_2)]}(\mathcal{F},\mathfrak{S},k_+,k_-).
	\] 
Here $\preSWF^{u}_{[n(i_1, i_2)]}(\cF, \mathfrak{S}, k_+, k_-)$ is the (unparameterized) Conley index with respect to the flow  $\varphi_{n(i_1,i_2), k_+, k_-}$. 
If $(i_1,i_2)$ is not sufficiently large, let $\unparamSWFn_{i_1,i_2}(\mathcal{F},\mathfrak{S},k_+,k_-)$ be a point.  Define the transition map \[\sigma_{(i,j),(i+1,j)}:\Sigma^{\mathbb{C}}\unparamSWFn_{i,j}\to \unparamSWFn_{i+1,j},\] where $i+1\neq D_n^2$ for any $n$, as the identity (With the $\mathbb{C}$ factor contributing to the leftmost factor of $\Sigma^{\mathbb{C}^{i_1-A(i_1,i_2)_1}}$), and similarly for transitions in the real coordinate.  If $i+1=D_n^2$ for some $n$, we use the $(\eta_n)_*$ as defined in Theorem \ref{thm change of approximation suspension}.  Note that the $(\eta_n)_*$ are only well-defined up to homotopy; we choose representatives in the homotopy class.
	
	In the event that the family has a self-conjugate $\mathrm{spin}^c$-structure, and so that the spectral section $\mathfrak{S}$ is preserved by $j$, we use $\mathbb{H}$ instead of $\mathbb{C}$ above, as appropriate, so that $\unparamSWFn$ is indexed on the $\mathrm{Pin}(2)$-universe described in Section \ref{subsec:homotopy1}.  To be more specific, we write $\SWFn^{u,\mathrm{Pin}(2)}(\mathcal{F},\mathfrak{S})$ for the $\mathrm{Pin}(2)$-spectrum invariant. In particular, $\SWFn^{u,\mathrm{Pin}(2)}_{i,j}$, viewed as an $S^1$-space, is identified with $\SWFn^{u}_{2i,j}$. 
\end{dfn}

We will often suppress some arguments of $\unparamSWFn$ from the notation where they are clear from context.

At the point-set level, there is a choice of index pairs (at each level $(i_1,i_2)$) involved in Definition \ref{def:swf}.  However, the space $\preSWF^u_{[n]}(\mathcal{F},\mathfrak{S},k_+,k_-)$ is well-defined up to canonical homotopy, since the Conley index forms a connected simple system, Theorem \ref{thm:fiberwise-deforming-conley-2}.  

\begin{rem}\label{rmk:parameterized-spectrum}
	We would be able to repeat Definition \ref{def:swf} in the parameterized setting, replacing the spectrum $\unparamSWFn$ with a parameterized spectrum $\SWFn$, except that it is not known that the parameterized Conley index forms a connected simple system in $\mathcal{K}_{G,B}$, the category considered in Appendix \ref{sec:homotopy}.  
\end{rem}

The spaces $\unparamSWFn_{(i_1,i_2)}(\mathcal{F})$ for $(i_1,i_2)$ not a pair $(D_n^2,D_n^4)$, for some $n$, seem to have rather an awkward definition, because they do not naturally represent the Conley index of some fixed flow.  However, they may be viewed as the Conley indices of a split flow on  $\underline{V}\times_{\mathrm{Pic}(\mathcal{F})}\preSWF_{[n]}(\mathcal{F})$, for $V= \mathbb{C}^{i_1-D_n^2}\oplus \mathbb{R}^{i_2-D_n^4}$ a vector space equipped with a linear (repelling) flow.

More generally, associated to a spectral system $\mathfrak{S}$, we define the virtual dimension of the vector bundle $F_n\oplus W_n$ as 
\[
  D_n=(\dim (P_n-P_0),\dim (Q_n-Q_0),\dim (W_n^+),\dim (W_n^-)).
\]  
We write $\mathfrak{S}(\vec{i})$ for the vector bundle of virtual dimension $\vec{i}=(i_1,i_2,i_3,i_4)$.  If the spectral section does not produce a vector bundle in that virtual dimension, we define
\[
\mathfrak{S}(i_1,i_2,i_3,i_4)=\underline{V}\oplus F_n\oplus W_n
\]
where $F_n\oplus W_n$ is the largest vector bundle coming from $\mathfrak{S}$ with virtual dimension at most $(i_1,i_2,i_3,i_4)$, and where we define $\underline{V}$ to be the trivial $S^1$ (or $\mathrm{Pin}(2)$, as appropriate) vector bundle with dimension $(i_1,i_2,i_3,i_4)-D_n$.  When we need to distinguish between the contributions of $F_n\oplus W_n$ and $\underline{V}$ to $\mathfrak{S}(\vec{i})$, we call $F_n\oplus W_n$ the \emph{geometric} bundle, and $\underline{V}$ the \emph{virtual} bundle.

We can treat $\mathfrak{S}(i_1,i_2,i_3,i_4)$ as a vector bundle with a split flow, as discussed above; its unparameterized Conley index is (canonically, up to homotopy) homotopy-equivalent to $\preSWF^u_{(i_2,i_4)}(\mathcal{F},\mathfrak{S})$.

Let \[\underline{V}(\vec{i},\vec{j})=\underline{\mathbb{C}}^{j_1-i_1}\oplus \underline{\mathbb{C}}^{j_2-i_2}\oplus \underline{\mathbb{R}}^{j_3-i_3}\oplus \underline{\mathbb{R}}^{j_4-i_4},\] viewed as a vector bundle with linear flow, outward in the even factors, inward in the odd factors.  Note that for any $\vec{j}\geq \vec{i}$ (that is, $j_1 \geq i_1, \dots, j_4 \geq i_4$), there is a vector bundle morphism \begin{equation}\label{eq:tautological-morphism}
\underline{V}(\vec{i},\vec{j})\oplus \mathfrak{S}(\vec{i})\to \mathfrak{S}(\vec{j}),
\end{equation} as follows.  Indeed, if $A(\vec{i})=A(\vec{j})$, then (\ref{eq:tautological-morphism}) is defined by:
\begin{align*}
\underline{V}(\vec{i},\vec{j})\oplus (\underline{V}(D_n,\vec{i})\oplus F_n\oplus W_n)
& =(\underline{V}(\vec{i},\vec{j})\oplus \underline{V}(D_n,\vec{i}))\oplus F_n\oplus W_n   \\
& \to \underline{V}(D_n,\vec{j})\oplus F_n\oplus W_n.
\end{align*}
If $\vec{j}=D_{n+1}$ and $\vec{i}=D_n$, the morphism (\ref{eq:tautological-morphism}) is just the structure map involved in the definition of a spectral system.  For more general $\vec{j},\vec{i}$, the morphism (\ref{eq:tautological-morphism}) is the composite coming from the sequence $\vec{i}\to D_{n_1}\to\ldots D_{n_k}=A(\vec{j})\to \vec{j}$, where the rightmost factors of $\underline{V}(\vec{i},\vec{j})$ are used first.

Similarly, we define $P(i_1)=\underline{\mathbb{C}}^{i_1-D_{A(i_1)}^1}\oplus P_{A(i_1)}$, etc.

\begin{dfn}\label{def:homotopy-of-spectral-sections}
	We call two spectral systems $\mathfrak{S}_1$ and $\mathfrak{S}_2$ for the same family $\mathcal{F}$ \emph{equivalent} if there exists a collection of bundle isomorphisms:
	\[
	\Phi_{P,i}: P^1(i)\to P^2(i),
	\]
	and similarly for $Q,W_P,W_Q$, for all $i$ sufficiently large, satisfying the following conditions.  First, there exists some sufficiently large $n$, so that the $\Phi_{P,i}$ (respectively $\Phi_{Q,i}$ etc.), as $i$ becomes large, must preserve the subbundles $P^j_n$ for $j=1,2$ (similarly for $Q^j_n$, etc.).  (Indeed, for $\vec{i}$ sufficiently large, $P^1_n$ (respectively $Q^1_n$ etc.) will be contained in the geometric bundles of $P^2(i)$ (respectively $Q^2(i)$ etc.).)
	
	Second, the $\Phi_{i}$ must be compatible with the structure maps of $\mathfrak{S}_1,\mathfrak{S}_2$ in that the following square commutes (as well as its analogs):
		\[
	\begin{tikzpicture}[baseline={([yshift=-.8ex]current  bounding  box.center)},xscale=8,yscale=1.5]
	\node (a0) at (0,0) {$ 	\underline{V}\oplus P^1(i)$};
	\node (a1) at (1,0) {$ 	\underline{V}\oplus P^2(i)$};
	\node (b0) at (0,-1) {$P^1(j)$};
	\node (b1) at (1,-1) {$P^2(j)$};

	\draw[->] (a0) -- (a1) node[pos=0.5,anchor=north] {\scriptsize
		$\operatorname{id}\oplus \Phi_{P,i}$}; \draw[->] (a0) -- (b0) node[pos=0.5,anchor=east] {\scriptsize $\eta $};
	\draw[->] (b0) -- (b1) node[pos=0.5,anchor=south east]{\scriptsize $\Phi_{P,j} $}; \draw[->] (a1) -- (b1) node[pos=0.5,anchor=west] {\scriptsize $\eta $};
	
	\end{tikzpicture}
	\]
	
	We do not require the isomorphisms $\Phi_{i}$ (etc.) to preserve all of the $P^j_n$ as $n$ varies.

\end{dfn}

Note that a morphism of spectral systems as in Definition \ref{def:homotopy-of-spectral-sections} also induces maps
\[
\Phi_{\vec{i}}: \mathfrak{S}_1(\vec{i})\to \mathfrak{S}_2(\vec{i})
\]
for $\vec{i}$ sufficiently large, which preserve the subbundles $F^1_n\oplus W^1_n$ (which lie in $\mathfrak{S}_2(\vec{i})$ for $\vec{i}$ sufficiently large naturally), for some fixed large $n$, for $\vec{i}$ sufficiently large.  There is also a commutative square: 

	\[
\begin{tikzpicture}[baseline={([yshift=-.8ex]current  bounding  box.center)},xscale=8,yscale=1.5]
\node (a0) at (0,0) {$ 	\underline{V}\oplus \mathfrak{S}_1(\vec{i})$};
\node (a1) at (1,0) {$ 	\underline{V}\oplus \mathfrak{S}_2(\vec{i})$};
\node (b0) at (0,-1) {$\mathfrak{S}_1(\vec{j})$};
\node (b1) at (1,-1) {$\mathfrak{S}_2(\vec{j})$};

\draw[->] (a0) -- (a1) node[pos=0.5,anchor=north] {\scriptsize
	$\operatorname{id}\oplus \Phi_{\vec{i}}$}; \draw[->] (a0) -- (b0) node[pos=0.5,anchor=east] {\scriptsize $\eta $};
\draw[->] (b0) -- (b1) node[pos=0.5,anchor=south east]{\scriptsize $\Phi_{\vec{j}} $}; \draw[->] (a1) -- (b1) node[pos=0.5,anchor=west] {\scriptsize $\eta $};

\end{tikzpicture}
\]

\begin{prop}\label{prop:levelwise-well-definedness}
	For $\mathcal{F}$ a family of $\mathrm{spin}^c$ $3$-manifolds, $n$ sufficiently large and $\Phi: \mathfrak{S}_1\to\mathfrak{S}_2$ an equivalence of spectral systems, there is a homotopy equivalence, well-defined up to homotopy:\[\Phi^u_{n,*}: 
	\preSWF_{[n]}^u(\mathcal{F},\mathfrak{S}_1)\to \preSWF_{[n]}^u(\mathcal{F},\mathfrak{S}_2).
	\]
	In fact, there is a fiberwise-deforming homotopy equivalence:
	\[
	\Phi_{n,*}: 
	\preSWF_{[n]}(\mathcal{F},\mathfrak{S}_1)\to \preSWF_{[n]}(\mathcal{F},\mathfrak{S}_2),
\]
	so that $\Phi^u_{n,*}=\nu_{!}\Phi_{n,*}$.  Here $\nu$ is the map $\mathrm{Pic}(\mathcal{F})\to *$ sending $\mathrm{Pic}(\mathcal{F})$ to a point, and $\nu_{!}$ is defined as in Section \ref{sec:homotopy}.  (Note that $\Phi_{n,*}$ is not claimed to be well-defined).  Analogous statements hold for $\mathrm{Pin}(2)$-equivariant spectral sections.
\end{prop}
\begin{proof}
	We consider the pullback of the flow $\varphi_2$ on $\mathfrak{S}_2(\vec{i})$ by the morphism (For some large $\vec{i}$)
	\[
	\Phi_{\vec{i}}: \mathfrak{S}_1(\vec{i})\to \mathfrak{S}_2(\vec{i}),
	\]
	defining a flow on $\mathfrak{S}_1(\vec{i})$.  Following the proof of Theorem \ref{thm change of approximation suspension}, we see that there is a well-defined, up to homotopy, deformation of $\Phi_{\vec{i}}^*\varphi_2$ to $\varphi_1$.  Deformation invariance of the Conley index gives a fiberwise-deforming homotopy equivalence
	\[
	I(\varphi_1)\to I((\Phi_{\vec{i}})^*\varphi_2)\cong I(\varphi_2),
	\]
	where the isomorphism is canonical (at the point-set level).  Passing to the unparameterized Conley index, the morphism
	\[
	I^u(\varphi_1)\to I^u((\Phi_{\vec{i}})^*\varphi_2)
	\] is canonical (up to homotopy).  This gives the proposition.
\end{proof}

We write $[\mathfrak{S}]$ for the equivalence class of a spectral system $\mathfrak{S}$.  

\begin{rem} \label{rmk:well-defined-transitions}
	As usual, if Conjecture \ref{conj:parameterized-connected-simple-system} holds, then $\Phi_{n,*}$ appearing in Proposition \ref{prop:levelwise-well-definedness} is well-defined.
\end{rem}

\begin{thm}\label{thm:well-definedness-of-swf-type}
	The equivariant parameterized stable homotopy type of 
	\[
	\Sigma_B^{\mathbb{C}^{-D^2_n}\oplus \mathbb{R}^{-D^4_n}} \preSWF_{[n]}(\mathcal{F},[\mathfrak{S}])
	\]
	 is independent of the choices in its construction.  That is, it is independent of:
	\begin{enumerate}\item The choice of $k_+,k_-$,
		\item the element $n \gg 0$,
		\item A choice of spectral system $\mathfrak{S}$ representing the equivalence class $[\mathfrak{S}]$.
		\item The family of metrics on $\mathcal{F}$.
		\end{enumerate}
 Here  $\Sigma_B^{\mathbb{C}^{-D^2_n}\oplus \mathbb{R}^{-D^4_n}}$ stands for the desuspension by $\mathbb{C}^{D_n^2} \oplus \mathbb{R}^{D_n^4}$ in the category $PSW_{S^1, B}$. See Section \ref{subsec:homotopy1}.  

If the $\mathrm{spin}^c$ structure is self-conjugate, a similar statement holds for 
\[
\Sigma_{B}^{\mathbb{H}^{-D_n^2} \oplus \tilde{\mathbb{R}}^{-D^4_n}} \mathcal{SWF}_{[n]}( \mathcal{F}, [\mathfrak{S}]).
\] 
\end{thm}
\begin{proof}
	Proposition \ref{prop:levelwise-well-definedness} addresses changes in the spectral section. Proposition \ref{prop:sobolev-independence-3} addresses varying of $k_\pm$.  The choice of $n$ was handled in Theorem \ref{thm change of approximation suspension}, and the metric was addressed in Theorem \ref{thm:families}.
\end{proof}

\begin{dfn} \label{def: SWF parameterized homotopy type}
	The \emph{Seiberg-Witten Floer parameterized homotopy type} 
	\[
	\preSWF(\mathcal{F},[\mathfrak{S}])
	\]
        is defined as the class of: 
        \[
	\Sigma_B^{\mathbb{C}^{-D^2_n}\oplus \mathbb{R}^{-D^4_n}} \preSWF_{[n]}(\mathcal{F},[\mathfrak{S}]),
	\]
	for any $n$.
	
When the $\mathrm{spin}^c$ structure is self-conjugate, the \emph{$\rm{Pin}(2)$-Seiberg-Witten Floer parameterized homotopy type} $\preSWF^{\mathrm{Pin}(2)}(\mathcal{F}, [\mathfrak{S}])$ is defined as the class of
	\[
	    \Sigma_{B}^{\mathbb{H}^{-D_n^2} \oplus \tilde{\mathbb{R}}^{-D^4_n}} \mathcal{SWF}_{[n]}( \mathcal{F}, [\mathfrak{S}]).
	\] 
\end{dfn}

Recall from Section \ref{subsec:spectra} that a \emph{weak} morphism of spectra is a (collection of) maps that is only defined in sufficiently high degrees (this is also the case for ordinary morphisms in Adams' \cite{adams} category of spectra).  
\begin{thm}\label{thm:well-definedness-of-seiberg-witten-invariant}
	For $\mathcal{F}$ a family of $\mathrm{spin}^c$ $3$-manifolds, and $\Phi\from \mathfrak{S}_1\to \mathfrak{S}_2$ an equivalence of spectral systems, there is a weak morphism which is a homotopy equivalence (see Section \ref{subsec:spectra}), well-defined up to homotopy:
	\[
	\Phi_*: \unparamSWFn(\mathcal{F},\mathfrak{S}_1)\to \unparamSWFn(\mathcal{F},\mathfrak{S}_2).
	\]
	That is, the collection of spectra \[\unparamSWFn(\mathcal{F},[\mathfrak{S}])=\{\unparamSWFn(\mathcal{F},\mathfrak{S})\}_{\mathfrak{S}}\] forms a connected simple system in spectra, if $\mathcal{F}$ admits a spectral system.
\end{thm}
\begin{proof}
	First, independence of $\unparamSWFn(\mathcal{F},[\mathfrak{S}])$ from the choice of Sobolev norms was handled in Proposition \ref{prop:sobolev-independence-3}.  Moreover, variation of metric, for a particular level $\preSWF^u_{[n]}(\mathcal{F},[\mathfrak{S}])$, was handled in Theorem \ref{thm:families}.  We then need only show that an equivalence of spectral systems induces a well-defined, up to homotopy, morphism
	\[
	\unparamSWFn(\mathcal{F},\mathfrak{S}_1)\to \unparamSWFn(\mathcal{F},\mathfrak{S}_2).
	\]
	For this, we use Proposition \ref{prop:levelwise-well-definedness} to define the maps levelwise, and we need only show that the following square homotopy commutes (the squares involving other vector bundles $\mathfrak{S}(i_1, i_2,i_3,i_4)$ are straightforward): 
	\[
	\begin{tikzpicture}[baseline={([yshift=-.8ex]current  bounding  box.center)},xscale=8,yscale=1.5]
	\node (a0) at (0,0) {$ 	\Sigma^{V_n}\preSWF^u_{[n]}(\mathcal{F},\mathfrak{S}_1)$};
	\node (a1) at (1,0) {$ 	\Sigma^{V_n}\preSWF^u_{[n]}(\mathcal{F},\mathfrak{S}_2)$};
	\node (b0) at (0,-1) {$\preSWF^u_{[n+1]}(\mathcal{F},\mathfrak{S}_1)$};
	\node (b1) at (1,-1) {$\preSWF^u_{[n+1]}(\mathcal{F},\mathfrak{S}_2)$};

	\draw[->] (a0) -- (a1) node[pos=0.5,anchor=north] {\scriptsize
		$\operatorname{id}\wedge \Phi_{n,*}$}; \draw[->] (a0) -- (b0) node[pos=0.5,anchor=east] {\scriptsize $\eta_{n,*} $};
	\draw[->] (b0) -- (b1) node[pos=0.5,anchor=south east]{\scriptsize $\Phi_{n+1,*} $}; \draw[->] (a1) -- (b1) node[pos=0.5,anchor=west] {\scriptsize $\eta_{n,*} $};
	
	\end{tikzpicture}
	\]
	Here $V_n = \mathbb{C}^{D_{n+1}^{2}-D_n^2}\oplus \mathbb{R}^{D_{n+1}^4-D_n^4}$. 
	This is a consequence of the two composites involved being Conley-index continuation maps associated to deformations of the flow.  Observe that the composite deformations are related to each other by a deformation of deformations.  By Section 6.3 of \cite{Salamon}, the square homotopy commutes (the necessary adjustments of Salamon's argument for equivariance are straightforward).
\end{proof}

As usual, subject to Conjecture \ref{conj:parameterized-connected-simple-system}, Theorem \ref{thm:well-definedness-of-seiberg-witten-invariant} would hold in the parameterized case.

Moreover, it is easy to determine when two spectral systems are equivalent:

\begin{lem}\label{lem:classification-of-spectral-systems}
	The set of spectral systems for a family $\mathcal{F}$ of $\mathrm{spin}^c$ three-manifolds up to equivalence, if nonempty, is affine-equivalent to $K(\mathrm{Pic}(\mathcal{F}))\times K(\mathrm{Pic}(\mathcal{F}))$, where the difference of systems $\mathfrak{S}_1,\mathfrak{S}_2$ is sent to $([P^1_0-P^2_0],[Q^1_0-Q^2_0])$.  
\end{lem}
\begin{proof}

By its construction, an equivalence of spectral systems is determined by its value $(\Phi_{P,i},\Phi_{Q,i},\Phi_{W_P,i},\Phi_{W_Q,i})$ for any sufficiently large $i$.  In the positive spectral section part of the spinor coordinate, to construct an equivalence $\mathfrak{S}_1\to \mathfrak{S}_2$ it is sufficient (and necessary) to construct an isomorphism $P^1(i)-P^1_n\to P^2(i)-P^1_n$ for some $i$ large, relative to a fixed (large) $n$.  By definition, $P^1(i)-P^1_n$ is canonically some number of copies of $\underline{\mathbb{C}}$, and so such an isomorphism exists if and only if \[ [P^2(i)-P^1_n]=[\underline{\mathbb{C}}^{\dim (P^1(i)-P^1_n)}].\]  This condition is satisfied exactly when $[P^1_0-P^2_0]=0\in K(\mathrm{Pic}(\mathcal{F}))$, as needed.

The $1$-form coordinate is handled similarly, but the bundles $W_n^\pm$ there are always trivial.
\end{proof}

In particular, we note that there is a canonical choice, subject to a choice of $Q_0$, and up to adding trivial bundles, of a spectral section $P_0$, by requiring $P_0-Q_0$ trivializable.  We call these \emph{normal} spectral sections; the set of equivalence classes of such is affine-equivalent to $K(\mathrm{Pic}(Y))$, as above.

\begin{dfn}\label{def:floer-framing}
	A(n) ($S^1$-equivariant) \emph{Floer framing} is an equivalence class of normal spectral sections.  A $\mathrm{Pin}(2)$-equivariant Floer framing is a ($\mathrm{Pin}(2)$)-equivalence class of normal spectral sections.  Here, a $\mathrm{Pin}(2)$-equivalence of ($\mathrm{Pin}(2)$-equivariant) spectral sections is a collection of isomorphisms as in Definition \ref{def:homotopy-of-spectral-sections} that are $\mathrm{Pin}(2)$-equivariant.
\end{dfn}


There are various extensions of Lemma \ref{lem:classification-of-spectral-systems}. Let us state a ${\rm Pin}(2)$- equivariant version of the Lemma.  

\begin{lem}  \label{lem:classification-of-Pin(2)-spectral-systems}
The set of ${\rm Pin}(2)$-spectral systems for a family $\mathcal{F}$ of $\mathrm{spin}^c$ three-manifolds up to equivalence, if nonempty, is affine-equivalent to 
\[
   KQ(\mathrm{Pic}(\mathcal{F}))  \times KQ(\mathrm{Pic}(\mathcal{F}) ), 
 \]
 where the difference of systems $\mathfrak{S}_1,\mathfrak{S}_2$ is sent to 
$([P^1_0-P^2_0],[Q^1_0-Q^2_0])$.   
Here $KQ$ is the Quaternionic K-theory defined in \cite{dupont},  \cite{lin_rokhlin}. 
\end{lem}

\begin{rem}\label{rmk:other-def}
We can define the spectrum $\mathbf{SWF}^u_{i_1, i_2}$ in a little different way. Fix a sufficiently large integer $n$ and put
\[
     \mathbf{SWF}^{u}_{i_1, i_2} =  \Sigma^{\C^{i_1-D_n^2} \oplus \R^{i_2-D_n^4}} \mathcal{SWF}^u_{[n]}
\]
for $(i_1, i_2) \in \mathbb{N}_0^2$ with $i_1, i_2 \geq n$.   The transition maps
\[
   \begin{split}
    & \sigma_{(i_1, i_2), (i_1 + 1, i_2)} : \Sigma^{\C} \mathbf{SWF}^{u}_{i_1, i_2} \rightarrow  \mathbf{SWF}^{u}_{i_1+1, i_2}, \\
  &  \sigma_{(i_1, i_2), (i_1, i_2+1)} : \Sigma^{\R}  \mathbf{SWF}^{u}_{i_1, i_2} \rightarrow  \mathbf{SWF}^{u}_{i_1, i_2+1}
   \end{split}
\]
are defined to be the identities.   This spectrum is homotopy equivalent to the previous one.  

In the previous definition of  $\mathbf{SWF}^{u}$, we introduced $A(i_1, i_2)$ which allows us to avoid choosing a large integer $n$. This makes the definition of $\mathbf{SWF}^{u}$  more natural. 
\end{rem}

%

%

In the construction of $\mathcal{SWF}_{[n]}(\cF,  \mathfrak{S})$, we have a frame  of the orthogonal complement of $Q_n$ in $Q_{n+1}$.   Using the frame, we have 
\[
           \mathcal{SWF}_{[n+1]} (\cF,  \mathfrak{S})  \cong 
           \Sigma_{B}^{\C^{k_{Q, n}}  \oplus \R^{k_{W,-, n}}} \mathcal{SWF}_{[n]}(\cF, \mathfrak{S}). 
\]
More generally, we can choose spectral sections $Q_n$ such that the orthogonal complement of $Q_n$ in $Q_{n+1}$ does not necessarily have a frame.  In this case, we have
\[
    \mathcal{SWF}_{[n+1]} (\cF,  \mathfrak{S}) 
    \cong \Sigma_{B}^{(Q_{n+1} / Q_n) \oplus \R^{k_{W, -,n}}}  \mathcal{SWF}_{[n]}(\cF,  \mathfrak{S}), 
\]
where $Q_{n+1} / Q_n$ may not be trivialized. 
See Theorem \ref{thm change of approximation suspension}. 
We can still define the Seiberg-Witten Floer stable homotopy type in a suitable stable homotopy category. The category is defined by taking $R$, $W$ to be finite dimensional,  virtual $G$-vector bundles  over $B$   in Definition \ref{def:sw-cat},  so that  we can take  desuspensions by non-trivial vector bundles. The Seiberg-Witten Floer stable homotopy type is  defined to be the class of
\[
         \Sigma_{B}^{ - (Q_n/Q_0) \oplus \R^{-D_n^4}  }  \mathcal{SWF}_{[n]}(\cF, \mathfrak{S})
\]
in the category, where $n$ is a fixed large integer.


\section{Elementary properties of $\preSWF(Y,\mathfrak{s})$}\label{subsec:elementary-properties}

Here we collect a few results about $\preSWF(Y,\mathfrak{s})$ that follow almost directly from the definitions.  We work only for a single $(Y,\mathfrak{s})$, but similar results hold in families.

\begin{prop}\label{prop:finiteness}
	The total space of $\preSWF^u_{[n]}(Y,\mathfrak{s})$  has the homotopy type of a finite $S^1$-CW complex; respectively the total space of $\preSWF^{u,\mathrm{Pin}(2)}_{[n]}(Y,\mathfrak{s})$, when defined, is a finite $\mathrm{Pin}(2)$-CW complex.  As a consequence, for $G=S^1$ or $\mathrm{Pin}(2)$, the Seiberg-Witten Floer spectrum  $\SWFn^{u,G}(Y,\mathfrak{s},\mathfrak{S})$ is a finite $G$-CW spectrum. 
\end{prop}
\begin{proof}
	For this, we need to consider perturbations of the Seiberg-Witten equations.  Recall the notion of cylinder functions from \cite[Chapter 11]{KM}.  As in Definition 2.1 of \cite{KLS1}, given a sequence of $\{C_j\}_{j=1}^\infty$ of positive real numbers and cylinder functions $\{\hat{f}_j\}_{j=1}^\infty$, let $\mathcal{P}$ be the Banach space
	\[
	\mathcal{P}=\Big\{\sum^\infty_{j=1}\eta_j\hat{f}_j  :  \eta_j\in \mathbb{R},\, \sum^\infty_{j=1}C_j |\eta_j|<\infty  \Big\}
	\]
	with norm defined by $\Vert\sum^\infty_{j=1}\eta_j\hat{f}_j\Vert=\sum^\infty_{j=1}|\eta_j|C_j$.  The elements of $\mathcal{P}$ are called \emph{extended cylinder functions}.  
	
	For $f$ an extended cylinder function, let $\grad f=\mathfrak{q}$ be the $L^2$-gradient over $L^2_k(\mathbb{S})\times \mathcal{H}^1(Y)\times L^2_k(\mathrm{im} \, d^*)$ of $f$.  We write $(\mathfrak{q}_V,\mathfrak{q}_H,\mathfrak{q}_W)$ for the vertical, horizontal, and one-form components of $\mathfrak{q}$.  Define the perturbed Seiberg-Witten equations by the downward gradient flow of $\mathcal{L}+f$, explicitly:
	\begin{equation}\label{eq:seiberg-witten-perturbed}
	\begin{aligned}
	\frac{d\phi}{d t} &= - D_a \phi(t) - c_1(\gamma(t))-\mathfrak{q}_V, \\
	\frac{d a}{dt} &= - X_{H}(\phi)-\mathfrak{q}_H, \\
	\frac{d\omega}{dt} &= - *d \omega - c_2(\gamma(t))-\mathfrak{q}_W. 
	\end{aligned}
	\end{equation}
	
	We may perform finite-dimensional approximation with the perturbed Seiberg-Witten equations in place of (\ref{eq:seiberg-witten}) (with the same spectral sections as for the unperturbed equations).  It is straightforward but tedious to check that the proof of Theorem \ref{thm isolating nbd} holds also for (\ref{eq:seiberg-witten-perturbed}), for $k$-extended cylinder functions $f$, where $k \geq \max\{ k_+, k_-\} + \frac{1}{2}$.  The key points are Proposition 2.2 of \cite{KLS1}, and Lemma 4.10 of \cite{LidmanManolescu}.
	
	Moreover, for a family of perturbations, the analog of Theorem \ref{thm isolating nbd} continues to hold, by a similar argument.  In particular, it is a consequence that $\preSWF^u_{[n]}(Y,\mathfrak{s})$ is well-defined up to canonical equivariant homotopy, independent of perturbation.
	
	Finally, the space of perturbations $\mathcal{P}$ attains transversality for the Seiberg-Witten equations, in the sense that for a generic perturbation from $\mathcal{P}$, there are finitely many (all non-degenerate) stationary points for the perturbed formal gradient flow.

	In particular, using the attractor-repeller sequence for the Conley index, together with the fact that the Conley index for a single non-degenerate critical point is a sphere, we observe that the the Conley index $I^u(\varphi_{n,k_+,k_-})$ for $n$ large is a finite $G$-CW complex.
\end{proof}

\begin{prop}\label{prop:orientation-reversal} 
	For $(Y,\mathfrak{s})$ a $\mathrm{spin}^c$, oriented closed $3$-manifold, and $\mathfrak{S}$ a spectral system, we have: 
	\[\preSWF^{u}(Y,\mathfrak{s},\mathfrak{S})^\vee\simeq \preSWF^{u}(-Y,\mathfrak{s},\mathfrak{S}^\vee),\]
		where the spectral system $\mathfrak{S^\vee}$ is obtained by reversing the roles of $P_n$ and $Q_n$ in $\mathfrak{S}$.
\end{prop}
\begin{proof}
	This follows from Spanier-Whitehead duality for the Conley index, Theorem \ref{thm:spanier-whitehead-duality-for-conley}.
\end{proof}

Note that it would be desirable in Proposition \ref{prop:orientation-reversal} to have a similar result in the parameterized setting; the analog of Theorem \ref{thm:spanier-whitehead-duality-for-conley} in the parameterized setting has not been established, but would suffice.

Using the latter parts of Theorem \ref{thm change of approximation suspension}, we have:
\begin{cor}\label{cor:change-of-suspensions}
	The homotopy type of $\preSWF_{[n]}(Y,\mathfrak{s},\mathfrak{S})$ is independent of the spectral sections $P_n$ for $n$ large.  That is, instead of $\preSWF_{[n]}(Y,\mathfrak{s},\mathfrak{S})$ depending on a choice in a set affine-equivalent to $K(\mathrm{Pic}(Y))\times K(\mathrm{Pic}(Y))$, $\preSWF_{[n]}(Y,\mathfrak{s},\mathfrak{S})$ is determined by a (relative) class in $K(\mathrm{Pic}(Y))$.
	
	Further, 
	\[\preSWF_{[n]}(Y,\mathfrak{s},\mathfrak{S}_1)\simeq \Sigma_B^{\mathfrak{S}_1-\mathfrak{S}_2}\preSWF_{[n]}(Y,\mathfrak{s},\mathfrak{S}_2),
	\]
	where $\mathfrak{S}_1-\mathfrak{S}_2$ is the bundle defined by Lemma \ref{lem:classification-of-spectral-systems}, and where suspension is defined as in Remark \ref{rmk:bad-suspension-1}.
	
\end{cor}

We can now prove some of the results from the introduction:

\emph{Proof of Theorem \ref{thm:main}:}

By \cite{kronheimer-manolescu}, the vanishing of the triple-cup product on $H^1(Y;\mathbb{Z})$ implies that the family index of the Dirac operator on $Y$ is trivial.  Using this, fix a Floer framing $\frak{P}$.  In that case, Theorems \ref{thm:well-definedness-of-swf-type} and \ref{thm:well-definedness-of-seiberg-witten-invariant} imply that $\preSWF(Y,\frak{s},\frak{P})$ and $\SWFn(Y,\frak{s},\frak{P})$ are well-defined.

Proposition \ref{prop:finiteness} gives the claim about finite CW structures.

Finally, when $b_1(Y) = 0$,  the relationship with $\mathit{SWF}(Y,\frak{s})$ is immediate from the definition of $\preSWF(Y,\frak{s},\frak{P})$, since the collection of linear subspaces used in the construction of $\mathit{SWF}(Y,\frak{s})$ defines a spectral system as in Definition \ref{def:spectral-system-3}.
\qed

\emph{Proof of Theorem \ref{thm:main-families}:}

The argument is completely parallel to the proof of Theorem \ref{thm:main}.
\qed

Finally, we address the claims in the introduction about  complex oriented cohomology theories. We start by reviewing the definition of an $E$-orientation of a vector bundle, where $E$ is a multiplicative cohomology theory (see \cite{Adams-book} for a discussion of orientability\footnote{nLab also has a nice discussion, which our presentation follows.}).  Indeed, let $V\to X$ be a topological vector bundle of rank $m$.  Then an $E$-orientation is a class
\[
u\in \tilde{E}^m(\mathrm{Th}(V)),
\]
so that, for all $x\in X$ and $i_x: S^m\to V$ the map associated to inclusion of a fiber over $x$, $i_x^*u$ is a unit in $\tilde{E}^m(S^m)=\tilde{E}^0(S^0)$ (The latter equality being the suspension isomorphism of the cohomology theory $E$).

Recall that a cohomology theory $E$ is \emph{complex oriented} if it is oriented on all complex vector bundles.  There is a universal such cohomology theory, complex cobordism $MU$, in the sense that for any complex-oriented cohomology theory $E$, there is a map of ring-spectra $MU\to E$ inducing the orientation on $E$.

The utility of a complex-oriented cohomology theory $E$ for studying the stable homotopy type  $\mathcal{SWF}(Y, \frak{s}, \mathfrak{S}_1)$ is as follows.  By Theorem \ref{thm change of approximation suspension}, we have, by changing the spectral system $\mathfrak{S}_1$ to $\mathfrak{S}_2$, that there is a ($S^1$-equivariant) parameterized equivalence
\begin{equation}\label{eq:change-of-suspension}
\mathcal{SWF}(Y,\frak{s},\mathfrak{S}_1)\to \Sigma^{\mathfrak{S}_1-\mathfrak{S}_2}\mathcal{SWF}(Y,\frak{s},\mathfrak{S}_2).
\end{equation}

In Chapter \ref{sec:froyshov}, after having considered the $4$-dimensional invariant, we will introduce a number $n(Y,\frak{s},g,P_0)$ associated to a spectral section $P_0$ of the Dirac operator over $Y$, and a metric $g$ on $(Y,\frak{s})$.  By its construction $n(Y,\frak{s},g,P_0)=n(Y,\frak{s},g,[\frak{S}])$ is an invariant of a spectral system up to equivalence $[\frak{S}]$, and its main property is that it changes appropriately to counteract the shift in (\ref{eq:change-of-suspension}).  That is:
\[
n(Y,\frak{s},g,[\frak{S}_1])-n(Y,\frak{s},g,[\frak{S}_2])=\dim [\frak{S}_1-\frak{S}_2],
\]
as follows immediately from (\ref{eq:spectral-index}).

For $E$ an $S^1$-equivariant cohomology theory, let 
\[
FE^*(Y,\frak{s},\frak{S}_1)=\tilde{E}^{*-2n(Y,\frak{s},g,\frak{S}_1)}(\nu_!\mathcal{SWF}(Y,\frak{s},\frak{S}_1)).
\]
We call $FE^*(Y,\frak{s},\frak{S}_1)$ the \emph{Floer E-cohomology} of the tuple $(Y,\frak{s},\frak{S}_1)$.  

More generally, we can also consider the notion of an \emph{equivariant complex orientation}.  This is more complicated to state; we follow \cite{Cole-Greenlees-Kriz} for the definition of equivariant complex orientability.  That is, let $A$ be an abelian compact Lie group, and fix a complete complex $A$-universe $\mathcal{U}$ (see the Appendix  \ref{sec:homotopy}).  A multiplicative equivariant cohomology theory $E^*_A(\cdot)$ is called \emph{complex stable} if there are suspension isomorphisms:
\[
\sigma_V: \tilde{E}^n_A(X) \to \tilde{E}^{n+\dim V}_A((V^+)\wedge X)
\]
for all complex (finite-dimensional) $A$-representations $V$ in $\mathcal{U}$.  The natural transitivity condition on the $\sigma_V$ is required, and the map $\sigma_V$ is required to be given by multiplication by an element of $\tilde{E}^{\dim V}(V^+)$ (necessarily a generator).  A \emph{complex orientation} of a complex stable theory $E_A$ is a cohomology class $x(\epsilon)\in E^*_A(\mathbb{C}P(\mathcal{U},\mathbb{C}P(\epsilon)))$ that restricts to a generator of 
\[
E_A^*(\mathbb{C}P(\alpha\oplus \epsilon),\mathbb{C}P(\epsilon))\cong \tilde{E}^*_A(S^{\alpha^{-1}}),
\]
for all one-dimensional representations $\alpha$.

Building on the equivalence (\ref{eq:change-of-suspension}), we have the following claim:
\begin{thm}\label{thm:monopole-complex-cobordism}
	Let $E$ be an equivariant complex-oriented (nonparameterized) homology theory.  Then, for any two spectral systems $\frak{S}_1,\frak{S}_2$, there is a canonical isomorphism
	\[
	\tilde{E}^{*}(\nu_!\mathcal{SWF}(Y,\frak{s},\mathfrak{S}_1))\to \tilde{E}^{*}(\nu_!\Sigma^{\frak{S}_2-\frak{S}_1}\mathcal{SWF}(Y,\frak{s},\mathfrak{S}_2)).
	\]
	In particular, $FE^*(Y,\frak{s},\frak{S}_1)$ is independent of $\frak{S}_1$, and defines an invariant $FE^*(Y,\frak{s})$.
\end{thm}
\begin{proof}

	The theorem is a consequence of the fact that, for an ex-space $(X,r,s)$ over a base $B$, and a complex $m$-dimensional vector bundle $V$ over $B$, with $\nu$ as usual the basepoint map $B\to *$:
	\begin{equation}\label{eq:tautology}
	\nu_! \Sigma^V_B X=\mathrm{Th}(r^*V).
	\end{equation}
	This equality is a direct exercise in the definitions.  In fact, if $(X,r,s)$ is an $S^1$-ex-space, with base $B$ on which $S^1$ acts trivially, the equality also holds at the level of $S^1$-spaces, where $V$ is an $S^1$-equivariant vector bundle over $B$, inherited from its complex structure (so that the pullback $r^*V$ is a $S^1$-equivariant vector bundle over the $S^1$-space $X$).
	
	We have by (\ref{eq:change-of-suspension}),
	\[
	\tilde{E}^{*}(\nu_!\mathcal{SWF}(Y,\frak{s},\mathfrak{S}_1))=\tilde{E}^{*}(\nu_!\Sigma^{\frak{S}_1-\frak{S}_2}\mathcal{SWF}(Y,\frak{s},\mathfrak{S}_2)).
	\]
	By (\ref{eq:tautology}),
	\[
	\tilde{E}^{*}(\nu_!\mathcal{SWF}(Y,\frak{s},\mathfrak{S}_1))=\tilde{E}^{*}(\mathrm{Th}(r^*(\frak{S}_1-\frak{S}_2))),
	\]
	where $r$ is the restriction map of the ex-space $\mathcal{SWF}(Y,\frak{s},\mathfrak{S}_2)$.  However, the complex orientation on $E$ induces an isomorphism:
	\[
	\tilde{E}^{*}(\mathrm{Th}(r^*(\frak{S}_1-\frak{S}_2)))\to \tilde{E}^{*-2 \dim(\frak{S}_1-\frak{S}_2) }(\mathcal{SWF}(Y,\frak{s},\mathfrak{S}_2)),
	\]
	which is exactly what we needed (The last  isomorphism above, in the equivariant case, follows from the construction of Thom classes in \cite[Theorem 6.3]{Cole-Greenlees-Kriz}).  
	
	The last claim of the Theorem is then a consequence of the definition of $FE^*$.

\end{proof}

The most important equivariant complex orientable cohomology theory for us will be equivariant complex cobordism $MU_G$, defined by tom Dieck \cite{tom-dieck-bordism} for a compact Lie group $G$.  It turns out, if $G$ is abelian, that $MU_G$ is the universal $G$-equivariant complex oriented cohomology theory, in the sense that any equivariant complex oriented cohomology theory $E_G$ accepts a unique ring map of ring spectra $MU_G\to E_G$ so that the orientation on $E_G$ is the image of the canonical orientation on $MU_G$. See \cite{Cole-Greenlees-Kriz}.

We define $\mathit{FMU}^*(Y,\frak{s})$ and $\mathit{FMU}^*_{S^1}(Y,\frak{s})$ by
\[
\begin{split}
\mathit{FMU}^*(Y,\frak{s})&=\widetilde{MU}^{*-2n(Y,\frak{s},g,\frak{S})}(\nu_!\mathcal{SWF}(Y,\frak{s},\frak{S})),\\
\mathit{FMU}^*_{S^1}(Y,\frak{s})&=\widetilde{MU}_{S^1}^{*-2n(Y,\frak{s},g,\frak{S})}(\nu_!\mathcal{SWF}(Y,\frak{s},\frak{S})),
\end{split}
\]
for some spectral sections $\frak{S}$.  By Theorem \ref{thm:monopole-complex-cobordism} and the complex orientation on $MU$ and $MU_{S^1}$, these are well-defined independent of a choice of $\frak{S}$, and this proves Theorem \ref{thm:monopole-complex-cobordism-introduction}.

For a spin structure $\mathfrak{s}$, we have the $\mathrm{Pin}(2)$-equivariant Seiberg-Witten Floer stable homotopy type $\mathcal{SWF}^{\mathrm{Pin}(2)}(Y, \mathfrak{s}, \mathfrak{S})$.  
To define $\mathrm{Pin}(2)$-equivariant cohomology  theory $\mathit{FMU}^*_{\mathrm{Pin}(2)}(Y, \mathfrak{s})$,  we need to show that
\[
         \widetilde{\mathit{MU}}_{\mathrm{Pin}(2)}^{*-2n(Y, \mathfrak{s}, \mathfrak{S})}(\nu_{!}\preSWF^{\mathrm{Pin}(2)}(Y, \mathfrak{s}, \mathfrak{S}))
\]
is independent of the choice of $\mathfrak{S}$, which requires an orientation on  $\widetilde{\mathit{MU}}^*_{\mathrm{Pin}(2)}$.     But we can not apply the argument in \cite{Cole-Greenlees-Kriz} to $\widetilde{\mathit{MU}}_{\mathrm{Pin}(2)}^*$ since $\mathrm{Pin}(2)$ is not abelian. We do not discuss orientations on $\widetilde{MU}^*_{\mathrm{Pin}(2)}$ in this memoir.

\chapter{Computation}\label{sec:computation}

In this chapter we provide a sample of calculations of the Seiberg-Witten Floer homotopy type.

\section{Seiberg-Witten Floer homotopy type in reducible case}

We will need the following lemma.

\begin{lem}    \label{lem (N, L_-)}
Let $\varphi : M \times \mathbb{R} \rightarrow M$ be a smooth flow on a smooth manifold $M$ and $N$  be a compact submanifold (with corners) of $M$ with $\dim M = \dim N$. Assume that the following conditions are satisfied:

\begin{enumerate}

\item
$\partial N = L_+ \cup L_-$, where $L_+, L_-$ are compact submanifolds (with corners) of $\partial N$ with $L_+ \cap L_- = \partial L_+ = \partial L_-$. 

\item
For $x \in \operatorname{int} (L_+)$, there is $\epsilon > 0$ such that $\varphi(x, t) \in \operatorname{int}(N)$ for $t \in (0,  \epsilon)$. 

\item
For $x \in L_-$, there is $\epsilon > 0$ such that $\varphi(x, t) \not\in N$ for $t \in (0, \epsilon)$. 

\end{enumerate}
Then $N$ is an isolating neighborhood and $(N, L_-)$ is an index pair of $\inv(N)$.  (See \cite{conley-easton} for a similar statement. ) 
\end{lem}

\begin{proof}
By Condition (2) and  (3), we have $\inv(N) \subset \operatorname{int}(N)$.   It is easy to see that $L_-$ is an exit set from the three conditions. Also Condition (3) implies that $L_-$ is positively invariant in $N$. 
\end{proof}

Fix a spin$^c$ 3-manifold $(Y,\mathfrak{s})$, along with a spectral system $\mathfrak{S}$, which we will usually suppress from the notation.  Let $k_+, k_- > 5$ be half-integers with $| k_+ - k_-| \leq \frac{1}{2}$, $k=\min\{k_+,k_-\}$ and
\[
     \varphi_{n} =  \varphi_{n, k_+, k_-} : (F_n \oplus W_n) \times \mathbb{R} \rightarrow F_n \oplus W_n
\]
be the flow induced by the Seiberg-Witten equations.

Fix $R \gg 0$.  Put
\[
       A_n(R) :=   (B_{k_+}(F_n^+; R) \times_B  B_{k_-}(F_n^-; R)) \times_{B}( B_{k_+}(W_n^+;R) \times_B B_{k_-}(W_n^-;R)). 
\]
Let $I_n \rightarrow B=Pic(Y)$ be the parameterized Conley index of $\operatorname{inv} (A_n(R), \varphi_n)$. 

\begin{thm} \label{thm reducible I_n}
	Assume that the following conditions are satisfied:

\begin{enumerate}
\item
$\ker (D : \mathcal{E}_{\infty} \rightarrow \mathcal{E}_{\infty}) = 0$. 

\item  
All solutions to the Seiberg-Witten equations (\ref{SW eq}) with finite energy are reducible. 

\end{enumerate}

Let $\frak{S}$ be a spectral system such that $P_0 = \mathcal{E}_0(D)^{0}_{-\infty}$. 
Then for all $n \gg 0$ we have
\[
         I_n \cong S_B^{F_n^- \oplus W_n^-},
\]
as an $S^1$-equivariant space, with the obvious projection to $B$.  
Hence the Seiberg-Witten Floer parameterized homotopy type is given by
\[
             \mathcal{SWF}(Y, \frak{s}, [\frak{S}]) \cong \Sigma_B^{\mathbb{C}^{-D_n^2}  \oplus \mathbb{R}^{-D_n^4} } I_n \cong S_B^{0}
\]
in $PSW_{S^1,B}$. Here $D_n^2 = \rank F_n$, $D_n^4 = \rank W_n^-$ and $PSW_{S^1,B}$ is the category defined in Definition \ref{def:sw-cat}.    

If the $\mathrm{spin}^c$ structure is self-conjugate,   the ${\rm Pin}(2)$-Seiberg-Witten Floer parameterized  homotopy type is given by
\[
  \mathcal{SWF}^{\rm{Pin}(2)}(Y, \frak{s}, [\frak{S}]) \cong S_B^{0}
\]
in $PSW_{{\rm Pin}(2), B}$.
\end{thm}

To prove this, we need the following: 

\begin{prop}
Assume that all solutions to (\ref{SW eq}) with finite energy are reducible. 
For any $\epsilon > 0$, there is $n_0$ such that for $n > n_0$ we have
\[
        \operatorname{inv}(A_n(R)) \subset  A_n(\epsilon). 
\]
\end{prop}

\begin{proof}
Put
\[
       \delta_n := \max \{   \| \phi^+ \|_{k_+}    :  (\phi, \omega) \in  \operatorname{inv} (A_n(R))  \}.
\]
Let 
\[
    \gamma_n = (\phi_n, \omega_n) : \mathbb{R} \rightarrow     A_n(R)
\]
be approximate Seiberg-Witten trajectories with
\[
        \Vert \phi_n^+(0) \Vert_{k_+} = \delta_n. 
\]
Then we have
\[
            \frac{d}{dt} \bigg|_{t=0} \Vert \phi_n^+(t) \Vert_{k_+}^2 = 0. 
\]
As we have seen before, after passing to a subsequence, $\gamma_n$ converges to a  Seiberg-Witten trajectory $\gamma$ with finite energy.  By assumption, $\gamma$ is reducible and we can write as $\gamma = (0, \omega)$. 
As in Lemma \ref{lem gamma_n L^2_{k+1/2}}, we can show that there is a constant $C > 0$ such that  $\| \phi_{n}^+(0) \|_{k_+ + \frac{1}{2}} < C$ for all $n$. By the Rellich lemma,  $\phi_n^+(0)$ converges to $0$ in $L^2_k$.  Therefore $\delta_n \rightarrow 0$. 

Similarly
  \begin{align*}
       &  \max \{ \| \phi^- \|_{k_-} : (\phi, \omega) \in \operatorname{inv} (A_n(R))  \},    \\
       & \max \{ \| \omega^+ \|_{k_+} : (\phi, \omega) \in \operatorname{inv} (A_n(R))  \},  \\
       & \max \{ \| \omega^- \|_{k_-} : (\phi, \omega) \in \operatorname{inv} (A_n(R))  \}
 \end{align*}
go to $0$ as $n \rightarrow 0$. 
\end{proof}

\noindent
{\it Proof of Theorem \ref{thm reducible I_n}}

Fix a small positive number $\epsilon$ with $\epsilon^2 \ll \epsilon$ and choose $n \gg 0$. By the proposition,
\[
          \operatorname{inv} (A_n(R)) \subset A_n(\epsilon). 
\]
Put
    \begin{align*}
         L_{n,-} (\epsilon) = &\left(  B_{k_+}(F_n^+; \epsilon) \times_{B} S_{k_-}(F_n^-; \epsilon) \right)   \times_{B}   (B_{k_+}(W_n^+; \epsilon) \times_B B_{k_-}(W_n^-; \epsilon) )  \\
          & \bigcup
          \left(  B_{k_+}(F_n^+; \epsilon) \times_{B} B_{k_-}(F_n^-; \epsilon) \right)   \times_{B} (  B_{k_+}(W_n^+; \epsilon) \times_{B} S_{k_-}(W_n^-; \epsilon)),     \\
       L_{n, +}(\epsilon)
        = &\left(  S_{k_+}(F_n^+; \epsilon) \times_{B} B_{k_-}(F_n^-; \epsilon) \right)   \times_{B} (  B_{k_+}(W_n^+; \epsilon) \times_{B} B_{k_-}(W_n^-; \epsilon) )  \\
          & \bigcup
          \left(  B_{k_+}(F_n^+; \epsilon) \times_{B}   B_{k_-}(F_n^-; \epsilon) \right)   \times_{B} (  S_{k_+}(W_n^+; \epsilon) \times_{B} B_{k_-}(W_n^-; \epsilon)).
    \end{align*}
Then we have
  \begin{align*}
     & \partial A_n(\epsilon) =    L_{n,-}(\epsilon) \cup L_{n, +}(\epsilon),    \\
     & L_{n, -}(\epsilon) \cap L_{n, +} (\epsilon) = \partial L_{n, -}(\epsilon) = \partial L_{n, +}(\epsilon). 
   \end{align*}
We will show that the pair $(A_n(\epsilon), L_{n,-}(\epsilon) )$ is an index pair. 
It is enough to check that   $A_n(\epsilon), L_{n,-}(\epsilon), L_{n, +}(\epsilon)$  satisfy the conditions (2), (3) in Lemma \ref{lem (N, L_-)}.
We consider the case when $k_+ \in \frac{1}{2} \mathbb{Z} \smallsetminus \mathbb{Z}$. 

Take an approximate Seiberg-Witten trajectory
\[
    \gamma = (\phi, \omega) : (-\delta, \delta) \rightarrow F_n \oplus W_n
\]
for a small positive number $\delta$.

Assume that 
\[
           \| \phi^+(0) \|_{k_+} = \epsilon. 
\]
We have
   \begin{align*}
   &  \frac{1}{2}     \frac{d}{dt}  \bigg|_{t = 0} \| \phi^+ (t) \|_{k_+}  \\
   & = \frac{1}{2}      \frac{d}{dt}  \bigg|_{t = 0} 
  \langle  |D|^{k_+ + \frac{1}{2}} \pi^+  \phi(t),  |D|^{k_+ - \frac{1}{2}}  \pi^+ \phi  (t)  \rangle_{0}  \\
& = 
\langle  (\nabla_{X_H} |D|^{k_+ + \frac{1}{2}}) \phi^+(0),   |D|^{k_+ - \frac{1}{2}} \phi^+(0) \rangle_0 \\ 
& \quad +  \langle   |D|^{k_+ + \frac{1}{2}} \phi^+(0),  (\nabla_{X_H} |D|^{k_+ - \frac{1}{2}}) \phi^+(0) \rangle_0   \\
 & \quad +\langle   (\nabla_{X_H} \pi^+) \phi(0), \phi^+(0)   \rangle_{k_+}   + 
     \Big\langle    \frac{d \phi}{dt}(0),   \phi^+ (0)   \Big\rangle_{k_+}.
    \end{align*}
Note that 
\[
         \| X_{H}(\phi) \| = \| q(\phi)_{\mathcal{H}} \| \leq C \epsilon^2. 
\]
Hence we have
    \begin{align*}
     &  \big| \langle  (\nabla_{X_H} |D|^{k_+ + \frac{1}{2}}) \phi^+(0),   |D|^{k_+ - \frac{1}{2}} \phi^+(0) \rangle_0  \big|
       \leq C \epsilon^4,     \\
  &  \big|   \langle   |D|^{k_+ + \frac{1}{2}} \phi^+(0),  (\nabla_{X_H} |D|^{k_+ - \frac{1}{2}}) \phi^+(0) \rangle_0 \big| \leq C \epsilon^4,   \\
  & \big| \langle   (\nabla_{X_H} \pi^+) \phi(0),  \phi^+(0)    \rangle_{k_+}   \big|  \leq C \epsilon^4. 
 \end{align*}
by Proposition \ref{prop L^2 nabla pi_P} and  Lemma \ref{lem nabla |D'|^k}.   Recall that $\pi^+ = 1 -\pi_{P_0}$, where $\pi_{P_0}$ is the $L^2$-projection onto $P_0$. 
We have 
   \begin{align*}
       \Big\langle    \frac{d \phi}{dt}(0),   \phi^+ (0)   \Big\rangle_{k_+}    
       =&  -  \langle  ( \nabla_{X_H} \pi_{F_n} ) \phi(0), \phi^+(0) \rangle_{k_+}  
              - \langle   \pi_{F_n} D \phi(0), \phi^+(0) \rangle_{k_+}   \\
         &     - \langle  \pi_{F_n} c_1(\gamma(0)), \phi^+(0) \rangle_{k_+}
   \end{align*}
and
      \begin{align*}
    &  \langle  ( \nabla_{X_H} \pi_{F_n} ) \phi(0), \phi^+(0) \rangle_{k_+} = 0, \\
    & \langle \pi_{F_n}  D \phi(0), \phi^+(0) \rangle_{k_+}  
           = \langle D  \phi(0), \phi^+(0) \rangle_{k_+} \geq C \epsilon^2, \\ 
    & \big| \langle   \pi_{F_n} c_1(\gamma(0)), \phi^+(0) \rangle_{k_+} \big| \leq C \epsilon^3. 
      \end{align*} 
Here we have used Lemma \ref{lem nabla pi  phi psi} for the first equality. 
Therefore
\[
        \frac{d}{dt}  \bigg|_{t=0} \| \phi^+(t) \|_{k_+}^2 
     \leq  -C \epsilon^2 +  C\epsilon^3 < 0. 
\]
Assume that
\[
   \| \phi^-(0) \|_{k_-} = \epsilon. 
\]
A similar calculation shows that
\[
         \frac{d}{dt} \bigg|_{t=0} \| \phi^-(t) \|_{k_-}^2 > 0. 
\]

Similarly,   if $\| \omega^+(0) \|_{k_+} = \epsilon$ then  $\left. \frac{d}{dt}  \right|_{t=0} \| \omega^+(t) \|_{k_+}^2 < 0$,  and if $\| \omega^-(0) \|_{k_-} = \epsilon$ then  $\left. \frac{d}{dt} \right|_{t=0} \| \omega^-(t) \|_{k_-}^2 > 0$.  From these, it is easy to see that the conditions (2), (3) in Lemma \ref{lem (N, L_-)} are satisfied and  we can apply Lemma \ref{lem (N, L_-)} to conclude that the pair $(A_n(\epsilon), L_n(\epsilon))$ is an index pair. 

Therefore we have
\[
             I_n =   A_n(\epsilon) \cup_{p_B} L_{n, -} (\epsilon)  \cong  S_B^{F_n^- \oplus W_n^-}.
\]
\qed

\section{Examples}

\begin{ex}
Suppose that $Y$ has a positive scalar curvature metric. Then the conditions of Theorem \ref{thm reducible I_n} are satisfied.  
\end{ex}

\begin{ex} \label{ex flat 3-mfd}
Let $Y$ be a non-trivial flat torus bundle over $S^1$  which is not the Hantzsche-Wendt manifold. Then $Y$ has a flat metric and $b_1(Y) = 1$. Take a torsion $\mathrm{spin}^c$ structure $\frak{s}$ of $Y$.  All solutions to the unperturbed Seiberg-Witten equations on $Y$  are reducible solutions $(A, 0)$ with $F_A  = 0$. Also,  all finite energy solutions to the unperturbed Seiberg-Witten equations on $Y \times \mathbb{R}$ are  the reducible solutions $(\pi_Y^* A, 0)$, where $A$ are the flat $\mathrm{spin}^c$  connections on $Y$ and $\pi_Y : Y \times \mathbb{R} \rightarrow Y$ is the projection.   Hence Condition (2) of Theorem \ref{thm reducible I_n} is satisfied. 

By Lemma 37.4.1 of \cite{KM}, if $\frak{s}$ is not the  torsion $\mathrm{spin}^c$ structure corresponding to the $2$-plane field tangent to the fibers,  Condition (1) of Theorem \ref{thm reducible I_n} is satisfied. 
\end{ex}

We consider the sphere bundle of a complex line bundle over  a surface $\Sigma$. We will  make use of results from \cite{MOY}, \cite{Nicolaescu_Eta_inv} and  \cite[Section 8]{KLS1}.

Let $\Sigma$ be a closed, oriented surface of genus $g$ and $p : N_d \rightarrow \Sigma$ be the complex line bundle on $\Sigma$ of degree $d$.  We will consider the sphere bundle $Y = S(N_d)$.   
We have
\[
       H^2(Y;\mathbb{Z}) \cong  \mathbb{Z}^{2g}  \oplus (\mathbb{Z} / d\mathbb{Z}). 
\]
The direct summand $\mathbb{Z}  /  d\mathbb{Z}$ corresponds to the image
\[
          \mathrm{Pic}^{t}(\Sigma) / \mathbb{Z}[N_d] \stackrel{p^*}{\rightarrow} 
          \mathrm{\Pic}^{t} (Y) 
          \stackrel{c_1}{\rightarrow}
          H^2(Y;\mathbb{Z}),
\]
where $\rm{Pic}^{t} (\Sigma)$ is the set of isomorphism classes of topological complex line bundles on $\Sigma$.

Fix a torsion $\mathrm{spin}^c$ structure $\frak{s}$.  We consider a metric 
\[
          g_{Y, r} =  (r  \eta)^{\otimes 2} \oplus g_{\Sigma}
\]
on $Y$ for  $r > 0$.  Here $i \eta \in i\Omega^1(Y)$ is a constant-curvature connection $1$-form of $S(N_d)$.   Following \cite{MOY} and \cite{Nicolaescu_Eta_inv},  we take the connection $\nabla^0$  on $TY$ which is trivial in the fiber direction and is equal to the pull-back of the Levi-Civita connection on $\Sigma$ on $\ker \eta$.  For $a \in \mathcal{H}^1(Y)$,  let $D_{r, a}$ be the Dirac operator induced by $\nabla^0$.  We have
\[
         D_{r, a}  = D_a + \delta_{r}, 
\]
where $\delta_{r} = \frac{1}{2} rd$.  See Section 5.1 of \cite{MOY} and  Section 2.1 of \cite{Nicolaescu_Eta_inv}. 
The family $\{ D_{r, a} \}_{a \in \mathcal{H}^1(Y)}$ induces an operator 
\[
        D_r: \mathcal{E}_{\infty} \rightarrow \mathcal{E}_{\infty}. 
\]
We consider the perturbed Seiberg-Witten equations for $\gamma = (\phi, \omega) : \mathbb{R} \rightarrow \mathcal{E}_{\infty} \times \operatorname{im} d^*$: 
\begin{equation}  \label{eq perturbed SW eq}
   \begin{aligned}
        \Big(  \frac{d\phi}{dt}(t)  \Big)_{H} &= - D_r \phi(t)  - c_1(\gamma(t))    \\
         \Big(  \frac{d \phi}{dt}(t) \Big)_{V} &= - X_{H}(\phi(t)) \\
         \frac{d\omega}{dt}(t) &= -*d \omega(t) - c_2(\gamma(t)). 
   \end{aligned}
\end{equation}
These equations are the gradient flow equation of the perturbed Chern-Simons-Dirac functional
\[
        CSD_r (\phi, \omega) = CSD (\phi, \omega) + \delta_r \| \phi \|_{L^2}^2. 
\]
The term $\delta_r \| \phi \|_{L^2}^2$ is a tame perturbation. See  p.171 of \cite{KM}.  We can apply Theorem \ref{thm isolating nbd} to the perturbed Seiberg-Witten equations (\ref{eq perturbed SW eq}).  

The following is a direct consequence of  Corollary 5.17 and Theorem 5.19  of \cite{MOY}. See also  Section 3.2 of \cite{Nicolaescu_Eta_inv}  and Proposition 8.1,  Section 8.2 of \cite{KLS1}.

\begin{prop} \label{prop critical points of CSD_r}
Let $\fs_0$ be the spin$^c$ structure of $Y$ with spinor bundle $\mathbb{S} = p^* K_{\Sigma}^{-1} \oplus \underline{\C}$.  Denote by $L_q$  the flat complex line bundle on $Y$ with $c_1 \equiv q \mod d$ in $\operatorname{Tor} H^2(Y;\Z)$.  Put $\fs_q := \fs_0 \otimes L_q$. 
Assume that $0 < g < d$.   Then for  $q \in \{ g, g+1, \dots, d-1 \}$,  all critical point of  the functional $CSD_r$ associated with $\frak{s}_q$  are reducible and non-degenerate. 
\end{prop}

Note that this proposition implies that $\ker D_r = 0$ and hence we have a natural spectral section $P_0$ of $D_r$:
\[
     P_0 = (\mathcal{E}_0(D_r))_{-\infty}^{0}. 
\]

The following proposition is proved in the proof of Theorem 7.5 of \cite{KLS1}. 

\begin{prop}
Under the same assumption as Proposition \ref{prop critical points of CSD_r},   any gradient trajectory of $CSD_r$  (ie,  solution to (\ref{eq perturbed SW eq}))  with finite energy is reducible.
\end{prop}

We can apply  the proof of Theorem \ref{thm reducible I_n} to the perturbed Seiberg-Witten equations (\ref{eq perturbed SW eq}) to show the following: 

\begin{thm}  \label{thm I_n sphere bundle}
Take $q \in \{ g, g+1, \dots, d-1 \}$. 
Let $\frak{S}$ be a spectral system with $P_0 = \mathcal{E}_0(D_r)_{-\infty}^{0}$. 
In the above notation,  for $r$ small,  we have
\[
              I_n \cong     S_B^{F_n^- \oplus W_n^-}
\]
Therefore we have
\[
     \mathcal{SWF}(Y, \frak{s}_q, [\frak{S}]) \cong S_B^0 
 \]
in $PSW_{S^1, B}$. 
If $\frak{s}$ is self-conjugate, 
\[
     \mathcal{SWF}^{\rm{Pin}(2)}(Y, \frak{s}_q, [\frak{S}]) \cong S_B^0 
\]
in $PSW_{\mathrm{Pin}(2), B}$.  
\end{thm}

In \cite{dai-sasahira-stoffregen},  Dai and the authors computed the Seiberg-Witten Floer stable homotopy type for almost rational plumbed 3-manifolds which have $b_1 = 0$.  The computation is based on surgery exact triangles   in \cite{sasahira-stoffregen_triangle}.   If we establish a surgery exact triangle for the Seiberg-Witten Floer stable homotopy type $\mathcal{SWF}(Y, \mathfrak{s}, \mathfrak{S})$ defined in this memoir, it would be possible to compute for more  3-manifolds with $b_1 > 0$.


\chapter{Finite dimensional approximation on $4$-manifolds}
\label{section 4-mfd with boundary}

\section{Construction of the relative Bauer-Furuta invariant}\label{subsec:construction-of-bf}

Let $(X,  \frak{t})$ be a compact $\mathrm{spin}^c$ 4-manifold with boundary $Y$. Take a Riemannian metric $\hat{g}$ of $X$ such that a neighborhood of $Y$ in $X$ is isometric to $Y \times (-1, 0]$.  We assume that the restriction $\frak{s}$ of $\frak{t}$ to $Y$ is a torsion $\mathrm{spin}^c$ structure.    Put
   \begin{align*}
           \mathcal{E}_{X, k}^{\pm}  
          &:=  \mathcal{H}^1(X)  \times_{H^1(X; \mathbb{Z})} L^2_k(\Gamma(\mathbb{S}^{\pm})),  \\
         \mathcal{W}_{X,k} &:= B_{X} \times L^2_k( \Omega_{CC}^1(X) ). 
    \end{align*}
 Here $B_{X} = \mathrm{Pic}(X)$ and  $\mathbb{S}^{\pm}$ are the spinor bundles on $X$ and $\Omega^1_{CC}(X)$ is the space of $1$-forms on $X$ in  double Coulomb gauge. See \cite{khandhawit} for the double Coulomb gauge condition. Note that $\mathcal{E}_{X, k}^{\pm}$, $\mathcal{W}_{X, k}$ are Hilbert bundles over $B_X$.
We have the Dirac operator
\[
     D_X : \mathcal{E}_{X,k}^+ \rightarrow \mathcal{E}_{X, k-1}^-
\]
on $X$, and as before,  we can define the fiberwise norm $\Vert \cdot \Vert_{k}$ on $\mathcal{E}_{X, k}^{\pm}$ for each non-negative number $k$.  Also we put
  \begin{align*}
     &  \mathcal{E}_{Y,k} := \mathcal{H}^1(Y) \times_{H^1(Y;\mathbb{Z})} L^2_k(\mathbb{S}),   \\
     & \mathcal{W}_{Y, k} := B_Y \times  L^2_{k}(\operatorname{im} d^*) \subset B_Y \times  L^2_k(\Omega^1(Y)). 
   \end{align*}
Here $P_Y = \mathrm{Pic}(Y)$.

\begin{prop} \label{prop compactness X, Y times R_{geq}}
For $k, l \geq 0$, there are constants  $R_{X, k}, R_{Y, l} > 0$ such that for any solution $x \in \mathcal{E}_{X, 2}^{+} \oplus \mathcal{W}_{X, 2}$ to the Seiberg-Witten equations on $X$ and  any Seiberg-Witten trajectory $\gamma : \mathbb{R}_{\geq 0} \rightarrow  \mathcal{E}_{Y, 2} \oplus \mathcal{W}_{Y, 2}$ with finite energy and   with
\[
         r_{Y}(x) = \gamma(0), 
\]
we have 
\[
      \| x \|_{k} \leq R_{X, k}, \quad \| \gamma(t) \|_{l} \leq R_{Y, l}
\]
for all $t \in \mathbb{R}_{\geq 0}$.  Here $r_Y$ stands for the restriction to the boundary $Y$.      
\end{prop}

See Section 4 of \cite{khandhawit} for  this Proposition.   

Let $D_Y$ be the family of Dirac operators on $Y$ parameterized by $B_Y$. 
Assume that $\operatorname{ind} D_Y = 0$ in $K^1(B_Y)$.  
Choose a spectral system $\frak{S}$. 
As usual,  put
\[
        F_n = P_n \cap Q_n,  \
        W_n = W_{P, n} \cap W_{Q, n}. 
\]
Then $F_n$, $W_n$ are subbundles of $\mathcal{E}_{Y,0}$, $\mathcal{W}_{Y, 0}$ with finite rank. 

From now on, we assume that $k$ is a half integer and $k > 5$ so that we can use the results in Chapter \ref{sec:findim-appx} and \ref{sec:well-def}. 
We consider the  map
\begin{equation}\label{eq:4-sw}
  \begin{aligned}
   &       SW_{X,n}:
      \mathcal{E}_{X, k}^+ \oplus \mathcal{W}_{X,k}
       \rightarrow   \\
    & \qquad   (\mathcal{E}_{X, k-1}^- \times L^2_{k-1}(\Omega^+(X)))  \times
       ( (P_n  \oplus W_{P, n}) \cap L^2_{k-\frac{1}{2}})  
    \end{aligned}
\end{equation}
defined by 
\[
SW_{X, n}  (\hat{\phi}, \hat{\omega})   = 
(D_X \hat{\phi} + \rho(\hat{\omega}) \hat{\phi},   F_{\hat{A}}^+ - q(\hat{\phi}), 
          \pi_{P_n} r_{Y} \hat{\phi}, 
          \pi_{W_{P, n}} r_{Y} \hat{\omega}).
\]
Here  $\pi_{P_n}, \pi_{W_{P, n}}$ are the $L^2$-projection,  where we have written $P_n$ also for the total space of the spectral section $P_n$.   
We will take subbundles $U_{n},  U_{n}'$ of $\mathcal{E}_{X, k}^+$, $\mathcal{E}_{X, k-1}^{-}$ with finite rank as follows.  The operator
\[
       (D_X, \pi_{P_0} r_Y) : \mathcal{E}_{X, k}^+ \rightarrow \mathcal{E}_{X, k-1}^{-} \oplus r_Y^* (P_0 \cap L^2_{k-\frac{1}{2}})
\]
is Fredholm. (See \cite{MP},  \cite[Section 17.2]{KM} and Section \ref{section:def of spectral section}. )   Hence there is a fiberwise linear operator
\[
     \frak{p} :   \underline{\mathbb{C}}^{m} \rightarrow \mathcal{E}_{X,k-1}^{-} \oplus r_Y^* (P_0 \cap L^2_{k-\frac{1}{2}})
\]
such that
\begin{equation}\label{eq:perturbation}
     (D_X, \pi_{P_0} r_Y ) + \frak{p} : 
     \mathcal{E}^+_{X, k} \oplus \underline{\mathbb{C}}^m
     \rightarrow 
     \mathcal{E}_{X, k-1}^- \oplus r_Y^* (P_0 \cap L^2_{k-\frac{1}{2}})
\end{equation}
is surjective. Here $\underline{\mathbb{C}}^m = B_X \times \mathbb{C}^m$ is the trivial bundle over $B_X$.

\begin{lem} \label{lem (D, i)}
For any $n$ and any subbundle $U'$ in $\mathcal{E}_{X, k-1}^{-}$,  $U' \oplus r_Y^* F_n$ and the image of
\[
         (D_X, \pi_{P_n} r_Y) + \mathfrak{p} : 
         \mathcal{E}_{X, k}^+ \oplus \underline{\mathbb{C}}^m
         \rightarrow
         \mathcal{E}_{X, k-1}^{-} \oplus r_Y^* (P_n \cap L^2_{k-\frac{1}{2}})  
\]
are transverse in $\mathcal{E}_{X, k-1}^- \oplus r_Y^* (P_n \cap L^2_{k-\frac{1}{2}})$. 
\end{lem}

\begin{proof}
Take any element $(x', y)$ from $\mathcal{E}_{X, k-1}^{-} \oplus r_Y^* ( P_n \cap L^2_{k-\frac{1}{2}} )$. There is $(x, v) \in \mathcal{E}_{X, k}^+ \oplus \underline{\mathbb{C}}^m$ such that
\[
        ( (D_X, \pi_{P_0} r_Y) + \frak{p}) (x, v) = (x',  \pi_{P_0} (y)). 
\]
Note that
\[
        P_n \cap (P_0)^{\bot} = F_n^+. 
\]
We  can write
\[
    (D_X, \pi_{P_n} r_Y + \frak{p}) (x, v) 
    = ((D_X,  (\pi_{P_0} + \pi_{F_n^+}) r_Y) + \frak{p})(x, v) 
    = (x', \pi_{P_0}(y) + z),
\]
where $z = \pi_{F_n^+} (r_Y x) \in  F_n^+ \subset F_n$.   Hence
   \begin{align*}
     (x', y) 
     & = (x', \pi_{P_0}(y) + z) + (0,  \pi_{ F_n^+  } (y)  - z)   \\
     & \in \operatorname{im} (( D_X, \pi_{P_n} r_Y) + \frak{p}) + F_n. 
  \end{align*}
\end{proof}

Take a sequence of finite-dimensional subbundles $U'_n$ of $\mathcal{E}_{X, k-1}^{-}$ such that $\pi_{U'_n} \rightarrow id_{\mathcal{E}_{X, k-1}^-}$ strongly as $n \rightarrow \infty$  and put
\begin{equation}\label{eq:4-manifold-appx-bundle}
        U_n :=   (( D_X, \pi_{P_n} r_Y  ) + \frak{p})^{-1}(U'_n \oplus r_Y^* F_n). 
\end{equation}
By Lemma \ref{lem (D, i)}, $U_n$ are subbundles of $\mathcal{E}_{X, k}^{+} \oplus \underline{\mathbb{C}}^m$.  Note that
\[
         [U_n] - [U_n' \oplus r_Y^* F_n] - [\underline{\mathbb{C}}^m]= [\operatorname{ind} (D_X, P_n)] \in K(B_X).
\]  
Here the right hand side is the index bundle defined in \cite[Section 6]{MP}.   

Choose finite dimensional subbundles 
\[
   V_n' = B_X \times V_{n, 0}'
\]
 of $B_{X} \times L^2_{k-1}(\Omega^+(X))$ with $\pi_{V_n'} \rightarrow \operatorname{id}_{B _{X} \times L^2_{k-1}(\Omega^+(X))}$ strongly as $n \rightarrow \infty$ and put
\[
       V_n := (d^+, \pi_{W_{P, n}} r_Y)^{-1} (  V_n' \oplus W_n   )
       \subset \mathcal{W}_{X, k}.
\]

We consider the following maps
\begin{equation}\label{eq:perturbed-sw-map}
    \begin{aligned}
      SW_{X, n, \frak{p}} & : = (D_X, d^+) + \frak{p} + \pi_{U_n' \oplus V_n'} c_X  : U_n \oplus V_n 
                      \rightarrow 
                      U_n' \oplus V_n',  \\
      \widetilde{SW}_{X,n,\frak{p}} & := (  SW_{X, n, \frak{p}}, \pi_{P_n} r_Y,\pi_{W_{P,n}} r_Y,  \operatorname{id}_{\underline{\mathbb{C}}^m}) :   \\
     &       \qquad    U_n \oplus V_n \rightarrow U_n' \oplus V_n' \oplus  r_Y^*(F_n \oplus  W_n) \oplus \underline{\mathbb{C}}^m,
    \end{aligned}
\end{equation}
where 
\[
       c_X(\hat{\phi}, \hat{\omega}) = (\rho(\hat{\omega}) \hat{\phi},   F_{\hat{A}_0}^+ + q(\hat{\phi}))
\]
for a fixed connection $\hat{A}_0$ on $X$. 
Fix positive numbers $R, R'$ with $0 \ll R' \ll R$.  Put
\[
     A_n : = ( B_{k - \frac{1}{2} }(F_n^+; R) \times_{B_Y} B_k(F_n^-; R)) \times_{B_Y}( B_{k-\frac{1}{2}}(W_n^+; R) \times_{B_Y} B_k(W_n^-; R) ).
\]
Here $B_{k-\frac{1}{2}}(F_n^+;R)$ is the ball in $F_n^+$ of radius $R$ with respect to $L^2_{k-\frac{1}{2}}$. Similarly for $ B_k(F_n^-;R), B_{k-\frac{1}{2}}(W_n^+; R)$, $B_{k}(W_n^-; R)$.   Note that we take different norms $L^2_{k-\frac{1}{2}}$ and $L^2_{k}$ for $F_n^{+}$, $W_n^+$ and $F_n^{-}$, $W_n^-$.   By Theorem \ref{thm isolating nbd}, for $n \gg 0$,  $A_n$ is an isolating neighborhood of the flow $\varphi_{n, k - \frac{1}{2},k}$, for suitable $k$.  
For $\epsilon > 0$, we define compact subsets $K_{n, 1}(\epsilon), K_{n, 2}(\epsilon)$ of $A_n$  by
\begin{align*}
    & K_{n, 1}(\epsilon)    \\
    &  :=  \Bigg\{ y \in A_n   :
      \begin{array}{l}
      \exists (\hat{\phi}, v, \hat{\omega}) \in B_k(U_n \oplus V_n  ; R'),  
      (\hat{\phi}, v) \in U_n  \subset \mathcal{E}_{X, k}^+ \oplus \underline{\mathbb{C}}^{m}, 
       \hat{\omega} \in V_n, \\
      \|  (SW_{X, n, \frak{p}}, \operatorname{id}_{\mathbb{C}^m})(\hat{\phi}, v, \hat{\omega})  \|_{k-1} \leq  \epsilon,  
      y = \pi_{ P_n \oplus \mathcal{W}_{-\infty}^{\mu_n} } r_Y ( \hat{\phi}, \hat{\omega} )
      \end{array}
      \Bigg\},  
\end{align*}
and
\begin{align*}
   & K_{n, 2}(\epsilon)      \\
     &:=  \Bigg\{   y \in A_n   : 
    \begin{array}{l}
     \exists (\hat{\phi}, v, \hat{\omega}) \in  \partial B_k(U_n \oplus V_n ;R'),    \\
     \|   (SW_{X,n, \frak{p}}, \operatorname{id}_{\underline{\mathbb{C}}^m} ) (\hat{\phi}, v, \hat{\omega}) \|_{k-1} \leq \epsilon,
      y = \pi_{  P_n \oplus \mathcal{W}_{-\infty}^{\mu_n} } r_Y ( \hat{\phi}, \hat{\omega} )
     \end{array}
       \Bigg\}    \\
&  \quad    \bigcup 
    \Big(  \partial A_n \bigcap K_{n, 1}(\epsilon) \Big). 
\end{align*} 
Here
\[
     \| (SW_{X, n,\frak{p}}, \operatorname{id}_{\underline{\mathbb{C}}^m})(\hat{\phi}, v, \hat{\omega})  \|_{k-1}
     = \| SW_{X, n,\frak{p}} (  \hat{\phi}, \hat{\omega}  ) \|_{k-1} + \| v \|.
\]
We will show that we can find a regular index pair containing $(K_{1, n}(\epsilon), K_{2, n}(\epsilon))$. 
See Section \ref{subsec:conley} for the definition of a regular index pair.

\begin{prop}\label{prop:bauer-furuta-existence}
There is a  $\epsilon_0 > 0$ such that  if $0 < \epsilon < \epsilon_0$, for $n$ large, we can find a regular index pair $(N_n, L_n)$ of $\operatorname{inv}(A_n; \varphi_{n,  k-\frac{1}{2},k})$ with 
\[
                K_{n, 1}(\epsilon) \subset N_n \subset A_n, \quad
                K_{n, 2}(\epsilon) \subset L_n. 
\]
 
\end{prop}

\begin{proof}
We write $\varphi_{n}$ for $\varphi_{n, k-\frac{1}{2},k}$. 
We denote by $A_n^{[0,\infty)}$  the set
\[
    \{  y \in A_n :  \forall t \in [0, \infty),  \varphi_{n}(y, t)    \in A_n  \}. 
\]
By Theorem 4 of \cite{Manolescu-b1=0},   it is sufficient to prove the following for $n$ large and $\epsilon$ small: 
\begin{enumerate}[(i)]

\item
if  $y \in K_{n, 1}(\epsilon) \cap A_n^{[0,\infty)}$ then we have  $\varphi_{n}(y, t) \not\in \partial A_n$ for all $t \in [0, \infty)$,  

\item
$K_{n, 2}(\epsilon) \cap A_n^{[0,\infty)} = \emptyset$.

\end{enumerate}
Furthermore, any index pair as constructed by Theorem 4 of \cite{Manolescu-b1=0} may be thickened to give a regular index pair still satisfying the conditions of the Proposition. See Remark 5.4 of \cite{Salamon}. 

Note that    for $y \in K_{n,1}(\epsilon)$ we have
\begin{equation}  \label{eq y_n^+ < R}
    \| y^{+} \|_{k-\frac{1}{2}} < R
\end{equation}
for all $n$ since the restriction $L^2_{k}(X) \rightarrow L^2_{k-\frac{1}{2}}(Y)$ is bounded and $R' \ll R$.


%
%
%
%
%
%

First, we will prove that (i) holds for $n$ large and $\epsilon$ small.  Assume that this is not true.  Then there is a sequence $\epsilon_n \rightarrow 0$ such that  after passing to a subsequence,  we have $y_n \in A_n$,   $(\hat{\phi}_n, v_n, \hat{\omega}_n) \in B_k(U_n \oplus V_{n}; R')$, $t_n \in [0, \infty)$ with 
   \begin{align*}
      & y_n = \pi_{ P_n \oplus W_{P, n}} r_Y (\hat{\phi}_n, \hat{\omega}_n), \\  
      & \| SW_{X, n,\frak{p}}( \hat{\phi}_n, \omega_n ) \|_{k-1}^2 + \| v_n \|^2  \leq \epsilon_n^2,  \\
      &  \varphi_n(y_n, [0, \infty) ) \subset A_n, \\
      &  \varphi_n(y_n, t_n) \in \partial A_n. 
   \end{align*}
Note that $v_n \rightarrow 0$. 
Let
\[
     \gamma_n = (\phi_n, \omega_n) : [0, \infty) \rightarrow F_n \oplus W_n
\] 
be the approximate Seiberg-Witten trajectory defined by
\[
       \gamma_n(t) = \varphi_{n}(y_n, t). 
\]
After passing to a subsequence,  one of the following holds for all $n$:

\begin{enumerate}[(a)]

\item
$\phi_{n}^+(t_n) \in S_{k-\frac{1}{2}}(F_n^+;R)$, 

\item
$\phi_n^{-} (t_n) \in S_{k}(F_n^-; R)$,

\item
$\omega_n^{+}(t_n) \in S_{k-\frac{1}{2}}(W_n^+; R)$,

\item
$\omega_n^{-}(t_n) \in S_k(W_n^-;R)$.

\end{enumerate}

Note that in the cases (a) and (c), we have $t_n > 0$ because of (\ref{eq y_n^+ < R}). 

As in the proof of Theorem \ref{thm isolating nbd}, we can show that there is a Seiberg-Witten trajectory 
\[
     \gamma = (\phi, \omega) : [0,\infty) \rightarrow  \mathcal{E}_{Y, k-\frac{3}{2}, k-1
     } \oplus  \mathcal{W}_{Y, k-\frac{3}{2}, k-1}
\] 
such that after passing to a subsequence, $\gamma_n$ converges to $\gamma$ uniformly in $L^2_{k-\frac{3}{2}}$ on each compact set in $[0, \infty)$. 
Also after passing to a subsequence, $(\hat{\phi}_n, \hat{\omega}_n)$ converges to a solution $(\hat{\phi}, \hat{\omega})$ to the Seiberg-Witten equations on $X$ uniformly in $L^2_{k-1}$  on each compact set in the interior of $X$.  We have
\[
            r_Y (\hat{\phi}, \hat{\omega}) = \gamma(0).
\]

Assume that the case (a) happens for all $n$.   As mentioned, $t_n > 0$.  Hence we have
\[
         \frac{d}{dt} \bigg|_{t = t_n}  \Vert \phi_n^+(t) \Vert_{k-\frac{1}{2}}^2 = 0. 
\]
As in Lemma \ref{lem gamma_n L^2_{k+1/2}}, we can show that there is $C > 0$ such that $\| \phi_{n}^+(t_n) \|_{k} < C$ for all $n$.     
After passing to a subsequence, $t_n \rightarrow t_{\infty} \in \mathbb{R}_{\geq 0}$ or $t_{n} \rightarrow \infty$. 
First assume that $t_n \rightarrow t_{\infty}$. By the Rellich lemma, $\phi_n^+(t_n)$ converges in $L^2_{k-\frac{1}{2}}$ strongly. This implies that 
\[
          \| \phi^+(t_{\infty}) \|_{k-\frac{1}{2}} = R, 
\]
which contradicts Proposition \ref{prop compactness X, Y times R_{geq}}. 

Next we consider the case $t_n \rightarrow \infty$. 
Let
\[
     \underline{\gamma}_n  = (\underline{\phi}_n, \underline{\omega}_n) 
     : [-t_n, \infty)  \rightarrow F_n \oplus W_n
\]
be the approximate Seiberg-Witten trajectory  defined by
\[
      \underline{\gamma}_n(t) := \varphi_{n}(y_n, t + t_n). 
\]
As before, we can show that there is a Seiberg-Witten trajectory
\[
       \underline{\gamma} : \mathbb{R} \rightarrow 
       \mathcal{E}_{Y, k-\frac{3}{2},k-1} \oplus \mathcal{W}_{Y, k-\frac{3}{2}, k-1}
\]
such that after passing to a subsequence,  $\underline{\gamma}_n$ converges to $\underline{\gamma}$ uniformly in $L^2_{k-\frac{3}{2}}$ on each compact set in $\mathbb{R}$. 
As before we can show that  the sequence $\| \underline{\phi}_{n}^{+}(0) \|_{k}$ is bounded and hence $\underline{\phi}^+_n(0)$ converges to $\underline{\phi}^+(0)$ in $L^2_{k-\frac{1}{2}}$ strongly.  Therefore $\| \underline{\phi}^{+}(0) \|_{k-\frac{1}{2}} = R$, which contradicts Proposition \ref{prop compactness}.    
Thus (a) can not happen.

Let us consider the case when (b) holds for all $n$. 
We have
\[
      \left. \frac{d}{dt}  \right|_{t = t_n}  \| \phi_n^-(t) \|_{k}^2 \leq 0. 
\]
As in the proof of Lemma \ref{lem phi^-}, 
  \begin{align*}
        0 
        &\geq  
              \frac{d}{dt} \bigg|_{t=t_n} \| \phi_n^-(t) \|_{k}^2   \\
        &\geq 
        -\langle  D'\phi_n^-(t_n), \phi_n^-(t_n) \rangle_k - C R^2\| \phi_n^-(t_n) \|_{k+\frac{1}{2}} - CR^2  \\
        & = \| \phi_{n}^-(t_n) \|_{k+\frac{1}{2}}^2  - CR^2\| \phi_n^-(t_n) \|_{k+\frac{1}{2}} - CR^2. 
 \end{align*}       
This implies that the sequence $\| \phi_n^-(t_n) \|_{k+\frac{1}{2}}$ is bounded and there is a subsequence such that $\phi_{n}^-(t_n)$ converges in $L^2_{k}$ strongly.  We have a contradiction as before. 

In the case when (c) or (d) holds for all $n$, we have a contradiction similarly.   We have proved that (i) holds for $n$ large and $\epsilon$ small. 

\vspace{3mm}

Next we will prove  that (ii) holds for $n$  large and $\epsilon$ small.  If this is  not true,   there is a sequence $\epsilon_n \rightarrow 0$ such that after passing to a subsequence,   one of the following cases holds for all $n$.

\begin{enumerate}[(a)]
\item
We have $(\hat{\phi}_n, v_n, \hat{\omega}_n)  \in \partial B_{k}(U_n \oplus V_n; R')$, $y_n \in A_n^{[0, \infty)}$ with
\[
     \|  SW_{X, n, \frak{p}}(\hat{\phi_n}, \hat{\omega}_n) \|_{k-1} + \| v_n \| \leq  \epsilon_n, \ 
     y_n = \pi_{ P_n \oplus W_{P,n} } r_Y (\hat{\phi}_n, \hat{\omega}_n). 
\]

\item
We have  $(\hat{\phi}_n, v_n, \hat{\omega}_n)  \in B_k(U_n \oplus V_n; R')$, $y_n \in \partial A_n \cap A_n^{[0, \infty)}$ with
\[
       \|  SW_{X, n, \frak{p}}(\hat{\phi}_n, \hat{\omega}_n) \|_{k-1} + \| v_n \|  \leq  \epsilon_n, \ 
     y_n = \pi_{ P_n \oplus W_{P, n} } r_Y (\hat{\phi}_n, \hat{\omega}_n). 
\]

\end{enumerate}

First we consider  the case (a). 
Let 
\[
  \gamma_n = (\phi_n, \omega_n)  :  [0,\infty) \rightarrow F_n \oplus W_n
\] 
be the approximate Seiberg-Witten trajectory defined by
\[
       \gamma_n(t) = \varphi_{n}(y_n, t). 
\]
As before, there is a Seiberg-Witten trajectory
\[
      \gamma = (\phi, \omega):  [0,\infty) \rightarrow \mathcal{E}_{Y, k-\frac{3}{2},k-1} \oplus \mathcal{W}_{Y, k-\frac{3}{2},k-1}
\]
such that after passing to a subsequence,  $\gamma_n$ converges to $\gamma$  uniformly in $L^2_{k-\frac{3}{2}}$ on each compact set in $[0, \infty)$.  Also there is a solution $(\hat{\phi}, \hat{\omega})$ to the Seiberg-Witten equations on $X$ such that after passing to a subsequence, $( \hat{\phi}_n, \hat{\omega}_n )$ converges to $(\hat{\phi}, \hat{\omega})$ in $L^2_{k-1}$ on each compact set in the interior of $X$.  We have
\[
       r_Y(\hat{\phi}, \hat{\omega}) = (\phi(0), \omega(0)). 
\]

Since $y_n \in A_n$, we have
\[
           \| y_n^{-} \|_{k} = \| (\phi_n^-(0), \omega_n^-(0)) \|_{k} \leq R. 
\]
Hence after passing to subsequence,  $(\phi_n^-(0), \omega_n^-(0))$ converges to $(\phi^-(0), \omega^{-}(0))$ in $L^2_{k-\frac{1}{2}}(Y)$ strongly. By the standard elliptic estimate,   we have
  \begin{align*}
  &  \| \hat{\phi}_{n} - \hat{\phi} \|_{L^2_k(X)}   \\
    \leq 
    & C  \Big(   \| \hat{\phi}_n - \hat{\phi} \|_{L^2(X)} + \| D_X (\hat{\phi}_n - \hat{\phi}) \|_{L^2_{k-1}(X)}  + \|   \phi_n^{-}(0) - \phi^{-}(0)    \|_{L^2_{k-\frac{1}{2}}(Y)}  \Big).
  \end{align*}
From the condition  that
\[
    \| SW_{X, n, \frak{p}}(\hat{\phi}_n, \hat{\omega}_n) \|_{k-1} + \| v_n \| \leq \epsilon_n, 
\]
we have
\[
     \| D_X(\hat{\phi}_n -  \hat{\phi}) \|_{k-1} \leq  C( \|   c_X(\hat{\phi_n}, \hat{\omega}_n) -  c_X(\hat{\phi}, \hat{\omega})  \|_{k-1} + \epsilon_n).
\]
Since $c_X(\hat{\phi_n}, \hat{\omega}_n)$ converges to $c_X(\hat{\phi}, \hat{\omega})$ in $L^2_{k-1}$ strongly,   $\hat{\phi}_n$ converges to $\hat{\phi}$ in $L^2_k$ strongly. 

Similarly, $\hat{\omega}_n$ converges to $\hat{\omega}$ in $L^2_k$ strongly. Hence, 
\[
       \| (\hat{\phi}, \hat{\omega}) \|_{k} = R'.
\]
This contradicts Proposition \ref{prop compactness X, Y times R_{geq}}, so case (a) cannot happen.

Next we consider the case (b).   Let
\[
      y_n = (\phi_n, \omega_n). 
\]
 After passing to a subsequence,  $\phi_n^- \in S_k(F_n^-; R)$ for all $n$ or,  $\omega_n^- \in S_k(W_n^-; R)$ for all $n$.  Note that the cases $\phi_n^+ \in S_{k-\frac{1}{2}}(F_n^+; R)$, $\omega_n^+ \in S_{k-\frac{1}{2}}(W_n^+;R)$ do not happen because of (\ref{eq y_n^+ < R}). 

We consider  the case $\phi_n^- \in S_k(F_n^-; R)$.   Put
\[
  \gamma_n(t) = (\phi_n(t), \omega_n(t))  = \varphi_n(y_n, t)
\]
for $t \geq 0$.  As in the proof of Lemma \ref{lem phi^-}, 
   \begin{align*}
         0 
        &\geq     \frac{d}{dt} \bigg|_{t=0} \|  \phi_n^-(t)  \|_{k}^2   \\
        & \geq  \| \phi_n^- \|_{k+\frac{1}{2}}^2 - CR^2 \| \phi_n^- \|_{k+\frac{1}{2}} - CR^2. 
   \end{align*}
Therefore the sequence $\| \phi_n^- \|_{k+\frac{1}{2}}$  is bounded. By the Rellich lemma, $\phi_n^-$ converges to $\phi^-$ in $L^2_k$ strongly and hence
\[
           \| \phi^- \|_{k}= R, 
\]
which contradicts Proposition \ref{prop compactness X, Y times R_{geq}}. 
Similarly, if $\omega_n^- \in S_k(W_n^-; R)$ for all $n$, we obtain a condtradition. We have proved that (ii) holds for $n$ large and $\epsilon$ small.

\end{proof}

\begin{rem}   \label{rem:L^2_{k-1/2, k}}
To get (\ref{eq y_n^+ < R}), we used the $L^2_{k-\frac{1}{2}}$-norm on the positive component. 
On the other hand, in the case (ii)-(a),  we  used the condition that $\| \phi_n^- (0)\|_{k}$ is bounded  (rather than $\| \phi_n^{-}(0) \|_{k-\frac{1}{2}}$) to have that $\phi_n^-(0)$ converges to $\phi^-(0)$ in $L^2_{k-\frac{1}{2}}$.  This is why we used the $L^2_k$-norm on the negative component to define $K_{n,1}(\epsilon)$, $K_{n,2}(\epsilon)$. 

In the case where $b_1(Y)= 0$,  we can use the $L^2_{k-\frac{1}{2}}$-norm on both of the positive and negative component. See the proofs of Proposition 6 of \cite{Manolescu-b1=0} and Lemma 4.4 of \cite{khandhawit}. In those proofs, to get the $L^2_{k-\frac{1}{2}}$-convergence of $\phi_n^{-}(0)$,   the following identity was used: 
\begin{equation}  \label{eq:e^{D} phi}
      e^{D}  \phi_{n}^{-}(1) -   \phi_n^{-}(0) =    \int_0^{1} \frac{d}{dt}  (e^{tD} \pi^{-} \phi_n(t) )    dt.
\end{equation}
In the case where $b_1 (Y) > 0$, we have 
\[
 \begin{split} 
        \frac{d}{dt}  (e^{tD} \pi^{-} \phi_n(t) )    
       &  =  e^{tD} (D + \nabla_{X_H} D) \pi^{-}  \phi_{n}(t) + e^{tD}  (\nabla_{X_H} \pi^{-}) \phi_n(t)    \\
        & \qquad  - e^{tD} \pi^{-} \{ ( \pi_n D + \nabla_{X_H} \pi_{F_n}) \phi_n(t)  + q(\phi_n(t))    \}. 
   \end{split} 
\] 
Since  $(\nabla_{X_H} \pi_{F_n})\phi_n(t)$ does not converges in $L^2_{k-\frac{1}{2}}$,  we can not deduce that $\phi_n^{-}(0)$ converges in $L^2_{k-\frac{1}{2}}$ from (\ref{eq:e^{D} phi}). 
\end{rem}

For $n$ large and $\epsilon$ small, let $(N_n, L_n)$ be a regular index pair of $\operatorname{inv}( \varphi_n, A_n )$ with
\[
                K_{1, n}(\epsilon) \subset N_n, \
                K_{2, n}(\epsilon) \subset L_n. 
\]
Put
\[ 
  \begin{split}
      & S^{U_n \oplus V_n}_{B_X} 
    := \bigcup_{a \in B_X} B( (U_n \oplus V_n)_a; R) / 
                  S ( (U_n \oplus V_n )_a; R),   \\
      & S_{B_X}^{ U_n' \oplus V_n' \oplus \underline{\mathbb{C}^m} }
     := \bigcup_{a \in B_X} B(  (U_n' \oplus V_n' \oplus \underline{\mathbb{C}^m} )_a; \epsilon  ) /
                    S (  (U_n' \oplus V_n' \oplus \underline{\mathbb{C}^m} )_a; \epsilon  ),
   \end{split} 
\]
which are sphere bundles over $B_X$,  and let $I_n$ be the Conley index:
\[
          I_n := N_n \cup_{ p_{B_Y} |_{ L_n } } B_Y.
\]
Here $p_{B_Y} : N_n \rightarrow B_Y$ is the projection. 
We obtain a map
\begin{equation}\label{eq:bauer-furuta}
              \preBF_{[n]}(X,\frak{t}):    S_{B_X}^{U_n \oplus V_n} \rightarrow 
                    S_{B_X}^{ U_n' \oplus V_n' \oplus \underline{\mathbb{C}}^m } \wedge_{B_X}    r^*_Y I_n
\end{equation}
defined by
\begin{align*}
     & \preBF_{[n]}(X, \frak{t})([\hat{\phi}, v, \hat{\omega}])    \\
    & =  \begin{cases} 
            [SW_{X, n, \frak{p}}(\hat{\phi}, v,  \hat{\omega}), v] \wedge [\pi_{P_n \oplus W_{P, n}} r_Y (\hat{\phi}, \hat{\omega})] &  \text{if (\ref{eq SW < epsilon})  holds},     \\
           *_{a}  & \text{otherwise.}
        \end{cases}
  \end{align*} 
Here $a = p_{B_X}(\hat{\phi}, \hat{\omega})$,  $*_a$ denotes the base point  of the sphere $S^{(U_n' \oplus V_n' \oplus \underline{\mathbb{C}}^m)_a}$  and  the condition (\ref{eq SW < epsilon}) is the following:
\begin{equation}  \label{eq SW < epsilon}
    \begin{aligned} 
        &  \|  SW_{X, n, \frak{p}}(\hat{\phi}, v, \hat{\omega}) \|_{k-1}^2 + \| v \|^2 \leq \epsilon, \\
        &  \pi_{P_n \oplus W_{P,n}} r_Y( \hat{\phi}, \hat{\omega} ) \in K_{n, 1}(\epsilon).
  \end{aligned} 
\end{equation}
We refer to the map $\preBF_n(X,\frak{t})$ as the (relative, $n$-th) \emph{pre-Bauer-Furuta invariant} of $(X,\frak{t})$, to emphasize that it is not yet an invariant of the construction (rather, its stable homotopy equivalence class will turn out to be an invariant).

An alternative version of this relative Bauer-Furuta invariant is obtained by instead considering the map of $B_Y$ spaces:
\[
\preBF_{[n]}(X, \frak{t}):  S_{B_X}^{U_n \oplus V_n} \rightarrow 
S_{B_X}^{ U_n' \oplus V_n' \oplus \underline{\mathbb{C}}^m } \wedge_{B_Y}   N_n/_{B_Y} L_n,
\]
where $S_{B_X}^{U_n \oplus V_n}$ is a $B_Y$ space using $r_Y$, and where $N_n/_B L_n$ is the fiberwise quotient.


 \section{Well-definedness of the relative Bauer-Furuta invariant}\label{subsec:bauer-furuta-well-def}
We next consider how the construction of the relative Bauer-Furuta invariant in (\ref{eq:bauer-furuta}) depends on the choices involved.  This is very similar to Chapter \ref{sec:well-def}, so we will abbreviate many of the arguments.  

First, we address the perturbation $\mathfrak{p}$.   

\begin{lem}\label{lem:perturbation-independence}
	Let $\mathfrak{p}_1$ be a perturbation for which (\ref{eq:perturbation}) is surjective.  Let $\mathfrak{q}$ be a linear operator $\underline{\mathbb{C}}^{m_2}\to \mathcal{E}_{X,k-1}^-\oplus r_Y^*(P_0\cap L^2_{k-\frac{1}{2}})$.  Let $U_n(\mathfrak{p})$, respectively $U_n(\mathfrak{p}+\mathfrak{q})$ be the bundles defined as in (\ref{eq:4-manifold-appx-bundle}) with respect to the perturbations $\mathfrak{p}$, respectively $\mathfrak{p}+\mathfrak{q}$.  Let $\preBF_{[n],\mathfrak{p}}(X,\frak{t})$, respectively $\preBF_{[n],\mathfrak{p}+\mathfrak{q}}(X,\frak{t})$, be the maps defined in (\ref{eq:bauer-furuta}) with respect to the perturbations $\mathfrak{p}$ and $\mathfrak{p}+\mathfrak{q}$.  Then there is a commutative diagram:
	
\[
	\begin{tikzcd}[column sep=large]
	\Sigma^{\mathbb{C}^{m_2}}S_{B_X}^{U_n(\mathfrak{p})\oplus V_n} \arrow[r, "\Sigma^{\mathbb{C}^{m_2}} \preBF_{[n],\frak{p}}"] \arrow[d]& S_{B_X}^{U_n'\oplus V_n'\oplus \mathbb{C}^m\oplus \mathbb{C}^{m_2}}\wedge_{B_Y}I_n \arrow{d}\\
	S_{B_X}^{U_n(\mathfrak{p}+\mathfrak{q})\oplus V_n} \arrow[r,"\preBF_{[n],\mathfrak{p}+\mathfrak{q}}"]& S_{B_X}^{U_n'\oplus V_n'\oplus \mathbb{C}^m\oplus \mathbb{C}^{m_2}}\wedge_{B_Y }I_n\\
	\end{tikzcd}  
\]
	Moreover, a choice of map $L: \mathbb{C}^{m_2}\to \mathcal{E}^+_{X,k}\oplus \mathbb{C}^m$ so that $((D_X, \pi_{P_0} r_Y ) + \mathfrak{p})\circ L =\mathfrak{q}$ determines the vertical arrows in the diagram.
\end{lem}
\begin{proof}
	Such a choice of $L$ as at the end of the statement exists for any such $\frak{p}$, $\frak{q}$, by surjectivity of (\ref{eq:perturbation}).  We show how to define maps as in the commutative diagram in terms of such $L$.  Of course, if $\mathfrak{q}=0$, this is obvious, with $L=0$.   
	
	More generally, we have the following commutative diagram:
		\begin{equation}\label{eq:perturbation-equivalence}
	\begin{tikzcd}  
	\mathcal{E}^+_{X,k}\oplus \mathbb{C}^m\oplus \mathbb{C}^{m_2} \arrow[r] \arrow[d,"\tilde{L}"]
	& \mathcal{E}^-_{X,k-1}\oplus r_Y^*(P_0\cap L^2_{k-\frac{1}{2}}) \arrow[d,"\mathrm{id}"] \\ \mathcal{E}^+_{X,k}\oplus \mathbb{C}^m\oplus \mathbb{C}^{m_2} \arrow[r]& \mathcal{E}^-_{X,k-1}\oplus r_Y^*(P_0\cap L^2_{k-\frac{1}{2}})\end{tikzcd}
	\end{equation}
	where $\tilde{L}$ is the identity on $\mathcal{E}^+_{X,k}\oplus \mathbb{C}^m$, and $L\oplus \operatorname{id}_{\mathbb{C}^{m_2}}$ on $\mathbb{C}^{m_2}$.  The horizontal arrows are $(D_X, \pi_{P_0} r_Y )\oplus \mathfrak{p}\oplus 0$ and $(D_X, \pi_{P_0} r_Y)\oplus \mathfrak{p}\oplus \mathfrak{q}$, respectively.
	
	Comparing with the definition of the Seiberg-Witten map (\ref{eq:4-sw}), we see that there is a commutative diagram analogous to (\ref{eq:perturbation-equivalence}), but with the maps $\widetilde{SW}_{X,n,\frak{p}}$ (and similarly for $\frak{q}$) from (\ref{eq:perturbed-sw-map}) along the horizontal arrows.  
	
	The definition of $\preBF_{[n]}(X,\frak{t})$ then gives the commutative diagram in the Lemma statement.  
	
\end{proof}

As in Section \ref{sec:well-def}, the proof of well-definedness is related to the definition of a families invariant.  Let $\mathcal{F}$ be a family of (metrized, $\mathrm{spin}^c$) $4$-manifolds with boundary, over a base $B$, with fiber $(X,\mathfrak{t})$, and let $\mathcal{G}$ be the boundary family (naturally over the base $B$), where we write $\partial (X,\mathfrak{t})=(Y,\mathfrak{s})$. See Section \ref{subsec:spinc family} for family of $\mathrm{spin}^c$-manifolds.   Assume that we have fixed a sequence of good spectral sections $P_n,Q_n$ on the boundary family.

 Assume also that we have fixed a sequence of good spectral sections $W_{P,n},W_{Q,n}$ of $\ast d$ of the boundary family, and assume $W_{P,0}$ is the orthogonal complement of $W_{Q,0}$. 
 
As at the beginning of the section, we now have bundles $\mathcal{E}_{\mathcal{F},k}^{\pm}$ and $\mathcal{W}_{\mathcal{F},k}$, where the fibers over $b\in B$ (with associated $4$-manifold $(X,\mathfrak{t})$) are:
\[
\begin{split}
\mathcal{E}_{\mathcal{F}, k,b}^{\pm}  
&:=  \mathcal{H}^1(\mathcal{F}_b)  \times_{H^1(X; \mathbb{Z})} L^2_k(\Gamma(\mathbb{S}_{b}^{\pm})),  \\
\mathcal{W}_{\mathcal{F},k,b} &:= \mathrm{Pic}(\mathcal{F}_b) \times  L^2_k( \Omega_{CC}^1(\mathcal{F}_b) ).
\end{split}
\] 
Furthermore, the space of sections $L^2_{k-1}(\Omega^+(\mathcal{F}))$ now defines a bundle over $B$ as well, with fiber $L^2_{k-1}(\Omega^+(\mathcal{F}_b))$, the $L^2_{k-1}$-self-dual $2$-forms on the fiber.

The $4$-dimensional Seiberg-Witten equations (\ref{eq:4-sw}) now define a fiberwise map:
\begin{equation}\label{eq:4-sw-fiberwise}
\begin{split}
&       SW_{\mathcal{F},n}:
\mathcal{E}_{\mathcal{F}, k}^+ \oplus \mathcal{W}_{\mathcal{F},k}
\rightarrow
( \mathcal{E}_{\mathcal{F}, k-1}^- \oplus L^2_{k-1}(\Omega^+(\mathcal{F})) )  \oplus
r_{ \mathcal{G}}^*(P_n  \oplus W_{P,n})
\end{split}
\end{equation}

Define $U_n$ as in (\ref{eq:4-manifold-appx-bundle}), and $V_n$ similarly.  Exactly as before, define $A_n$; note that $A_n$ is now a fiber bundle over the total space of the fibration $\mathrm{Pic}(\mathcal{F})\to B$, a fiber of this latter fibration is $\mathrm{Pic}(\mathcal{F}_b)$.  Define subspaces (themselves spaces over the total space of $\mathrm{Pic}(\mathcal{G})\to B$)) $K_{n,1}(\epsilon)$ and $K_{n,2}(\epsilon)$ with fibers $K_{n,1,b}(\epsilon)$ and $K_{n,2,b}(\epsilon)$ according to:

\begin{align*}
&K_{n, 1,b}(\epsilon)     \\
& :=  \Bigg\{       y \in A_n      : 
\begin{array}{l}
\exists (\hat{\phi}, v, \hat{\omega}) \in B_k(U_n \oplus V_n  ; R'),     
(\hat{\phi}, v) \in U_n  \subset \mathcal{E}_{X, k}^+ \oplus \underline{\mathbb{C}}^{m}, 
\hat{\omega} \in V_n, \\
\|  (SW_{X, n, \frak{p},b}, \operatorname{id}_{\mathbb{C}^m})(\hat{\phi}, v, \hat{\omega})  \|_{k-1} \leq  \epsilon,  
y = \pi_{ P_n \oplus W_{P,n} } r_{\mathcal{G}_b} ( \hat{\phi}, \hat{\omega} )
\end{array}
    \Bigg \},  
\end{align*}
and
\begin{align*}
&K_{n, 2,b}(\epsilon)    \\
& :=  \Bigg\{  y \in A_n   : 
\begin{array}{l}
\exists (\hat{\phi}, v, \hat{\omega}) \in  \partial B_k(U_n \oplus V_n ;R'),    \\
\|   (SW_{X,n, \frak{p},b}, \operatorname{id}_{\underline{\mathbb{C}}^m} ) (\hat{x}, v, \hat{\omega}) \|_{k-1} \leq \epsilon, 
y = \pi_{ P_n \oplus W_{P,n} } r_{ \mathcal{G}_{b}} ( \hat{\phi}, \hat{\omega} )
\end{array}
      \Bigg\}    \\
&  \quad \bigcup
   \big(  \partial A_n \bigcap K_{n, 1,b}(\epsilon)   \big)
\end{align*} 

The proof of Proposition \ref{prop:bauer-furuta-existence} is only changed in this setting according to the procedure in Section \ref{sec:well-def}.  In particular, the following proposition also relies on a families version of Theorem 4 of \cite{Manolescu-b1=0}; the proof thereof is only notationally different from that appearing in \cite{Manolescu-b1=0}.   A families version of Proposition \ref{prop compactness X, Y times R_{geq}} is also used, its proof is a modification of that in \cite[Section 4]{khandhawit}.
We obtain:
\begin{prop}\label{prop:bauer-furuta-existence-families}
	There is an $\epsilon_0 > 0$ such that  if $0 < \epsilon < \epsilon_0$, for $n$ large, we can find a regular fiberwise index pair $(N_n, L_n)$ of $\operatorname{inv}(A_n; \varphi_{n, k, k-\frac{1}{2}})$ with 
	\[
	K_{n, 1}(\epsilon) \subset N_n \subset A_n, \quad
	K_{n, 2}(\epsilon) \subset L_n. 
	\]
\end{prop}

Put 
\begin{align*}
& S^{U_n \oplus V_n}_{\mathrm{Pic}(\mathcal{F})} 
:= \bigcup_{a \in \mathrm{Pic}(\mathcal{F})} B( (U_n \oplus V_n)_a; R) / 
S ( (U_n \oplus V_n )_a; R),   \\
& S_{\mathrm{Pic}(\mathcal{F})}^{ U_n' \oplus V_n' \oplus \underline{\mathbb{C}^m} }
:= \bigcup_{a \in \mathrm{Pic}(\mathcal{F})} B(  (U_n' \oplus V_n' \oplus \underline{\mathbb{C}^m} )_a; \epsilon  ) /
S (  (U_n' \oplus V_n' \oplus \underline{\mathbb{C}^m} )_a; \epsilon  ),
\end{align*} 
Let
\[I_n(\mathcal{G})=N_n\cup_{p_{\mathrm{Pic}(\mathcal{G})\mid_{L_n}}}\mathrm{Pic}(\mathcal{G}),\] 
where $p_{\mathrm{Pic}(\mathcal{G})}$ is the projection to $\mathrm{Pic}(\mathcal{G})$ of $F\times W$.  

We obtain a fiber-preserving map over $\mathrm{Pic}(\mathcal{G})$:
\begin{equation}\label{eq:pre-families-bf}
\preBF_{[n]}(\mathcal{F}): S^{U_n \oplus V_n}_{\mathrm{Pic}(\mathcal{G})}\to S_{\mathrm{Pic}(\mathcal{G})}^{ U_n' \oplus V_n' \oplus \underline{\mathbb{C}^m} }\wedge_{\mathrm{Pic}(\mathcal{G})} I_n(\mathcal{G}).
\end{equation}
Here, $S^{U_n \oplus V_n}_{\mathrm{Pic}(\mathcal{G})}$ and $ S_{\mathrm{Pic}(\mathcal{G})}^{ U_n' \oplus V_n' \oplus \underline{\mathbb{C}^m} }$ are spaces over $\mathrm{Pic}(\mathcal{G})$ by pushing forward $S^{U_n \oplus V_n}_{\mathrm{Pic}(\mathcal{F})}$ and $ S_{\mathrm{Pic}(\mathcal{F})}^{ U_n' \oplus V_n' \oplus \underline{\mathbb{C}^m} }$ along the restriction map $\mathrm{Pic}(\mathcal{F})\to \mathrm{Pic}(\mathcal{G})$ (see Section \ref{subsec:homotopy1}).    

In particular, we obtain that the homotopy class of the map $\preBF_{[n]}(X,\frak{t})$ in (\ref{eq:bauer-furuta}) is independent of the metric on $X$ used in its construction.  To be more precise:

\begin{lem}\label{lem:bf-metric-independence}
	Let $(X,\frak{t})$ be a compact $\mathrm{spin}^c$ $4$-manifold with boundary (admitting a Floer framing) $(Y,\mathfrak{s})$.  Let $g_t$ for $t\in [0,1]$ be a path of metrics on $X$, along with a path of perturbations $\frak{p}_t$ with surjectivity in (\ref{eq:perturbation}) for all $t$.  There exist good spectral sections $P_{n,t},Q_{n,t},W_{P,n,t},W_{Q,n,t}$ on the boundary $Y$, say forming a spectral system $\mathfrak{S}$.  Let $I_n=\mathcal{SWF}_{[n]}(Y,\mathfrak{s},\mathfrak{S})$ denote the family Seiberg-Witten invariant of the boundary.  Let $p$ denote the projection $p: B_Y \times I  \to B_Y$, where $I = [0,1]$.  Then there exists a map
	\[
	\preBF_{[n],I}(X,\mathfrak{t}): S_{B_X\times I}^{U_{n}\oplus V_n}\to S_{B_X\times I}^{U'_n\oplus V_n'\oplus \underline{\mathbb{C}}^m}\wedge_{B_Y \times I}p^* I_n.
	\]
 	The map $\preBF_{[n],I}(X,\mathfrak{t})$ is a map respecting the projection on each side to $B_Y \times I$.  
 	
 	In particular, for a \emph{fixed trivialization} of the families $U_{n,t},V_{n,t},U_{n,t}',V_{n,t}'$ and $I_n$ over $I_+$, together with a path of perturbations $\mathfrak{p}_t$, there is an (equivariant) homotopy equivalence from $\preBF_{[n],0,\frak{p}_0}$ and $\preBF_{[n],1,\frak{p}_1}$ which is well-defined up to (equivariant) homotopy.
\end{lem}

\begin{proof}
	The existence of the spectral sections follows from Section \ref{sec:findim-appx}.  Otherwise the Lemma is a restatement of the definition of the families relative Bauer-Furuta invariant.  There is no issue in choosing a good spectral section for $*d$ of the boundary family in this situation, since on $[0,1]$, each $*d$ may be written as a (small) compact perturbation of $*_g d$, where $g$ is some fixed metric.
	
\end{proof}

Further, the homotopy class of $\preBF_{[n]}(X,\mathfrak{t})$ does not depend on the Sobolev norm used in its construction.  The proof of the following Lemma is analogous to the work in Section \ref{subsec:sobolev}, and is left to the reader.  We state the result for the unparameterized case; the parameterized case is not substantially different.
\begin{lem}\label{lem:sobolev-independence-4}
	Let  $(X,\frak{t})$ be a compact $\mathrm{spin}^c$ $4$-manifold with boundary (admitting a Floer framing) $(Y,\mathfrak{s})$.
	Let $U_n'$ be a sequence of finite dimensional subbundles of $\mathcal{E}^-_{X,k}$ for $k>11/2$, and $V'_n = B_X \times V_{n,0}'$ be a sequence of  finite dimensional subbundles of $B_X \times L^2_{k}(\Omega^+(X))$ where $V_{n,0}' \subset L^2_{k}(\Omega^+(X))$, with $\pi_{U_n'}\to \operatorname{id}_{\mathcal{E}^-_{X,k}}$ and $\pi_{V_n'}\to \operatorname{id}_{B_X \times L^2_{k}(\Omega^+(X))}$ strongly.  Let $\preBF_{[n],k+1}(X)$ and $\preBF_{[n],k}(X)$ be the pre-Bauer-Furuta invariants defined with respect to $L^2_{k+1}$ and $L^2_k$-norm respectively. 
	Write $I$ for the interval $[0,1]$.  Then there is a family of maps over the interval:
	\[
	\preBF_{[n],I}(X,\mathfrak{t}): S^{U_n\oplus V_n}_{B_{X} \times I }\to 
	  S^{U'_n\oplus V_n'\oplus \underline{\mathbb{C}}^m}_{B_{X} \times I } \wedge_{B_{Y} \times I}\preSWF_{[n]}(Y)_{I},
	\]
	where $\preSWF_{[n]}(Y)_{I}$ is the parameterized Conley index coming from the $I$-family of flows used in the proof of Proposition \ref{prop:sobolev-independence-3}.  In particular, for the given homotopy equivalence in Proposition \ref{prop:sobolev-independence-3}, the maps $\preBF_{[n],k}(X,\mathfrak{t})$ and $\preBF_{[n],k+1}(X,\mathfrak{t})$ are homotopic by a homotopy well-defined up to homotopy.
\end{lem}

We next consider the effect of stabilization on $\mathcal{BF}_{[n]}$.  There are two separate stabilizations: increasing $U_n',V_n'$, or increasing $P_n,Q_n,W^{\pm}_n$.  Fix trivializations of $U_{n+1}'/U_n'= \underline{\mathbb{C}}^{c_n}$ and $V_{n+1}'/V_n'= \underline{\mathbb{R}}^{d_n}$.  Recall the definition of a \emph{spectral system} from Definition \ref{def:spectral-system-3}.  By construction, $U_{n+1}$ is naturally identified with $U_n\oplus \underline{\mathbb{C}}^{k_{Q_n}+c_n}$ for $k_{P,n},k_{Q,n}$ as in Theorem \ref{thm change of approximation suspension}, using the isomorphism $\eta: P_{n+1}\to P_n\oplus \underline{\mathbb{C}}^{k_{P,n}}$, similarly for $k_{Q,n}$.  Analogously, $V_{n+1}$ is identified with $V_n \oplus \underline{\mathbb{R}}^{k_{W,-,n}+d_n}$.  Let $\varphi_{n+1,t}$ denote the family of flows as in Theorem \ref{thm change of approximation suspension}, with $n$ chosen large enough.  Recall that there is an induced homotopy equivalence
\[
\Sigma_{B_Y}^{\underline{\mathbb{C}}^{k_{Q,n}}\oplus \underline{\mathbb{R}}^{k_{W,-,n}}}\preSWF_{[n]}(Y)\to \preSWF_{[n+1]}(Y)
\] 
as in Theorem \ref{thm change of approximation suspension}.  

Stabilization of the Bauer-Furuta invariant is as follows.  Let $c_n'=c_n+k_{Q,n}$ and $d_n'=d_n+k_{W,-,n}$.

\begin{prop}\label{prop:bf-stabilization}
	For appropriate choices of index pairs, there is a homotopy-commuting square of parameterized spaces, defined by Conley index continuation maps:

		\begin{equation}\label{eq:commutation-of-bauer-furuta} 
	\begin{tikzpicture}[xscale=6.5,yscale=1.5]
	\node (b0) at (0,-1) {$ 	S_{B_X}^{\underline{\mathbb{C}}^{c_n'}\oplus  \underline{\mathbb{R}}^{d_n'} \oplus T_n' \oplus \underline{\mathbb{C}}^m}  \preSWF_{[n]}(Y)$};
	\node (b1) at (1,-1) {$ 	S_{B_X}^{T_{n+1}' \oplus \underline{\mathbb{C}}^m}\wedge_{B_Y}\preSWF_{[n+1]}(Y)$};
	\node (a0) at (0,0) {$S_{B_X}^{ \underline{\mathbb{C}}^{c_n'}\oplus \underline{\mathbb{R}}^{d_n'}}\wedge_{B_X}S_{B_X}^{T_n}$};
	\node (a1) at (1,0) {$S^{T_{n+1}}_{B_X}$};

	\draw[->] (a0) -- (a1) node[pos=0.5,anchor=north] {}; \draw[->] (a0) -- (b0) node[pos=0.5,anchor=east] {\scriptsize $\operatorname{id}\wedge_{B_X}\preBF_{[n]}$};
	\draw[->] (b0) -- (b1) node[pos=0.5,anchor=south east]{}; \draw[->] (a1) -- (b1) node[pos=0.5,anchor=west] {\scriptsize $\preBF_{[n+1]}$};
	
	\end{tikzpicture}
	\end{equation}
	$T_n = U_n \oplus V_n$, $T_n' = U_n' \oplus V_{n}'$. 
	In particular, (\ref{eq:commutation-of-bauer-furuta}) is a homotopy-commuting square of (unparameterized) connected simple systems.
\end{prop}
\begin{proof}
	The proof is similar to the proof of Theorem \ref{thm change of approximation suspension}, and we will only roughly sketch the details.  Indeed, the bottom arrow of (\ref{eq:commutation-of-bauer-furuta}) is exactly the map defined in that theorem.
	
	Recall that we have fixed identifications $U_{n+1}/U_n= \underline{\mathbb{C}}^{c_n+k_{Q,n}}$
	To obtain that (\ref{eq:commutation-of-bauer-furuta}) homotopy-commutes, we deform $\widetilde{SW}_{X,n+1,\frak{p}}=\widetilde{SW}_{X,n+1,\frak{p},0}$ by a family $\widetilde{SW}_{X,n+1,\frak{p},t}$, by removing (linearly in $t$) the nonlinear terms in $SW_{X,n,\frak{p}}$ on the $U_{n+1}/U_n$ and $V_{n+1}/V_n$-factors to a map $\widetilde{SW}_{X,n+1,\frak{p},1}$ which is the sum of maps
\[	\begin{split}
	H: U_{n+1}/U_n\oplus V_{n+1}/V_n&\to U_{n+1}'/U_n'\oplus V_{n+1}'/V_n'\oplus \underline{\mathbb{C}}^{k_{Q,n}}\oplus \underline{\mathbb{R}}^{k_{W,-,n}}\\
	&\mbox{ and }\\ \widetilde{SW}_{X,n,\frak{p}}: U_n\oplus V_n&\to U_n'\oplus V_n'\oplus r_Y^* (F_n\oplus W_n) \oplus \underline{\mathbb{C}}^m.
	\end{split}
	\]
	Here $H$ is some linear isomorphism (from the linearization of $SW_{X,n}$).

	We define $A_n$ as before, and require that $A_n$ is an isolating neighborhood of the flow $\varphi_{n+1,t}$ for all $t\in [0,1]$.
	
	We then define 
   \begin{align*} 
	&K_{n, 1}(\epsilon)     \\
	& :=  \Bigg\{        (y,t)  \in A_n\times [0,1]      :  \\
	& \qquad \qquad 
	   \begin{array}{l}
	\exists (\hat{\phi}, v, \hat{\omega}) \in B_k(U_n \oplus V_n  ; R'),   
	\\
	\|  (SW_{X, n, \frak{p},t}, \operatorname{id}_{\mathbb{C}^m})(\hat{\phi}, v, \hat{\omega})  \|_{k-1} \leq  \epsilon,  
	y = \pi_{ P_n \oplus W_{P,n} } r_Y ( \hat{\phi}, \hat{\omega} )
	\end{array}
	   \Bigg\},  
    \end{align*} 
	and
	\begin{align*} 
	& K_{n, 2}(\epsilon)     \\
	& :=  \Big \{  (y,t) \in A_n\times [0,1]     :    \\ 
	& \qquad \qquad 
	\begin{array}{l}
	\exists (\hat{\phi}, v, \hat{\omega}) \in  \partial B_k(U_n \oplus V_n ;R'), \\  
	\|   (SW_{X,n, \frak{p},t}, \operatorname{id}_{\underline{\mathbb{C}}^m} ) (\hat{x}, v, \hat{\omega}) \|_{k-1} \leq \epsilon, 
	y = \pi_{ P_n \oplus W_{P,n} } r_Y ( \hat{\phi}, \hat{\omega} )
	\end{array}
         \Bigg\} \\
	& \qquad \bigcup 
	\Big(  ((\partial A_n)\times [0,1]) \bigcap K_{n, 1}(\epsilon) \Big).
	\end{align*}
	  One then establishes the analog of Proposition \ref{prop:bauer-furuta-existence} for the family of flows $\varphi_{n+1,t}$.  	

Writing $I=[0,1]$, there results a map 
\begin{align*}
  \preBF_{[n+1],I}(X,\mathfrak{t})&\from S_{B_X\times I}^{ \underline{\mathbb{C}}^{c_n+k_{Q,n}}\oplus \underline{\mathbb{R}}^{d_n+k_{W,-,n}}}\wedge_{B_Y \times I}S_{B_X\times I}^{U_{n}\oplus V_{n}}\\& \to S_{B_X\times I}^{U_{n+1}'\oplus V_{n+1}'\oplus \underline{\mathbb{C}}^m}\wedge_{B_Y \times I}\preSWF_{[n+1]}(Y).
\end{align*} 
At $t=1$ this is the composite from first going down in (\ref{eq:commutation-of-bauer-furuta}), while for $t=0$, this restricts to $\preBF_{[n+1]}$.  The homotopy commutativity of (\ref{eq:commutation-of-bauer-furuta}) follows.
	
	The claim on the well-definedness of the maps in (\ref{eq:commutation-of-bauer-furuta}) follows from Theorem \ref{thm:fiberwise-deforming-conley-2}.
\end{proof}

\begin{prop}
The map $\preBF_{[n]}$ is independent of the choice of regular index pair $(N_n, L_n)$ with $K_{n,1}(\epsilon) \subset N_n,  K_{n, 2}(\epsilon) \subset L_n$ for $n$ large and $\epsilon$ small, up to isomorphisms in $PSW_{S^1, B}$.  
\end{prop}

\begin{proof}
We will follow the argument in  \cite[Apeendix]{khandhawit}.  Take another regular index pair $(N_n', L_n')$ with $K_{1, n}(\epsilon) \subset N_n',    K_{2,n}(\epsilon) \subset L_n'$ for $n$ large  and $\epsilon$ small.  Let $I_n'$ denote the parameterized Conley index associated to $(N_n', L_n')$. 

First we consider the case when $(N_n, L_n) \subset (N_n', L_n')$.  The map
\[
         \iota_n :  I_n \rightarrow I_n'
\]
induced by the inclusion is an isomorphism in $PSW_{S^1, B}$ by \cite[Theorem 6.2]{MRS} and the following diagram is commutative:
\[
          \xymatrix{
                                          &     S_{B_X}^{ U_n' \oplus V_n' \oplus  \underline{\C}^{m} } \wedge_{B_Y}  I_n  \ar[dd]^{\operatorname{id} \wedge \iota_n}  \\
                   S_{B_X}^{ U_n \oplus V_n  } \ar[ur]^{\preBF_{[n]}}  \ar[dr]_{\preBF_{[n]}'}  &      \\
                            &   S_{B_X}^{ U_n' \oplus V_n' \oplus  \underline{\C}^{m} } \wedge_{B_Y} I_n'
          } 
\]

Next we consider the general case.  As shown in p1653 of \cite{khandhawit}, we have index pairs $(\tilde{N}_n, \tilde{L}_n)$, $(N_{n,1}, L_{n,1})$, $(N_{n,1}', L_{n}')$ such that
  \begin{align*}
     & (N_n, L_n) \subset (N_{n,1}, L_{n,1}), \ (N_n', L_{n}') \subset (N_{n,1}', L_{n,1}'), \\
     &   (K_{n,1}(\epsilon), K_{2,n}(\epsilon)) \subset (\tilde{N}_n, \tilde{L}_n) \subset (N_{n,1}, L_{n, 1}) \cap (N_{n, 1}', L_{n, 1}').
  \end{align*}
We can assume that $(\tilde{N}_n, \tilde{L}_n), (N_{n,1}, L_{n,1}),  (N_{n,1}', L_{n,1}')$ are all regular by thickening the exists slightly (\cite[Remark 5.4]{Salamon}). The statement follows from the  commutative diagram:
\[
    \xymatrix{
      (N_{n,1}, L_{n,1})  &     &  (N_{n, 1}', L_{n,1}')   \\
        & (\tilde{N}_{n}, \tilde{L}_n)  \ar@{_{(}->}[lu]  \ar@{^{(}->}[ru]  &   \\
     (N_n, L_n) \ar@{^{(}->}[uu] &  & (N_{n}', L_{n}') \ar@{_{(}->}[uu]
    }
\]

\end{proof}

Recall that we have defined the virtual bundle $\ind(D_X,P)$ following equation (\ref{eq:4-manifold-appx-bundle}).  For a normal spectral system $\frak{P}$ whose $n$-th section is $P_n$, we write $\ind(D_X,\frak{P})$, since $\ind(D_X,P_{n})$ and $\ind(D_X,P_{n+1})$ are canonically identified for all $n$.  For $V=V_1\ominus V_2$ a virtual vector bundle over a base $B$, we define an element $S^V_B$ of the stable-homotopy category $\fpsw_B$ (see Definition \ref{def:sw-cat}) by $(S^{V_1}_B,-V_2)$ where $S^{V_1}_B$ is the sphere bundle associated to $V_1$; the stable-homotopy-type of this space does not depend on a choice of universe. 

For $V$ a vector bundle over $B$, let $\thom^V_B$ denote the Thom space of $V$; we will abuse notation and also write $\thom^V_B$ for the suspension spectrum of $\thom^V_B$.  Write $\ker(D_X,\frak{P})$ for the kernel of the map in (\ref{eq:perturbation}), which depends on the perturbation $\frak{p}$.

For topological spaces $W,Z$, a \emph{map class} from $W$ to $Z$ will refer to a homotopy class $W \to Z$, up to self-homotopy-equivalence of $W,Z$. 
 We can now prove Theorem \ref{thm:bauer-furuta-main} from the introduction, which we restate as follows:

\begin{cor}\label{cor:bauer-furuta-main-equivalent}
	Fix a Floer framing $\mathfrak{P}$ on $Y$.  There is a well-defined (parameterized, equivariant, stable) map class \[
	\preBF(X,\mathfrak{t}): S^{\ind(D_X,\frak{P})}_{\mathrm{Pic}(X)}\to \preSWF(Y,\mathfrak{P}).
	\]
	For a choice of perturbation $\frak{p}$ as in (\ref{eq:perturbation}), there is a well-defined (equivariant, unparameterized) weak map of spectra:
	\[
	\BF_\frak{p}(X,\mathfrak{t}): \thom_{\mathrm{Pic}(X)}^{\ker(D_X,\frak{P})}\to \Sigma^{\mathbb{C}^m}\SWFn^u(Y,\mathfrak{P}).
	\]    
	Moreover, if $\frak{p}_0$ and $\frak{p}_1$ are related by a family $\frak{p}_t$ of perturbations satisfying (\ref{eq:perturbation}), $\BF_{\frak{p}_0}$ is homotopic to $\BF_{\frak{p}_1}$.  
\end{cor}
\begin{proof}
	The class $\preBF_{\frak{p}}$ is well-defined by Proposition \ref{prop:bf-stabilization}.  Independence (as a map class) from $\frak{p}$ follows from Lemma \ref{lem:perturbation-independence}.
	
	The unparameterized case follows from Proposition \ref{prop:bf-stabilization}, and an argument for families as before.  
\end{proof}

Analogous results hold for the $\mathrm{Pin}(2)$-equivariant versions, \textit{mutatis mutandis}.


\chapter{Fr{\o}yshov type invariants}\label{sec:froyshov}

In this chapter, we will generalize the Fr{\o}yshov type invariants \cite{froyshov}, \cite{Manolescu_Intersection_form} defined for rational homology 3-spheres to 3-manifolds with $b_1 > 0$, making use of the Seiberg-Witten Floer stable homotopy type constructed in this memoir. As applications, we will prove restrictions on the intersection forms of smooth 4-manifolds with boundary. 

It may be of interest to compare the material of this section with work of Levine-Ruberman, where similar invariants are defined in the Heegaard Floer setting \cite{Levine-Ruberman}; see also \cite{Behrens-Golla} for further work in the Heegaard Floer setting.

\section{Equivariant cohomology}
We will recall a basic fact about the $S^1$-equivariant Borel cohomology.  For a pointed $S^1$-CW complex $W$, we let $\tilde{H}^*_{S^1}(W;\mathbb{R})$ be the reduced  $S^1$-equivariant Borel cohomology:
\[
       \tilde{H}_{S^1}^*(W;\mathbb{R}) = \tilde{H}^*(W \wedge_{S^1} ES_+^1; \mathbb{R}), 
\]
where $ES^1_+$ is a union of $ES^1$ and a disjoint base point. 
Note that  $\tilde{H}_{S^1}^*(S^0; \mathbb{R})$ is isomorphic to $\mathbb{R}[T]$ and that  $\tilde{H}^*_{S^1}(W;\mathbb{R})$ is an $\mathbb{R}[T]$-module.   We have the following (See Proposition 1.18.2 of \cite{CW} and  Proposition 2.2 of \cite{Manolescu-triangulation}):

\begin{prop} \label{prop equivariant suspension thm}

Let $V$ be an $S^1$-representation space and $\mathcal{V}$ be the vector bundle
\[
       \mathcal{V} = (W \times ES^1) \times_{S^1} V \rightarrow  W \times_{S^1} ES^1
\]
over $W \times_{S^1} ES^1$. 
The Thom isomorphism for $\mathcal{V}$ induces an $\mathbb{R}[T]$-module isomorphism 
\[
        \tilde{H}_{S^1}^{* + \dim_{\mathbb{R}} V} (\Sigma^{V} W;\mathbb{R}) \cong \tilde{H}_{S^1}^*(W; \mathbb{R}). 
\]
\end{prop}


\section{Fr{\o}yshov type invariant}

  Let $B$ be a compact CW-complex and choose a base point $b_0 \in B$.   We view $B$ as an $S^1$-CW-complex, with the trivial action of $S^1$.  The following definition is an $S^1$-ex-space version of  \cite[Definition 2.7]{Manolescu-triangulation}.  

\begin{dfn}
Let ${\bf U} = (W, r, s)$ be a well-pointed $S^1$-ex-space over $B$ such that $W$ is $S^1$-homotopy equivalent to an $S^1$-CW complex.   We say that ${\bf U}$ is of SWF type at level $t$ if there is an equivalence, as ex-spaces, from $W^{S^1}\to S_B^{\mathbb{R}^t}$, and so that the $S^1$-action on $W \smallsetminus W^{S^1}$ is free. 
\end{dfn}

Note that in the situation above, $W^{S^1}$ inherits the structure of an ex-space, as a subspace of $W$, naturally.  Spaces of SWF type are meant to be the class of spaces that are produced by the Seiberg-Witten Floer homotopy type construction.  Indeed, note that in the case that $B$ is a point, spaces of SWF type over $B$ are exactly spaces of SWF type as in \cite{Manolescu-triangulation}.  For us, $B$ will always be a Picard torus.
	
Moreover, for $\mathbf{U}=\preSWF (Y)$ for some $3$-manifold $Y$ admitting a spectral section (with torsion $\mathrm{spin}^c$ structure and spectral section suppressed from the notation), more is true, in that the fixed point set $W^{S^1}$ is actually fiber-preserving homotopy-equivalent, relative to $s(B)$, to $S^{\mathbb{R}^t}_B$, although for the definition of the Fr{\o}yshov invariant, this is not strictly needed.

\begin{dfn}
Let ${\bf U} = (W, r, s)$ be a well-pointed $S^1$-ex-space of SWF type at level $t$ over $B$.  We denote by $\mathcal{I}_{\Lambda}({\bigu})$  the submodule in $\tilde{H}^*(B_+;\mathbb{R})\otimes \mathbb{R}[[T]]$, viewed as a module over the formal power series ring $\mathbb{R}[[T]]$, generated by the image of the homomorphism induced by the inclusion $\iota : W^{S^1} \hookrightarrow W $: 
\begin{align*}
      \tilde{H}_{S^1}^{* + t}(W/s(B); \mathbb{R}) &\stackrel{\iota^*}{\rightarrow }
      \tilde{H}_{S^1}^{*+t}(W^{S^1}/s(B);\mathbb{R})
      \cong \tilde{H}_{S^1}^{*+t} (S^{\mathbb{R}^{t}}\wedge B_+; \mathbb{R})\\
      &= H^*(B;\mathbb{R})\otimes \mathbb{R}[T]
      \hookrightarrow H^*(B;\mathbb{R})\otimes \mathbb{R}[[T]]. 
      \end{align*}
  We obtain a more specific invariant by considering only $H^0(B;\mathbb{R})$, in the case that $B$ is connected; we impose this condition on $B$ from now on.  Let $\mathcal{I}(\bf{U})$ denote the ideal in $\mathbb{R}[[T]]$ which is the image of 
 
\begin{align*}
   \tilde{H}_{S^1}^{* + t}(W/s(B); \mathbb{R}) 
   &\stackrel{\iota^*}{\rightarrow }
   \tilde{H}_{S^1}^{*+t}(W^{S^1}/s(B);\mathbb{R})  \\
  & \cong \tilde{H}_{S^1}^{*+t} (S^{\mathbb{R}^{t}}\wedge B_+; \mathbb{R})\to \tilde{H}^{*+t}_{S^1}(S^{\mathbb{R}^t};\mathbb{R})=\mathbb{R}[T]\hookrightarrow \mathbb{R}[[T]]
 \end{align*}
  obtained using the inclusion of a fiber $S^{\mathbb{R}^t}\to S^{\mathbb{R}^t}\wedge B_+$.
  
Then there is a non-negative integer $h$ such that $\mathcal{I}({\bf U}) = (T^{h})$.  Here $(T^{h})$ is the ideal generated by $T^h$.  We denote this integer by $h({\bf U})$.   
\end{dfn}

The invariant $h({\bf U})$ defined above is most similar to $d_{\mathrm{bot}}$ as in \cite{Levine-Ruberman}; while $\mathcal{I}_{\Lambda}(\bf{U})$ is, roughly, in line with the collection of their ``intermediate invariants".  

\begin{rem}
We also note that the cohomology group $\tilde{H}^{*}_{S^1}(W/s(B);\mathbb{R})$
admits an action by $H^*(B)$
  as follows.  Using the projection map $r: W\to B$, we have an algebra morphism $r^*: H^*(B;\mathbb{R})\to H^*(W;\mathbb{R})$.  The Mayer-Vietoris sequence for $(B,W)$ splits because of the map $s: B\to W$, and we obtain
\[
H^*(W;\mathbb{R})=H^*(W/s(B);\mathbb{R})\oplus H^*(B;\mathbb{R}),
\]
and in fact this splitting is at the level of $H^*(B;\mathbb{R})$-modules, so that the cohomology group $H^*(W/s(B);\mathbb{R})$ inherits a $H^*(B;\mathbb{R})$-action.  This is not strictly necessary in the definition of invariants from $\mathcal{I}_{\Lambda}(\bf{U})$ above, but is indicative of the structure of $\mathcal{I}_\Lambda(\bf{U})$.
\end{rem}

From Proposition \ref{prop equivariant suspension thm},  we can see the following:

\begin{lem}  \label{lem suspension h}
Let ${\bf U} = (W, r, s)$ be a well-pointed $S^1$-ex-space of SWF type over $B$. If $V$ is a real vector space, we have
\[
    h(\Sigma_B^{V} {\bf U}) = h({\bf U}). 
\]
If $V$ is a complex vector space, we have
\[
      h(\Sigma_B^{V}  {\bf U}) = h({\bf U}) + \dim_{\mathbb{C}} V. 
\]
\end{lem}

\begin{prop}  \label{prop h W_0 W_1}
Let ${\bf U}_0 = (W_0, r_0, s_0)$, ${\bf U}_1 = (W_1, r_1, s_1)$ be well-pointed $S^1$-ex-spaces of SWF type at level $t$ over $B_0$ and $B_1$, and assume given a map $\rho: B_0 \to B_1$.  Let $\rho_!{\bf U}_0$ denote the pushforward of $\bf{U}_0$, as an ex-space over $B_1$.  Assume that there is a fiberwise-deforming $S^1$-map
\[
     f :  \rho_!\bf{U}_0  \rightarrow \bf{U}_1
\]
such that the restriction to 
\[
	f^{S^1}: \rho_! W_0^{S^1}\to W_1^{S^1},
\]
as a fiberwise-deforming morphism over $B_1$, is homotopy-equivalent to 
\[
\mathrm{Id}\wedge \rho:  (\mathbb{R}^t)^+\times B_0\cup_{B_0}B_1 \to (\mathbb{R}^t)^+\times B_1.
\]
Then
\[
       h({\bf U}_0) \leq h({\bf U}_1).  
\]

\end{prop}

As a special case, if $B_0$ is a point, the hypothesis is that the map $f$, restricted to fixed point sets, $f^{S^1}: W_0^{S^1}\to W_1^{S^1}/s(W_1)$, be homotopic to the inclusion of a fiber.
\begin{proof}
We have the following diagram:
\[
    \xymatrix{
          \tilde{H}^{* + t}(W_0/s(B_0);\mathbb{R})  \ar[d]_{}  & & \tilde{H}^{* + t}(W_1/s(B_1); \mathbb{R}) \ar[d]^{} \ar[ll]_{f^*} \\ 
          \tilde{H}^{* + t}((\mathbb{R}^t)^+\times B_0/s(B_0);\mathbb{R})  \ar[d]  & & \tilde{H}^{* + t}((\mathbb{R}^t)^+\times B_1/s(B_1); \mathbb{R}) \ar[d] \ar[ll]_{\rho^*} \\
     \tilde{H}^{* + t} ((\mathbb{R}^{t})^+;\mathbb{R}) = \mathbb{R}[T] \ar[dr] & &  \tilde{H}^{*+t}(  (\mathbb{R}^t)^{+}; \mathbb{R}) = \mathbb{R}[T]  \ar[ll]_{f^*}^{\cong}   \ar[ld]     \\
           &  \mathbb{R}[[T]]   
    } 
\]
From this diagram, we obtain
\[
      (T^{h({\bf U}_0)}) \supset (T^{h({\bf U}_1)}),
\]
which implies that $h({\bf U}_0) \leq h({\bf U}_1)$. 
\end{proof}

\begin{dfn}
For   $m, n \in \mathbb{Z}$ and $S^1$-ex-space ${\bf U}$ of SWF type over $B$, we define
\[
      h(\Sigma_{B}^{\mathbb{R}^m  \oplus \mathbb{C}^n }  {\bf U}) = h({\bf U}) + n. 
\]
\end{dfn}

Note that this definition is compatible with Lemma \ref{lem suspension h}.

\begin{dfn}
For  $m_0, n_0, m_0, n_1 \in \mathbb{Z}$ and $S^1$-ex-spaces ${\bf U}_0, {\bf U}_1$  of SWF type over $B$,   we say that $\Sigma_{B}^{\mathbb{R}^{m_0} \oplus \mathbb{C}^{n_0} } {\bf U}_0$ and $\Sigma_{B}^{ \mathbb{R}^{m_1} \oplus \mathbb{C}^{n_1}  } {\bf U}_1$ are locally equivalent if  there is $N \in \mathbb{Z}_{\geq 0}$  with $N + m_0, N + n_0, N + m_1, N + n_1 \geq 0$   and fiberwise-deforming maps 
\begin{align*} 
        &  f :   \Sigma_B^{\mathbb{R}^{N + m_0} \oplus \mathbb{C}^{N+n_0}} {\bf U}_0 \rightarrow \Sigma_B^{ \mathbb{R}^{N + m_1}  \oplus  \mathbb{C}^{N+n_1}}{\bf U}_1, \\
       &  g : \Sigma_B^{ \mathbb{R}^{N + m_1} \oplus \mathbb{C}^{N+n_1}} {\bf U}_1 \rightarrow \Sigma_B^{\mathbb{R}^{N + m_0} \oplus \mathbb{C}^{N+n_0}}  {\bf U}_0
   \end{align*}
such that the restrictions 
\begin{align*}
   & f^{S^1} :   \Sigma_B^{\mathbb{R}^{N + m_0} } ({\bf U}_0)^{S^1} \rightarrow  \Sigma_B^{\mathbb{R}^{N + m_1} }({\bf U}_1)^{S^1},   \\
   & g^{S^1} : \Sigma_B^{\mathbb{R}^{N + m_1} } ({\bf U}_1)^{S^1} \rightarrow \Sigma_B^{\mathbb{R}^{N + m_0} } ({\bf U}_0)^{S^1}
\end{align*}
are   homotopy equivalent to
\[
      \mathrm{Id} : B \times (\mathbb{R}^{t}) \rightarrow B \times (\mathbb{R}^{t})^+
\]
as fiberwise-deforming morphisms over $B$. 
\end{dfn}

It is easy to see that the local equivalence is an equivalence relation.

\begin{cor} \label{cor local eq h}
If $\Sigma_{B}^{ \mathbb{R}^{m_0} \oplus  \mathbb{C}^{n_0} } {\bf U}_0$ and $\Sigma_{B}^{  \mathbb{R}^{m_1}  \oplus  \mathbb{C}^{n_1}} {\bf U}_1$ are locally equivalent, 
\[
          h(\Sigma_{B}^{  \mathbb{R}^{m_0} \oplus  \mathbb{C}^{n_0}}{\bf U}_0) 
          = h(\Sigma_B^{ \mathbb{R}^{m_1} \oplus   \mathbb{C}^{n_1} }{\bf U}_1). 
\]
\end{cor}

\begin{proof}
This is a direct consequence of Proposition \ref{prop h W_0 W_1}. 
\end{proof}

Let $Y$ be a closed $3$-manifold, $g$ be a Riemannian metric, $\frak{s}$  be a torsion $\mathrm{spin}^c$ structure on $Y$. Let $B_Y$ be the Picard torus $\mathrm{Pic}(Y)$ of $Y$.  Assume that $\operatorname{ind} D_Y = 0$ in $K^1(B_Y)$.  
We take  a spectral system
\[
      \mathfrak{S}=(\mathbf{P},\mathbf{Q},\mathbf{W}_P,\mathbf{W}_Q,\{\eta^P_n\}_n,\{\eta_n^Q\},\{\eta^{W_P}_n\}_n,\{\eta^{W_Q}_n\}_n)
\]
for $Y$. See Definition \ref{def:spectral-system-3}. 
Put
\[
     F_n = P_n  \cap Q_n, \quad
     W_n = W_{P, n} \cap W_{Q, n}
\]
as before.  Take half integers $k_+, k_-$ with $k_+, k_- > 5$ and with $| k_+ - k_- | \leq \frac{1}{2}$. 
We have the approximate Seiberg-Witten flow
\[ 
         \varphi_{n} = \varphi_{n, k_+, k_-} : (F_n \oplus W_n) \times \mathbb{R} \rightarrow F_n \oplus W_n.
\] 
Put
\[
     A_n = ( B_{k_+}(F_n^+ ; R) \times_{B_Y} B_{k_-}(F_n^-; R) ) \times_{B_Y} (  B_{k_+}(W_n^+;R) \times_{B_Y} B_{k_-}(W_n^-; R) )
\]
for $R \gg 0$. 
Recall that $A_n$ is an isolating neighborhood for $n \gg 0$   (Theorem \ref{thm isolating nbd}).

\begin{lem} \label{lem I_n S^1}
Let ${\bf U}_n = (I_n, r_n, s_n)$ be the $S^1$-equivariant Conley index for the isolated invariant set $\operatorname{inv}(A_n, \varphi_n)$ for $n \gg 0$. 
Then ${\bf U}_n$ is of SWF type at level $\rank_{\mathbb{R}} W_n^{-}$. 

\end{lem}

\begin{proof}
We first note that $I_n$ is of the homotopy-type of a $S^1$-CW complex by Proposition \ref{prop:finiteness}.  The $S^1$-fixed point set $(I_n, r_n, s_n)^{S^1}$ is the Conley index for  
\[
   \operatorname{inv}( \varphi_n|_{W_n},    B_{k_+}(W_n^+; R) \times_{B_Y} B_{k_-}(W_n^-;R)).
\]
 Note that if $\phi = 0$, the quadratic terms $c_1(\gamma)$, $c_2(\gamma)$, $X_H(\phi)$ are all zero.  See (\ref{eq quadratic terms}). Hence  the restriction of the flow $\varphi_n$ to $W_n$ is the flow induced by the linear map $- \! * \! d |_{W_n}$.   In particular, the flow $\varphi_{n}|_{W_n}$ preserves each fiber of the trivial bundle $W_n = B_Y \times L^2_k(\im d^*)_{\lambda_{n}}^{\mu_{n}}$ over $B_Y$.  Hence there is an equivalence, as ex-spaces, $(I_n)^{S^1} \cong S_B^{W_n^-}$.  (In fact, more is true; there is a fiber preserving homotopy equivalence $(I_n)^{S^1} \cong S_B^{W_n^-}$.) 
\end{proof}

    Let $\mathcal{SWF}(Y, \frak{s}, [\frak{S}])$ be the Seiberg-Witten Floer parameterized homotopy type (Definition \ref{def: SWF parameterized homotopy type}).   
    
    Recall that $\eta_n^{P}$, $\eta_{n}^{Q}$, $\eta_{n}^{W_{P}}$, $\eta_{n}^{W_{Q}}$ are isomorphisms 
  \begin{align*}
       & P_{n+1}  \stackrel{\cong}{\rightarrow} P_n \oplus \C^{k_P, n},   \\
       & Q_{n+1} \stackrel{\cong}{\rightarrow} Q_n \oplus \C^{k_Q, n},  \\
       & W_{n+1}^{P} \stackrel{\cong}{\rightarrow } W_{n}^{+} \oplus \R^{k_{W,+,n}},  \\
       & W_{n+1}^Q \stackrel{\cong}{ \rightarrow } W_{n}^{-} \oplus \R^{k_{W, -, n}}.
  \end{align*}
These induce   an $S^1$-equivariant homotopy equivalence
\[ 
           I(\varphi_{n+1}) \cong  \Sigma_B^{\C^{k_{Q, n}}   \oplus \R^{k_{W, -, n}}   } I(\varphi_{n})
\]
for $n \gg 0$, whose restriction to the $S^1$-fixed point set is a fiber-preserving homotopy equivalence.  See Theorem \ref{thm change of approximation suspension}. This implies that the number
\[
       h (\mathcal{SWF}(Y, \mathfrak{s}, [\mathfrak{S}])) = 
       h(I(\varphi_n))  -  D_n^2
\]
is independent of the choice of $n \gg 0$ by Lemma \ref{lem suspension h} and Corollary \ref{cor local eq h}. Here $D_n^2 = \dim  (Q_n - Q_0)$.  

Also  it follows from Proposition \ref{prop:sobolev-independence-3}  that $h(\mathcal{SWF}(Y, \mathfrak{s}, [\mathfrak{S}]))$ is independent of $k_{\pm}$. 
Hence  $h(\mathcal{SWF}(Y, \frak{s}, [\frak{S}]))$ is well-defined.

We will introduce another number. 
We can take a $\mathrm{spin}^c$ 4-manifold $(X, \frak{t})$ with boundary $(Y, \frak{s})$.  Since $c_1(\frak{t})|_{Y}$ is torsion in $H^2(Y;\mathbb{Z})$, there is a positive integer $m$ such that 
\[
        mc_1(\frak{t}) \in H^2(X, Y; \mathbb{Z}). 
\] 
Put
\[
     c_1(\frak{t})^2 := \frac{1}{m} \langle ( m c_1(\frak{t}) ) \cup c_1(\frak{t}), [X] \rangle 
     \in \mathbb{Q}, 
\]
where $\langle \cdot, \cdot \rangle$ is the pairing 
\[
  H^4(X, Y ; \mathbb{Z}) \otimes H_4(X;\mathbb{Z}) \rightarrow \mathbb{Z}. 
\]
We define 
\begin{equation}\label{eq:spectral-index}
\begin{aligned}
      n(Y, g, \frak{s}, P_0) 
       &  :=\dim \operatorname{ind} (D_X, P_0) - \frac{c_1(\frak{t})^2- \sigma(X)}{8}  \in \mathbb{Q}  \\
       & = \frac{1}{2} \eta_{D, P_0} - \frac{1}{8}  \eta_{Y, sign}   . 
\end{aligned}
\end{equation}
Here $D_X$ is the Dirac operator on $X$, $\ind (D, P_0)$ is the index defined in Proposition \ref{prop:Ind(D,P)} and  $\eta_{D, P_0}, \eta_{Y, sing}$ are the $\eta$-invariants of the Dirac operator and signature operator. We have used the index formula \cite{APS1}, \cite{MP}.  See also \cite[Section 6]{Manolescu-b1=0}.

\begin{dfn}\label{def:froyshov}
We define $h(Y, \frak{s}) \in \mathbb{Q}$  by
\[
       h(Y, \frak{s}) := h( \mathcal{SWF}(Y, \frak{s}, [\frak{S}])  ) - n(Y,  g, \frak{s}, P_0). 
\]
\end{dfn}
A priori, the expression in Definition \ref{def:froyshov} may depend on both the metric and the spectral system.  However, for two spectral systems $\mathfrak{S}_0,\mathfrak{S}_1$ with $\dim \ind (D_X,P^0_0)=\dim \ind (D_X,P^1_0)$, we see that the $h$-invariants agree, since $\mathcal{SWF}(Y,\frak{s},[\frak{S}_0])$ differs from $\mathcal{SWF}(Y,\frak{s},[\frak{S}_1])$ by suspension by a virtual complex vector bundle of formal dimension zero.  In order to see this, we first note that $S^1$-equivariant Borel cohomology is an $S^1$-equivariant complex orientable cohomology theory by \cite{Cole-Greenlees-Kriz}, so that for an $S^1$-equivariant complex vector bundle $V$  over $B$ and  a $S^1$-ex-space $(X, r, s)$ over $B$, there is a canonical isomorphism
\[
   \tilde{H}_{S^1}^{*+ 2\rank_{\C} V} ( \nu_! \Sigma_B^{V} X) \cong \tilde{H}_{S^1}^{*+ 2 \rank_{\C} V}(\mathrm{Th}(r^*V))\cong \tilde{H}_{S^1}^*(X). 
\]
Here $\nu : B \rightarrow *$ and we have used (\ref{eq:tautology}).  This implies that
\[
     h(\Sigma_{B}^{V} X) = h(X) + 2\rank_{\C} V.
\]
It follows in particular that: 
\[
h(\mathcal{SWF}(Y,\frak{s},[\frak{S}_0]))=h(\mathcal{SWF}(Y,\frak{s},[\frak{S}_1])).
\]

Changes in the metric and changes in $\dim \ind(D_X,P_0)$ are treated in a similar way, so we only address the latter.  Indeed, if we replace $\frak{S}_0$ with a spectral system $\frak{S}_1$ so that the $K$-theory class is 
\[
[\frak{S}_1-\frak{S}_0]=\underline{\mathbb{C}}\in K(B_Y),
\]
then 
\[h(\mathcal{SWF}(Y,\frak{s},[\frak{S}_1]))=h(\mathcal{SWF}(Y,\frak{s},[\frak{S}_0]))-1,\]
but $n(Y,g,\frak{s},P^1_0)=n(Y,g,\frak{s},P^0_0)-1$, as needed. 

Finally, in the case that $b_1(Y)=0$, this agrees (by definition) with the $\delta$-invariant defined in \cite{Manolescu-triangulation}.

In particular, it is natural to consider the parameterized equivariant homotopy type of the formal desuspension:
\[
\Sigma^{-n(Y,g,\frak{s},P_0)\mathbb{C}}_{B_Y}\mathcal{SWF}(Y,\frak{s},[\frak{S}]),
\]
which one can think of as a desuspension so that the grading of a reducible element of $\mathcal{SWF}(Y,\frak{s},[\frak{S}])$ has been specified.  We note that $n(S^1\times S^2,g,\frak{s},P_0)=0$, where $g$ is the product metric on $S^1\times S^2$, $\frak{s}$ is the torsion $\mathrm{spin}^c$ structure, and $P_0$ is the standard spectral section (since the Dirac operator has trivial kernel for each flat connection, this is specified).  That is, with our conventions, the grading of each reducible in \[\Pic(S^1\times S^2)\simeq \mathcal{SWF}(Y,\frak{s},[\frak{S}])\]
is zero.  This differs from the convention in Heegaard-Floer homology, for which each reducible should be $-\frac{1}{2}$-graded, as in \cite{Ozsvath-Szabo-absolute-gradings}.

We will prove a generalization of \cite[Theorem 4]{froyshov}.

\begin{thm} \label{thm negative def 4-mfd}
Let $Y_0$ be a rational homology $3$-sphere and  $Y_1$ be a closed, oriented $3$-manifold such that the triple-cup prodcut
\[
   \begin{array}{ccc}
           \Lambda^3 H^1(Y_1; \mathbb{\Z}) &  \rightarrow & \mathbb{Z}  \\
             \alpha_1 \wedge \alpha_2 \wedge \alpha_3 & \mapsto & \langle \alpha_1 \cup \alpha_2 \cup \alpha_3, [Y_1] \rangle
   \end{array}
\]
is zero. 
Let $(X, \frak{t})$ be a compact, $\mathrm{spin}^c$ negative semi-definite 4-manifold with boundary $-Y_0 \coprod Y_1$ such that $c_1(\frak{t})|_{\partial X}$ is torsion.  Then we have
\[
                    \frac{c_1(\frak{t})^2 + b_2^-(X)}{8} +   h(Y_0, \frak{t}|_{Y_0}) \leq h(Y_1, \frak{t}|_{Y_1}).
\]

\end{thm}

\begin{proof}
Since the triple-cup product is zero, we have $\operatorname{ind} D_{Y_1} = 0$ in $K^1(B_{Y_1})$ by the index formula. (See \cite[Proposition 6]{kronheimer-manolescu}.) 
 Note that the map $\mathcal{BF}_{[n]}(X, \frak{t})$ constructed in Chapter \ref{section 4-mfd with boundary} is a fiber-preserving map.  We  consider the restriction of  $\mathcal{BF}_{[n]}(X, \frak{t})$ to the fiber over a point $[0] \in  B_X$. The restriction $\mathcal{BF}_{[n]}(X, \mathfrak{t})$ to the fiber and   the duality map
\[
        I_n(Y_0) \wedge I_n(-Y_0) \rightarrow  S^{F_n(Y_0) \oplus W_n(Y_0)}, 
\]
defined in \cite[Section 2.5]{manolescu_gluing}, induce an $S^1$-map
\[
      f_n :  \Sigma^{\mathbb{R}^{m_0} \oplus \mathbb{C}^{n_0+a} }  I_n(Y_0)    \rightarrow 
                 \Sigma^{\mathbb{R}^{m_1} \oplus \mathbb{C}^{n_1}}    ( I_n(Y_1)/s_n(B_{Y_1}))
\]
for $n \gg 0$,  where
    \begin{align*}
       & m_0 - m_1 
         =  \rank_{\mathbb{R}} W_{n}(Y_1)^{-} - \dim_{\mathbb{R}} W_{n}(Y_0)^{-}  \\
       & n_0 - n_1 
       =  \rank_{\mathbb{C}} F_n(Y_1)^{-} - \dim_{\mathbb{C}} F_{n}(Y_0)^-, \\
       & a = \dim \operatorname{Ind} D_{X, P_0} = \frac{c_1(\frak{t})^2 + b_2^-(X)}{8} + n(Y_1, g|_{Y_1}, \frak{t}|_{Y_1}, P_0) -  n(Y_0, g|_{Y_0},  \frak{t}|_{Y_0}). 
     \end{align*}
The restriction of $f_n$ to the $S^1$-fixed point set $\Sigma^{\mathbb{R}^{m_0} } (I_n(Y_0))^{S^1}$ is induced by the operator 
\[
     D'=(d^+,  \pi_{-\infty}^{0} r_{-Y_0},  \pi_{-\infty}^{0} r_{Y_1})  : \Omega_{CC}^{1}(X) \rightarrow 
       \Omega^+(X) \oplus  (\mathcal{W}_{-Y_0})_{-\infty}^{0} \oplus (\mathcal{W}_{Y_1})_{-\infty}^{0}.
\]
The operator $D'$ is an isomorphism.  Therefore the restriction
\[
     f_n^{S^1} : \Sigma^{\mathbb{R}^{m_0} } (I_n(Y_0))^{S^1} 
     \rightarrow 
     \Sigma^{\mathbb{R}^{m_1}} (I_n(Y_1))_{[0]}^{S^1}. 
\]
is a homotopy equivalence.
Here $[0] \in B_{Y_1}$  is the restriction of  $[0] \in B_X$ to $Y$ and $(I_n(Y_1))_{[0]}^{S^1}$ is the fiber over $[0]$. 

By Lemma \ref{lem suspension h} and  Proposition \ref{prop h W_0 W_1}, we have
\[
          \frac{c_1(\frak{t})^2 + b_2^-(X)}{8} + h(Y_0, \mathfrak{t}|_{Y_0}) \leq  h(Y_1,  \frak{t}|_{Y_1}). 
\]

\end{proof}

\begin{rem} 
	There is an apparent discrepancy with the statement of \cite[Theorem 4.7]{Levine-Ruberman}.  We note that in the translation between these statements, we expect $h(Y,\frak{s})$ to correspond to $\frac{d_{\mathrm{bot}}(Y,\frak{s})}{2}+\frac{b_1(Y)}{4}$, due to the difference in the grading conventions on the reducible; with this observation, the statements are consistent.  
\end{rem}

\begin{rem}  \label{rem:duality}
In order to  generalize Theorem \ref{thm negative def 4-mfd} to the case $b_1(Y_0) > 0$,  we need to establish the duality for the Seiberg-Witten Floer parameterized homotopy types  $\preSWF(Y_0,  \mathfrak{t}|_{Y_0}, [\mathfrak{S}])$ and $\preSWF(-Y_0, \mathfrak{t}|_{Y_0}, [\mathfrak{S}^{\vee}_0])$  to get  the parameterized  Bauer-Furuta map  
\[
    \mathcal{SWF}(Y_0, \mathfrak{t}|_{Y_0}, [\mathfrak{S}_0]) \to \mathcal{SWF}(Y_1, \mathfrak{t}|_{Y_1}, [\mathfrak{S}_1]).  
\]    
We do not discuss it in this memoir. See  Proposition \ref{prop:orientation-reversal}.

\end{rem}

\begin{cor}
Let $Y$ be a closed, connected,  oriented $3$-manifold such that 
the triple-cup product is zero. 
Let $(X,  \frak{t})$ be a compact, negative semi-definite,  $\mathrm{spin}^c$ $4$-manifold with $\partial X = Y$ such that $c_1(\frak{t})|_{Y}$ is torsion. Then we have
\[
       \frac{c_1(\frak{t})^2  + b_2^-(X)}{8}  \leq h(Y,  \frak{t}|_{Y}). 
\]
\end{cor}

\begin{proof}
Removing a small ball from $X$, we get a compact $\mathrm{spin}^c$ 4-manifold $X'$ with boundary $S^3 \coprod Y$. Applying Theorem \ref{thm negative def 4-mfd} to $X'$, we get the inequality. 
\end{proof}

\begin{ex} \label{ex flat 3-mfd h inv}
Let $T^2$ be a torus $(\mathbb{R} / \mathbb{Z})  \times ( \mathbb{R} / \mathbb{Z})$. 
Put
\[
        Y := \mathbb{R} \times  T^2 / (x, \theta_1, \theta_2 ) \sim (x+1, -\theta_1, - \theta_2). 
\]
Then $Y$ is a flat $T^2$ bundle over $S^1$, which has a flat metric and $b_1(Y) = 1$.  We have
\[
       H^2(Y;\mathbb{Z}) \cong H_1(Y; \mathbb{Z}) 
       \cong \mathbb{Z} \oplus  ( \mathbb{Z} / 2 \mathbb{Z}) \oplus ( \mathbb{Z} / 2 \mathbb{Z}). 
\]
There are four $\mathrm{spin}^c$ structures $\frak{s}_0, \dots, \frak{s}_3$. Let $\frak{s}_0$ be the $\mathrm{spin}^c$ structure corresponding to the $2$-plane field tangent to the fibers.    As stated in Example \ref{ex flat 3-mfd}, for $j = 1, 2, 3$,   $(Y, \frak{s}_j)$ satisfies the conditions of Theorem \ref{thm reducible I_n}.
We have
\[
     \mathcal{SWF}(Y,\frak{s}, [\frak{S}]) \cong S^0_{B_Y}. 
\]
Here $\frak{S}$ is a spectral system with $P_0 = \mathcal{E}_0(D)_{-\infty}^{0}$. 
As stated in p.2112 of \cite{KLS1}, 
\[
      n(Y, \frak{s}_j, g, P_0) =  0
\]
for $j = 1, 2, 3$. 
Therefore we obtain
\[
       h(Y, \frak{s}_j) =  h( \mathcal{SWF}(Y, \frak{s}, [\frak{S}])  )  - n(Y, \frak{s}_j, g, P_0)  = 0. 
\] 
\end{ex}

\begin{ex}  \label{ex sphere bundle}
Let $\Sigma$ be a closed, oriented surface with $g(\Sigma) > 0$ and $Y$ be the sphere bundle of the complex line bundle over $\Sigma$ of degree $d$.  Suppose that $0 < g  <  d$, where $g := g(\Sigma)$.   Let $\fs_q$ be the spin$^c$ structure in Proposition \ref{prop critical points of CSD_r}.   For $q \in \{ g, g+1, \dots, d-1 \}$,  we have
\[ 
          \mathcal{SWF}(Y, \frak{s}_q, [\frak{S}]) \cong  S_B^0
\]
by Theorem \ref{thm I_n sphere bundle}. Here $\frak{S}$ is a spectral system with $P_0 = \mathcal{E}_0(D_r)_{-\infty}^{0}$. 
The value of $n(Y, g_r, \frak{s}_q, P_0)$ was computed in \cite[Section 8.2]{KLS1} and  we have
\begin{equation} \label{eq:n(Y,s_q)}
        n(Y, g_r, \frak{s}_q, P_0) = - \frac{d-1}{8} -  \frac{(g-1-q)(d+g-1-q)}{2d}. 
\end{equation}
(Note that the definition of $n(Y, g, \frak{s}_q, P_0)$ of this memoir is $-1$ times that of \cite{KLS1}.)
Hence
 \begin{align*}
      h(Y, \frak{s}_q, g) 
     & =   h(\mathcal{SWF}(Y, \frak{s}_q, [\frak{S}]))  - n(Y, g, \frak{s}_q, P_0)  \\
    &  = \frac{d-1}{8} + \frac{(g-1-q)(d+g-1-q)}{2d}.
      \end{align*}
\end{ex}


\section{K-theoretic Fr{\o}yshov invariant} 

 In analogy to the previous section on the (homological) Fr{\o}yshov invariant, we now generalize the invariant $\kappa(Y)$ constructed in \cite{Manolescu_Intersection_form}.  For details on $\mathrm{Pin}(2)$-equivariant complex $K$-theory, we refer to \cite{Manolescu_Intersection_form}.

Let $\tilde{\mathbb{R}}$ be the non-trivial real   representation of ${\rm{Pin}}(2) = S^1 \coprod j S^1$.  Let $B$ be a compact, connected ${\rm Pin}(2)$-CW complex with a  $\mathrm{Pin}(2)$-fixed marked (though we do not consider $B$ itself to be an object in the category of pointed spaces) point $b_0 \in B^{{\rm Pin}(2)}$, such that  the $S^1$-action on $B$ is trivial and the action of $j$ is an involution.

\begin{dfn}
Let ${\bf U} = (W, r, s)$ be a well-pointed ${\rm Pin}(2)$-ex-space over $B$ such that $W$ is ${\rm Pin}(2)$-homotopy equivalent to a ${\rm Pin}(2)$-CW complex. We say that ${\bf U}$ is of SWF type at level $t$ if there is a ex-space ${\rm Pin}(2)$-homotopy equivalence  from  $W^{S^1}$ to $S_B^{\tilde{\mathbb{R}}^{t}}$  and if the $\mathrm{Pin}(2)$-action on $W \smallsetminus W^{S^1}$ is free. 

\end{dfn}
As before, in fact for us there is the stronger condition that there is a fiber-preserving (equivariant) homotopy equivalence $W^{S^1}\to S^{\tilde{\mathbb{R}}^t}_B$.

Let $R({\rm Pin}(2))$ be the representation ring of ${\rm Pin}(2)$. That is,
\[
      R({\rm Pin}(2)) \cong \mathbb{Z}[z, w]/(w^2-2w, zw - 2w), 
\]
where 
\[ 
        w =  1 - [\tilde{\mathbb{C}}], \quad
        z = 2 - [\mathbb{H}]. 
\]

We will generalize  \cite[Definition3]{Manolescu_Intersection_form} to ${\rm Pin}(2)$-ex-spaces:

\begin{dfn}
	Let ${\bf U} = (W, r, s)$ be a well-pointed $\mathrm{Pin}(2)$-ex-space of SWF type at level $2t$ over $B$ so that $W$ is $\mathrm{Pin}(2)$-homotopy equivalent to a $\mathrm{Pin}(2)$-CW complex.  We denote by $\mathcal{I}_{\Lambda}({\bigu})$  the submodule in $K_{\mathbb{Z}/2}(B)$, viewed as a module over $R(\mathrm{Pin}(2))$, generated by the image of the homomorphism induced by the inclusion $\iota : W^{S^1} \hookrightarrow W $: 
	\begin{align*}
	\tilde{K}_{\mathrm{Pin}(2)}(W/s(B)) &\stackrel{\iota^*}{\rightarrow }
	\tilde{K}_{\mathrm{Pin}(2)}(W^{S^1}/s(B))
	\cong \tilde{K}_{\mathrm{Pin}(2)} (S^{\tilde{\mathbb{C}}^{t}}\wedge B_+)\\
	&= K_{\mathbb{Z}/2}(B)
	\end{align*}
	We obtain a more specific invariant by considering only a single fiber.
	Let $\mathcal{I}(\bf{U})$ denote the ideal in $R(\mathrm{Pin}(2))$ which is the image of 
	\begin{align*}
	\tilde{K}_{\mathrm{Pin}(2)}(W/s(B))    &\stackrel{\iota^*}{\rightarrow }
	\tilde{K}_{\mathrm{Pin}(2)}(W^{S^1}/s(B))     \\
	 & \cong \tilde{K}_{\mathrm{Pin}(2)} (S^{\tilde{\mathbb{C}}^{t}}\wedge B_+)\to \tilde{K}_{\mathrm{Pin}(2)}(S^{\tilde{\mathbb{C}}^t};\mathbb{R})=R(\mathrm{Pin}(2))
	\end{align*}
	obtained using the inclusion of a fiber $S^{\mathbb{R}^t}\to S^{\mathbb{R}^t}\wedge B_+$, over  the marked  point   $b_0\in B^{\mathrm{Pin}(2)}$.  
	In particular, the invariant $k(\bf{U})$ depends on a choice of the  point $b_0 \in B$, which does not appear in the notation.  
	
	We define $k({\bf U}) \in \mathbb{Z}_{\geq 0}$ by
	\[
	k({\bf U}) =  \min \{ k \in \mathbb{Z}_{\geq 0} :  \exists x \in \mathcal{I}({\bf U}),   wx =  2^{k} w  \}. 
	\]
	If $\mathcal{I}({\bf U})$ is of the form $(z^{k})$ for some non-negative inter $k$, we say that ${\bf U}$ is $K_{{\rm Pin}(2)}$-split. 
\end{dfn}

%

\begin{lem} \label{lem k suspension}
\[
          k(\Sigma_{B}^{\tilde{\mathbb{C}}} {\bf U}) = k({\bf U}), \quad
          k( \Sigma_{B}^{\mathbb{H}} {\bf U}) = k({\bf U}) + 1. 
\]
\end{lem}

\begin{proof}
Since 
  \begin{align*}
     &   (\Sigma_{B}^{\tilde{\mathbb{C}}} W)  / s(B)
     = \Sigma^{\tilde{\mathbb{C}}} [W/ s(B)], \\
    &    (\Sigma_{B}^{\mathbb{H}}  W)  / s(B)
     = \Sigma^{\mathbb{H}}[W/ s(B)], 
  \end{align*}
we can apply Lemma 3.4 of \cite{Manolescu_Intersection_form}.  
\end{proof}

\begin{prop}  \label{prop k W_0 W_1}
	Let ${\bf U}_0 = (W_0, r_0, s_0)$, ${\bf U}_1 = (W_1, r_1, s_1)$ be $\mathrm{Pin}(2)$-ex-spaces of SWF type at level $2t_0,2t_1$ over $B_0$ and $B_1$, and assume given an inclusion $\rho: B_0 \to B_1$.  Let $\rho_!{\bf U}_0$ denote the pushforward of $\bf{U}_0$, as an ex-space over $B_1$.  Assume that there is a fiberwise-deforming $S^1$-map
	\[
	f :  \rho_! \bf{U}_0  \rightarrow \bf{U}_1
	\]
	such that the restriction to 
	\[
	f^{S^1}: \rho_! W_0^{S^1}\to W_1^{S^1},
	\]
	as a fiberwise-deforming morphism over $B_1$, is homotopy-equivalent to 
	\[
	\ell\cup \rho:  ((\tilde{\mathbb{C}}^{t_0})^+\times B_0)\cup_{B_0} B_1 \to (\tilde{\mathbb{C}}^{t_1})^+\times B_1,
	\]
	where $\ell$ is the map on one-point compactifications induced by a map of representations $\tilde{\mathbb{C}}^{t_0}\to \tilde{\mathbb{C}}^{t_1}$, which is an inclusion if $t_0\leq t_1$.  Say that $\rho$ sends the marked point $b_0\in B_0$ to $b_1\in B_1$.
\begin{enumerate}

	\item\label{itm:k-1}
	If $t_0  \leq  t_1$, we have
	\[
	k({\bf U}_0) + t_0 \leq k({\bf U}_1) + t_1. 
	\]

	\item\label{itm:k-2}
	If $t_0 < t_1$ and  ${\bf U}_0$ is $K_{{\rm Pin}(2)}$-split, we have
	\[
	k({\bf U}_0) + t_0 + 1 \leq k({\bf U}_1) + t_1. 
	\]
\end{enumerate}
\end{prop}

%
%
%
%
%
%

\begin{proof} 
We have the following commutative diagram:
\[
     \xymatrix{
       \tilde{K}_{{\rm Pin}(2)}(W_0/s_0(B_0))  \ar[d]_{\iota_0^*}   &  & \tilde{K}_{{\rm Pin}(2)}(W_1/s_1(B_1))  \ar[d]^{\iota_1^*} \ar[ll]_{f^*}  \\
        \tilde{K}_{{\rm Pin}(2)}( ((\tilde{\mathbb{C}}^{t_0})^+\times B_0)\cup_{B_0}B_1/s(B_1)) \ar[d]_{ \iota^*} &    &  \tilde{K}_{{\rm Pin}(2)}( (\tilde{\mathbb{C}}^{t_1})^+\times B_1/s(B_1) ) \ar[d]^{\iota^*}  \ar[ll]_-{(\ell\cup \rho)^*}  \\
        \tilde{K}_{{\rm Pin}(2)}((\tilde{\mathbb{C}}^{t_0})^+ ) \ar[d]_{ \cdot w^{t_0}} &    &  \tilde{K}_{{\rm Pin}(2)}( ((\tilde{\mathbb{C}}^{t_1})^+ ) \ar[d]^{\cdot w^{t_1}}  \ar[ll]_{\ell^*} \\
        \tilde{K}_{{\rm Pin}(2)}(  S^0 ) &  & \tilde{K}_{{\rm Pin}(2)}( S^0  )  \ar[ll]_{\id}  \\
          }
\]
 Here we have used $\iota$ to denote various inclusions.  Note that $f^*$ in the first row is well-defined, because $s_0(B_0)\subset s_0(B_1)$, using the definition of the push-forward $\rho_!\mathbf{U}_0$ (this does not require that $\rho$ be an inclusion).  In fact, more is true, in that $\rho_!W_0/s_0(B_1)$ is exactly $W_0/s_0(B_0)$.

We can apply the arguments in the proofs of Lemma 3.10 and Lemma 3.11 of \cite{Manolescu_Intersection_form} so that the result follows. 
\end{proof}

\begin{dfn}
For  $m, n \in \mathbb{Z}$ and  ${\rm Pin}(2)$-ex-space ${\bf U}$ of SWF type at even level,  we define
\[
      k(\Sigma_B^{ \tilde{\mathbb{R}}^{2m} \oplus \mathbb{H}^{n} } {\bf U}) = 
      k({\bf U}) + n. 
\]
\end{dfn}

Note that this definition is compatible with Lemma \ref{lem k suspension}. 

\begin{dfn}
For $m_0, n_0, m_1, n_1 \in \mathbb{Z}$ and ${\rm Pin}(2)$-ex-spaces ${\bf U}_0, {\bf U}_1$ of SWF type at even level over $B$,    we say that $\Sigma_B^{ \tilde{\mathbb{R}}^{2m_0}  \oplus \mathbb{H}^{n_0}} {\bf U}_0$ and $\Sigma_B^{\tilde{\mathbb{R}}^{2m_1}  \oplus \mathbb{H}^{n_1}  } {\bf U}_1$  are locally equivalent if there are $N \in \mathbb{Z}$ with  $N + m_0, N+n_0, N + m_1, N + n_1 \geq 0$ and  ${\rm Pin}(2)$-fiberwise deforming maps
  \begin{align*}
         & f : \Sigma_B^{ \tilde{\mathbb{R}}^{2(N + m_0)} \oplus  \mathbb{H}^{N + n_0} }{\bf U}_0 \rightarrow \Sigma_B^{ \tilde{\mathbb{R}}^{2(N + m_1)} \oplus  \mathbb{H}^{N + n_1}  } {\bf U}_1, \\   
        & g :  \Sigma_B^{\tilde{\mathbb{R}}^{2(N + m_1)} \oplus  \mathbb{H}^{N + n_1} }{\bf U}_1 \rightarrow \Sigma_B^{ \tilde{\mathbb{R}}^{2(N + m_0)} \oplus \mathbb{H}^{N + n_0} } {\bf U}_0
    \end{align*}
such that  the restrictions 
\[
       f^{S^1} : \Sigma_B^{ \tilde{\mathbb{R}}^{2(N+m_0)} } {\bf U}_0^{S^1} \rightarrow \Sigma_B^{ \tilde{\mathbb{R}}^{2(N+m_1)} } {\bf U}_{1}^{S^1}, \
       g^{S^1} : \Sigma_B^{ \tilde{\mathbb{R}}^{2(N+m_1)} }{\bf U}_1^{S^1} \rightarrow \Sigma_B^{ \tilde{\mathbb{R}}^{2(N+m_0)} }{\bf U}_0^{S^1} 
\]
 are homotopy equivalent to
\[
    \mathrm{Id} : B \times (\mathbb{R}^{t})^+ \rightarrow B \times (\mathbb{R}^t)^+
\]
as  $\mathrm{Pin}(2)$-fiberwise-deforming  morphisms. 
\end{dfn}

\begin{cor} \label{cor local eq k}
If  $\Sigma_B^{ \tilde{\mathbb{R}}^{2m_0} \oplus \mathbb{H}^{n_0} }{\bf U}_0$ and  $\Sigma_B^{ \tilde{\mathbb{R}}^{2m_1} \oplus  \mathbb{H}^{n_1}  }{\bf U}_1$ are locally equivalent,  we have
\[
     k(\Sigma_B^{\tilde{\mathbb{R}}^{2m_0} \oplus \mathbb{H}^{n_0}  }{\bf U}_0) = 
     k(\Sigma_B^{\tilde{\mathbb{R}}^{2m_1} \oplus \mathbb{H}^{n_1}  }{\bf U}_1). 
\]
\end{cor}

\begin{proof}
This is a direct consequence of Proposition \ref{prop k W_0 W_1}. 
\end{proof}

Let $\frak{s}$ be a  spin structure (not just a self-conjugate $\mathrm{spin}^c$ structure, although we will also write $\frak{s}$ for the induced self-conjugate $\mathrm{spin}^c$ structure) of $Y$. Then the Seiberg-Witten equations (\ref{SW eq}) and the finite dimensional approximations (\ref{eq for gamma})  have Pin(2)-symmetry. Let $B_Y$ be the Picard torus of $Y$, which is homeomorphic to the torus $\tilde{\mathbb{R}}^{b_1(Y)} / \mathbb{Z}^{b_1(Y)}$, where we have chosen coordinates so that $0\in \tilde{\mathbb{R}}^{b_1(Y)}$ corresponds to the selected spin structure on $Y$.  We choose $[0] \in B_{Y}$ as base point.  Assume that $\operatorname{ind} D_Y = 0$ in $KQ^1(B_Y)$.  By Theorem \ref{thm Pin(2) spectral sections}, we can  choose a  ${\rm Pin}(2)$-spectral system
\[
\mathfrak{S}=(\mathbf{P},\mathbf{Q},\mathbf{W}_P,\mathbf{W}_Q,\{\eta^P_n\}_n,\{\eta_n^Q\},\{\eta^{W_P}_n\}_n,\{\eta^{W_Q}_n\}_n)
\]
for $Y$. Put
\[
       F_n = P_n \cap Q_n, \ W_n = W_{P, n} \cap W_{Q, n}. 
\]
We have the ${\rm Pin}(2)$-equivariant Conley index $(I_n, r_n, s_n)$ for the isolated invariant set $\operatorname{inv}(A_n, \varphi_{k_+, k_-, n})$ for $n \gg 0$. 

\begin{lem}
The ${\rm Pin}(2)$-equivariant Conley index $(I_n, r_n, s_n)$ is of SWF type at level $\rank_{\mathbb{R}} W_n^{-}$ for $n \gg 0$. 
\end{lem}

\begin{proof}
The proof is similar to that of Lemma \ref{lem I_n S^1} and omitted. 
\end{proof}

Let  $\mathcal{SWF}^{\rm{Pin}(2)}(Y, \frak{s}, [\frak{S}])$ be  the ${\rm Pin}(2)$-Seiberg-Witten Floer parameterized homotopy type.  As before, the local equivalence class of $\mathcal{SWF}^{\rm{Pin}(2)}(Y, \frak{s}, [\frak{S}])$ is independent of $k_{\pm}, n$.  
See \cite{stoffregen-connectedsum} for the study of the local equivalence class of the $\rm{Pin}(2)$-Seiberg-Witten Floer homotopy type in the case $b_1(Y) = 0$. 
We may assume that $\dim_{\mathbb{R}} W_n^-$ are even for all $n$.  Then we have the well-defined number 
 \[
      k( \mathcal{SWF}^{\rm{Pin}(2)}(Y, \frak{s}, [\frak{S}]) ) \in \mathbb{Z}.
\]

\begin{dfn}
Fix $(Y,\frak{s})$ as above.  We define $\kappa(Y, \frak{s}) \in  \mathbb{Q} \cup \{ -\infty \}$ by 
\[
   \kappa(Y, \frak{s}) := 
    \inf_{g, \frak{S}}    2\Big(  k( \mathcal{SWF}^{\rm{Pin}(2)}(Y, \frak{s}, [\frak{S}]) ) -  \frac{1}{2} n(Y, g, \frak{s}, P_0) \Big). 
\]
We say that $(Y, \frak{s})$ is Floer $K_{{\rm Pin}(2)}$-split if $(I_n, r_n,s_n)$ is $K_{{\rm Pin}(2)}$-split for $n$ large,  where $(I_n,r_n,s_n)$ realizes equality in the definition of $\kappa(Y,\frak{s})$. 

 Note that this invariant indeed depends a priori on $\frak{s}$ as a spin structure, in what we have chosen as the marked point in $B_Y$ that is used in the definition of $\kappa$.
\end{dfn}
Unlike the case for homology, we have not shown that the invariant 
\[
            k( \mathcal{SWF}^{\rm{Pin}(2)}(Y, \frak{s}, [\frak{S}]))
\]
 is invariant under changes of spectral section that lie in $\tilde{KQ}(B)$, (essentially since we do not have access to a notion of $\rm{Pin}(2)$-complex orientable cohomology theories).  We expect that the quantity appearing in the $\inf$ is, in fact, independent of $[\frak{S}]$, however.
	
	We do not know if a self-conjugate $\mathrm{spin}^c$ structure may have different $\kappa$-invariants associated to different underlying spin structures.  The invariant $\kappa(Y,\frak{s})$ for $Y$ a rational homology $3$-sphere, agrees with Manolescu's definition \cite{Manolescu_Intersection_form}, by construction.  

\begin{cor}\label{cor:rokhlin-kappa}
	The reduction mod $2$ of the $\kappa$ invariant satisfies:
	\[
	\mu(Y,\frak{s})=\kappa(Y,\frak{s})\bmod{2},
	\]
	where $\mu(Y,\frak{s})$ is the Rokhlin invariant of $(Y,\frak{s})$.
\end{cor}
\begin{proof}
	Indeed, $n(Y,g,\frak{s},P_0)\bmod{2}$ is the Rokhlin invariant of $(Y,\frak{s})$ by its construction.  The corollary then follows from the definition of $\kappa$ and the fact that $k$ is an integer.
\end{proof}

Corollary \ref{cor:rokhlin-kappa} indicates that $\kappa(Y,\frak{s})$ may depend on $\frak{s}$, as a spin structure. Note that if $(Y,\frak{s})$ admits a $\mathrm{Pin}(2)$-equivariant spectral section, for a self-conjugate $\mathrm{spin}^c$ structure $\frak{s}$, then $\mu(Y,-)$ is constant on all spin structures underlying $\frak{s}$; by Lin's result \cite{lin_rokhlin}, this condition, coupled with the triple cup product vanishing, characterizes $3$-manifolds which admit a $\mathrm{Pin}(2)$-equivariant spectral section.  However, if the $\mathrm{Pin}(2)$-equivariant $K$-theory could be extended to $3$-manifolds without a $\mathrm{Pin}(2)$-spectral section, so that Corollary \ref{cor:rokhlin-kappa} held, it would of course also imply that $\kappa(Y,\frak{s})$ depends on the spin structure and not just the $\mathrm{spin}^c$ structure.

Using our invariant $\kappa(Y,\frak{s})$, we can prove a $\frac{10}{8}$-type inequality  for smooth 4-manifolds with boundary, which generalizes the results of \cite{furuta} and \cite{Manolescu_Intersection_form}.


\begin{thm} \label{thm 10/8 inequality}
Let  $(Y_0, \frak{s}_0)$ be a spin,  rational homology $3$-sphere and $(Y_1, \frak{s}_1)$ be a closed,  spin 3-manifold such that the index $\operatorname{ind} D_{Y_1}$ is zero in $KQ^1(B_{Y_1})$.

\begin{enumerate}
\item
 Let $(X, \frak{t})$ be a compact,  smooth, spin, negative semidefinite 4-manifold with boundary   $-(Y_0, \frak{s}_0) \coprod (Y_1, \frak{s}_1)$. Then we have
\[
        \frac{1}{8} b_2^-(X)  + \kappa(Y_0, \frak{s}_0) \leq \kappa(Y_1, \frak{s}_1). 
\]

\item
Let $(X, \frak{t})$ be a compact, smooth, spin 4-manifold with boundary  
$-(Y_0, \frak{s}_0) \coprod$ $(Y_1, \frak{s}_1)$. Then we have
we have 
\[
        - \frac{\sigma(X)}{8} +  \kappa(Y_0, \frak{s}_0)  - 1
        \leq   b^+(X) + \kappa(Y_1, \frak{s}_1).
\]
Moreover, if $Y_0$ is Floer $K_{{\rm Pin}(2)}$-split  and $b^+(X) > 0$, we have
\[
    - \frac{\sigma(X)}{8} +  \kappa(Y_0, \frak{s}_0)  + 1
        \leq   b^+(X) + \kappa(Y_1, \frak{s}_1).
\] 
\end{enumerate}

\end{thm}

\begin{proof}
Let $[0] \in B_X = \mathrm{Pic}(X)$ be the element corresponding to the flat spin connection.  
Recall that $\mathcal{BF}_{[n]}$ is a fiber-preserving map. 
The restriction $\mathcal{BF}_{[n]}(X,\frak{t})$ to the fiber over $[0]$ and the duality map
\[
      I_n(Y_0) \wedge I_n(-Y_0) \rightarrow   S^{F_n(Y_0) \oplus W_n(Y_0)}
\]
defined in \cite[Section 2.5]{manolescu_gluing},   give a ${\rm Pin}(2)$-map
\[
       f_n :  \Sigma^{\tilde{\mathbb{R}}^{m_0} \oplus \mathbb{H}^{n_0}}  I_n(Y_0)  
                 \rightarrow 
                 \Sigma^{\tilde{\mathbb{R}}^{m_1} \oplus \mathbb{H}^{n_1}} (I_n(Y_1)  / s_n(B_{Y_1}))
\]
such that 
 \begin{align*}
   &  f_n(   (\Sigma^{\tilde{\mathbb{R}}^{m_0} \oplus \mathbb{H}^{n_0}}  I_n(Y_0))^{S^1}  )   
       \subset    
       (\Sigma^{\tilde{\mathbb{R}}^{m_1} \oplus \mathbb{H}^{n_1}} I_n(Y_1)_{[0]})^{S^1}, \\
   &  f_n( (\Sigma^{\tilde{\mathbb{R}}^{m_0} \oplus \mathbb{H}^{n_0}}  I_n(Y_0))^{{\rm Pin}(2)}  ) \subset  (\Sigma^{\tilde{\mathbb{R}}^{m_1} \oplus \mathbb{H}^{n_1}}I_n(Y_1)_{[0]})^{{\rm Pin}(2)}. 
 \end{align*}
Here $[0] \in \mathrm{Pic}(Y_1)$ is the element corresponding to the flat spin  connection,  and
   \begin{align*}
         m_0 - m_1 &= \rank_{\mathbb{R}} W_{n}(Y_1)^-  - \dim_{\mathbb{R}} W_n(Y_0)^- - b^+(X), \\
         n_0 - n_1  &= \rank_{\mathbb{H}} F_n(Y_1)^{-} - \dim_{\mathbb{H}} F_n(Y_0)^-   \\
                         & \qquad +  \frac{1}{2}n(Y_1, g|_{Y_1}, \frak{t}|_{Y_1}, P_0) - \frac{1}{2} n(Y_0, g|_{Y_0}, \frak{t}|_{Y_0}) - \frac{\sigma(X)}{16}. 
   \end{align*}
The restriction of $f_n$ to $(\Sigma^{\tilde{\mathbb{R}}^{m_0} \oplus \mathbb{H}^{n_0}}   I_n(Y_0))^{S^1}$ is induced by the operator
\[
      (d^+, \pi_{-\infty}^{0} r_{-Y_0},   \pi_{-\infty}^{0}  r_{Y_1}) : \Omega^1_{CC}(X) \rightarrow \Omega^+(X) \oplus  ({\mathcal W}_{-Y_0})_{-\infty}^{0} \oplus (\mathcal{W}_{Y_1, [0]})_{-\infty}^{0}
\]
and is a homotopy equivalence
\[
          (\Sigma^{\tilde{\mathbb{R}}^{m_0} \oplus \mathbb{H}^{n_0}}  I_n(Y_0))^{{\rm Pin}(2)} 
          \rightarrow 
          (\Sigma^{\tilde{\mathbb{R}}^{m_1} \oplus \mathbb{H}^{n_1}} I_n(Y_1)_{[0]})^{{\rm Pin}(2)},
\]
indeed, both of these are just $S^0$ consisting of $0$ and the base point. 
Moreover if $b^+(X) = 0$, the restriction of $f_n$ to $(\Sigma^{\tilde{\mathbb{R}}^{m_0} \oplus \mathbb{H}^{n_0}}  I_n(Y_0))^{S^1}$ is a ${\rm Pin}(2)$-homotopy equivalence
\[
        \Sigma^{\tilde{\mathbb{R}}^{m_0}}  I_n(Y_0)^{S^1} 
        \rightarrow   
      \Sigma^{\tilde{\mathbb{R}}^{m_1}}  I_n(Y_1)_{[0]}^{S^1}.
\]
We may assume that $m_0, m_1$ are even and can use Proposition \ref{prop k W_0 W_1} (1) to get the first statement.  

If $b^+(X)$ is even,  $\Sigma^{\tilde{\mathbb{R}}^{m_0} \oplus \mathbb{H}^{n_0}}  I_n(Y_0)$ and  $\Sigma^{\tilde{\mathbb{R}}^{m_1} \oplus \mathbb{H}^{n_1}} I_n(Y_1)$ are of SWF type at even levels and  we can apply Proposition \ref{prop k W_0 W_1} (\ref{itm:k-1}), (\ref{itm:k-2}) to $f_n$ to obtain the second statement. 
If $b^+(X)$ is odd, we take a connected sum $X \# S^2 \times S^2$, then we can apply Proposition \ref{prop k W_0 W_1}.   In this second part, we take advantage of the fact that $\kappa(Y,\frak{s})\bmod{2}$ agrees with the Rokhlin invariant, as is used in \cite[Proof of Theorem 1.4]{Manolescu_Intersection_form}.

\end{proof}

\begin{cor} \label{cor:10/8}
Let $(X, \frak{t})$ be a compact spin $4$-manifold with boundary $Y$. Assume that  the index bundle $\operatorname{ind} D_{Y}$  is zero in $KQ^1(B_Y)$.  Then we have
\[
        -\frac{\sigma(X)}{8} - 1 \leq b^+(X) + \kappa(Y, \frak{t}|_{Y}). 
\]
Moreover if $b^+(X) > 0$ we have
\[
      -\frac{\sigma(X)}{8} + 1 \leq b^+(X) + \kappa(Y, \frak{t}|_{Y}).
\]
\end{cor}

\begin{proof}
Removing a small disk from $X$, we get a bordism $X'$  with boundary $S^3 \coprod Y$. Since  $\kappa(S^3) = 0$ and $S^3$ is Floer $K_{{\rm Pin}(2)}$-split,  applying Theorem \ref{thm 10/8 inequality} to $X'$, we obtain the inequalities. 
\end{proof}

Since the spin bordism group $\Omega^{\mathrm{spin}}_3$ is zero, we obtain the following: 

\begin{cor}
$\kappa(Y, \frak{s}) > - \infty$. 
\end{cor}

\begin{ex}
Let $\frak{s}$ be a spin structure on $S^1 \times S^2$. 
Since $S^1 \times S^2$ has a positive scalar curvature metric $g$,  the conditions of Theorem \ref{thm reducible I_n} are satisfied.  Hence  $\mathcal{SWF}(Y,\frak{s}, \frak{S}) \cong S_{B_Y}^{0}$. Here $\frak{S}$ is a spectral system with $P_0 = \mathcal{E}_0(D)_{-\infty}^{0}$.   Also we have $n(S^1 \times S^2, g, \frak{s}, P_0) = 0$, because there is an orientation reversing diffeomorphism of $S^1 \times S^2$. So we obtain
\[
    \kappa(S^1 \times S^2, \mathfrak{s}) \leq 0. 
\]
Note that $\mathfrak{s}$ extends to a spin structure $\mathfrak{t}$ on $S^1 \times D^3$.  Applying Theorem \ref{cor:10/8} to $(S^1 \times D^3) \# (S^2 \times S^2)$, we get $\kappa(S^1 \times S^2, \mathfrak{s}) \geq 0$.  Hence 
\[
  \kappa (S^1 \times S^2, \mathfrak{s} ) = 0.
\]

If $X$ is an compact, oriented,  spin  $4$-manifold with boundary $S^1 \times S^2$ and with $b^+(X) > 0$, we have
\[
          -\frac{\sigma(X)}{8} + 1 \leq b^+(X)
\]
by Corollary \ref{cor:10/8}. 
This inequality can be also obtained from the $\frac{10}{8}$-inequality \cite{furuta} for the closed $4$-manifold $X \cup (S^1 \times D^3)$  and the additivity of the signature. 
\end{ex}

\begin{ex}
Let $Y$ be the flat 3-manifold and $\frak{s}_1, \frak{s}_2, \frak{s}_3$ be the $\mathrm{spin}^c$ structures  in Example \ref{ex flat 3-mfd h inv}.  As in Example \ref{ex flat 3-mfd h inv}, for any underlying spin structure, we have
\[
    \kappa(Y, \frak{s}_j) \leq 0 
\]
for $j = 1, 2, 3$. 
\end{ex}

\begin{ex}
Let $p : Y \rightarrow \Sigma$ be the sphere bundle  of the complex line bundle $N_d$ on a closed, oriented surface $\Sigma$ of degree $d$.     Assume that $d$ is even and that $0 < g(\Sigma) < \frac{d}{2} + 1$.  Using a connection on $N_d$, we have an identification  
\[
TN_d =   p^* T\Sigma \oplus p^* N_d.
\]
Let $s : Y  \rightarrow p^* N_d|_{Y}$ be the tautological section. Then we have
\begin{equation}  \label{eq:TY TSigma} 
     TY = p^* T\Sigma \oplus i \R s. 
\end{equation}
Choose  spin structures of $\Sigma$ and $N_d$. This is equivalent to choosing  complex line bundles $K_{\Sigma}^{\frac{1}{2}}$, $N_d^{\frac{1}{2}}$ and  isomorphisms $K_{\Sigma}^{\frac{1}{2}} \otimes K_{\Sigma}^{\frac{1}{2}} \cong K_{\Sigma}$, $N_d^{\frac{1}{2}} \otimes N_d^{\frac{1}{2}} \cong N_d$.   Also we consider the natural  spin structure of the trivial bundle $i\R s$.  The spin structures of $\Sigma$,  $i\R s$ and (\ref{eq:TY TSigma}) induce a spin structure $\fs'$ on $Y$. Note that $p^* (N_d^{\frac{1}{2}} \otimes N_d^{\frac{1}{2}} ) \cong p^* N_d =  \underline{\C}$ and hence the structure group of $p^* N_d^{\frac{1}{2}}$ is $\{ \pm 1 \}$.  
Put $\fs := \fs' \otimes p^*N_d^{\frac{1}{2}}$.  Then $\fs$ is a spin structure of $Y$ with spinor bundle $\mathbb{S} = p^* ( (K_{\Sigma}^{-\frac{1}{2}} \oplus K_{\Sigma}^{\frac{1}{2}})  \otimes N_d^{\frac{1}{2}})$.  The spin$^c$ structure induced by  $\fs$ is $\fs_{g-1+\frac{d}{2}}$ of Proposition \ref{prop critical points of CSD_r}.  Since $g \leq g-1+\frac{d}{2} < d$, we can apply Theorem \ref{thm I_n sphere bundle} and we get 
\[
      \mathcal{SWF}^{Pin(2)}(Y, \fs, [\mathfrak{S}]) \cong S^0_B. 
\]
Here $\mathfrak{S}$ is as in Theorem \ref{thm I_n sphere bundle}. 
Taking $q$ to be $g-1+\frac{d}{2}$ in (\ref{eq:n(Y,s_q)}), we have
\[
     n(Y, \fs, g_r, P_0) = \frac{1}{8}. 
\]
Thus we obtain
\[
    \kappa(Y,\frak{s}) \leq -\frac{1}{8}. 
\]

\end{ex}

\appendix
\chapter{The Conley Index and Parameterized Stable Homotopy}
\label{sec:homotopy}

In this chapter we define the category in which the Seiberg-Witten stable homotopy type lives, and variations thereon, as well as some background on the Conley index.    Let $G$ be a compact Lie group for this section.   In section \ref{subsec:homotopy1}, we define parameterized homotopy categories we will be interested in.  In section \ref{subsec:conley}, we give basic definitions for the Conley index.  In section \ref{subsec:spectra}, we give a definition of spectra suitable for the construction.  The main point is Theorem \ref{thm:fiberwise-deforming-conley-1}, which states that the parameterized homotopy class of the (parameterized) Conley index is well-defined as a parameterized equivariant homotopy class in $\exsp$.

\section{The Unstable parameterized homotopy category}\label{subsec:homotopy1}
This section is intended both to introduce some notation, and also to point out that the notions introduced in \cite{MRS} are compatible with parameterized, equivariant homotopy theory, as considered in \cite{CW},\cite{MS}.\footnote{Establishing that \cite{MRS} and \cite{CW},\cite{MS} are compatible is, in fact, straightforward.  However, at the time that \cite{MRS} appeared, the May-Sigurdsson parameterized homotopy category had not yet appeared.}  In the first part, we follow the discussion of Costenoble-Waner \cite[Chapter II]{CW} and \cite[Section 3]{MRS}.  In particular, we will occasionally use the notation of model categories, but the reader unfamiliar with this language may safely ignore these aspects.  The main points are Lemma \ref{lem:fib-hom-good}, which lets us translate properties from the language of \cite{MRS} to that of \cite{MS}, and Proposition \ref{prop:310}, which is used in describing the change of the Conley index of approximate Seiberg-Witten flows upon changing the finite-dimensional approximation.  

\begin{dfn}\label{def:fib-pointed-space}
	Fix a compactly generated space $Z$ with a continuous $G$-action.  A triple $\bigu=(U,r,s)$ consisting of a $G$-space $U$ and $G$-equivariant continuous maps $r\from U \to Z$ and $s\from Z \to U$ such that $r\circ s=\id_Z$ is called a(n) (equivariant) \emph{ex-space} over $Z$. \footnote{In \cite{MRS}, ex-spaces are called \emph{fiberwise-deforming spaces}.}  Let $\exsp$ be the category of ex-spaces, where morphisms $(U,r,s)\to (U',r',s')$ are given by maps $f\from U \to U'$ so that $r'f=r$ and $fs=s'$.
\end{dfn}

In comparison to the ordinary homotopy category, passing to the parameterized homotopy category results in many more maps (for a highbrow definition of the parameterized homotopy category, refer to Remark \ref{rmk:model-strucs}.  Indeed, let $(X\times I,r,s)$ and $(Y,r',s')$ be ex-spaces over $Z$.  

\begin{dfn}\label{def:fib-map}
	A \emph{fiberwise-deforming map} $f\from \bigu\to \bigu'$ is an equivariant continuous map $f\from (U,s(Z))\to (U',s'(Z))$ so that $r'\circ f$ is (equivariantly) homotopic to $r$, relative to $s(Z)$.  We say that fiberwise-pointed spaces $\bigu$ and $\bigu'$ are \emph{fiberwise-deforming homotopy equivalent}  
	if there exist continuous $G$-equivariant maps $f\from \bigu\to \bigu', g\from \bigu'\to\bigu$ so that 
	\begin{align*}
		f\circ s = s', & \qquad g\circ s'=s,\\
		r' \circ f \simeq r\, \mbox{rel}\, s(Z), & \qquad r\circ g \simeq r' \,\mbox{rel}\, s'(Z),\\
		g\circ f\simeq\id_U\, \mbox{rel}\, s(Z), &\qquad f\circ g\simeq \id_{U'}\, \mbox{rel} \, s'(Z).
	\end{align*}
	We write $[\bigu]$ for the fiberwise homotopy type of $\bigu$.  
	We will call a fiberwise-deforming map, along with the choice of a homotopy $h$ between $r'\circ f$ and $r$, a \emph{lax map}, following \cite{CW}.  
\end{dfn}

We can also consider homotopies of fiberwise-deforming maps.  A \emph{homotopy} of fiberwise-deforming maps will mean a collection of fiberwise-deforming maps  $F_t : \bigu \to \bigu'$, so that $F: U \times I \to U'$ is continuous.  Homotopy of lax maps is similar, but requiring that the homotopy involved in the definition of a lax map is compatible, as we will define below. 

\begin{rem}\label{rmk:model-strucs}
	There is a model structure (what May-Sigurdsson call the $q$-\emph{model structure}) on $\exsp$ given by declaring a map in $\exsp$ to be a weak equivalence, fibration, or cofibration, if it is such after forgetting the base $Z$, but May-Sigurdsson point out technical difficulties with this model structure.  They define a variant, the $qf$-\emph{model structure} on $\exsp$, whose weak equivalences are those of the $q$-model structure, but with a smaller class of cofibrations.  Let $\hoex$ denote the homotopy category of the $qf$-model structure; we call this the \emph{parameterized homotopy category} and write $[X,Y]_{G,Z}$ for the morphism sets of $\hoex$ - these turn out to be the same as the lax maps $X$ to $Y$ up to homotopy, as in \cite[Section 2.1]{CW}.
\end{rem}

Let $\Lambda Z$ denote the set of \emph{Moore paths} of $Z$:
\[
\Lambda Z=\{(\lambda,\ell)\in Z^{[0,\infty]}\times [0,\infty) \mid \lambda(r)=\lambda(\ell) \mbox{ for } r\geq \ell\}.
\]
Recall that Moore paths have a strictly associative composition:
\[
(\lambda\mu)(t) =\begin{cases}
\lambda(t) \mbox{ if } t\leq \ell_\lambda, \\ \mu(t-\ell_\lambda) \mbox{ if } t \geq \ell_\lambda.
\end{cases}
\]
Given $r\from X \to Z$, the \emph{Moore path fibration} $LX=L(X,r)$ is defined by
\[
LX =X \times_Z \Lambda Z,
\]
and there is an inherited projection map $Lr\from LX \to Z$ by $Lr((x,\lambda))=\lambda(\infty)$, as well as an inherited section map $Ls\from Z \to LX$ given by $Ls(b)=(s(b),b)$, the path with length zero at $s(b)$.  Finally, there is a natural inclusion $\iota \from X \to LX$, which is a weak-equivalence on total spaces, and hence a weak equivalence in the $qf$-model structure.  

Note that a lax map $X\to Y$ is equivalent to the data of a genuine map $X\to LY$ in $\exsp$ (using that $Y$ and $LY$ are weakly equivalent, and basic properties of model categories).  In particular, any lax map defines an element of $[X,Y]_{G,Z}$, which may or may not be represented by a map $X\to Y$ in $\exsp$.  The following lemma is then immediate from the definitions:

\begin{lem}\label{lem:fib-hom-good}
	Fiberwise-deforming homotopy equivalent spaces are weakly equivalent in $\exsp$.
\end{lem}


A \emph{homotopy} between lax maps $f_0: X \to Y$ and $f_1: X \to Y$ is a lax map $X\wedge_Z [0,1]_+\to Y$ so that $f|_{X\wedge i}=f_i$ for $i=0,1$.  By \cite[2.1]{CW} the homotopy classes of lax maps are in agreement with $[X,Y]_{G,Z}$.



We will encounter collections of fiberwise-deforming spaces related by suspensions.  We have the following definition.

\begin{dfn}[Section 3.10, \cite{MRS}]\label{def:fiberwise-smash}
	Let $\bigu=(U,r,s)$ and $\bigu'=(U',r',s')$ be ex-spaces over $Z,Z'$, where $U,Z$ are $G$-spaces and $U',Z'$ are $G'$-spaces, for $G,G'$ compact Lie groups.  Define an equivalence relation $\sim_\wedge$ on $U \times U'$ by $(u,u')\sim_\wedge (v,v')$ if $(u,u')=(v,v')$ or $u=v\in s(Z),r'(u')=r'(v')$ or $r(u)=r(v),u'=v'\in s'(Z')$.  Define the \emph{fiberwise smash product} by
	\[
	U \wedge U':= U\times U'/\sim_\wedge.
	\]
\end{dfn}

We call an ex-space $\bigu$ \emph{well-pointed} if the inclusion $s(Z)\to U$ is a cofibration in the category of $G$-spaces.  That is, we require that $s(Z)\subset U$ admits a $G$-equivariant Str{\o}m structure (for a definition see \cite[Section 3]{MRS}).  We record the following result from \cite{MRS} (the proof in the equivariant case is identical to that for the nonequivariant case).
\begin{prop}[Proposition 3.10, \cite{MRS}]\label{prop:310}
	Assume that $\bigu,\bigu',\mathbf{V},\mathbf{V}'$ are fiberwise well-pointed spaces, with $[\bigu]=[\bigu']$ and $[\mathbf{V}]=[\mathbf{V}']$.  Then $[\bigu \wedge \mathbf{V}]=[\bigu'\wedge \mathbf{V}']$.  
\end{prop}

There is also a pushforward for ex-spaces defined in \cite{MS}.  Fix an ex-object $\bigu$ given by $Z\to^s U\to^r Z$ and a map $f\from Z \to Y$.  Define $f_!\bigu=(f_! U,t,q)$ by the retract diagram
\[
\begin{tikzcd}
Z \arrow[r]{f} \arrow[d]& Y\arrow[d]\\
\bigu \arrow[r]\arrow[d]& f_! U\arrow[d]\\
Z \arrow[r]& Y,
\end{tikzcd}  
\]
where the top square is a pushout, and the bottom is defined by the universal property of pushouts, along with the requirement that $q\circ t=\operatorname{id}$.

\begin{prop}[Proposition 7.3.4 \cite{MS}]\label{prop:descent}
	Say that $\bigu$ and $\bigu'$ are weakly-equivalent $G$-ex-spaces.  Then $f_!\bigu\simeq f_! \bigu'$.
\end{prop} 
Note the simple example that for $\bigu$ a sectioned spherical fibration over $Z$, and $f\from Z \to *$ the collapse, $f_!\bigu$ is the Thom complex.

For $W$ a real $G$-vector space and $\bigu\in \exsp$, we define $\Sigma^W\bigu=\bigu\wedge W^+$, where $W^+$ is considered as a parameterized space over a point (we consider $\bigu\wedge W^+$ as a $G$-fiberwise deforming space by pulling back along the diagonal map $G\to G\times G$).  By Proposition \ref{prop:descent}, this is well-defined on the level of homotopy categories.

\begin{rem}\label{rmk:bad-suspension-1}
	For two ex-spaces $\bigu,\bigu'$, there is a fiberwise product $\bigu\times_Z \bigu'$, which is naturally an ex-space (whose structure maps are inherited from the universal properties of pullbacks), and similarly we obtain a fiberwise smash product $\bigu \wedge_Z \bigu'$.  That is, we have a functor $\wedge_Z\from \hoex\times \hoex\to\hoex$.  By Proposition 7.3.1 of \cite{MS}, $\wedge_Z$ descends to homotopy categories.  The main implication of this from our perspective is that it is legitimate to suspend Conley indices by nontrivial sphere bundles over the base $Z$.  
\end{rem}

\begin{dfn}\label{def:sw-cat}
	Fix $B$ a finite $G$-CW complex.  The $G$-equivariant \emph{parameterized Spanier-Whitehead category} $\fpsw_B$ is defined as follows.  The objects are pairs $(\bigu,R)$,  also denoted by $\Sigma_B^{R} {\bf U}$,  for $\bigu$ an element of $\exsp$ (with total space $U$ a finite $G$-CW complex) and $R$ a virtual real finite-dimensional $G$-vector space (in a fixed universe).  Morphisms are given by
	\begin{equation}\label{eq:sw-morph}
	\hom((\bigu,R),(\bigu',R'))=\colim_{W} [\Sigma^{W+R}\bigu,\Sigma^{W+R'}\bigu']_{G,B},
	\end{equation} 
	where the colimit is over sufficiently large $W$.  A \emph{stable homotopy equivalence} in $\fpsw_{G, B}$ will be a stable map that admits some representative which is a weak equivalence.  We write $(\bigu,R)\simeq_{\fpsw} (\bigu',R')$ to denote stable homotopy equivalence, omitting the subscript if clear from context.
	A \emph{parameterized $G$-equivariant stable homotopy type} is an equivalence class of objects in $\fpsw_{G,B}$ up to stable homotopy equivalence.
\end{dfn}

In Definition \ref{def:sw-cat}, the colimit may be taken over any sequence of representations which is cofinal in the universe.  In particular, in the case of $S^1$ and $\mathrm{Pin}(2)$-spaces, we will fix the following definitions.

Let $\mathcal{U}_{S^1}=\underline{\mathbb{C}}^{\oplus \infty}\oplus \underline{\mathbb{R}}^{\oplus \infty}$, where $\mathbb{C}$ is the standard representation of $U(1)$, and $\mathbb{R}$ is the trivial representation.  Let $\mathcal{U}_{\mathrm{Pin}(2)}=\underline{\mathbb{H}}^{\oplus \infty}\oplus \underline{\tilde{\mathbb{R}}}^{\oplus \infty}$, where $\mathbb{H}$ is the quaternion representation of $\mathrm{Pin}(2)$, and $\tilde{\mathbb{R}}$ is the sign representation.  There is a full subcategory  $\spanierwhitehead_{S^1}$ of $\fpsw_{S^1,B}$ obtained by considering only those spaces $(\bigu,R)$ with $R=\underline{\mathbb{C}}^{\oplus n}\oplus \underline{\mathbb{R}}^{\oplus m}$, with $m,n\in\mathbb{Z}$; we use the shorthand $(\bigu,-2n,-m)$ to denote $(\bigu,R)$ in $\spanierwhitehead_{S^1}$.  Note that every element of $\fpsw_{S^1,B}$ on $\mathcal{U}_{S^1}$ is stable homotopy equivalent to an element of $\spanierwhitehead_{S^1}$.  Similarly, we write $\spanierwhitehead_{\mathrm{Pin}(2)}$ for the subcategory whose objects are tuples $(\bigu,R)$ in $\fpsw_{\mathrm{Pin}(2),B}$ with \[
R=\underline{\mathbb{H}}^{\oplus n}\oplus \underline{\tilde{\mathbb{R}}}^{\oplus m}.
\]
We write $(\bigu,-4n,-m)$ for the resulting element (so that the notation is consistent with the forgetful functor from $\mathrm{Pin}(2)$-spaces to $S^1$-spaces).



We note that $\fpsw_*$, the parameterized Spanier-Whitehead category over a point, is exactly the ordinary Spanier-Whitehead category.  The next lemma follows from the definitions:
\begin{lem}\label{lem:spanier-whitehead-descent}
	Let $f\from B \to *$.  There is an induced functor $f_!\from \fpsw_B \to \fpsw_*$ defined by $f_!(\bigu,R)=(f_!\bigu,R)$ so that $(\bigu,R)\simeq_{\fpsw_B} (\bigu',R')$ implies $f_!(\bigu,R)\simeq_{\fpsw_*} f_!(\bigu',R')$.  
\end{lem}



We have the following corollary:
\begin{cor}\label{cor:types}
	Let $f\from B \to *$.  Then stable-homotopy equivalence classes in $\fpsw_B$ give well-defined stable-homotopy classes in $\fpsw_*$.  
\end{cor}

Finally, we remark that May-Sigurdsson \cite[Chapter 20-22]{MS} define many parameterized homology theories, suitably generalizing the usual definition of a (usual) homology theory, and giving convenient invariants from objects of $\fpsw_*$.



\section{The Parameterized Conley Index}\label{subsec:conley}
In this subsection, we review the \emph{parameterized Conley index} from \cite{MRS} (see also work of Bartsch \cite{bartsch}); we note that we work in considerably less generality than they present.  We start by giving the basic definitions in Conley index theory, following \cite[Section 5]{Manolescu-b1=0}.  Note that \cite{MRS} work nonequivariantly; the proofs in the equivariant case are similar.   

Let $M$ be a finite-dimensional manifold and $\varphi$ a flow on $M$; for a subset $N\subset M$, we define the following sets:
\[
\begin{split}
N^+&=\{x\in N: \forall t>0, \, \varphi_t(x)\in N\}\\
N^-&=\{x\in N: \forall t<0, \, \varphi_t(x)\in N\}\\
\operatorname{inv} \; N&= N^+\cap N^-.
\end{split}
\]

A compact subset $S\subset M$ is called an \textit{isolated invariant set} if there exists a compact neighborhood $S\subset N$ so that $S=\operatorname{inv}(N)\subset \operatorname{int} (N)$.  Such a set $N$ is called an \textit{isolating neighborhood} of $S$.  

A pair $(N,L)$ of compact subsets $L\subset N\subset M$ is an \textit{index pair} for $S$ if the following hold:
\begin{enumerate}
	\item $\operatorname{inv}\; (N \smallsetminus L) = S\subset \operatorname{int}(N \smallsetminus L)$.
	\item $L$ is an exit set for $N$, that is, for any $x\in N$ and $t>0$ so that $\varphi_t(x)\notin N,$ there exists $\tau\in [0,t)$ with $\varphi_\tau(x)\in L$.
	\item $L$ is \textit{positively invariant} in $N$.  That is, for $x\in L$ and $t>0$, if $\varphi_{[0,t]}(x)\subset N$, then $\varphi_{[0,t]}(x)\subset L$.  
\end{enumerate}

For an index pair $P=(P_1,P_2)$ of an isolated invariant set $S$, we define $\tau_P: P_1\to [0,\infty]$ by
\[
\tau_P(x)=\begin{cases}
\sup\{t\geq 0 \mid \varphi_{[0,t]}(x)\subset P_1 \smallsetminus P_2 \} \mbox{ if } x\in P_1 \smallsetminus  P_2 \\ 0 \mbox{ if } x\in P_2.
\end{cases}
\]
We say that an index pair $P$ is \emph{regular} if $\tau_P$ is continuous.  

For $Z$ a Hausdorff space, $\omega : M \to Z$ a continuous map, and a regular index pair $P=(P_1,P_2)$, define the \emph{parameterized Conley index} $I_\omega(P)$ as $P_1\cup_{\omega|_{P_2}}Z$, namely:
\[
I_\omega(P)=(Z\times 0 )\cup (P_1 \times 1) / \sim
\]
where $(x,1)\sim (\omega(x),0)$ for all $x\in P_2\times 1$.

The space $I_\omega(P)$ is naturally an ex-space, with embedding $s_P: Z \to I_{\omega}(P)$ given by $z\to [z,0]$, and projection $r_P: I_\omega(P)\to Z$ given by $r_P([x,1])=\omega(x)$, $r_P([z,0])=z$.  By construction, $r_P\circ s_P=\operatorname{id}_Z$.  

For $Z=*$, we sometimes write $I^u(P)$ for $I_{\omega}(P)$, to specify the ``unparameterized" Conley index.

\begin{thm}[{\cite[Theorem 2.1]{MRS}}]\label{thm:fiberwise-deforming-conley-1}
	If $P$ and $Q$ are two regular index pairs for an isolated invariant set $S$, then $(I_\omega(P),r_P,s_P)$ and $(I_\omega(Q),r_Q,s_Q)$ have the same equivariant homotopy type over $Z$, and are both fiberwise well-pointed.
\end{thm}
\begin{proof}
	In \cite{MRS}, it is proved that the two indices have the same fiberwise-deforming type; Lemma \ref{lem:fib-hom-good} then implies the statement.  The well-pointedness is \cite[Proposition 6.1]{MRS}
\end{proof}

\begin{dfn}[{\cite{Conley},\cite[Definition 2.6]{Salamon}}]\label{def:connected-simple-system}
	A \textit{connected simple system} is a collection $I_0$ of pointed spaces along with a collection of $I_h$ of homotopy classes of maps among them, so that:
	\begin{enumerate}
		\item\label{itm:css1} For each pair $X,X'\in I_0$, there is a unique class $[f]\in I_h$ from $X\to X'$.
		\item\label{itm:css2} For $f,f'\in I_h$ with $f: X\to X'$ and $f': X'\to X''$, the composite $f'\circ f$ is in $I_h$.
		\item\label{itm:css3} For each $X\in I_0$, the morphism $f: X \to X$ is $[\operatorname{id}]$.  
	\end{enumerate}
\end{dfn}
Of course, the notion of a connected simple system has an obvious generalization in any category with an associated homotopy category.

\begin{thm}[{\cite{Salamon}}]\label{thm:fiberwise-deforming-conley-2}
	Fix notation as in Theorem \ref{thm:fiberwise-deforming-conley-1}.  The \emph{unparameterized Conley indices} $I^u(P)=I_\omega(P)/Z$, ranging over regular index pairs for $S$, form a \emph{connected simple system}.
\end{thm}

We conjecture that in fact the parameterized Conley indices also have this property:
\begin{conj}\label{conj:parameterized-connected-simple-system}
	Fix notation as in Theorem \ref{thm:fiberwise-deforming-conley-1}.  Then the parameterized Conley indices $(I_\omega(P),r_P,s_P)$, running over all regular index pairs for the isolated invariant set $S$, form a connected simple system.
\end{conj}
In Section \ref{sec:well-def}, we encounter the parameterized Conley indices for product flows.  We have:
\begin{thm}[{\cite[Theorem 2.4]{MRS}}]\label{thm:product-flows}
	Let $S,S'$ be isolated invariant sets for $\varphi,\varphi'$.  Then\[
	I_{\omega\times \omega'}(S\times S',\varphi\times \varphi')\simeq I_{\omega}(S,\varphi)\wedge I_{\omega'}(S',\varphi').
	\]
\end{thm}

Moreover, the usual deformation invariance of the Conley index continues for the parameterized Conley index:

\begin{thm}[{\cite[Theorem 2.5]{MRS},\cite[Corollary 6.8]{Salamon}}]\label{thm:deformation-invariance}
	If $N$ is an isolating neighborhood with respect to flows $\varphi^\lambda$ continuously depending on $\lambda\in [0,1]$, with a continuous family of isolated invariant sets $S^\lambda$ inside of $N$, then the fiberwise-deforming homotopy type of $I_{\omega}(S^\lambda,\varphi^\lambda)$ is independent of $\lambda$.

In the case of the unparameterized Conley index, for each $\lambda_1,\lambda_2\in [0,1]$, there is a well-defined, up to homotopy, map of connected simple systems:
\[
F(\lambda_1,\lambda_2): I^u(S^{\lambda_1},\varphi^{\lambda_1})\to I^u(S^{\lambda_2},\varphi^{\lambda_2}). 
\]
Furthermore, for all $\lambda_1,\lambda_2,\lambda_3\in [0,1]$, 
\begin{align*}
F(\lambda_2,\lambda_3)\circ F(\lambda_1,\lambda_2)&\sim F(\lambda_1,\lambda_3)\\
F(\lambda_1,\lambda_1)&\sim \operatorname{id}.
\end{align*}
  \end{thm}

\begin{lem}\label{lem:conley-index-pushforward}
	Fix a flow $\varphi$ on a manifold $X$, along with a map $p: X \to B$, and write $\pi: B\to *$ be the map collapsing $B$ to a point.  Then the pushforward of the parameterized Conley index $I(\varphi)$, namely $\pi_{!}I(\varphi)$, is the ordinary Conley index $I^u(\varphi)$. 
\end{lem}
\begin{proof}
	This is immediate from the definitions.
\end{proof}


We also note the behavior under time reversal:
\begin{thm}[{\cite[Theorem 3.5]{Cornea},\cite[Proposition 3.8]{Manolescu-triangulation}}]\label{thm:spanier-whitehead-duality-for-conley}
	Let $M$ be a stably parallelized $G$-manifold for a compact Lie group $G$.  For $\varphi$ a flow on $M$, the (unparameterized) Conley index of an isolated invariant set $S$ with respect to the time-reversed flow $-\varphi$, denoted $I^u(S,-\varphi)$, is equivariantly Spanier-Whitehead dual to $I^u(S,\varphi)$.
\end{thm}


\section{Spectra}\label{subsec:spectra}

For $G$ a compact Lie group, we define a $G$-\emph{universe} $\mathcal{U}$ to be a countably infinite-dimensional orthogonal representation of $G$.  

\begin{dfn}\label{def:spectra}
      Let $\mathcal{U}$ be a universe with a direct sum decomposition $\mathcal{U}=\oplus_{i=1}^n V_i^\infty$, for finite dimensional $G$-representations $V_i$. 
	A sequential $G$-spectrum $X$  on $\mathcal{U}$ 
	is a collection $X(V)$ of spaces, indexed on the subspaces of $\mathcal{U}$ of the form $V=\oplus_{i=1,\ldots,n}V_i^{k_i}$ for some $k_i\geq 0$, along with transition maps, whenever $W\subset V$,
	\[
	\sigma_{V-W}: \Sigma^{V-W}X(W)\to X(V),
	\]
	where $V-W$ is the orthogonal complement of $W$ in $V$.  For $V=W$, the transition map is required to be the identity, and the maps $\sigma$ are required to be transitive in the usual way.  The space $X(V)$ is sometimes referred to as the $V$-th \emph{level} of the spectrum.
	
	If $\sigma_{V-W}$ is a homotopy equivalence for $V,W$ sufficiently large, we say that $X$ is a $G$-\emph{suspension spectrum}.  
\end{dfn}
We will only work with suspension spectra in this memoir.

A \emph{morphism} of spectra $X\to Y$ will be a collection of morphisms 
\[
\phi_V: X(V) \to Y(V)
\] 
compatible with the transition maps.

We will also consider a generalization of morphisms, as follows:
\begin{dfn}
	A \emph{weak} morphism of spectra $\phi: X \to Y$ is a collection of morphisms
	\[
	\phi_V: X(V)\to Y(V)
	\]
	for $V$ sufficiently large, so that the diagram 
		\[
	\begin{tikzpicture}[baseline={([yshift=-.8ex]current  bounding  box.center)},xscale=8,yscale=1.5]
	\node (a0) at (0,0) {$ 	\Sigma^{W-V}X(V)$};
	\node (a1) at (1,0) {$ 	\Sigma^{W-V}Y(V)$};
	\node (b0) at (0,-1) {$X(W)$};
	\node (b1) at (1,-1) {$Y(W)$};

	\draw[->] (a0) -- (a1) node[pos=0.5,anchor=north] {\scriptsize
		$\Sigma^{W-V} \phi_V$}; \draw[->] (a0) -- (b0) node[pos=0.5,anchor=east] {};
	\draw[->] (b0) -- (b1) node[pos=0.5,anchor=south east]{\scriptsize $\phi_W$}; \draw[->] (a1) -- (b1) node[pos=0.5,anchor=west] {};
	
	\end{tikzpicture}
	\]
	homotopy commutes for $W$ sufficiently large.  Weak morphisms $\phi_0,\phi_1$ are said to be \emph{homotopic} if there exists a weak morphism $\phi_{[0,1]}: X\wedge [0,1]_+\to Y$ restricting to $\phi_{j}$ at $X\wedge \{j\}$ for $j=0,1$.  
\end{dfn}

We will also need the notion of a \emph{connected simple system of spectra}.  Indeed, instead of using the direct generalization for spaces, the Seiberg-Witten Floer spectrum, as currently defined, requires that we work with weak morphisms instead, as follows:

\begin{dfn}\label{def:connected-simple-system-of-spectra}
	A \textit{connected simple system of $G$-spectra} is a collection $I_0$ of $G$-spectra, along with a collection $I_h$ of weak homotopy classes of maps between them, so that the analogs of (\ref{itm:css1})-(\ref{itm:css3}) of Definition \ref{def:connected-simple-system} are satisfied.
\end{dfn}

\begin{rem}\label{rmk:why-sequential}
	In Section \ref{subsec:defn-of-invar}, we could have used non-sequential $G$-spectra, but we have no need for the added generality in the memoir, and it slightly complicates the notation.
\end{rem}

\begin{rem}\label{rmk:improvements}
	If higher naturality is established for the Conley index, then it would be possible to replace \emph{weak} morphisms in the definition of $\SWFn$ and Definition \ref{def:connected-simple-system-of-spectra} could be replaced with ordinary morphisms of spectra.
\end{rem}


\chapter{Afterword: Finite-dimensional Approximation in other settings}

Outside of Seiberg-Witten theory, we expect that the notion of parameterized finite-dimensional approximation may be applicable in some cases in symplectic topology.  The methods of this memoir rely, roughly speaking, on a few special features of the Seiberg-Witten equations, relative to other Floer-type problems:
\begin{enumerate}
	\item\label{itm:sw-bundle} The configuration space is naturally a bundle over a compact, finite dimensional manifold.
	\item\label{itm:bubble} Bubbling phenomena do not occur.
	\item\label{itm:fibers} With respect to the bundle structure in (\ref{itm:sw-bundle}), the Seiberg-Witten equations are ``close to linear" on the fibers.
	\item\label{itm:spectrum} There is a relatively good understanding of the spectrum of the Dirac operator.
\end{enumerate}
Perhaps the item most likely to elicit worry more generally is (\ref{itm:sw-bundle}).  However, we note that it is classical that for any compact subset $K$ of a Hilbert manifold, there is an open sub-Hilbert-manifold $B$ containing $K$ which is diffeomorphic to the total space of a Hilbert bundle over a compact finite-dimensional manifold:
\begin{Lemma}\label{lem:hilbert}
	Let $M$ a separable Hilbert manifold, and $K\subset M$ a compact subset.  Then there exists some open $U\supset K$ diffeomorphic to $V\times H$, where $H$ is a separable Hilbert space, and $V$ is a finite-dimensional smooth manifold.
\end{Lemma}
\begin{proof}
	By compactness, choose a good open cover $\mathcal{C}'$ of $K$, with finite subcover $\mathcal{C}=\{U_i\}_i$, which is once again good, with $U=\cup_i U_i$.  The nerve $N(\mathcal{C})$ is then homotopy-equivalent to $U$.  Moreover, $N(\mathcal{C})$ may be embedded in some finite-dimensional Euclidean space and has a regular neighborhood which is a smooth manifold $V$, with $N(\mathcal{C})\simeq V$.  By \cite{burghelea-kuiper}, \cite{moulis}, separable infinite-dimensional homotopy-equivalent Hilbert manifolds are diffeomorphic.  Then $V\times H$ is diffeomorphic to $U$, as needed.
\end{proof}

  In particular, (\ref{itm:sw-bundle}) holds locally around the moduli space of (gradient flows of the Chern-Simons functional, symplectic action, etc.) in many situations of interest (there is the technical point that a version of Lemma \ref{lem:hilbert} which respected $L^2_k$-norms for multiple values of $k$ would be more appropriate, but we have not attempted it).   Although it is not at all clear how to perform finite-dimensional approximation in the presence of bubbling, nonetheless items (\ref{itm:bubble}) and (\ref{itm:spectrum}) also hold in various geometric situations.  The problem then amounts to establishing appropriate versions of (\ref{itm:fibers}) in specific situations; this appears challenging except when the configuration space is very special.

We finally note that the finite-dimensional approximation process of this memoir can also be applied locally.  In particular, it can be applied in the neighborhood of a broken trajectory.  Here, the base space is some smooth trajectory very close to the broken trajectory, so that there is a neighborhood containing the broken trajectory, and on which (\ref{itm:sw-bundle})-(\ref{itm:spectrum}) hold.  Finite-dimensional approximation then produces a sequence of flows, whose finite-energy integral curves converge to solutions of the Seiberg-Witten equations.  Assuming nondegeneracy, one may be able to assemble these locally-constructed approximating submanifolds into the data of a flow category as in \cite{CJS}.  The hoped-for result of this process would be replacing the need to give a smooth structure to the corners for the moduli spaces of the Seiberg-Witten equations themselves, with the problem of putting a smooth structure on the trajectory spaces of a finite-dimensional approximation.  The main obstruction to this approach is likely the need to establish that the approximating submanifolds constructed this way are suitably independent of the choices involved in their construction, which may be difficult.

\backmatter

\bibliographystyle{alpha}
\bibliography{bib}

\begin{thebibliography}{MOY97}

\bibitem[AB21]{Abouzaid-Blumberg}
Mohammed Abouzaid and Andrew~J Blumberg.
\newblock Arnold conjecture and morava k-theory.
\newblock {\em arXiv preprint arXiv:2103.01507}, 2021.

\bibitem[Ada84]{adams}
J.~F. Adams.
\newblock Prerequisites (on equivariant stable homotopy) for {C}arlsson's
  lecture.
\newblock In {\em Algebraic topology, {A}arhus 1982 ({A}arhus, 1982)}, volume
  1051 of {\em Lecture Notes in Math.}, pages 483--532. Springer, Berlin, 1984.

\bibitem[Ada95]{Adams-book}
J.~F. Adams.
\newblock {\em Stable homotopy and generalised homology}.
\newblock Chicago Lectures in Mathematics. University of Chicago Press,
  Chicago, IL, 1995.
\newblock Reprint of the 1974 original.

\bibitem[APS75]{APS1}
Michael~F Atiyah, Vijay~K Patodi, and Isadore~M Singer.
\newblock Spectral asymmetry and {R}iemannian geometry. {I}.
\newblock In {\em Mathematical Proceedings of the Cambridge Philosophical
  Society}, volume~77, pages 43--69. Cambridge University Press, 1975.

\bibitem[AS69]{AS_skew-adjoint}
M.~F Atiyah and I.~M Singer.
\newblock Index theory for skew-adjoint fredholm operators.
\newblock {\em Publ. Math. Inst. Hautes {\'E}tudes Sci.}, 37:5--26, 1969.

\bibitem[Ati85]{Atiyah}
Michael Atiyah.
\newblock Eigenvalues of the {D}irac operator.
\newblock In {\em Arbeitstagung Bonn 1984}, pages 251--260. Springer, 1985.

\bibitem[Bar92]{bartsch}
Thomas Bartsch.
\newblock The {C}onley index over a space.
\newblock {\em Math. Z.}, 209(2):167--177, 1992.

\bibitem[BF04]{bauer-furuta}
Stefan Bauer and Mikio Furuta.
\newblock A stable cohomotopy refinement of {S}eiberg-{W}itten invariants. {I}.
\newblock {\em Invent. Math.}, 155(1):1--19, 2004.

\bibitem[BG18]{Behrens-Golla}
Stefan Behrens and Marco Golla.
\newblock Heegaard {F}loer correction terms, with a twist.
\newblock {\em Quantum Topol.}, 9(1):1--37, 2018.

\bibitem[BK69]{burghelea-kuiper}
Dan Burghelea and Nicolaas~H. Kuiper.
\newblock Hilbert manifolds.
\newblock {\em Ann. of Math. (2)}, 90:379--417, 1969.

\bibitem[CE71]{conley-easton}
C~Conley and R~Easton.
\newblock Isolated invariant sets and isolating blocks.
\newblock {\em Trans. Amer. Math. Soc}, 158:35--61, 1971.

\bibitem[CGK02]{Cole-Greenlees-Kriz}
Michael Cole, J.~P.~C. Greenlees, and I.~Kriz.
\newblock The universality of equivariant complex bordism.
\newblock {\em Math. Z.}, 239(3):455--475, 2002.

\bibitem[CJS95]{CJS}
R.~L. Cohen, J.~D.~S. Jones, and G.~B. Segal.
\newblock Floer's infinite-dimensional {M}orse theory and homotopy theory.
\newblock In {\em The {F}loer memorial volume}, volume 133 of {\em Progr.
  Math.}, pages 297--325. Birkh\"auser, Basel, 1995.

\bibitem[Con78]{Conley}
Charles Conley.
\newblock {\em Isolated invariant sets and the {M}orse index}, volume~38 of
  {\em CBMS Regional Conference Series in Mathematics}.
\newblock American Mathematical Society, Providence, R.I., 1978.

\bibitem[Cor00]{Cornea}
Octavian Cornea.
\newblock Homotopical dynamics: suspension and duality.
\newblock {\em Ergodic Theory Dynam. Systems}, 20(2):379--391, 2000.

\bibitem[CW16]{CW}
Steven~R Costenoble and Stefan Waner.
\newblock {\em Equivariant ordinary homology and cohomology}.
\newblock Springer, 2016.

\bibitem[Dei97]{Deimling}
Klaus Deimling.
\newblock {\em Ordinary Differential Equations in Banach spaces}.
\newblock Springer-Verlag, 1997.

\bibitem[DSS23]{dai-sasahira-stoffregen}
Irving Dai, Hirofumi Sasahira, and Matthew Stoffregen.
\newblock Lattice homology and {S}eiberg-{W}itten-{F}loer spectra.
\newblock {\em arXiv:2309.01253}, 2023.

\bibitem[Dup69]{dupont}
Johan.~L Dupont.
\newblock Symplectic bundles and {$KR$}-theory.
\newblock {\em Math. Scand.}, 24:27--30, 1969.

\bibitem[Fr{\o}10]{froyshov}
Kim Fr{\o}yshov.
\newblock Monopole floer homology for rational homology 3-spheres.
\newblock {\em Duke Math. J.}, 155:519--576, 2010.

\bibitem[Fur01]{furuta}
M.~Furuta.
\newblock Monopole equation and the {$\frac{11}8$}-conjecture.
\newblock {\em Math. Res. Lett.}, 8(3):279--291, 2001.

\bibitem[Kat13]{Kato}
Tosio Kato.
\newblock {\em Perturbation theory for linear operators}, volume 132.
\newblock Springer Science \& Business Media, 2013.

\bibitem[Kha15]{khandhawit}
Tirasan Khandhawit.
\newblock A new gauge slice for the relative bauer--furuta invariants.
\newblock {\em Geom. Topol.}, 19:1631--1655, 2015.

\bibitem[KLS18]{KLS1}
Tirasan Khandhawit, Jianfeng Lin, and Hirofumi Sasahira.
\newblock Unfolded {S}eiberg--{W}itten {F}loer spectra, {I}: Definition and
  invariance.
\newblock {\em Geometry \& Topology}, 22(4):2027--2114, 2018.

\bibitem[KLS23]{KLS2}
Tirasan Khandhawit, Jianfeng Lin, and Hirofumi Sasahira.
\newblock Unfolded {S}eiberg-{W}itten {F}loer spectra, {II}: Relative
  invariants and the gluing theorem.
\newblock {\em J. Differential Geom}, 124:231--316, 2023.

\bibitem[KM02]{kronheimer-manolescu}
Peter~B Kronheimer and Ciprian Manolescu.
\newblock Periodic {F}loer pro-spectra from the {S}eiberg-{W}itten equations.
\newblock {\em arXiv preprint math/0203243}, 2002.

\bibitem[KM07]{KM}
Peter~B Kronheimer and Tomasz Mrowka.
\newblock {\em Monopoles and three-manifolds}, volume~10.
\newblock Cambridge University Press Cambridge, 2007.

\bibitem[KM10]{Kronheimer-Mrowka-excision}
Peter Kronheimer and Tomasz Mrowka.
\newblock Knots, sutures, and excision.
\newblock {\em J. Differential Geom.}, 84(2):301--364, 2010.

\bibitem[Kra18]{Kragh}
Thomas Kragh.
\newblock The {V}iterbo transfer as a map of spectra.
\newblock {\em J. Symplectic Geom.}, 16(1):85--226, 2018.

\bibitem[KT20]{konno-taniguchi-rel-boundary}
Hokuto Konno and Masaki Taniguchi.
\newblock The groups of diffeomorphisms and homeomorphisms of 4-manifolds with
  boundary.
\newblock {\em arXiv preprint arXiv:2010.00340}, 2020.

\bibitem[Lin18a]{lin-pin2}
Francesco Lin.
\newblock A {M}orse-{B}ott approach to monopole {F}loer homology and the
  triangulation conjecture.
\newblock {\em Mem. Amer. Math. Soc.}, 255(1221):v+162, 2018.

\bibitem[Lin18b]{lin_rokhlin}
Francesco Lin.
\newblock Pin(2)-monopole {F}loer homology and the {R}okhlin invariant.
\newblock {\em Compositio Math.}, 154:2681--2700, 2018.

\bibitem[LM18]{LidmanManolescu}
Tye Lidman and Ciprian Manolescu.
\newblock The equivalence of two {S}eiberg-{W}itten {F}loer homologies.
\newblock {\em Ast\'{e}risque}, (399):vii+220, 2018.

\bibitem[LR14]{Levine-Ruberman}
Adam~Simon Levine and Daniel Ruberman.
\newblock Generalized {H}eegaard {F}loer correction terms.
\newblock In {\em Proceedings of the {G}\"{o}kova {G}eometry-{T}opology
  {C}onference 2013}, pages 76--96. G\"{o}kova Geometry/Topology Conference
  (GGT), G\"{o}kova, 2014.

\bibitem[Man03]{Manolescu-b1=0}
Ciprian Manolescu.
\newblock Seiberg-{W}itten-{F}loer stable homotopy type of three-manifolds with
  {$b_1=0$}.
\newblock {\em Geom. Topol.}, 7:889--932, 2003.

\bibitem[Man07]{manolescu_gluing}
Ciprian Manolescu.
\newblock A gluing theorem for the relative {B}auer-{F}uruta invariants.
\newblock {\em J. Differential Geom}, 76:117--153, 2007.

\bibitem[Man14]{Manolescu_Intersection_form}
Ciprian Manolescu.
\newblock On the intersection forms of spin four-manifolds with boundary.
\newblock {\em Math. Ann.}, 359(3-4):695--728, 2014.

\bibitem[Man16]{Manolescu-triangulation}
Ciprian Manolescu.
\newblock Pin(2)-equivariant {S}eiberg-{W}itten {F}loer homology and the
  triangulation conjecture.
\newblock {\em J. Amer. Math. Soc.}, 29(1):147--176, 2016.

\bibitem[Mou68]{moulis}
Nicole Moulis.
\newblock Sur les vari\'{e}t\'{e}s {H}ilbertiennes et les fonctions non
  d\'{e}g\'{e}n\'{e}r\'{e}es.
\newblock {\em Nederl. Akad. Wetensch. Proc. Ser. A 71 = Indag. Math.},
  30:497--511, 1968.

\bibitem[MOY97]{MOY}
Tomasz Mrowka, Peter Ozsv{\'a}th, and Baozhen Yu.
\newblock Seiberg-{W}itten monopoles on {S}eifert fibered spaces.
\newblock {\em Comm. Anal. Geom.}, 5(4):685--791, 1997.

\bibitem[MP97]{MP}
Richard~B Melrose and Paolo Piazza.
\newblock Families of {D}irac operators, boundaries and the b-calculus.
\newblock {\em J. Differential Geom}, 46(1):99--180, 1997.

\bibitem[MRS00]{MRS}
Marian Mrozek, James Reineck, and Roman Srzednicki.
\newblock The {C}onley index over a base.
\newblock {\em Transactions of the American Mathematical Society},
  352(9):4171--4194, 2000.

\bibitem[MS06]{MS}
J.~P. May and J.~Sigurdsson.
\newblock {\em Parametrized homotopy theory}, volume 132 of {\em Mathematical
  Surveys and Monographs}.
\newblock American Mathematical Society, Providence, RI, 2006.

\bibitem[Nic98]{Nicolaescu_Eta_inv}
Liviu~I. Nicolaescu.
\newblock Adiabatic limits of the {S}eiberg-{W}itten equations on {S}eifert
  manifolds.
\newblock {\em Comm. Anal. Geom.}, 6(2):331--392, 1998.

\bibitem[OS03]{Ozsvath-Szabo-absolute-gradings}
Peter Ozsv\'{a}th and Zolt\'{a}n Szab\'{o}.
\newblock Absolutely graded {F}loer homologies and intersection forms for
  four-manifolds with boundary.
\newblock {\em Adv. Math.}, 173(2):179--261, 2003.

\bibitem[RS80]{reed-simon}
Michael Reed and Barry Simon.
\newblock {\em Methods of modern mathematical physics. {I}}.
\newblock Academic Press, Inc. [Harcourt Brace Jovanovich, Publishers], New
  York, second edition, 1980.
\newblock Functional analysis.

\bibitem[Sal85]{Salamon}
Dietmar Salamon.
\newblock Connected simple systems and the {C}onley index of isolated invariant
  sets.
\newblock {\em Trans. Amer. Math. Soc.}, 291(1):1--41, 1985.

\bibitem[SS]{sasahira-stoffregen_triangle}
Hirofumi Sasahira and Matthew Stoffregen.
\newblock Exact triangles in {S}eiberg-{W}itten {F}loer spectra.
\newblock In preparation.

\bibitem[Sto17]{stoffregen-connectedsum}
Matthew Stoffregen.
\newblock Manolescu invariants of connected sums.
\newblock {\em Proc. London Math. Soc}, 115:1072--1117, 2017.

\bibitem[tD70]{tom-dieck-bordism}
Tammo tom Dieck.
\newblock Bordism of {$G$}-manifolds and integrality theorems.
\newblock {\em Topology}, 9:345--358, 1970.

\bibitem[Wit94]{witten}
Edward Witten.
\newblock Monopoles and four-manifolds.
\newblock {\em Math. Res. Lett.}, 1(6):769--796, 1994.

\end{thebibliography}









\end{document}